\documentclass[11pt, a4paper, openany, DIV = 14]{scrbook}

\usepackage[american]{babel}
\usepackage[utf8]{inputenc}

\usepackage{thesis_style}

\makenomenclature

\newcommand{\N}{\mathbb{N}}
\newcommand{\Z}{\mathbb{Z}}
\newcommand{\Q}{\mathbb{Q}}
\newcommand{\R}{\mathbb{R}}
\newcommand{\C}{\mathbb{C}}
\renewcommand{\S}{\mathbb{S}}
\newcommand{\indexone}{j}
\newcommand{\indextwo}{k}
\newcommand{\indexthree}{l}
\newcommand{\indexfour}{i}
\newcommand{\Ind}[1]{\mathbbm{1}_{\text{#1}}}
\newcommand{\eps}{\varepsilon}
\newcommand{\Lp}{L^p}
\newcommand{\LpCn}{L^p(\T^d, \, \C^N)}
\newcommand{\LtwoTCN}{L^2(\T^d, \, \C^N)}

\newcommand{\ltwoZCN}{l^2(\Z^d, \, \C^N)}
\newcommand{\T}{\mathbb{T}}
\newcommand{\half}{\frac{1}{2}}
\newcommand{\intT}{\int_{\T^d}}
\newcommand*\colvec[1]{\begin{pmatrix}#1\end{pmatrix}}
\newcommand*\LpCnnorm[2]{\| #2 \|_{L^#1(\T^d, \, \C^N)}}
\newcommand*\pnorm[2]{\| #2 \|_{L^#1(\T^d)}}
\newcommand*\Wkp[2]{W^{#1,#2}(\T^d)}
\newcommand*\WkpCn[2]{W^{#1,#2}(\T^d, \, \C^N)}
\newcommand*\Lpnorm[2]{\|#2\|_{L^#1(\T^d)}}
\newcommand\restr[2]{{
		\left.\kern-\nulldelimiterspace 
		#1 
		\vphantom{\big|} 
		\right|_{#2} 
}}

\def\Xint#1{\mathchoice
	{\XXint\displaystyle\textstyle{#1}}%
	{\XXint\textstyle\scriptstyle{#1}}%
	{\XXint\scriptstyle\scriptscriptstyle{#1}}%
	{\XXint\scriptscriptstyle\scriptscriptstyle{#1}}%
	\!\int}
\def\XXint#1#2#3{{\setbox0=\hbox{$#1{#2#3}{\int}$ }
		\vcenter{\hbox{$#2#3$ }}\kern-.6\wd0}}

\def\dashint{\Xint-}

\date{August 16, 2021}

\begin{document}
	\renewcommand{\bibname}{Bibliography}
	
	\tcolorboxenvironment{theorem}{breakable, enhanced jigsaw, colback={shadecolor},
		sharp corners, frame hidden}
	
	\KOMAoptions{twoside=false}
	
	\maketitle
	
	
	\let\tmp\oddsidemargin
	\let\oddsidemargin\evensidemargin
	\let\evensidemargin\tmp
	\reversemarginpar
	
	\thispagestyle{empty}
	\vspace*{2cm}
	\noindent
	At first, I would like to express my gratitude to my advisors Dr. Christina Lienstromberg and Prof. Dr. Juan J. L. Velázquez. They always took time to provide guidance and feedback, which led to many fruitful discussions.\newline
	\newline
	I would also like to thank my parents and my sister. The intense work on my thesis during the last months would not have been possible without their continuous support. 
	\newpage
	
	\tableofcontents
	\newpage
	
	\listoffigures
	

	\chapter{Introduction}\label{chapter_Introduction}
	
	Transport-reaction equations arise in the mathematical modeling of the movement of populations. In contrast to parabolic equations, hyperbolic models realistically account for a finite propagation speed of organisms. Therefore, scientists increasingly focus on hyperbolic and kinetic models \cite[p. 36]{eftimie2012hyperbolic}. \newline
	A particularly interesting question is whether mathematical models are able to describe the variety of patterns observed in nature: rippling waves in myxobacteria swarms, milling schools of fish, swarms of insects or pedestrian traffic jams, to name only a few of them \cite{eftimie2012hyperbolic}. Diffusive-generated patterns have originally been studied by Turing in his pioneer work \cite{turing1990chemical} and his idea has been applied many times, for instance in the context of morphogenesis \cite{gierer1972theory}. Patterns generated by transport on the other hand are less well understood. \newline
	
	The aim of this master's thesis is to study the linear transport-reaction system
	\begin{equation}\label{intro_transport_reaction_eq}
		\partial_t \colvec{u_1 \\ \vdots \\ u_N}(t,\mathbf{x}) + \colvec{\mathbf{v_1} \cdot \nabla u_1 \\ \vdots \\ \mathbf{v_N} \cdot \nabla u_N}(t,\mathbf{x}) = B \colvec{u_1 \\ \vdots \\ u_N}(t,\mathbf{x})
	\end{equation}
	on the $d$-dimensional hypercube $[0,1]^d$ with periodic boundary conditions, i.e. on the $d$-dimensional torus $\T^d.$ Concerning notation and interpretation, the parameter $N$ is the number of population subgroups and $\mathbf{v_1},\cdots, \mathbf{v_N} \in \R^d$ are the transport directions of each component. The matrix $B\in \R^{N\times N}$ describes linear interactions of the subgroups.\newline
	Typically, \eqref{intro_transport_reaction_eq} appears as a linearization of the non-linear partial differential equation
	\begin{equation}\label{intro_transport_reaction_eq_nonlinear}
		\partial_t \colvec{u_1 \\ \vdots \\ u_N}(t,\mathbf{x}) + \colvec{\mathbf{v_1} \cdot \nabla u_1 \\ \vdots \\ \mathbf{v_N} \cdot \nabla u_N}(t,\mathbf{x}) = F(u(t,\mathbf{x})),
	\end{equation}
	where $u=(u_1,\cdots,u_N)$ and $F\colon \R^N \to \R^N$ is a sufficiently smooth function. There is a large amount of research which deals with the question of stability of a given equilibrium point of \eqref{intro_transport_reaction_eq_nonlinear}; see for instance \cite{hillen1996turing,lutscher2002emerging,Primi2013}. Many times, however, only the one dimensional case $d=1$ is studied. \newline
	The transport-reaction equation \eqref{intro_transport_reaction_eq} can also be interpreted as a velocity discretization of
	\begin{equation}\label{intro_integral_interpret}
		\partial_t u(t,\mathbf{x}, \mathbf{v}) + \mathbf{v} \cdot \nabla_x u(t,\mathbf{x},\mathbf{v}) = - \sigma(\mathbf{v}) + \int_V \kappa(\mathbf{v}, \mathbf{v}') u(t,\mathbf{x},\mathbf{v}) \, d\mathbf{v}',
	\end{equation}
	where the velocity $\mathbf{v}$ can take values in a set $V\subseteq \R^d.$ The function $\sigma\colon V \to \R_{\geq 0}$ describes collisions causing absorption \cite{Batkai2017} and corresponds to the diagonal entries of $B$. The so called scattering kernel $\kappa\colon V \times V \to \R_{\geq 0}$ governs the transition of particles with incoming velocity $\mathbf{v}'$ to particles with outgoing velocity $\mathbf{v}$ \cite{Batkai2017}. It corresponds to the off-diagonal entries of $B.$ The linear Boltzmann-type equation \eqref{intro_integral_interpret} is, for example, connected to velocity-jump processes and neutron transport, see \cite{hillen2002hyperbolic,eftimie2012hyperbolic,Batkai2017}. \newline
	
	Throughout the whole thesis, we try to be as general as possible in the sense that we consider the transport-reaction equation \eqref{intro_transport_reaction_eq} on the spaces $\LpCn$ for all $1\leq p <\infty$ and all dimensions $d\in\N$ simultaneously. Whenever we need to restrict ourselves to $p=2$ or $d=1,$ we try to at least indicate how future research can approach the other cases.\newline
	The thesis is structured as follows: \Cref{chapter_transport_semigroup} and \Cref{chapter_transport_reaction_SG} deal with well-posedness of \eqref{intro_transport_reaction_eq} and basic properties such as positivity and mass conservation. We use a semigroup approach and show that
	\begin{equation*}
		-A_p + B_p = -\colvec{\mathbf{v_1} \cdot \nabla u_1 \\ \vdots \\ \mathbf{v_N} \cdot \nabla u_N} + B \colvec{u_1 \\ \vdots \\ u_N}
	\end{equation*}
	with an appropriately chosen domain generates a semigroup on $\LpCn,$ which we call the transport-reaction semigroup $(R_p(t))_{t\geq0}.$ We continue our study with a detailed spectral analysis of both the generator $-A_p+B_p$ and the transport-reaction semigroup $(R_p(t))_{t\geq0}$. \Cref{chapter_side_length_neumann} deals with the incorporation of different side lengths $L>0$ of the torus $\T^d$ and Neumann boundary conditions for symmetric models in dimension $d=1$. In \Cref{chapter_application}, we apply our results to, for example, a Goldstein-Kac-type equation \cite{kac1974stochastic} and a sophisticated model from \cite{hillen1996turing}. Finally, \Cref{chapter_pattern_formation} is devoted to pattern formation and transport-driven instabilities. \newline
	We tried our best to write every chapter as self-contained as possible and to motivate all sections of the thesis with a detailed introduction at the beginning of each chapter. \newline	
	
	One main contribution of the thesis is the weak spectral mapping theorem
	\begin{equation*}
		\sigma(R_2(t)) = \overline{e^{t\sigma(-A_2+B_2)}} \quad \text{ for all } t\in\R
	\end{equation*}
	and a description of $\sigma(-A_2+B_2)$ in terms of roots of polynomials of order $N$ for all dimensions $d\in\N$. This result allows a very detailed study of the long-term behavior of solution to \eqref{intro_transport_reaction_eq}. To the best of our knowledge, all current results deal with the one-dimensional case $d=1.$ The proofs are based on Fourier transform and so called matrix multiplication operators from \cite{holderrieth1991matrix}.\newline
	The other main contribution is the chapter on pattern formation and transport-driven instabilities. We introduce the notion of hyperbolic instabilities, which allows us to prove a dichotomy: every transport-driven instability in $d=1$ and with arbitrary number of components $N$ is either a Turing pattern or a hyperbolic instability. Intuitively, hyperbolic instabilities correspond to increasingly oscillating and chaotic behavior. We do not hold expertise in biological modeling but we believe that our finding reflects the nature of movement of populations: either there is a regular pattern or one observes bustle. As a byproduct of our analysis, we find an algebraic condition in terms of the transport speeds $v_1,\cdots,v_N$ and $B$ which ensures the existence of Turing patterns. The mathematical details of this chapter are based on perturbation theory from \cite{kato2013perturbation}. \newline

	\section{Function Spaces and Notation}
		
	In order to rigorously define standard $\Lp$ spaces for $1\leq p < \infty$ on a 
	$d$-dimensional torus $\T^d$, we follow the approach from \cite[p. 238f.]{folland1999real} and consider the quotient map
	\begin{equation*}
		q\colon [0,1]^d \rightarrow \T^d,
	\end{equation*}
	which identifies opposite sides of the hypercube. This mapping induces the
	measurable space $(\T^d, \, q_\sharp\mathcal{B}([0,1]^d), \,
	q_\sharp\mathcal{L}^d)$ where $\mathcal{B}([0,1]^d)$ is the Borel
	$\sigma$-algebra on the hypercube and $\mathcal{L}^d$ is the $d$-dimensional
	Lebesgue measure. Moreover, the canonical norm on $\Lp(\T^d, \C^N)$ is given by 
	\begin{equation}\label{definition_Lp_norm}
		\LpCnnorm{p}{u} \coloneqq \left( \sum_{\indexone=1}^{N}
		\pnorm{p}{u_\indexone}^2\right)^\half
	\end{equation}
	for $u=(u_1,\cdots,u_N) \in \LpCn.$ To give a better visualization, it should be mentioned that functions on $\T^d$ may be considered as periodic functions on $\R^d$ or as functions on $[0,1)^d.$
	The more appropriate point of view will follow by the context when it matters
	\cite[p. 238f.]{folland1999real}. Generally, we are interested in functions
	$u\colon \R_{\geq0} \times \T^d \to \C^N,$ where the value of $u$ at $(t,\mathbf{x})$  is for instance the concentration of population subgroups at time $t$. The reason we are
	working with complex-valued functions is that we approach the transport-reaction equation \eqref{intro_transport_reaction_eq} from a
	semigroup and spectral theory perspective. This may be counterintuitive at first
	glance, but ultimately the only difference is that also complex initial functions are
	allowed and that no inconsistencies in the spectral analysis occur. Given
	\textit{real} initial values $u_0\colon \T^d \to \R^N,$ every solution of the
	equations we study will also be real valued. \newline
	
	It should also be mentioned that we almost always mean strongly continuous (semi-)group, when we use the term (semi-)group. A relatively detailed introduction to semigroup theory is given in Appendix \ref{appendix_semigroup_theory} to Appendix \ref{appendix_Long_time_behavior_SG}. \newline
	Last but not least, the thesis contains a list of symbols, which can be found at the very end.

	\chapter{The Transport Semigroup}\label{chapter_transport_semigroup}
	
	The first questions regarding the transport-reaction model
	\eqref{intro_transport_reaction_eq} are to ask for existence and uniqueness of
	solutions. This chapter is devoted to answer these questions in the non-reacting case $B=0$ by using a semigroup approach. In the process, we study the domain of the generator of the
	transport semigroup and identify $C_c^\infty(\T^d, \C^N)$ as a core. The linear reactions in \eqref{intro_transport_reaction_eq} are incorporated in the next chapter.\newline	
	
	Throughout the whole chapter, we assume that the dimension $d\in\N$, the number
	of components $N\in\N$, the transport directions
	$\mathbf{v_1},\cdots,\mathbf{v_N}\in \R^d$ and $1\leq p < \infty$ are
	arbitrarily chosen fixed parameters, if not stated otherwise. We work on the
	Banach space $\LpCn$ with norm $\LpCnnorm{p}{\cdot}$, which is defined in
	\eqref{definition_Lp_norm}. \newline
	
	We start by introducing the transport semigroup on $\LpCn$ explicitly via
	\begin{equation}\label{transport_semigroup_definition}
		(T(t)u)(\mathbf{x}) \coloneqq \colvec{u_1(\mathbf{x} - t \mathbf{v_1}) \\
		\vdots \\ u_N(\mathbf{x} - t\mathbf{v_N})}
	\end{equation}
	for $t\geq0$, $N\in\N$ directions $\mathbf{v_1}, \cdots, \mathbf{v_N} \in \R^d$
	and $u = (u_1,\cdots,u_N) \in \LpCn.$ Strictly speaking, we should include the indices $d$ and $p$
	into the notation as the semigroup depends on the Banach space, but we will
	often drop these two indices to improve readability.
	
	\begin{lemma}
		$(T(t))_{t\geq0}$ defines a contraction semigroup on $\LpCn.$
	\end{lemma}
	\begin{proof}
		The lemma with domain $\R^d$ instead of $\T^d$ is a standard example for a
		contraction semigroup and strong continuity follows by the density of
		$\C_c^\infty(\R^d, \, \C^N);$ see for example \cite[Proposition 8.5, p. 238]{folland1999real}. On a torus, the proof works similarly. A density argument
		and Fubini can be used to show
		\begin{equation*}
			\intT |u_\indexone(\mathbf{x})|^p \, d\mathbf{x} = \intT |u_\indexone(\mathbf{x}-t\mathbf{v_\indexone})|^p \, d\mathbf{x}
		\end{equation*}
		for all $\indexone=1,\cdots,N$ and all $t\geq0.$ Moreover, the semigroup
		property is a consequence of 
		\begin{equation*}
			((a \mod 1) + b) \mod 1 = (a+b) \mod 1 \qquad \text{for all } a, \,b\in\R.
		\end{equation*}
	\end{proof}

	In the following, we denote the generator of the transport semigroup $(T(t))_{t\geq0}$ by $(-A, \, D(-A)).$
	Similarly to the semigroup, the generator depends on $d$ and $p$. This should be
	kept in mind for later chapters on the spectral analysis of the generator. In
	this chapter, however, the dependence does not play a significant role, which is
	why we drop the index. By definition, it holds
	\begin{equation}\label{generator_by_definition}
		D(-A) = \left\{ u \in \LpCn \, \colon \, \lim_{t \searrow 0}
		\frac{u_\indexone(\cdot - t\mathbf{v_\indexone}) - u_\indexone(\cdot)}{t} \text{
		exists in } \Lp(\T^d) \text{ for all } \indexone=1,\cdots,N \right\}.
	\end{equation}
	For continuously differentiable functions $u$, the limits in $L^p(\T^d)$ are simply given by 
	\begin{equation*}
		\lim_{t \searrow 0}	\frac{u_\indexone(\cdot - t\mathbf{v_\indexone}) - u_\indexone(\cdot)}{t} = - \mathbf{v}_\indexone \cdot \nabla u_\indexone
	\end{equation*}
	for all $\indexone=1,\cdots,N.$ We start with a functional analytic
	approach to extend this result to weakly differentiable functions $u\in
	W^{1,p}(\T^d, \, \C^N)$. Studying convergence of difference quotients with
	respect to the $L^p$-norm is very common and can for example be found in
	\cite[Part 2; 5.8.2, p. 293f.]{evans10}. Nevertheless, we decided to include this
	approach because in the literature, the arguments are typically only given for the
	standard unit vectors as directions. In our case, we have arbitrary directions
	$\mathbf{v_1},\cdots,\mathbf{v_N} \in \R^d.$ This functional analytic
	perspective is somewhat technical in the case $p=1$ since $L^1$ is not
	reflexive. Therefore, these proofs also motivate a pure semigroup approach: on
	$\T^d$, abstract semigroup theory can be applied to obtain a shorter and easier
	proof for the $L^p$-convergence of the difference quotients.\newline
	
	Consider a function $u \in \Lp(\T^d)$ and define the difference quotients
	\begin{equation}\label{definition_difference_quotient}
		D_\indextwo^h u(\cdot) \coloneqq \frac{u(\cdot + h\mathbf{e_\indextwo} ) -
		u(\cdot)}{h}
	\end{equation}
	for $h \in \R $ and $\indextwo = 1,\cdots, d.$ In this definition,
	$\mathbf{e_\indextwo}$ is the $k$-th standard unit vector of $\R^d.$ The next lemma
	deals with the limit behavior as $h\to 0$ (note that $h$ is not assumed to have
	a particular sign).
	
	\begin{lemma}\label{lemma_convergence_unit_difference_quotients}
		Let $u \in \Lp(\T^d)$ and let $k=1,\cdots,d.$
		\begin{enumerate}
			\item
			Let $D_\indextwo^h u \rightharpoonup v$ in $\Lp(\T^d)$. Then the
			$\indextwo$-th weak derivative exists in $\Lp(\T^d)$ and $\partial_\indextwo u =
			v.$
			\item
			Conversely, if $u \in \Wkp{1}{p}$, it follows that $D_\indextwo^h u \to
			\partial_\indextwo u$ in $\Lp(\T^d)$ for $p>1$ and $D_\indextwo^h u
			\rightharpoonup \partial_\indextwo u$ for $p=1$.
		\end{enumerate}
		All convergences are with respect to $h \to 0.$
	\end{lemma}
	\begin{proof}
		The details are given in Appendix \ref{proof_lemma_diff_quot}.
	\end{proof}
	
	Although the directions $\mathbf{v_1}, \cdots, \mathbf{v_N} \in \R^d$ are
	arbitrary in the definition of the generator of the transport semigroup
	\eqref{generator_by_definition}, a corresponding result can be shown in the exact same manner. That is, if we define the directional difference quotient
	
	\begin{equation}\label{definition_directional_difference_quotient}
		D_{\mathbf{v}}^h u(\cdot) \coloneqq \frac{u(\cdot + h\mathbf{v} ) -
		u(\cdot)}{h}.
	\end{equation}
	for $u\in \Lp(\T^d), \, \mathbf{v} \in \R^d$ and $h \in\R,$ the following lemma
	holds true.
	
	\begin{lemma}\label{lemma_convergence_diff_quotients_p>1}
		Let $1 \leq p < \infty, \, u \in \Wkp{1}{p}$ and let $\mathbf{v} \in \R^d.$
		Then $D_{\mathbf{v}}^h u \to \mathbf{v} \cdot \nabla u$ in $\Lp(\T^d)$ for $p>1$
		and $D_{\mathbf{v}}^h u \rightharpoonup \mathbf{v} \cdot \nabla u$ in $\Lp(\T^d)$ for $p=1$ as
		$h\to0.$
	\end{lemma}
	\begin{proof}
		Use the inequality
		\begin{equation*}
			\| D_{\mathbf{v}}^h u \|_{L^p(\T^d)} \leq | \mathbf{v} | \, \Lpnorm{p}{|\nabla u|}
		\end{equation*}
		for $u \in \Wkp{1}{p}$ and argue as in the proof of
		\Cref{lemma_convergence_unit_difference_quotients}. Also the case $p=1$ works
		analogously.
	\end{proof}
	
	We define the operator $-\mathcal{A}$ on $\LpCn $ by
	\begin{equation}\label{definition_generatorA}
		-\mathcal{A} u = -\colvec{\mathbf{v_1} \cdot \nabla u_1 \\ \vdots \\
		\mathbf{v_N} \cdot \nabla u_N}.
	\end{equation}
	Here, we did not specify the domain intentionally. It will vary and we will
	specify it later on for different scenarios. With the functional analytic arguments
	from above, we obtain the following result concerning the domain of the
	generator of $(T(t))_{t\geq0}.$
	
	\begin{lemma}
		For $1<p<\infty,$ the generator $(-A, \, D(-A))$ is an extension of
		$(-\mathcal{A}, \, \WkpCn{1}{p})$.	In the case $p=1,$ the generator is an
		extension of $(-\mathcal{A}, \, C^1(\T^d, \, \C^N)).$
	\end{lemma}
	\begin{proof}
		The statement for $p=1$ follows immediately from the chain rule and dominated
		convergence. For $p>1,$ we can apply \Cref{lemma_convergence_diff_quotients_p>1} componentwise to obtain
		\begin{equation*}
			\frac{u_\indexone(\cdot - t\mathbf{v_\indexone}) - u_\indexone(\cdot)}{t} = -
			D_{\mathbf{v_\indexone}}^{-t} u_\indexone \longrightarrow -\mathbf{v_\indexone}
			\cdot \nabla u_j
		\end{equation*}
		in $L^p(\T^d)$ for $\indexone=1,\cdots,N.$
	\end{proof}
	
	Remarkably enough, the previous functional analytic arguments are insufficient
	to show strong convergence in the case $p=1$. An alternative and shorter proof
	that uses semigroup theory works for all cases $1 \leq p < \infty.$
	
	\begin{lemma}\label{lemma_A_is_an_extension}
		Let $1 \leq p < \infty$. Then $(-A, \, D(-A))$ extends $(-\mathcal{A}, \,
		\WkpCn{1}{p}).$
	\end{lemma}
	\begin{remark}
		This implies in particular, that strong convergence in
		\Cref{lemma_convergence_diff_quotients_p>1} also holds for $p=1.$
	\end{remark}
	\begin{proof}
		Let $1 \leq p < \infty$ and let $u \in \WkpCn{1}{p}.$ Moreover, take an
		approximation $(u^{(n)})_{n\in\N}\subset C_c^\infty(\T^d, \, \C^N)$ such that
		$u^{(n)}\to u $ in $\WkpCn{1}{p}.$ By the convergence of the derivatives, it
		follows that $-Au^{(n)} = -\mathcal{A}u^{(n)} \to -\mathcal{A}u$ in $L^p(\T^d,
		\, \C^N).$ Since generators of $C_0$-semigroups are closed by Hille-Yosida, see
		\Cref{hille-yosida}, we obtain $u \in D(-A)$ and $-Au = -\mathcal{A} u.$
	\end{proof}
	
	At this point, it is unclear, whether the generator of the transport semigroup
	is uniquely determined by $(-\mathcal{A}, \, \WkpCn{1}{p}).$ However, the
	semigroup is extremely well behaved in the sense that it is easy to identify
	$C_c^\infty(\T^d, \, \C^N)$ as a core, i.e. $C_c^\infty(\T^d, \, \C^N)$ is dense
	in $D(-A)$ w.r.t. the graph norm $\|\cdot\|_{-A}$.
	
	\begin{theorem}\label{theorem_Ccinfty_core_A}
		$C_c^\infty(\T^d, \, \C^N)$ is a core for $(-A, \, D(-A)).$ In particular,
		$(T(t))_{t\geq0}$ is the only strongly continuous semigroup on $\Lp(\T^d, \C^N)$
		with a generator that extends $(-\mathcal{A}, \, C_c^\infty(\T^d, \, \C^N)).$
	\end{theorem}
	\begin{proof}
		Clearly, we have $T(t)(C_c^\infty(\T^d, \, \C^N)) \subset C_c^\infty(\T^d,
		\C^N).$ This implies that $C_c^\infty(\T^d, \, \C^N)$ is a core for $(-A, \, D(-A))$ by
		\cite[\Romannum{2}. Proposition 1,7, p. 53]{engel2001one}. \newline
		The addendum is well known but we recall its short and nice proof. Let
		$(\widetilde{T}(t))_{t\geq0}$ be a semigroup with generator $(-\widetilde{A}, \,
		D(-\widetilde{A})) \supseteq (-\mathcal{A}, \, C_c^\infty(\T^d, \, \C^N))$. The
		fact that $C_c^\infty(\T^d,\, \C^N)$ is a core and the closedness of
		$-\widetilde{A}$ due to Hille-Yosida, see \Cref{hille-yosida}, imply
		$(-\widetilde{A}, \, D(-\widetilde{A})) \supseteq (-A, \, D(-A))$. Hille-Yosida
		also yields existence of some $\lambda \in \rho(-A) \cap \rho(-\widetilde{A})$ and
		hence, both the operators
		\begin{equation*}
			\lambda+\widetilde{A} \colon D(-\widetilde{A}) \longrightarrow \LpCn
		\end{equation*}
		and
		\begin{equation*}
			\restr{\lambda+\widetilde{A}}{D(-A)} = \lambda + A \colon D(-A) \longrightarrow \LpCn
		\end{equation*}
		are bijective. We obtain $-A = -\widetilde{A}$, so the semigroups
		$(T(t))_{t\geq0}$ and $(\widetilde{T}(t))_{t\geq0}$ coincide by \cite[Chapter 1,
		Theorem 2.6, p. 6]{pazy2012semigroups}.
		
	\end{proof}
	
	\begin{remark}\label{remark_-A_generates_C0_group}
		By defining \eqref{transport_semigroup_definition} for all $t\in\R,$ one can
		directly see that $(T(t))_{t\geq0}$ can be extended to a $C_0$-group
		$(T(t))_{t\in\R}$ and that
		\begin{equation}\label{equation_domain_transport_group}
		D(-A) = \left\{ u \in \LpCn \, \colon \, \lim_{t \to 0} \frac{u_\indexone(\cdot
			- t\mathbf{v_\indexone}) - u_\indexone(\cdot)}{t} \text{ exists in } \Lp(\T^d)
		\text{ for all } \indexone=1,\cdots,N \right\}.
		\end{equation}
		The difference is subtle but notice that we have $t\to 0$ in the above formula
		instead of $t \searrow 0$ as in \eqref{generator_by_definition}. This equality
		on the domain follows by an explicit computation because
		\begin{equation*}
		\lim_{t \searrow 0} \frac{u_\indexone(\cdot - t\mathbf{v_\indexone}) -
			u_\indexone(\cdot)}{t} \text{ exists in } \Lp(\T^d) \text{ for all }
		\indexone=1,\cdots,N
		\end{equation*}
		if and only if
		\begin{equation*}
		\lim_{t \nearrow 0} \frac{u_\indexone(\cdot - t\mathbf{v_\indexone}) -
			u_\indexone(\cdot)}{t} \text{ exists in } \Lp(\T^d) \text{ for all }
		\indexone=1,\cdots,N.
		\end{equation*}
		In particular, the operator $A$ with $D(A)\coloneqq D(-A)$ generates the
		$C_0$-semigroup $(T(-t))_{t\geq0}$ given by
		\begin{equation*}
		(T(-t)u)(\mathbf{x}) \coloneqq \colvec{u_1(\mathbf{x} + t \mathbf{v_1}) \\
			\vdots \\ u_N(\mathbf{x} + t\mathbf{v_N})}
		\end{equation*}
		for $t\geq0,$ cf. \cite[\Romannum{2}. 3.a, p. 79]{engel2001one}.
	\end{remark}
	
	We now compute the adjoint semigroup $(T_p(t)^*)_{t\geq0}$ of
	$(T_p(t))_{t\geq0}$ on $\LpCn^*$ and the adjoint $-A_p^*$ of the generator
	$-A_p$, which will be pretty useful for the spectral analysis performed in
	\Cref{chapter_spectral_analysis}. We would like to point out that the index $p$
	does play a crucial role here, in contrast to all other results in the chapter.
	\newline
	Recall the classical result $\LpCn^* = L^q(\T^d, \, \C^N)$ where $1 < p <
	\infty$ and $q$ is the dual exponent of $p,$ i.e. $\tfrac{1}{p} + \tfrac{1}{q} =
	1,$ see \cite[\Romannum{4}.8, Theorem 1, p. 286]{dunford1988linear}. Moreover,
	the adjoint of the bounded operator $T_p(t)$ is defined by
	\begin{equation*}
		\langle T_p(t)u, \, v \rangle = \langle u, \, T_p(t)^* v\rangle \quad \text{
		for all } u\in L^p(\T^d, \, \C^N) \text{ and all } v \in L^q(\T^d, \, \C^N).
	\end{equation*}
	Again, the case $p=1$ is somewhat a special case because $L^1(\T^d, \, \C^N)$ is
	not reflexive.
	
	\begin{proposition}\label{adjoint_transport_gen}
		Let $1<p<\infty$ and let $(T_p(t))_{t\geq0}$ be the semigroup generated by
		$(-A_p, \, D(-A_p)).$ Then $(T_p(t)^*)_{t\geq0}$ is strongly continuous on
		$L^q(\T^d, \, \C^N)$ and it holds
		\begin{align*}
			(T_p(t)^*)_{t\geq 0} &= (T_q(-t))_{t\geq 0}, \\
			(-A_p^*, \, D(-A_p^*)) &= (A_q, \, D(A_q)).
		\end{align*}
	\end{proposition}
	\begin{proof}
		The space $L^p(\T, \, \C^N)$ is reflexive by for $1<p<\infty,$ see
		\cite[\Romannum{4}.8, Corollary 2, p. 288]{dunford1988linear}. Hence, the
		adjoint semigroup is again strongly continuous by \cite[\Romannum{1} 5.14
		Proposition, p. 44]{engel2001one} and the sun dual semigroup coincides with the
		adjoint semigroup, cf. \cite[\Romannum{2} 2.6, p. 62]{engel2001one}. In
		particular, the generator of the adjoint (strongly continuous) semigroup is the
		adjoint of $-A_p,$ namely $-A_p^*,$ by \cite[\Romannum{2} 2.6 proposition, p. 63]{engel2001one}. \newline
		It is left to show $-A_p^* = A_q$ and to characterize the domain. Recall that
		we have
		\begin{align*}
			D(-A_p^*) = \left\{ v \in L^q(\T^d, \, \C^N) \, \colon \, \exists w \in
			L^q(\T^d, \, \C^N) \text{ with } \langle -A_pu, v\rangle = \langle u, w\rangle
			\text{ for all } u \in D(-A_p) \right\}
		\end{align*}
		and
		\begin{equation*}
			-A_p^*v = w
		\end{equation*}
		for all $v\in D(-A_p^*).$ Using that $C_c^\infty(\T^d, \, \C^N)$ is a core for
		$(-A_p, \, D(-A_p))$ by \Cref{theorem_Ccinfty_core_A}, it follows
		\begin{equation*}
			D(-A_p^*) = \left\{ v \in L^q(\T^d, \, \C^N) \, \colon \, \exists w \in
			L^q(\T^d, \, \C^N) \text{ with } \langle -A_pu, v\rangle = \langle u, w\rangle
			\text{ for all } u \in C_c^\infty(\T^d, \, \C^N) \right\}.
		\end{equation*}
		The next step is to show
		\begin{equation}\label{auxeq34}
			(-A_p^*, \, D(-A_p^*)) \supseteq (A_q, \, C_c^\infty(\T^d, \, \C^N)).
		\end{equation}
		To this end, let $u,v\in C_c^\infty(\T^d, \, \C^N).$ We use the notation
		$\mathbf{v_{\indexone,\indextwo}}$ for the $\indextwo$-th component of
		$\mathbf{v_\indexone}$ and compute
		\begin{align*}
			\langle -A_pu, v \rangle &= \sum_{\indexone = 1}^N \intT
			(-\mathbf{v_\indexone} \cdot \nabla u_\indexone) \overline{v_\indexone} \,
			d\mathbf{x} =  - \sum_{\indexone=1}^N \sum_{\indextwo=1}^d
			\mathbf{v_{\indexone,\indextwo}} \intT \partial_\indextwo u_\indexone
			\overline{v_\indexone} \, d\mathbf{x} \\
			&= \sum_{\indexone = 1}^N \sum_{\indextwo=1}^d
			\mathbf{v_{\indexone,\indextwo}} \intT u_\indexone \partial_\indextwo
			\overline{v_\indexone} \, d\mathbf{x} = \sum_{\indexone=1}^N \intT u_\indexone
			\overline{\mathbf{v_\indexone} \cdot \nabla v_\indexone} \, d\mathbf{x} =
			\langle u, \, A_q v\rangle.
		\end{align*}
		This shows \eqref{auxeq34}, i.e. $(T_p^*(t))_{t\geq0}$ is a semigroup with a
		generator that extends $(\mathcal{A}_q, \, C_c^\infty(\T^d, \, \C^N)).$ \newline
		On the other hand, $(T_q(-t))_{t\geq0}$ is the unique semigroup on $L^q(\T^d,
		\, \C^N)$ with a generator that extends $(\mathcal{A}_q, \, C_c^\infty(\T^d, \,
		\C^N))$ by \Cref{remark_-A_generates_C0_group} and
		\Cref{theorem_Ccinfty_core_A}, which yields
		\begin{equation*}
			(T_p(t)^*)_{t\geq 0} = (T_q(-t))_{t\geq 0}.
		\end{equation*}
		In particular, the generators of both semigroups coincide.
	\end{proof}
	
	The dual characterization $L^1(\T^d, \, \C^N)^* = L^\infty(\T^d, \, \C^N)$  via
	the dual exponent also holds for $p=1,$ see \cite[\Romannum{4}.8, Theorem 5, p. 289]{dunford1988linear}. The difference to the cases $1<p<\infty$ is that the
	adjoint semigroup fails to be strongly continuous, hence we have not considered
	the generators on $L^\infty(\T^d, \, \C^N)$ yet. We restrict ourselves to the case
	$d=1$ with $v_1,\cdots,v_N \neq 0$ and make up this technical leeway with the
	definitions
	\begin{equation}\label{definition_generator_p_infinity}
		-A_\infty u \coloneqq -\colvec{v_1 u_1^\prime \\ \vdots \\ v_Nu_N^\prime }
	\end{equation}
	and
	\begin{equation}
		D(-A_\infty) \coloneqq  \big\{ u \in L^\infty(\T, \, \C^N) \, \colon \, u \text{ is
		weakly differentiable, } -A_\infty u \in L^\infty(\T, \, \C^N) \big\}.
	\end{equation}
	We obtain the following result concerning the dual of $(-A_1, \, D(-A_1)).$
	
	\begin{proposition}\label{proposition_adjoint_peq1}
		Let $p=1$ and let $d=1$. Assume that the transport directions $v_1,\cdots, v_N$
		are non-vanishing, i.e. $v_1,\cdots,v_N\neq0$. Then it holds
		\begin{align*}
		D(-A_\infty) &= W^{1,\infty}(\T, \, \C^N), \\
		(-A_1^* ,\, D(-A_1^*)) &= (A_\infty, W^{1,\infty}(\T, \, \C^N)).
		\end{align*}
	\end{proposition}
	\begin{proof}
		The characterization of the domain follows directly from
		\begin{equation*}
			u_\indexone^\prime = - \frac{1}{v_\indexone}  (-A_\infty u)_\indexone
		\end{equation*}
		for all $\indexone=1,\cdots, N$, so we are left to compute the adjoint $-A_1^*.$
		Let $\widetilde{u}\in D(-A_1^*).$ Then there exists $w\in L^\infty(\T, \, \C^N)$ with
		\begin{equation*}
			\langle -A_1 u, \, \widetilde{u}\rangle = \langle u, \, w\rangle
		\end{equation*}
		for all $u\in D(-A_1).$ In particular, given $\indexone=1,\cdots,N,$ we can choose $\varphi \in \C_c^\infty(\T)$ and 
		\begin{equation*}
			u_\indexone \coloneqq \frac{1}{v_\indexone} \varphi, \qquad
			u_i \coloneqq 0 \quad \text{ for } i\neq j
		\end{equation*}
		as the transport directions are non-vanishing. Testing with $u$ implies
		\begin{equation}\label{equation_dual_generator_explicit_form}
			-\int_{\T} \varphi' \overline{\widetilde{u}_\indexone} \, dx = \frac{1}{v_\indexone} \int_{\T}
			\varphi \overline{w_\indexone} \, dx
		\end{equation}
		for all test functions $\varphi \in C_c^\infty(\T).$ We obtain that $\widetilde{u}$ is
		weakly differentiable with $\widetilde{u} \in W^{1,\infty}(\T, \, \C^N)$. Now, given $\widetilde{u} \in	W^{1,\infty}(\T, \, \C^N)$ and $u\in C_c^\infty(\T, \, \C^N),$ we can compute
		\begin{align*}
			\langle -A_1u, \, \widetilde{u} \rangle &= -\sum_{\indexone=1}^N v_\indexone \int_{\T}
			u_\indexone^\prime \overline{\widetilde{u}_\indexone} \, dx = \sum_{\indexone = 1}^N
			v_\indexone \int_{\T} u_\indexone \overline{\widetilde{u}_\indexone^\prime} \, dx = \langle
			u,\, A_\infty \widetilde{u} \rangle
		\end{align*}
		This shows $-A_1^*\widetilde{u} = A_\infty \widetilde{u}$ because $C_c^\infty(\T, \, \C^N)$ is a core for $(-A_1, \, D(-A_1))$ by \Cref{theorem_Ccinfty_core_A}.
	\end{proof}

	In dimension $d=1,$ we can even show $D(-A) = W^{1,p}(\T, \, \C^N)$ if we assume
	that the transport directions are non-vanishing, i.e. $v_\indexone \neq 0$ for
	all $\indexone=1,\cdots,N.$ It should be noted that the assumption cannot be
	dropped. By definition \eqref{generator_by_definition} of the domain of the
	generator, there are no restrictions on $u_\indexone$ if $v_\indexone=0$ and the
	component could be an arbitrary function in $L^p(\T).$
	
	\begin{corollary}\label{corollary_domain_generator_one_dimensional}
		Let $1\leq p < \infty$ and let $d=1.$ Assume that $v_1, \cdots, v_N \in \R$ are
		non-vanishing, i.e. $v_1,\cdots,v_N\neq 0.$ Then 
		\begin{equation*}
			(-A, \, D(-A)) = (-\mathcal{A}, \, W^{1,p}(\T, \,\C^N)).
		\end{equation*}
	\end{corollary}
	\begin{proof}
		One direction follows from \Cref{lemma_A_is_an_extension}. Concerning the other
		direction, let $u\in D(-A)$. By \Cref{theorem_Ccinfty_core_A}, we can choose a
		sequence $(u^{(n)})_n \subset C_c^\infty(\T, \, \C^N)$ such that $u^{(n)} \to u$
		and $\mathcal{A}u^{(n)} = Au^{(n)} \to Au$ in $\Lp(\T, \, \C^N).$ Since we
		assumed non-vanishing transport directions, it follows that
		\begin{equation*}
			\| {u_\indexone^\prime}^{(n)} - {u_\indexone^\prime}^{(m)} \|_{\Lp(\T)} =
			\frac{1}{|v_\indexone|} \| (\mathcal{A} u^{(n)})_\indexone - (\mathcal{A}
			u^{(m)})_\indexone \|_{\Lp(\T)} \longrightarrow 0 \quad \text{as } n, \, m \to \infty
		\end{equation*}
		holds for all $j=1,\cdots,N$. Hence, the sequence $(u^{(n)})_n$ is Cauchy in
		$W^{1,p}(\T, \, \C^N).$ We obtain $u \in W^{1,p}(\T, \, \C^N)$ and $Au =
		\mathcal{A}u$ by \Cref{lemma_A_is_an_extension}. \newline
	\end{proof}
	
	For higher dimensions, there are in some sense too many directions. On an
	intuitive level, a function $u\in\LpCn$ such that the components $u_\indexone$
	``behave nicely" along the directions $\mathbf{v_\indexone}$ has enough
	regularity to be an element of $D(-A).$ In general, this should be insufficient
	for $u$ to be weakly differentiable for dimensions $d\geq2$ and one would expect
	$W^{1,p}(\T^d, \, \C^N) \subsetneq D(-A)$. However, we can only prove this under
	the assumption that one component is transport periodic. Let us clarify what we
	mean by that.
	
	\begin{definition}\label{definition_transport_periodic}
		Let $\mathbf{v_1}, \cdots, \mathbf{v_N} \in \R^d$ be non-vanishing, i.e.
		$\mathbf{v_1,} \cdots, \mathbf{v_N} \neq 0.$ Moreover, let  $\indexone =
		1,\cdots,N$. 
		\begin{enumerate}
			\item
			Component $\indexone$ is transport periodic if there exists $t^*>0$ such that
			$t^*\mathbf{v_\indexone} \in \Z^d$. In this case,
			\begin{equation*}
			\tau_\indexone \coloneqq \min \{ t > 0 \, \colon \, t\mathbf{v_\indexone} \in
			\Z^d \}.
			\end{equation*}
			is called the period of the $\indexone$-th component.
			\item
			The model is jointly periodic if there exists $t^*>0$ such that
			$t^*\mathbf{v_\indexone} \in \Z^d$ for all $\indexone=1, \cdots, N$. In this
			case,
			\begin{equation*}
			\tau \coloneqq \min\{ t>0 \, \colon \, t\mathbf{v_\indexone} \in \Z^d \text{
				for all } \indexone=1,\cdots,N\}.
			\end{equation*}
			is called the period of the model.
		\end{enumerate}
	\end{definition}
	
	\begin{remark}
		The definition of jointly periodic coincides with the transport semigroup
		$(T(t))_{t\geq0}$ defined in \eqref{transport_semigroup_definition} being
		periodic in the canonical sense of \cite[\Romannum{4}. Definition 2.23, p. 266]{engel2001one}. We need a concept to distinguish between the different
		components in order to give a relatively weak condition for $\WkpCn{1}{p} =
		D(-A)$ to fail in higher dimensions.
	\end{remark}
	\begin{remark}
		In dimension $d=1,$ every component is transport periodic with $\tau_\indexone
		= 1/|v_\indexone|,$ assuming that $v_\indexone \neq 0.$ This is not the case in
		higher dimensions.
	\end{remark}
	
	Let us characterize transport periodicity algebraically. To this end, we use the
	notation $\mathbf{v_{\indexone,\indextwo}}$ for the $\indextwo$-th component of
	$\mathbf{v_\indexone}.$
	
	\begin{lemma}\label{lemma_component_transport_periodic_characterization}
		Let $\mathbf{v_1}, \cdots, \mathbf{v_N} \in \R^d$ be non-vanishing, i.e.
		$\mathbf{v_1,} \cdots, \mathbf{v_N} \neq 0.$ Moreover, let $\indexone = 1,
		\cdots, N.$ The following assertions are equivalent:
		\begin{enumerate}
			\item
			Component $j$ is transport periodic.
			\item
			$\mathbf{v_{\indexone,\indextwo}} \in \mathbf{v_{\indexone,\indextwo^*}} \Q$
			for all $\indextwo=1,\cdots,d$ and one (and hence all) $\indextwo^*$ with
			$\mathbf{v_{\indexone,\indextwo^*}} \neq 0.$
			\item
			$ \displaystyle\bigcap^d_{\substack{\indextwo=1 \\
			\mathbf{v_{\indexone,\indextwo}} \neq 0}} \mathbf{v_{\indexone,\indextwo}} \Z
			\neq \{0\}.$
		\end{enumerate}
	\end{lemma}
	\begin{proof}
		The details are given in Appendix \ref{proof_lemma_periodicity}.
	\end{proof}
	
	Joint periodicity can also be characterized in an algebraic way.
	
	\begin{lemma}\label{lemma_jointly_periodic}
		Let $\mathbf{v_1}, \cdots, \mathbf{v_N} \in \R^d$ be non-vanishing, i.e.
		$\mathbf{v_1,} \cdots, \mathbf{v_N} \neq 0.$ The following assertions are
		equivalent:
		\begin{enumerate}
			\item
			The model is jointly periodic.
			\item
			$\mathbf{v_{\indexone,\indextwo}} \in \mathbf{v_{\indexone^*,\indextwo^*}} \Q$
			for all   $\indexone=1,\cdots,N,$ all $\indextwo=1,\cdots,d$ and one (and hence
			all) $\indexone^*$ and $\indextwo^*$ with $\mathbf{v_{\indexone^*,\indextwo^*}}
			\neq 0.$
			\item
			$ \displaystyle\bigcap_{\indexone=1}^N
			\displaystyle\bigcap^d_{\substack{\indextwo=1 \\
					\mathbf{v_{\indexone,\indextwo}} \neq 0}} \mathbf{v_{\indexone,\indextwo}} \Z
			\neq \{0\}.$
		\end{enumerate}
	\end{lemma}
	\begin{proof}
		Notice that the model is jointly periodic if and only if there exists $t^*>0$
		such that
		\begin{equation*}
		t^* \colvec{\mathbf{v_1} \\ \vdots \\ \mathbf{v_N}} \in \Z^{Nd}.
		\end{equation*}
		The result follows from an application of
		\Cref{lemma_component_transport_periodic_characterization} to the concatenated
		vector.
	\end{proof}
	
	\begin{remark}
		For $d\geq2,$ the above results show that there are very few transport
		directions such that one component is transport periodic or the model is jointly
		periodic in the sense that these directions are a nullset in $\R^d$ and
		$\R^{Nd}$ respectively. However, these are the natural cases a human would
		choose as an example when trying to show results for $d\geq2.$ Even for these
		"simple" periodic cases, many properties which hold true for $d=1$ fail in
		higher dimensions.
	\end{remark}
	
	The first negative result for higher dimensions is the next theorem.
	
	\begin{theorem}\label{theorem_domain_generator_higher_dimensions}
		Let $d\geq2$ and let $\mathbf{v_1}, \cdots, \mathbf{v_N}$ be non-vanishing,
		i.e. $\mathbf{v_1},\cdots,\mathbf{v_N}\neq 0$. Assume that there is one
		transport periodic component. Then 
		\begin{equation*}
			(-\mathcal{A}, \, \WkpCn{1}{p}) \subsetneq (-A, \, D(-A)).
		\end{equation*}
	\end{theorem}
	\begin{proof}
		Let component $\indexone$ be transport periodic with period $\tau_\indexone$.
		We show existence of a function $u\in D(-A)$ with $u\notin \WkpCn{1}{p}$ for all
		$1\leq p < \infty$. The idea is to take a non-differentiable function $u \in
		\LpCn$ such that $u_\indexone$ is constant along direction
		$\mathbf{v_\indexone}.$ \newline
		Let $\mathbf{c} = (\half, \cdots, \half) \in \T^d$ be the center of the torus.
		By assumption, the curve
		\begin{equation*}
			\gamma\colon [0,\tau_j] \rightarrow \T^d \quad \text{with} \quad \gamma(t)
			\coloneqq \mathbf{c} + t \mathbf{v_j}
		\end{equation*}
		fulfills $\gamma(0)=\gamma(\tau_j).$ Let $K=\range(\gamma)$ be the range of
		$\gamma$. Then there exists $\delta>0$ with
		\begin{equation}\label{domain_generator_construction_function}
			U_\delta = \{\mathbf{x}\in \T^d \, \colon \, \dist(\mathbf{x}, \, K) < \delta \} \subsetneq
			\T^d.
		\end{equation}
		We define $u_j \coloneqq \Ind{$U_\delta$}$ and $u_i \equiv 0$ for all
		$i=1,\cdots, N$ with $i\neq j.$ Notice that $u \in \LpCn$ and that $u_j$ is
		constant along direction $\mathbf{v_j},$ i.e.
		\begin{equation*}
			u_j(\cdot - t\mathbf{v_j}) - u_j(\cdot) \equiv 0.
		\end{equation*}
		Hence, $u\in D(-A)$ with $-Au=0$ and it is left to show that $u_j \notin
		W^{1,p}(\T^d)$ for all $1\leq p < \infty.$ Assume that for all
		$\indextwo=1,\cdots,d$ there exist $g_\indextwo \in L^1_{\loc}(\T^d)$  such that
		\begin{equation*}
			- \int_{\T^d} g_k \varphi \, d\mathbf{x} = \int_{\T^d} u_j \partial_k \varphi
			\, d\mathbf{x} = \int_{U_\delta} \partial_k \varphi \, d\mathbf{x} =
			\int_{\partial U_\delta} \varphi \mathbf{n_k} \, d\mathcal{H}^{d-1}(\mathbf{x})
		\end{equation*}
		holds for all $\varphi \in C_c^\infty(\T^d)$ with outer normal $\mathbf{n}$ of
		$U_\delta.$ Notice that the last step follows from Gauss's theorem. Testing with
		$\varphi \in C_c^\infty(U_\delta)$ and $\varphi \in
		C_c^\infty(\overline{U_\delta}^c)$ respectively implies $g_k = 0$ almost
		everywhere on $(\partial U_\delta)^c$ by the fundamental lemma of calculus of
		variations. The boundary $\partial U_\delta$ is a Lebesgue-nullset and hence,
		$g_k = 0$ holds almost everywhere in $\T^d$ for all $k=1,\cdots,d.$
		Consequently, $\nabla u_j = 0$ and $u_j$ is constant by the connectedness of the
		torus. This contradicts the definition of $u_j$.
	\end{proof}
	
	\begin{figure}[H]
		\centering
		\includegraphics{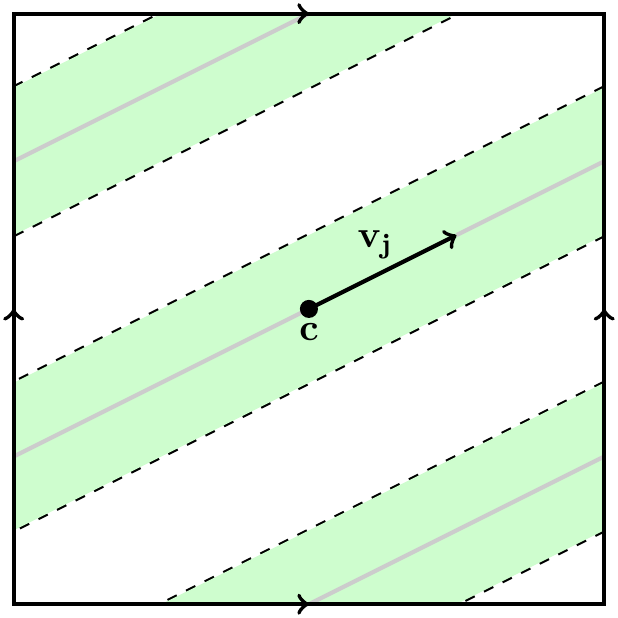}
		\caption[Domain Generator $\subsetneq W^{1,p}$: An Illustration of the Idea]{An
			illustration of the idea behind the function constructed in
			\eqref{domain_generator_construction_function} for $d=2$ in the exemplary case
			$\mathbf{v_j} = (\tfrac{1}{4}, \tfrac{1}{8}).$ The whole area is the torus
			$\T^2$ and the green area is the support of $u_j$ for
			$\delta=\tfrac{1}{4\sqrt{5}}$.}
		\label{fig:theorem_domain_generator_higher_dimensions}
	\end{figure}
	
	After achieving a better understanding of the domain of the generator, it would
	be reasonable to turn our attention to the question of well-posedness for two
	abstract Cauchy problems. Let $T>0$ and $I = [0,T) \subseteq \R_{t\geq0}$ be an interval. The linear hyperbolic Cauchy problem on $I$ reads
	\begin{equation}
		\begin{cases}
		\begin{array}{rrll}
		\dot{u}(t) + Au(t) & = & F(t) & \qquad 0<t<T, \\
		u(0) & = & u_0
		\end{array}
		\end{cases}
	\end{equation}
	for $u_0 \in \LpCn,$ an interval  and $F \in C(I, \,
	\LpCn)$. The semilinear hyperbolic Cauchy problem on $I$ is given by
	\begin{equation}
		\begin{cases}
		\begin{array}{rrll}
		\dot{u}(t) + Au(t) & = & F(t,u(t)) & \qquad 0<t<T, \\
		u(0) & = & u_0
		\end{array}
		\end{cases}
	\end{equation}
	for $u_0 \in \LpCn$ and $F \in C(I \times\LpCn, \, \LpCn).$ \newline
	
	We decided to postpone the overview on well-posedness results for these
	equations to \Cref{chapter_transport_reaction_SG} because the transport
	semigroup $(T(t))_{t\geq0}$ with generator $(-A, \, D(-A))$ is a special case of
	the transport-reaction semigroup studied in the next chapter. \newline
	
	\newpage
	
	\chapter{The Transport-Reaction Semigroup}\label{chapter_transport_reaction_SG}
	
	As in the previous chapter, we assume that the dimension $d\in\N$, the number of
	components $N\in\N$, the transport directions
	$\mathbf{v_1},\cdots,\mathbf{v_N}\in \R^d$ and $1\leq p < \infty$ are
	arbitrarily chosen fixed parameters, if not stated otherwise. We work on the
	Banach space $\LpCn$ with norm $\LpCnnorm{p}{\cdot}$, which is defined in
	\eqref{definition_Lp_norm}. \newline
	
	Eventually, our goal is to study the stability and long-time behavior of
	solutions to the transport-reaction model \eqref{intro_transport_reaction_eq_nonlinear}. Given an equilibrium state $u = c = \const \in \R^N$ with $F(c) = 0,$ we follow the
	standard approach by defining $u = c + w$ and linearizing the system for the
	perturbation $w.$ For a solution of the form $u=c+w$, one has
	\begin{equation*}
		\partial_t \colvec{w_1 \\ \vdots \\ w_N} +  \colvec{\mathbf{v_1} \cdot \nabla
		w_1\\ \vdots \\ \mathbf{v_N} \cdot \nabla w_N} = F(u) = DF(c)w +
		\mathcal{O}(|w|^2)
	\end{equation*}
	where $DF(c) \in \R^{N\times N}$ is the Jacobian of $F$ evaluated at $c.$ The
	(relabeled) linearized system then reads
	\begin{equation*}
		\partial_t \colvec{u_1 \\ \vdots \\ u_N} +  \colvec{\mathbf{v_1} \cdot \nabla
		u_1\\ \vdots \\ \mathbf{v_N} \cdot \nabla u_N} = DF(c)u.
	\end{equation*}
	For this reason, we will study the equation
	\begin{equation}\label{linear_transport_reaction_eq}
		\partial_t u + \colvec{\mathbf{v_1} \cdot \nabla u_1\\ \vdots \\ \mathbf{v_N}
		\cdot \nabla u_N} = Bu
	\end{equation}
	for an a priori arbitrary matrix $B \in \R^{N \times N}.$ It should be noted
	that $B$ is a \textit{real} matrix and that it induces a bounded linear operator
	on $\LpCn$ by pointwise matrix multiplication. We will use the letter $B$ for
	the matrix \textit{and} the operator. From a semigroup point of view, there are
	two different perspectives on \eqref{linear_transport_reaction_eq}. One possibility is to consider the abstract Cauchy problem
	\begin{equation}\label{perturbed_PDE_perspective_1_non_vanishing_rhs}
		\dot{u}(t) + Au(t) = Bu(t),
	\end{equation}
	and to solve it via analyzing the generator $(-A, \, D(-A))	$ of the transport
	semigroup. This approach can be carried out with the knowledge from
	\Cref{chapter_transport_semigroup} and Appendix \ref{appendix_semigroup_theory}.
	The other possibility is to consider the abstract Cauchy problem
	\begin{equation}\label{perturbed_PDE_perspective_2_perturbed_generator}
		\dot{u}(t) + (A-B)u(t) = 0
	\end{equation}
	and to study the semigroup generated by $-A+B.$ This section addresses the
	second approach and we study the reaction and transport-reaction semigroups
	generated by $B$ and $-A+B$ respectively. The main advantage is that standard
	perturbation theory gives a detailed description of the perturbed semigroup,
	which allows us to prove qualitative properties of the solutions to
	\eqref{linear_transport_reaction_eq}. \newline
	\indent An overview of the results from perturbation theory for semigroups we
	apply, including references and adaptations to our simple case, can be found in
	Appendix \ref{appendix_perturbation_theory}. \newline
	
	As a first step, we take a closer look to the semigroup generated by $(B, \,
	\LpCn)$ and characterize positivity. The main idea for this characterization was
	found on \cite{Conifold2014}.
	
	\begin{definition}
		Let $1\leq p < \infty.$ A function $u\in\LpCn$ is called non-negative if
		\begin{equation*}
			u_\indexone(\mathbf{x}) \geq 0 \text{ for almost all } \mathbf{x}\in\T^d
			\text{ and all } \indexone=1,\cdots,N.
		\end{equation*}
		In this case, we write $u\geq0$ and we define $\LpCn_+ \coloneqq \{ u \in \LpCn
		\, \colon \, u \geq0\}.$
	\end{definition}
	\begin{remark}\label{remark_positive_functions_closed_subset}
		The set $\LpCn_+$ is a closed subset of $\LpCn$ since convergence in
		$\|\cdot\|_{\LpCn}$ implies pointwise convergence almost everywhere along a
		subsequence.
	\end{remark}
	\begin{definition}\label{definition_positive}
		A semigroup $(T(t))_{t\geq0}$ on $\LpCn$ is called positive if each operator
		$T(t)$ is positive, i.e. if
		\begin{equation*}
			u \geq 0 \implies T(t)u \geq 0
		\end{equation*}
		for all $t\geq0.$ Similarly, a group $(T(t))_{t\in\R}$ is called positive if
		$T(t)$ is positive for all $t\in\R.$
	\end{definition}
	\begin{remark}
		These defintions make $\LpCn$ a Banach lattice and are just special cases of
		already existing general theory, cf. \cite{schaefer2012banach}. Postive
		semigroups arise naturally in biology and physics when $u$ describes non-negative
		quantities, like for example concentrations. Moreover, there are strong results concerning
		specral theory and long-time asymptotics of positive semigroups, which is why we
		are interested in a characterization of all matrices $B$ such that the semigroup
		generated by $-A+B$ becomes positive. The fundamentals in a concrete context can
		be found in \cite{engel2001one} and the detailed theory is presented in
		\cite{arendt1986one}.
	\end{remark}
	
	\begin{lemma}\label{lemma_S(t)_positive_SG_iff_B_off_diag_non-neg}
		The operator $(B, \, \LpCn)$ generates the uniformly continuous semigroup
		$(S(t))_{t\geq0} = (e^{tB})_{t\geq0}$ on $\LpCn,$ where $e^{tB}$ is the operator
		on $\LpCn$ given by pointwise matrix multiplication with the matrix $e^{tB}.$
		The semigroup is positive if and only if the off-diagonal entries of $B$ are
		non-negative.
	\end{lemma}
	\begin{proof}
		Notice that $B$ defines a bounded linear operator on $\LpCn$ due to the
		estimate
		\begin{equation*}
			\LpCnnorm{p}{Bu} \lesssim \| B \|_\infty \LpCnnorm{p}{u}.
		\end{equation*}
		Hence, $(B, \, \LpCn)$ generates the uniformly continuous semigroup
		$(S(t))_{t\geq0} = (e^{tB})_{t\geq0}$ by \cite[Chapter 1, Theorem 1.2, p. 2]{pazy2012semigroups}. \newline
		Let us show that this semigroup is positive if the off-diagonal entries of $B$
		are non-negative. To this end, we firstly assume that the off-diagonal entries
		are positive, i.e.
		\begin{equation}\label{equation_off_diagonal_positive}
			b \coloneqq \min_{\indexfour \neq \indexone} B_{\indexfour\indexone} > 0.
		\end{equation}
		Let $u \in \LpCn_+$ and let $\mathbf{x}\in\T^d$ such that $u(\mathbf{x})
		\geq0$. By definition, we have
		\begin{equation*}
			(S(t)u)(\mathbf{x}) = e^{tB}u(\mathbf{x}) = (I_{\C^{N\times N}} +
			tB)u(\mathbf{x}) + R_tu(\mathbf{x}) 
		\end{equation*}
		for a matrix $R_t$ with $\|R_t\|_\infty = \mathcal{O}(t^2)$ as $t\to0.$ This implies
		existence of $t_1 = t_1(B) > 0$ and $C = C(B) > 0$ such that
		\begin{equation*}
			\left( e^{tB}u(\mathbf{x}) \right)_\indexone \geq \big( (I_{\C^{N\times N}} +
			tB)u(\mathbf{x}) \big)_\indexone- \| R_t\|_\infty \|u(\mathbf{x})\|_\infty \geq \big(
			(I_{\C^{N\times N}} + tB)u(\mathbf{x}) \big)_\indexone - Ct^2
			\|u(\mathbf{x})\|_\infty
		\end{equation*}
		for all $\indexone=1,\cdots, N$ and $0\leq t\leq t_1.$ By our assumption
		\eqref{equation_off_diagonal_positive}, we also have existence of $t_2=t_2(B)>0$
		with
		\begin{equation*}
			\big( I_{\C^{N\times N}} + tB \big)_{\indexfour \indexone} \geq t b >0
		\end{equation*}
		for all $\indexfour,\indexone =1,\cdots, N$ and all $0\leq t\leq t_2.$ Notice
		that the identity matrix dominates the diagonal for small $t.$
		Both estimates together imply
		\begin{equation*}
			\left( e^{tB}u(\mathbf{x}) \right)_\indexone \geq t b \sum_{\indexfour=1}^{N}
			u_\indexfour(\mathbf{x}) - Ct^2 \|u(\mathbf{x})\|_\infty \geq t(b - Ct)
			\|u(\mathbf{x})\|_\infty \geq 0
		\end{equation*}
		for all $\indexone=1, \cdots, N$ and $t\leq t^* = t^*(B) \coloneqq \min\{t_1,
		\, t_2, \, b/C \}.$ This implies that the operators $S(t)$ are positive for
		$0\leq t\leq t^*.$\newline
		For $t\in\R_{\geq0},$ choose $m\in \N_0$ and $r\in [0, t^*)$ such that $t=mt^* + r.$
		Then
		\begin{equation*}
			S(t) = S(r) S(t^*)^m
		\end{equation*}
		is a positive operator, since compositions preserve positivity. \newline
		Secondly, let the off-diagonal entries of $B$ be non-negative. In this case, we
		take a sequence $(B^{(n)})_{n\in\N} $ of matrices with positive off-diagonal entries
		such that $B^{(n)} \to B$ in $\|\cdot\|_\infty.$ Let $u \geq 0, \,
		\mathbf{x}\in\T^d$ with $u(\mathbf{x})\geq0$ and $t\geq0.$ Our first step
		implies
		\begin{equation*}
			(S(t)u)(\mathbf{x}) = e^{tB^{(n)}}u(\mathbf{x}) + (e^{tB} - e^{t B^{(n)}})
			u(\mathbf{x}) \geq - \| e^{tB} - e^{tB^{(n)}} \|_\infty \|u(\mathbf{x})\|_\infty
			\longrightarrow 0
		\end{equation*}
		as $n\to\infty,$ i.e. $(S(t))_{t\geq0}$ defines a positive semigroup. \newline
		Conversely, if there exists one negative off-diagonal entry
		$B_{\indexfour\indexone} < 0$, we define the non-negative function $u \equiv
		e_\indexone \in \LpCn$ for the standard unit vector $e_\indexone \in \C^N.$
		Similar to before, we obtain the estimate
		\begin{equation*}
			(S(t)u)_\indexfour (\mathbf{x}) \equiv \left(
			e^{tB}e_\indexone\right)_\indexfour \leq 	\big((I_{\C^{N\times N}} +
			tB)e_\indexone\big)_\indexfour + Ct^2 = t(B_{\indexfour\indexone} + Ct) < 0
		\end{equation*}
		for some constant $C=C(B)>0$ and small $t.$ The last inequality is a
		consequence of the negativity of $B_{\indexfour\indexone}$. In particular,
		$S(t)$ is not a positive operator for small enough $t>0$ and the semigroup
		$(S(t))_{t\geq0}$ cannot be positive.
	\end{proof}
	
	Given this knowledge, we can use perturbation theory to show that $(-A+B, \,
	D(-A))$ generates a strongly continuous semigroup, which we call the
	transport-reaction semigroup $(R(t))_{t\geq0}.$
	
	\begin{theorem}\label{theorem_semigroup_generated_by_-A+B}
		The operator $(-A+B, \, D(-A))$ generates a $C_0$-semigroup $(R(t))_{t\geq0}$
		on $\LpCn$. Moreover, the following properties hold:
		\begin{enumerate}
			\item
			$\|R(t)\|_{\mathcal{L}(\LpCn)} \leq e^{\| B \|_{\mathcal{L}(\LpCn)}t}$
			\item
			$R(t)u = T(t)u + \int_{0}^{t} T(t-s) B R(s)u \, ds$
			\item
			$R(t)u = T(t)u + \int_{0}^{t} R(t-s) B T(s) \, ds$
			\item
			\begin{equation*}
			R(t) = \sum_{k=0}^{\infty} R_k(t)
			\end{equation*}
			with
			\begin{equation*}
			R_0(t) = T(t), \quad R_{k+1}(t)u = \int_{0}^{t} T(t-s)BR_k(s)u \, ds
			\end{equation*}
			\item
			\begin{equation*}
			R(t)u = \lim_{n \to \infty} \left(T(\tfrac{t}{n}) e^{t/n \, B}\right)^n u
			\end{equation*}
		\end{enumerate}
		for all $t\geq 0$ and $u \in \LpCn.$ Moreover, the graph norms of $-A$ and
		$-A+B$ on $D(-A)$ are equivalent, the series in $(iv)$ converges in
		$\mathcal{L}(\LpCn)$ uniformly on compact subsets of $\R_{\geq0}$ and for every
		fixed $u\in \LpCn,$ the convergence in $(v)$ is uniform on compact subsets of
		$\R_{\geq0}.$\newline
		In all of the above statements, $(T(t))_{t\geq0}$ is the transport semigroup on $\LpCn$
		defined in \eqref{transport_semigroup_definition} and with generator $(-A, \, D(-A))$.
	\end{theorem}
	\begin{proof}
		Let us recall that $(-A, \, D(-A))$ generates the transport semigroup
		$(T(t))_{t\geq0}$ from \Cref{chapter_transport_semigroup} and that $(B, \,
		\LpCn)$ generates the uniformly continuous semigroup $(e^{tB})_{t\geq0},$ see
		\Cref{lemma_S(t)_positive_SG_iff_B_off_diag_non-neg}. \newline
		The fact that $(-A+B, \, D(-A))$ generates a $C_0$-semigroup and properties
		$(i)-(iii)$ directly follow from
		\Cref{perturbation_theory_-A+B_gen_integral_identities}. Property $(iv)$ is a
		consequence of \Cref{perturbation_theory_-A+B_dyson_phillips_series}. Also the
		addenda, except for the uniform convergence of $(v)$, are consequences of these
		two theorems. \newline
		Furthermore, property $(v)$ and its corresponding addendum can be shown by
		applying \Cref{perturbation_theory_Lie_Trotter_formula}. Notice that
		$(T(t))_{t\geq0}$ is a contraction semigroup, so we have the stability bound
		\begin{align*}
		\| (T(\tfrac{t}{n}) e^{t/n \, B} )^n \|_{\mathcal{L}(\LpCn)} &\leq
		\left(\|T(\tfrac{t}{n})\|_{\mathcal{L}(\LpCn)} \|e^{t/n \,
			B}\|_{\mathcal{L}(\LpCn)}  \right)^n \\
		&\leq (e^{t/n \, \|B\|_{\mathcal{L}(\LpCn)} })^n = e^{t
			\|B\|_{\mathcal{L}(\LpCn)}}.
		\end{align*}
		Also, we already know that $(-A+B, \, D(-A))$ generates a $C_0$-semigroup from
		the previous steps. Hence, we can find some $\lambda > \|B\|_{\mathcal{L}(\LpCn)}$
		such that $(\lambda I_{\LpCn} - (-A+B))D(-A) = \LpCn$ by Hille-Yosida, see
		\Cref{hille-yosida}. Another consequence of Hille-Yosida is that the closure of
		$(-A+B, \, D(-A))$ is again $(-A+B, \, D(-A)).$ Therefore, in view of
		\Cref{perturbation_theory_Lie_Trotter_formula}, the formula on the right-hand
		side of $(v)$ is indeed a formula for the transport-reaction semigroup. 
	\end{proof}

	\begin{corollary}\label{corollary_Ccinfty_core_A+B}
		$C_c^\infty(\T^d, \C^N)$ is a core for $(-A+B, \, D(-A)).$ In particular,
		$(R(t))_{t\geq0}$ is the only strongly continuous semigroup on $\LpCn$ with a
		generator that extends $(-A+B, \, C_c^\infty(\T^d, \, \C^N)).$
	\end{corollary}
	\begin{proof}
		We know from \Cref{theorem_Ccinfty_core_A} that $C_c^\infty(\T^d, \, \C^N)$ is
		a core for $(-A, \, D(-A)).$ Therefore, $C_c^\infty(\T^d, \, \C^N)$ is also a
		core for $(-A+B, \, D(-A))$ because the graph norms of both operators are
		equivalent by \Cref{theorem_semigroup_generated_by_-A+B}. The uniqueness property is a general result for cores, cf.
		\Cref{theorem_Ccinfty_core_A}.
	\end{proof}
	
	For completeness, we give a brief intermezzo on the well-posedness of two abstract
	Cauchy problems.  Let $T>0$. The linear hyperbolic Cauchy problem on the interval $I=[0,T)$ reads
	\begin{equation}\label{equation_abstract_cauchy_problem_linear}
		\begin{cases}
		\begin{array}{rrll}
		\dot{u}(t) + (A-B)u(t) & = & F(t) &  \qquad 0<t<T, \\
		u(0) & = & u_0
		\end{array}
		\end{cases}
	\end{equation}
	for $u_0 \in \LpCn$ and $F \colon (0,T) \to \LpCn$. The semilinear hyperbolic
	Cauchy problem on $I$ is given by
	\begin{equation}\label{equation_abstract_cauchy_problem_semilinear}
		\begin{cases}
		\begin{array}{rrll}
		\dot{u}(t) + (A-B)u(t) & = & F(t,u(t)) & \qquad 0<t<T, \\
		u(0) & = & u_0
		\end{array}
		\end{cases}
	\end{equation}
	for $u_0 \in \LpCn$ and $F\colon (0,T) \times \LpCn \to \LpCn.$ \newline
	The standard solution concepts for equations of type
	\eqref{equation_abstract_cauchy_problem_linear} and
	\eqref{equation_abstract_cauchy_problem_semilinear} are classical and mild
	solutions, as well as maximal existence times $T_{\text{max}}(u_0)$. A detailed
	overview of all concepts we use, including precise definitions and references,
	can be found in Appendix \ref{appendix_semigroup_theory} and we
	strongly recommend reading Appendix \ref{appendix_semigroup_theory} for further details
	on the topic. \newline
	Without any a priori bounds of the solution in concrete examples, we can only
	apply the standard existence and uniqueness results from
	Appendix \ref{appendix_semigroup_theory}.
	
	\begin{theorem}[cf. {\Cref{classical_sol_linear}}]
		The initial value problem \eqref{equation_abstract_cauchy_problem_linear} has a
		unique classical solution $u\in C(I, \, \LpCn) \cap C^1((0,T), \, \LpCn)$ on $I$
		for every $u_0\in D(-A),$ if $F$ satisfies one of the following two properties:
		\begin{enumerate}
			\item
			$F\in C^1([0,T], \, \LpCn),$
			\item
			$F \in C((0,T), \, \LpCn) \cap L^1((0,T), \, \LpCn) \quad \text{ and } \quad
			F(s) \in D(-A) \text{ for all } s\in (0,T) \text{ with } AF(\cdot) \in
			L^1((0,T), \, \LpCn).$
		\end{enumerate}
	\end{theorem}
	\begin{remark}
		The homogeneous equation with $F\equiv0$ has the global classical solution
		$u(\cdot)=R(\cdot)u_0 \in C^1([0,\infty), \, \LpCn)$ for $u_0\in D(-A),$ see
		\Cref{existence_homogeneous_eq}.
	\end{remark}
	
	In the semilinear case \eqref{equation_abstract_cauchy_problem_semilinear}, the
	existence theorem for mild solutions reads as follows.
	
	\begin{theorem}[cf. {\Cref{theorem_mildsol}}]
		Let $F\colon [0,\infty)\times \LpCn \to \LpCn$ be continuous in $t$ and locally
		Lipschitz continuous in $u$, uniformly in $t$ on bounded intervals. Then for
		every $u_0\in \LpCn,$ there is a time $T_{\text{max}} =
		T_{\text{max}}(\|u_0\|_{\LpCn}) > 0,$ possibly $T_{\text{max}}=T,$ such that the
		initial value problem \eqref{equation_abstract_cauchy_problem_semilinear} has a
		unique mild solution $u$ on $[0,T_{\text{max}}).$ Moreover, if $T_{\text{max}} <
		T,$ then
		\begin{equation*}
			\lim_{t \nearrow T_{\text{max}}} \| u(t) \|_{\LpCn} = \infty.
		\end{equation*}
	\end{theorem}
	
	\begin{remark}
		If for each $t'>0,$ there exists a constant $C(t') >0$ such that
		\begin{equation*}
			\|F(t,u)\|_{\LpCn} \leq C (1+\|u\|_{\LpCn})
		\end{equation*}
		holds for all $t\in [0,t']$ and all $u\in \LpCn$, then
		\eqref{equation_abstract_cauchy_problem_semilinear} has a mild solution on
		$[0,\infty),$ i.e. one can choose $T=\infty$ and it holds $T_{\text{max}}=\infty,$
		see \Cref{remark_global_mild_sol}.
	\end{remark}
	
	Under stronger regularity assumptions on the right-hand side and the initial
	state, one obtains existence of strong solutions.
	
	\begin{theorem}[cf. {\Cref{semilinear_strong_sol}}]
		Let $F\in C([0,\infty)\times \LpCn, \, D(-A))$ be uniformly (in t) locally
		Lipschitz continuous in $u$. Then for every $u_0\in D(-A),$ the initial value
		problem \eqref{equation_abstract_cauchy_problem_semilinear} possesses a unique
		maximal strong solution
		\begin{equation*}
			u \in C([0,T_{\text{max}}), \, D(-A)) \cap C^1((0,T_{\text{max}}), \, \LpCn).
		\end{equation*}
	\end{theorem}
	
	Moreover, the solution depends continuously on the initial data in the sense of
	the next theorem.
	
	\begin{corollary}[cf. {\Cref{semilinear_lipschitz_dependence}}]
		Let $F\in C([0,\infty)\times \LpCn, \, \LpCn)$ be globally Lipschitz continuous
		in $u.$  Then the mild solution depends Lipschitz continuously on $u_0\in
		\LpCn,$ i.e. for all $T>0,$ the mapping $u_0\mapsto u$ is Lipschitz continuous
		from $\LpCn$ into $C([0,T], \, \LpCn)$.
	\end{corollary}
	
	After this summary of well-posedness results for the general transport reaction
	equation \eqref{intro_transport_reaction_eq_nonlinear}, we again focus on the linear case
	\eqref{linear_transport_reaction_eq}. \newline
	
	Similar to the transport semigroup $(T(t))_{t\geq0},$ we can extend the
	semigroup $(R(t))_{t\geq0}$ to a strongly continuous group $(R(t))_{t\in\R}$ and
	show an analogous result to \eqref{equation_domain_transport_group}. However, we
	have to be a little bit more careful as in \Cref{remark_-A_generates_C0_group}
	because the semigroup is less explicit.
	
	\begin{lemma}\label{lemma_-A+B_generates_C_0_group}
		The semigroup $(R(t))_{t\geq0}$ can be (uniquely) extended to the $C_0$-group
		$(R(t))_{t\in\R}$ where $(R(-t))_{t\geq0}$ is the semigroup $(R_-(t))_{t\geq0}$
		generated by $(A-B, \, D(A))$ with $D(A) = D(-A).$ Moreover, we have
		\begin{equation}\label{equation_domain_group_gen_by_-A+B}
		D(-A) = \left\{ u \in \LpCn \, \colon \, \lim_{t \to 0} \frac{R(t)u - u}{t}
		\textup{ exists in } \LpCn \right\}.
		\end{equation}
	\end{lemma}
	\begin{proof}
		Both $(-A+B, \, D(-A))$ and $(A-B, \, D(A))$ generate the $C_0$-semigroups
		$(R(t))_{t\geq0}$ and $(R_-(t))_{t\geq0}$ respectively by
		\Cref{theorem_semigroup_generated_by_-A+B}, since the change of signs simply
		changes the direction of movement and the sign of $B,$ i.e. the model is still
		of the type covered in \Cref{theorem_semigroup_generated_by_-A+B}. \newline
		Now, \cite[\Romannum{2} 3.11, p. 79]{engel2001one} implies that $(-A+B,
		\, D(-A))$ generates a strongly continuous group $(\widetilde{R}(t))_{t\in\R}$
		with
		\begin{equation}\label{auxeq1}
		D(-A) = \left\{ u \in \LpCn \, \colon \, \lim_{t \to 0}
		\frac{\widetilde{R}(t)u - u}{t} \text{ exists in } \LpCn \right\}.
		\end{equation}
		Notice that the generator theorems in \cite{engel2001one} are insufficient to
		show that $(\widetilde{R}(t))_{t\in\R}$ is in fact the $C_0$-group we would
		expect. We only have that $(\widetilde{R}(t))_{t\geq0}$ is a semigroup with a
		generator that \textit{extends} $(-A+B, \, D(-A))$ and that
		$(\widetilde{R}(-t))_{t\geq0}$ is a semigroup whose generator \textit{extends}
		$(A-B, \, D(A)).$ But this has the consequence that the generators of these
		semigroups extend $(-A+B, \, C_c^\infty(\T^d, \, \C^N))$ and $(A-B, \,
		C_c^\infty(\T^d, \, \C^N))$ respectively. However, $C_c^\infty(\T^d, \, \C^N)$
		is a core of $(-A+B, \, D(-A))$ and $(A-B, \, D(A))$ by
		\Cref{corollary_Ccinfty_core_A+B}, so we obtain
		\begin{equation*}
		(\widetilde{R}(t))_{t\geq 0} = (R(t))_{t\geq 0} \quad \text{ and } \quad
		(\widetilde{R}(-t))_{t\geq 0} = (R_-(t))_{t\geq 0}.
		\end{equation*}
		That is, $(\widetilde{R}(t))_{t\in\R} = (R(t))_{t\in\R}$ is an extension of
		$(R(t))_{t\geq0}$. Given this equality,
		\eqref{equation_domain_group_gen_by_-A+B} follows from \eqref{auxeq1}, which
		completes the proof.
	\end{proof}
	
	Alternatively, one can prove the above result in a natural way by explicitly
	defining the extension $(R(t))_{t\in\R}$ and showing the group property. In this
	case, one can use a comparable argument with cores to show
	\eqref{equation_domain_group_gen_by_-A+B}. The logic is interesting in itself
	and we also present this approach.
	
	\begin{proof}[A second proof of \Cref{lemma_-A+B_generates_C_0_group}]
		Let $(R_-(t))_{t\geq0}$ be the semigroup generated by $(A-B, \, D(A)).$ We
		define
		\begin{equation*}
		R(t) = \begin{cases}
		R(t) & t\geq 0, \\
		R_-(-t) & t\leq 0.
		\end{cases}
		\end{equation*}
		This defines a strongly continuous group by $R(-t)=R(t)^{-1}.$ The latter
		follows from
		\begin{equation*}
		\frac{d}{dt} R(-t)R(t)u = (A-B)R(-t)R(t)u + R(-t)(-A+B)R(t)u = 0
		\end{equation*}
		for all $u\in D(-A)$ and density of $D(-A)$ in $\LpCn$. Let $(\widetilde{A}, \,
		D(\widetilde{A}))$ be the generator of the $C_0$-group $(R(t))_{t\in\R}$ in the
		sense of \cite[\Romannum{2} 3.11, p. 79]{engel2001one}, i.e.
		\begin{equation*}
		D(\widetilde{A}) = \left\{ u \in \LpCn \, \colon \, \lim_{t \to 0} \frac{R(t)u
			- u}{t} \text{ exists in } \LpCn \right\}
		\end{equation*}
		and
		\begin{equation*}
		\widetilde{A}u = \lim_{t \to 0} \frac{R(t)u - u}{t}
		\end{equation*}
		for $u\in D(\widetilde{A}).$ It is left to show that $(\widetilde{A}, \,
		D(\widetilde{A})) = (-A+B, \, D(-A))$ holds. The direction ``$\subseteq$"
		follows directly from the definition of a generator, see \Cref{def_generator}.
		For the other direction, notice that $(\widetilde{A}, \, D(\widetilde{A}))$
		generates a strongly continuous semigroup $(\widetilde{R}(t))_{t\geq0}$ by
		\cite[\Romannum{2} 3.11, p. 79]{engel2001one} with $C_c^\infty(\T^d, \, \C^N)
		\subseteq D(\widetilde{A})$ because $C_c^\infty(\T^d, \, \C^N)$ lies in the
		domains of the generators of $(R(t))_{t\geq0}$ and $(R_-(t))_{t\geq0}$ by
		\Cref{lemma_A_is_an_extension}. Hence, $(\widetilde{A}, \, D(\widetilde{A}))$ is
		closed by \Cref{hille-yosida} and extends $(-A+B, \, C_c^\infty(\T^d, \,
		\C^N)).$ The set $C_c^\infty(\T^d, \, \C^N)$ is a core of $(-A+B, \, D(-A))$,
		see \Cref{corollary_Ccinfty_core_A+B}, and as a consequence, we have
		$(\widetilde{A}, \, D(\widetilde{A})) \supseteq (-A+B, \, D(-A))$ by the
		closedness of $\widetilde{A}.$ 
	\end{proof}
	
	The next chapter deals with the spectral properties of the generators
	$(-A_p+B_p, \, D(-A_p))$ and the generated semigroups $(R_p(t))_{t\geq0}$ on
	$\LpCn$, mostly in the case $p=2.$ \newline
	We also present methods which enable some extensions of the results to $1\leq p<\infty.$ At this	point, it is unclear whether the spectra are independent of $p$ and a key
	strategy will be to start by showing equality of the spectra for all $p\geq2.$
	In a second step, one can extend the results to $1\leq p \leq 2$ by using a
	duality argument. To this end, it is crucial to study the adjoint transport-reaction semigroup $(R_p(t)^{*})_{t\geq0}$ on $L^p(\T^d, \, \C^N)^*.$ Recall the classical result
	$\LpCn^* = L^q(\T^d, \, \C^N)$, where $1 < p < \infty$ and $q$ is the dual
	exponent of $p,$ i.e. $\tfrac{1}{p} + \tfrac{1}{q} = 1,$ see
	\cite[\Romannum{4}.8, Theorem 1, p. 286]{dunford1988linear}. Moreover, the
	adjoint of the bounded operator $R_p(t)$ is defined by
	\begin{equation*}
	\langle R_p(t)u, \, v \rangle = \langle u, \, R_p(t)^* v\rangle \quad \text{
		for all } u\in L^p(\T^d, \, \C^N) \text{ and all } v\in L^q(\T^d, \, \C^N).
	\end{equation*}
	Similar to \Cref{adjoint_transport_gen} and \Cref{proposition_adjoint_peq1}, we
	split our analysis into the cases $1<p<\infty$ and $p=1.$
	
	\begin{proposition}\label{proposition_adjoint_semigroup_pneq1}
		Let $1<p<\infty$ and let $(R_p(t))_{t\geq0}$ be the semigroup on $\LpCn$
		generated by $(-A_p+B_p, \, D(-A_p)).$ Then $(R_p(t)^*)_{t\geq0}$ is a
		strongly continuous semigroup on $L^q(\T^d, \, \C^N)$ with generator $(A_q+B_q^T, \,
		D(A_q))$. Moreover, it holds
		\begin{equation*}
			((-A_p+B_p)^*, \, D((-A_p+B_p)^*)) = (A_q+B_q^T, \, D(A_q))^.
		\end{equation*}
	\end{proposition}
	\begin{proof}
		The space $\LpCn$ is reflexive, see \cite[\Romannum{4}.8, Corollary 2, p. 288]{dunford1988linear}. Hence, the adjoint semigroup is again strongly
		continuous by \cite[\Romannum{1} 5.14 Proposition, p. 44]{engel2001one} and
		the sun dual semigroup coincides with the adjoint semigroup, cf.
		\cite[\Romannum{2} 2.6, p. 62]{engel2001one}. In particular, the generator of
		the adjoint (strongly continuous) semigroup is the adjoint $(-A_p + B_p)^*$ by
		\cite[\Romannum{2} 2.6 proposition, p. 63]{engel2001one}. \newline
		It is left to show $(-A_p+B_p)^* = A_q + B_q^T$. First of all, it should be
		noted that
		\begin{align*}
			\langle B_p u, \, v\rangle &= \sum_{\indexone = 1}^N \intT
			(Bu(\mathbf{x}))_\indexone \overline{v_\indexone(\mathbf{x})} \, d\mathbf{x} =
			\intT \langle Bu(\mathbf{x}), \, v(\mathbf{x}) \rangle_{\C^N} \, d\mathbf{x} =
			\int_{\T} \langle u(\mathbf{x}), \, B^T v(\mathbf{x}) \rangle_{\C^N} \,
			d\mathbf{x} \\
			&= \sum_{\indexone = 1}^N \int_{\T} u_\indexone(\mathbf{x})
			\overline{(B^Tv(\mathbf{x}))_\indexone} \, d\mathbf{x} = \langle u, \, B_q^T v
			\rangle
		\end{align*}
		holds for all $u\in L^p(\T,\, \C^N)$ and all $v\in L^q(\T, \, \C^N),$ i.e.
		$B_p^* = B_q^T.$ \newline
		Secondly, recall that we have already shown $(-A_p^*, \, D(-A_p^*)) = (A_q, \,
		D(A_q))$ in \Cref{adjoint_transport_gen}. Finally, the boundedness of $B_p$
		implies
		\begin{equation*}
			((-A_p + B_p)^*, \, D((-A_p+B_p)^*)) = (-A_p^*+B_p^*, \, D(-A_p^*)) =
			(A_q+B_q^T, \, D(A_q)).
		\end{equation*}
	\end{proof}
	The dual characterization $L^1(\T^d, \, \C^N)^* = L^\infty(\T^d, \, \C^N)$  via
	the dual exponent also holds for $p=1,$ see \cite[\Romannum{4}.8, Theorem 5, p. 289]{dunford1988linear}. The difference to the cases $1<p<\infty$ is that the
	adjoint semigroup fails to be strongly continuous. As in
	\Cref{chapter_transport_semigroup}, we restrict ourselves to the case $d=1$ with
	$v_1,\cdots,v_N \neq 0$ and make up this technical leeway with the definitions
	\begin{equation}\label{def_transport_reaction_gen_p_eq_infty}
		(-A_\infty+B_\infty) u \coloneqq -\colvec{v_1u_1^{'} \\ \vdots \\ v_Nu_N^{'} }
		+ Bu
	\end{equation}
	and
	\begin{equation}\label{def_domain_transport_reaction_gen_p_eq_infty}
		D(-A_\infty+B_\infty) \coloneqq D(-A_\infty) = \big\{ u \in L^\infty(\T, \, \C^N)
		\, \colon \, u \text{ is weakly differentiable, } -A_\infty u \in L^\infty(\T,
		\, \C^N) \big\}.
	\end{equation}
	We obtain the following result concerning the dual of $(-A_1+B_1, \, D(-A_1)).$
	
	\begin{proposition}\label{proposition_adjoint_generator_peq1}
		Let $p=1$ and let $d=1$. Assume that the transport directions $v_1,\cdots, v_N$
		are non-vanishing, i.e. $v_1,\cdots,v_N\neq0$. Then it holds
		\begin{align*}
			D(-A_\infty+B_\infty) &= W^{1,\infty}(\T, \, \C^N), \\
			((-A_1+B_1)^* ,\, D((-A_1+B_1)^*)) &= (A_\infty+B_\infty^T, \,
			W^{1,\infty}(\T, \, \C^N)).
		\end{align*}
	\end{proposition}
	\begin{proof}
		The characterization of the domain was proven in
		\Cref{proposition_adjoint_peq1}. Moreover, we have $B_1^* = B_\infty^T,$ which
		can be obtained just like in \Cref{proposition_adjoint_semigroup_pneq1}.
		\Cref{proposition_adjoint_peq1} also yields
		\begin{equation*}
			((-A_1+B_1)^* ,\, D((-A_1+B_1)^*)) = (-A_1^* + B_1^* ,\, D(-A_1^*)) =
			(A_\infty+B_\infty^T, \, W^{1,\infty}(\T, \, \C^N))
		\end{equation*}
		because $B_1$ is bounded.
	\end{proof}

	We will now focus our attention on two important qualitative properties of
	transport-reaction semigroups: positivity and mass conservation. It turns out
	that both properties can be fully characterized by conditions on the matrix $B$.
	These conditions are the same as for the ODE model
	\begin{equation}
		\begin{cases}
		\begin{array}{rrll}
		\dot{y}(t) & = & By(t) & \qquad t>0, \\
		y(0) & = & y_0
		\end{array}
		\end{cases}
	\end{equation}
	with given initial state $y_0\in\C^N,$ which is very reasonable because
	transport should not have any impact on the positivity of the concentrations or
	the total mass of the system.
	
	\begin{theorem}\label{theorem_semigroup_positive}
		The semigroup $(R(t))_{t\geq0}$ on $\LpCn$ generated by $(-A+B, \, D(A))$ is
		positive if and only if the off-diagonal entries of $B$ are non-negative.
	\end{theorem}
	\begin{proof}
		Let us begin by showing that $(R(t))_{t\geq0}$ is positive if the off-diagonal
		entries of $B$ are non-negative. The strategy is to apply $(v)$ from
		\Cref{theorem_semigroup_generated_by_-A+B}. Recall that $\LpCn_+$ is a closed
		subset of $\LpCn$, see \Cref{remark_positive_functions_closed_subset}. Let
		$u\geq0$ and let $t\geq0.$ The result simply follows from the fact that the
		transport semigroup $(T(t))_{t\geq0}$ and $(S(t))_{t\geq0} = (e^{tB})_{t\geq0}$
		are positive by \Cref{lemma_S(t)_positive_SG_iff_B_off_diag_non-neg}. Hence,
		\begin{equation*}
			\left(T(\tfrac{t}{n}) e^{t/n \, B}\right)^n u \geq 0
		\end{equation*}
		for all $n\in\N$ and the limit $n\to\infty$ implies $R(t)u \geq 0$. \newline
		Conversely, assume that the semigroup $(R(t))_{t\geq0}$ is a positive
		semigroup. This time, the idea is to conclude that $B = (-A+B) + A$ generates a
		positive semigroup. If this is the case, the statement will follow from
		\Cref{lemma_S(t)_positive_SG_iff_B_off_diag_non-neg}. Some caution is required,
		since the perturbation of the generator $-A+B$ of the positive semigroup
		$(R(t))_{t\geq0}$ is unbounded. Nevertheless, we can apply
		\Cref{perturbation_theory_Lie_Trotter_formula}. More precisely, we
		know that $(-A+B, \, D(-A))$ and $(A, \, D(A))$ generate the positive
		$C_0$-semigroups $(R(t))_{t\geq0}$ and $(T(-t))_{t\geq0}$ respectively. The
		latter is a consequence of the fact that $(-A, \, D(-A))$ generates a positive
		$C_0$-group, cf. \Cref{remark_-A_generates_C0_group}. We have the stability
		bound
		\begin{align*}
			\| (R(\tfrac{t}{n}) T(-\tfrac{t}{n}))^n \|_{\mathcal{L}(\LpCn)} \leq \|
			R(\tfrac{t}{n}) \|_{\mathcal{L}(\LpCn)}^n \leq e^{\|B\|_{\mathcal{L}(\LpCn)} t},
		\end{align*}
		since $(T(-t))_{t\geq0}$ is a $C_0$-semigroup of contractions and by an
		application of statement $(i)$ from \Cref{theorem_semigroup_generated_by_-A+B}.
		Next, we notice that $B \in \mathcal{L}(\LpCn)$ implies that
		\begin{equation*}
			(\lambda I_{\LpCn} - B) \LpCn = \LpCn
		\end{equation*}
		holds for $\lambda > \|B\|_{\mathcal{L}(\LpCn)}$ and therefore, a standard density
		argument implies density of $(\lambda I_{\LpCn} - B)D(-A)$ by the density of
		$D(-A)$ and the continuity of $B$. We can finally apply
		\Cref{perturbation_theory_Lie_Trotter_formula} and conclude that $(B, \, D(-A))$
		has a closure which generates a $C_0$-semigroup $(\widetilde{S}(t))_{t\geq0}$.
		This is unsurprising as the closure of $(B, \, D(-A))$ is the operator $B$ with
		full domain. As a consequence, $(\widetilde{S}(t))_{t\geq0} = (S(t))_{t\geq0}$
		by \cite[\Romannum{2}. 1.4, p. 51]{engel2001one}. More importantly, we obtain
		the formula
		\begin{equation*}
			S(t)u = \lim_{n \to \infty} \left(R(\tfrac{t}{n}) T(-\tfrac{t}{n})\right)^nu
		\end{equation*}
		for all $u\in\LpCn$ and all $t\geq0.$ It now follows similarly to the first
		part of the proof that $(S(t))_{t\geq0}$ is a positive semigroup, i.e. $B$
		generates a positive semigroup.
	\end{proof}

	\begin{corollary}\label{corollary_group_positive}
		The group $(R(t))_{t\in\R}$ on $\LpCn$ generated by $(-A+B, \, D(A))$ is
		positive if and only if $B$ is a diagonal matrix.
	\end{corollary}
	\begin{proof}
		By \Cref{lemma_-A+B_generates_C_0_group}, the group $(R(t))_{t\in\R}$ is
		positive if and only if the two semigroups $(R(t))_{t\geq0}$ and
		$(R_-(t))_{t\geq0}$, which are generated by $(-A+B, \, D(-A))$ and $(A-B, \,
		D(A))$ respectively, are positive semigroups. \Cref{theorem_semigroup_positive}
		implies that this is the case if and only if the diagonal entries of $B$ and
		$-B$ are non-negative, i.e. iff $B$ is a diagonal matrix.
	\end{proof}
	
	\begin{definition}\label{definition_conservation}
		We stay that $y\in\C^N$ is associated to conservation if the quantity
		\begin{equation*}
			\langle y, \, \overline{\intT (R(t)u)(\mathbf{x}) \, d\mathbf{x}}
			\rangle_{\C^N} \in\C
		\end{equation*}
		is independent of $t\geq0$ for all $u\in \LpCn.$ The model is called mass
		conserving if $(1,\cdots,1)^T$ is associated to conservation.
	\end{definition}
	\begin{remark}
		Mass conservation essentially means
		\begin{equation*}
			\sum_{\indexone = 1}^N \intT u_\indexone(t, \mathbf{x}) \, d\mathbf{x} =
			\sum_{\indexone = 1}^N \intT u_0(\mathbf{x}) \, d\mathbf{x}
		\end{equation*}
		for all $t>0$, the solution $u$ of
		\begin{equation*}
			\partial_t u + \colvec{\mathbf{v_1} \cdot \nabla u_1\\ \vdots \\ \mathbf{v_N}
			\cdot \nabla u_N} = Bu
		\end{equation*}
		and ``any" initial value $u_0.$ This definition via the PDE has the
		disadvantage that one needs to specify an appropriate set of admissible initial
		functions, whereas the semigroup is defined on the whole space $\LpCn.$
	\end{remark}
	
	We now characterize all vectors $y\in\C^N$ associated to conservation.
	Interestingly enough, the proof relies on the structure of the semigroup and
	does not use $\tfrac{d}{dt}R(t)u = (-A+B)R(t)u$ for $u\in D(-A).$
	
	\begin{lemma}\label{time_development_averages}
		Let $(R(t))_{t\geq0}$ be the semigroup on $\LpCn$ generated by $(-A+B, \,
		D(-A))$. Then
		\begin{equation*}
			\intT (R(t)u)(\mathbf{x}) \, d\mathbf{x} = e^{tB} \intT u(\mathbf{x}) \, d\mathbf{x}
		\end{equation*}
		holds for all $u\in\LpCn$ and all $t\geq0.$
	\end{lemma}
	\begin{proof}
		The two key ingredients to this observation are
		\begin{equation}\label{auxeq32}
		\intT (T(t)u)(\mathbf{x}) \, d\mathbf{x} = \intT u(\mathbf{x}) \, dx
		\end{equation}
		and
		\begin{equation}\label{auxeq33}
		\intT (e^{tB}u)(\mathbf{x}) \, d\mathbf{x} = e^{tB} \intT u(\mathbf{x}) \,
		d\mathbf{x}
		\end{equation}
		for all $u\in\LpCn$ and all $t\geq0.$ They follow from the independence of
		integrals of periodic functions w.r.t. translations and the definition of the
		operator $e^{tB}$ as a pointwise matrix multiplication. The equations
		\eqref{auxeq32} and \eqref{auxeq33} imply
		\begin{align*}
			\intT (R(t)u)(\mathbf{x}) \, d\mathbf{x} &= \lim_{n \to \infty} \intT
			\big((T(\tfrac{t}{n})e^{t/n \,B})^n u\big)(\mathbf{x}) \, d\mathbf{x} = \lim_{n \to \infty}\intT \big(e^{t/n \, B} (T(\tfrac{t}{n})e^{t/n \,
			B})^{n-1} u\big)(\mathbf{x}) \, d\mathbf{x} \\
			&= \lim_{n \to \infty} e^{t/n \, B} \intT  \big((T(\tfrac{t}{n})e^{t/n \,
			B})^{n-1} u\big)(\mathbf{x}) \, d\mathbf{x} = e^{tB} \intT u(\mathbf{x}) \,
			d\mathbf{x}.
		\end{align*}
		for all $u\in\LpCn$ and all $t\geq0$ by an application of property $(v)$ from
		\Cref{theorem_semigroup_generated_by_-A+B} and a simple induction.
	\end{proof}
	\begin{remark}
		If $B$ is a stable matrix, i.e. $\sigma(B) \subseteq \C_-=\{\lambda \in \C \,
		\colon \, \real(\lambda) < 0 \}$, the above lemma shows that the averages of the
		solution converge to $0$.
	\end{remark}
	
	\begin{theorem}\label{theorem_conservation}
		A vector $y\in\C^N$ is associated to conservation if and only if $B^T y = 0.$
	\end{theorem}
	\begin{proof}
		$`` \impliedby "$ Let $y\in\C^N$ with $B^Ty = 0$ and let $u\in \LpCn.$ We can
		use \Cref{time_development_averages} to compute
		\begin{align*}
			\frac{d}{dt} \langle y, \, \overline{\intT (R(t)u)(\mathbf{x}) \, d\mathbf{x}}
			\rangle_{\C^N} &= \frac{d}{dt} \langle y, \, e^{tB} \overline{\intT
			u(\mathbf{x}) \, d\mathbf{x}} \rangle_{\C^N} =\langle y, \, Be^{tB}\overline{\intT u(\mathbf{x}) \, d\mathbf{x}}
			\rangle_{\C^N} \\
			&= \langle B^T y, \, e^{tB} \overline{\intT u(\mathbf{x}) \,
			d\mathbf{x}} \rangle_{\C^N} = 0
		\end{align*}
		for $t>0.$ The continuity of $t\mapsto R(t)u$ on $[0,\infty)$ yields
		\begin{equation}\label{auxeq18}
		\langle y, \, \overline{\intT (R(t)u)(\mathbf{x}) \, d\mathbf{x}}
		\rangle_{\C^N} = \langle y, \, \overline{\intT u(\mathbf{x}) \, d\mathbf{x}}
		\rangle_{\C^N}
		\end{equation}
		for all $t\geq0.$ \newline
		
		\noindent $`` \implies "$ Let $y\in\C^N$ be associated to conservation, i.e.
		\begin{equation*}
		\langle y, \, \overline{\intT (R(t)u)(\mathbf{x}) \, d\mathbf{x}}
		\rangle_{\C^N} = \langle y, \, \overline{\intT u(\mathbf{x}) \, d\mathbf{x}}
		\rangle_{\C^N}
		\end{equation*}
		holds for all $u\in\LpCn$ and all $t\geq0.$ Differentiating with respect to $t$
		implies
		\begin{equation*}
		\langle B^T y, \, e^{tB} \overline{\intT u(\mathbf{x}) \, d\mathbf{x}}
		\rangle_{\C^N} = 0
		\end{equation*}
		for all $t>0$ because of the computation done in $`` \impliedby "$ and the fact
		that the right-hand side is independent of $t$. Letting $t\searrow0$ results in
		\begin{equation*}
		\langle B^T y, \, \overline{\intT u(\mathbf{x}) \, d\mathbf{x}} \rangle_{\C^N}
		= 0.
		\end{equation*}
		Finally, taking $u\equiv \overline{z}$ for $z\in\C^N$ yields $\langle B^Ty, \,
		z\rangle_{\C^N} = 0$ for all $z\in\C^N$, which finishes the proof.
	\end{proof}

	\chapter{Spectral Analysis of the Transport-Reaction Semigroup}\label{chapter_spectral_analysis}
	
	As always, we assume that the dimension $d\in\N$, the number of components
	$N\in\N$, the transport directions $\mathbf{v_1},\cdots,\mathbf{v_N}\in \R^d$
	and $1\leq p < \infty$ are arbitrarily chosen fixed parameters, if not
	stated otherwise. We work on the Banach space $\LpCn$ with norm
	$\LpCnnorm{p}{\cdot}$, which was defined in \eqref{definition_Lp_norm}. \newline
	
	Our long-term goal is to study the stability of solutions of the
	transport-reaction model \eqref{intro_transport_reaction_eq_nonlinear}. Let us recall from the
	introduction of \Cref{chapter_transport_reaction_SG} that the linearized
	equation reads as
	\begin{equation*}
		\partial_t u + \colvec{\mathbf{v_1} \cdot \nabla u_1 \\ \vdots \\ \mathbf{v_N} \cdot \nabla u_N} = Bu
	\end{equation*}
	for an a priori arbitrary matrix $B \in \R^{N \times N}.$ In
	\Cref{chapter_transport_reaction_SG}, we considered the abstract Cauchy problem
	\begin{equation}\label{equation_abstract_cauchy_problem}
		\dot{u}(t) + (A-B)u(t) = 0
	\end{equation}
	on $\LpCn$, where $(-A, \, D(-A))$ is the generator of the transport semigroup
	introduced in \Cref{chapter_transport_semigroup} and $B$ is the bounded linear
	operator on $\LpCn$ induced by the matrix $B$. Notice that we
	slightly abuse notation and again use the letter $B$ for the matrix \textit{and} the
	operator. Distinction between both of them will always follow within the context
	of usage. \newline
	\newline
	Semigroup and spectral theory offer a great tool to study stability properties
	both qualitatively and quantitatively with precise convergence rates. The most
	important definitions and results we will apply can be found in Appendix
	\ref{appendix_Long_time_behavior_SG}, including additional references. Roughly
	speaking, if $(R(t))_{t\geq0}$ is the semigroup generated by $(-A+B, \, D(-A))$,
	certain spectral properties of the operators $R(t)$ for $t\geq0$ imply
	stability. It is then natural to try to find a connection between the spectrum
	of the generator $(-A+B, \, D(-A))$ and the spectrum of $R(t)$ for $t\geq0.$
	This connection does indeed exist via the spectral mapping theorem
	\begin{equation}\label{motivation_spectral_mapping_theorem}
		``\quad\sigma(R(t)) = e^{t\sigma(-A+B)}\quad".
	\end{equation}
	Standard results in the literature write $\sigma(R(t)) \backslash \{0\}$ in
	\eqref{motivation_spectral_mapping_theorem}. But we have seen that
	$(R(t))_{t\geq0}$ extends to a strongly continuous group, which implies $0\in
	\rho(R(t))$ for all $t\in\R,$ see \Cref{lemma_-A+B_generates_C_0_group}.
	However, the equation is put in quotation marks as it does not always hold. In
	our case, the PDE is hyperbolic and the semigroup $(R(t))_{t\geq0}$ does not
	have powerful properties like analyticity. It is in general not even eventually
	norm continuous and one cannot apply \Cref{smtm}. In fact, \eqref{motivation_spectral_mapping_theorem}
	fails for the transport semigroup $(T(t))_{t\geq0}$ and the best one can hope
	for is a weaker version of \eqref{motivation_spectral_mapping_theorem}, the so
	called weak spectral mapping theorem
	\begin{equation}\label{equation_goal_weak_spectral_mapping_theorem}
	\sigma(R(t)) = \overline{e^{t\sigma(-A+B)}}
	\end{equation}
	The objective of the whole chapter is to show
	\eqref{equation_goal_weak_spectral_mapping_theorem} for the transport-reaction
	semigroups $(R(t))_{t\geq0}$ on $\LpCn$. Unfortunately, we often need to
	restrict ourselves to the cases $p=2$ or $d=1.$ \newline
	
	This chapter is structured as follows: we start by performing a complete
	spectral analysis of the transport semigroup $(T(t))_{t\geq0}$ for $d=1$. These
	computations mainly serve as a motivating example and show that
	\eqref{motivation_spectral_mapping_theorem} cannot be expected to hold for the
	transport-reaction semigroups. Beyond that, the computations indicate that a
	weak spectral mapping theorem should hold. \newline
	We continue by characterizing the spectrum of the generator $(-A_2+B_2, \,
	D(-A_2))$ for arbitrary dimension $d\in\N$. Afterwards, we study the spectra of the
	generators $(-A_p+B_p, \, D(-A_p))$ for dimension $d=1$ and arbitrary $1\leq p <
	\infty$. It may seem counter intuitive to start this way, since these results
	are insufficient to perform a stability analysis for the model without knowing
	that \eqref{equation_goal_weak_spectral_mapping_theorem} holds. However, a great understanding of the spectrum of the generator for all $1\leq p < \infty$ is
	required to extend a weak spectral mapping theorem from $p=2$ to different $p$.
	In particular, independence of $\sigma(-A_p+B_p)$ of $p$ plays a crucial
	role. When studying the spectrum for $p\neq 2$, we assume $d=1$ because in this case,
	the generator has compact resolvent if the transport directions are
	non-vanishing. This implies that the spectrum consists only of eigenvalues. For
	higher dimensions, we show that $(-A+B, \, D(-A))$ does not have
	compact resolvent if one component $\indexone$ is transport periodic, see
	\Cref{definition_transport_periodic}. Rigorously proving stability results then
	turns into a difficult task for $p\neq 2$ and $d\geq 2$. A computation of the
	eigenvalues of $(-A+B, \, D(-A))$ could yield an incomplete picture without
	further work on the problem $\sigma(-A+B) = \sigma_p(-A+B)$. \newline
	We finally show  \eqref{equation_goal_weak_spectral_mapping_theorem} for $p=2$
	and all dimensions $d\in \N$. For $d=1,$ we can partially extend the result to
	other $p$ by using the main ideas from \cite{latrach2008weak}. More precisely,
	we show that the weak spectral mapping theorem all $1\leq p < \infty$
	is basically equivalent to the weak spectral mapping theorem for $p=1$, cf. \cref{wsmthm_deq1_equiv} for the precise statement. In particular, \eqref{equation_goal_weak_spectral_mapping_theorem} extends to all $1\leq p < \infty$ if the group $(R(t))_{t\in\R}$ is positive but unfortunately, the model
	is quite unspectacular if $B$ is a diagonal matrix, cf. \Cref{corollary_group_positive}. \newline
	
	We conjecture that \eqref{equation_goal_weak_spectral_mapping_theorem} holds for
	all dimensions $d\in \N$ and all $1\leq p < \infty$ but proving this seems to be
	a difficult task.
	
	\section{Spectral Analysis of the Transport Semigroup}
	
	The next example is a complete spectral analysis in the non-reacting case in dimension
	$d=1$ and with non-vanishing transport directions $v_1, \cdots, v_N \neq 0$. It shows
	that one cannot expect \eqref{motivation_spectral_mapping_theorem} to hold,
	even in the most simplest scenario. Essentially, the example we formulate is a
	detailed and generalized version of \cite[\Romannum{4}. Example 2.6 (iv), p. 254]{engel2001one} and \cite[Example 4.12, p. 100]{schnaubelt2011lecture}. The
	main ideas are already presented there.

	\begin{example}[$N=1$]\label{example_spectral_analysis_transport_1d}
		Let us consider $d=1, \, N=1$ and $B=0$ with non-vanishing transport direction
		$v \neq 0.$ In this case, we have
		\begin{equation*}
			-Au = -vu' \quad \text{and} \quad (T(t)u)(x) = u(x-tv)
		\end{equation*}
		with $D(-A)=W^{1,p}(\T)$ by \Cref{corollary_domain_generator_one_dimensional}.
		Clearly, the semigroup is periodic with period $\tau = 1/|v|.$ By the spectral
		inclusion \Cref{spectral_inclusion_theorem}, we obtain
		\begin{equation*}
			e^{\tau \sigma(-A)} \subseteq \sigma(T(\tau)) = \sigma(I_{L^p(\T, \, \C^N)}) =
			\{ 1 \}
		\end{equation*}
		and thus $\sigma(-A) \subseteq 2\pi i v \Z.$ In fact, there holds equality:
		for $\lambda \in 2\pi i v \Z$ we define the function
		\begin{equation*}
			u(x) = e^{-\tfrac{\lambda}{v}x}.
		\end{equation*}
		Then
		\begin{equation*}
			-Au = -v u' = \lambda u
		\end{equation*}
		and $u\in W^{1,p}(\T)$ as it fulfills the periodic boundary conditions. That is, we have shown
		\begin{equation*}
			\sigma(-A) = \sigma_p(-A) = 2\pi i v \Z
		\end{equation*}
		explicitly.\newline
		In the following, we compute the spectrum of $T(t)$ for $t\geq0.$ To this end,
		let us introduce some notation. For a rational number $p/q \in \Q$, we choose
		$p\in\Z , \, q\in\N$ and define the $q$-th roots of unit $\Gamma_q \coloneqq \{z
		\in \C \, \colon \, z^q =1\}.$ The goal is to show
		\begin{equation*}
			\sigma(T(t)) = \begin{cases}
			\S^1 \qquad t v \in \R \backslash \Q, \\
			\Gamma_q \qquad t v = \frac{p}{q} \text{ with } p \text{ and } q \text{
				coprime}
			\end{cases}.
		\end{equation*}
		Recall that $(T(t))_{t\geq0}$ defines a group of isometries and hence
		\begin{align}\label{example_spectral_radius_bounded_from_above}
			r(T(t)) = \max \{ |\lambda| \, \colon \, \lambda \in \sigma(T(t))\} = \lim_{n
				\to \infty} \| T(t)^n \|^{\frac{1}{n}} = \lim_{n \to \infty} \| T(nt)
			\|^{\frac{1}{n}} = 1,
		\end{align}
		where $r(T(t))$ is the spectral radius of $T(t),$ see
		\Cref{long_time_behavior_definition_spectral_radius}, and the second equality follows from \cite[Theorem 1.16, p. 7]{schnaubelt2012lecture}. Equation \eqref{example_spectral_radius_bounded_from_above} holds for
		all $t\in\R$ and we particularly obtain
		\begin{equation*}
			r(T(t)^{-1}) = r(T(-t)) = 1.
		\end{equation*}
		Now, note that boundedness and invertibility of $T(t)$ implies
		\begin{equation*}
			\sigma(T(-t)) = \frac{1}{\sigma(T(t))}.
		\end{equation*}
		Therefore, we also get
		\begin{equation}\label{example_spectral_radius_bounded_from_below}
			1 = r(T(-t)) = \max \{ |\lambda| \, \colon \, \lambda \in \sigma(T(-t))\} =
			\max \left\{ \tfrac{1}{|\lambda|} \, \colon \, \lambda \in \sigma(T(t)) \right\},
		\end{equation}
		so \eqref{example_spectral_radius_bounded_from_above} and \eqref{example_spectral_radius_bounded_from_below} imply $\sigma(T(t))
		\subseteq \S^1$ for all $t\geq0.$ \newline
		At this point, we again use the spectral inclusion \Cref{spectral_inclusion_theorem} to deduce
		\begin{equation}\label{example_spectral_inclusion_theorem}
			e^{t\sigma(-A)} \subseteq \sigma(T(t))
		\end{equation}
		for all $t\geq0.$ Recall $\sigma(-A)= 2\pi i v \Z$. In the case $tv \in
		\R\backslash\Q,$ it is well known and relatively easy to see that
		\begin{equation*}
			e^{t 2 \pi i v \Z} \subseteq \S^1
		\end{equation*}
		is a dense inclusion. By the closedness of spectra, it follows
		\begin{equation*}
			\S^1 = \overline{e^{t\sigma(-A)}} \subseteq \overline{\sigma(T(t))} =
			\sigma(T(t)) \subseteq \S^1.
		\end{equation*}
		In the second case, let $tv = \tfrac{p}{q} \in \Q$ for coprime $p\in\Z$ and
		$q\in\N.$ We can use the spectral mapping theorem \cite[Theorem 5.3, p. 84]{schnaubelt2012lecture} for bounded linear operators and the holomorphic map
		$z \mapsto z^q$ to obtain
		\begin{equation*}
			\sigma(T(t))^q = \sigma\big(T(\tfrac{p}{q v})\big)^q =
			\sigma\big(T(\tfrac{p}{qv})^q\big) = \sigma\big(T(\tfrac{1}{v})^p\big) =
			\sigma(T(\sign(v)\tau)^p) = \sigma(I_{L^p(\T, \, \C^N)}) = \{1\},
		\end{equation*}
		i.e. $\sigma(T(t)) \subseteq \Gamma_q.$ On the other hand,
		\begin{equation*}
			\Gamma_q = e^{2 \pi i \tfrac{p}{q} \Z} =  e^{t 2 \pi i v \Z} = e^{t\sigma(-A)} \subseteq \sigma(T(t))
		\end{equation*}
		by \eqref{example_spectral_inclusion_theorem} and the fact that $p$ and $q$ are
		coprime ($e^{2\pi i p/q}$ is a primitive $q$-th root if and only if $p$ and $q$
		are coprime, which is a classical result from Algebra).
	\end{example}
	\begin{remark}
		The result on the spectrum of $T(t)$ for $t\geq0$ is an interesting result in
		itself. It gives precise information on the solvability of a class of difference
		equations on $L^p(\T),$ i.e. solvability in $L^p(\T)$ of the equation
		\begin{equation*}
		\lambda u(x) - u(x-tv) = \big( (\lambda-T(t))u \big)(x) = v(x) \quad \text{ for
			almost all } x\in\T,
		\end{equation*}
		given $\lambda \in \C$ and $v\in L^p(\T).$ Note that also the question of
		uniqueness of the solution can be answered precisely due to $\sigma_p(T(t)) =
		e^{t\sigma_p(-A)} = e^{t\sigma(-A)}$ by \Cref{spectral_theorem_point_spec}.
	\end{remark}
	
	\begin{example}[arbitrary $N\in\N$]\label{example_spectrum_transport_sg_N}
		The result from \Cref{example_spectral_analysis_transport_1d} can be naturally
		extended to the case $d=1, \, N \in \N$ and $B=0$ with non-vanishing transport
		directions $v_j \neq 0$ for $\indexone=1,\cdots,N.$ That is, $N$ species move
		along a circle without reacting with each other. We have
		\begin{equation*}
			-Au = -\colvec{v_1 u_1' \\ \vdots \\ v_N u_N'} \quad \text{ and } \quad
			(T(t)u)(x) = \colvec{u_1(x-tv_1) \\ \vdots \\ u_N(x-tv_N)}
		\end{equation*}
		with $D(-A) = W^{1,p}(\T, \, \C^N)$ by
		\Cref{corollary_domain_generator_one_dimensional}. Both the generator and the
		semigroup act on the $N$ different components independently and therefore, we get
		\begin{equation*}
			\sigma(-A) = \sigma_p(-A) = 2 \pi i \bigcup_{\substack{\indexone=1}}^N v_j \Z
		\end{equation*}
		and
		\begin{equation}\label{example_spectrum_T(t)}
			\sigma(T(t)) = \begin{cases}
			\S^1  &\exists \indexone=1,\cdots,N \text{ with } tv_\indexone \in \R
			\backslash \Q \\
			\cup_{\indexone=1}^N \Gamma_{q_\indexone}  &tv_\indexone =
			\tfrac{p_\indexone}{q_\indexone} \text{ with } \, p_\indexone \text{ and }
			q_\indexone \text{ coprime}
			\end{cases}.
		\end{equation}
		If we additionally assume that the semigroup is periodic, the first condition
		in \eqref{example_spectrum_T(t)} reduces to a condition on $tv_1$ because
		$tv_\indexone$ is irrational any $\indexone=1,\cdots,N$ if and only if $tv_1$ is irrational, see \Cref{lemma_jointly_periodic}.
	\end{example}
	
	Let us emphasize that the above computation verify the weak spectral mapping
	theorem
	\begin{equation*}
		\sigma(T(t)) =
		\begin{rcases}
		\begin{dcases}
		\S^1  &\exists \indexone=1,\cdots,N\colon tv_\indexone \in \R \backslash \Q
		\\
		\cup_{\indexone=1}^N \Gamma_{q_\indexone}  &tv_\indexone =
		\tfrac{p_\indexone}{q_\indexone} \text{ with } \, p_\indexone \text{ and } q_\indexone \text{
			coprime}
		\end{dcases}
		\end{rcases} = \bigcup_{\indexone=1}^N \overline{e^{t2\pi i v_\indexone \Z}} =
		\overline{e^{t\sigma(-A)}}
	\end{equation*}
	and show that the spectral mapping theorem
	\eqref{motivation_spectral_mapping_theorem} does not hold. The equality in
	\eqref{motivation_spectral_mapping_theorem} fails for all $t\geq0$ such that
	there exists $\indexone=1,\cdots,N$ with $tv_\indexone \in \R \backslash \Q.$ In
	addition, we see that the spectrum of the transport semigroup and its generator
	are independent of $1 \leq p < \infty.$
	
	\section{Spectral Analysis of the Generator}
	
	As mentioned in the introduction of the chapter, we start with a
	characterization of $\sigma(-A_2+B_2)$ for all dimensions $d\in\N$. The key
	observation, which will also be crucial for the proof of the weak
	spectral mapping theorem, is that Fourier transformation turns linear differential
	operators acting on $L^2(\T^d, \, \C^N)$ to \textit{matrix multiplication
		operators} on $l^2(\Z^d, \, \C^N)$, the space of the Fourier coefficients. Then, we
	can use existing theory of matrix multiplication operators from
	\cite{holderrieth1991matrix} and obtain a complete characterization of
	$\sigma(-A_2+B_2)$ in terms of roots of polynomials of degree $N$. \newline
	
	Consider the Fourier transformation $\mathcal{F}\colon L^2(\T^d, \, \C^N)
	\rightarrow l^2(\Z^d, \, \C^N)$ with
	\begin{equation}\label{def_fourier_transform}
		(\mathcal{F}u)(\indextwo) = \hat{u}(k) \qquad \text{ for } \indextwo \in \Z^d,
	\end{equation}
	where $\hat{u}(k) = (\hat{u}_1(\indextwo), \cdots, \hat{u}_N(\indextwo))^T \in
	\C^N$ are the unique Fourier coefficients
	\begin{equation*}
		\hat{u}(\indextwo) = \intT u(x) e^{-2\pi i \indextwo \cdot \mathbf{x}} \,
		d\mathbf{x}
	\end{equation*}
	with
	\begin{equation}\label{fourier_series}
		u(\mathbf{x}) = \sum_{\indextwo \in \Z^d} \hat{u}(\indextwo) e^{2\pi i
		\indextwo \cdot \mathbf{x}}.
	\end{equation}
	It is well known that the Fourier transformation defines an isometric
	isomorphism and that the Fourier series \eqref{fourier_series} of $u\in
	L^2(\T^d, \, \C^N)$ converges to $u$ in the $L^2(\T^d, \, \C^N)$ norm, see
	\cite[Theorem 8.20, p. 248]{folland1999real}. In particular, equation
	\begin{equation}\label{parseval}
		\intT |u(\mathbf{x})|^2 \, d\mathbf{x} = \sum_{\indextwo \in \Z^d}
		|\hat{u}(\indextwo)|^2
	\end{equation}
	is called Parseval's identity. \newline
	For our application, it is of great importance that Fourier series of
	$C^\infty(\T^d, \, \C^N)$ functions are very well-behaved: for any $u\in
	C^\infty(\T^d, \, \C^N)$ and all $m\in\N,$ there exists a constant
	$C=C(d,m,u)>0$ such that
	\begin{equation}\label{decay_fourier_smooth}
		|\hat{u}(\indextwo)| \leq C (1+|\indextwo|)^{-m}
	\end{equation}
	holds, see \cite[Theorem 3.3.3, p. 196]{grafakos2008classical}. Moreover, for
	$u\in C^\infty(\T^d, \, \C^N),$ the Fourier series \eqref{fourier_series}
	converges to $u$ in the $C^\infty$ norm by \cite[Inversion theorem for
	$C^\infty$ 5.4, p. 72]{roe1999elliptic}, i.e. all derivatives of the Fourier series of $u$ converge uniformly to the derivatives of $u$. \newline
	
	At a first glance, the connection of our problem to the abstract setting in
	\cite{holderrieth1991matrix} is not evident. We recall the setting from the
	paper to make it clear and comprehensible. \newline
	Let $(X, \, \Sigma, \, \mu)$ be a $\sigma$-finite measure space, let $1\leq p <
	\infty$ and define $E\coloneqq L^p(X, \, \C^N)$. Let $q\colon X \rightarrow
	\C^{N\times N}$ be a measurable function.
	
	\begin{definition}[{\cite[Definition 1]{holderrieth1991matrix}}]\label{def_matrix_mult_op}
		An operator $(\mathcal{A}_q, \, D(\mathcal{A}_q))$ defined on $L^p(X, \, \C^N)$
		by $\mathcal{A}_q\colon f \mapsto qf,$ i.e.
		\begin{equation}\label{matrix_mult_op}
			\mathcal{A}_q f(x) = q(x) f(x)
		\end{equation}
		for $x\in X$ and for all $f \in D(\mathcal{A}_q) = \left\{ f \in L^p(X, \,
		\C^N) \, \colon \, qf \in L^p(X, \, \C^N)\right\}$ is called matrix
		multiplication operator.
	\end{definition}
	
	\begin{proposition}[{\cite[Proposition 1]{holderrieth1991matrix}}]
		If $(\mathcal{A}_q, \, D(\mathcal{A}_q))$ is a matrix multiplication operator
		on $L^p(X, \, \C^N)$ with non-void resolvent set $\rho(\mathcal{A}_q)$, its
		spectrum is given by
		\begin{equation}\label{spec_matrix_mult_op}
			\sigma(\mathcal{A}_q) = \bigcap_{q' \in [q]} \overline{\bigcup_{x\in X}
				\sigma(q'(x))}.
		\end{equation}
		Here, $[q]$ is the equivalence class of all measurable functions $\mu$-a.e.
		equal to $q$.
	\end{proposition}
	
	In our application, we take $(X, \Sigma, \, \mu) = (\Z^d, \mathcal{P}(\Z^d), \,
	\#)$, where $\#$ is the counting measure and $p=2$. Then $L^2(X, \, \C^N) =
	l^2(\Z^d, \, \C^N), \, q\colon \Z^d \rightarrow \C^{N\times N}$ and
	\eqref{matrix_mult_op} reduces to
	\begin{equation}
		\mathcal{A}_q f(\indextwo) = q(\indextwo) f(\indextwo)
	\end{equation}
	for $\indextwo\in\Z^d$ and $f\in D(\mathcal{A}_q) = \left\{ f \in l^2(\Z^d, \,
	\C^N) \, \colon \, qf \in l^2(\Z^d, \, \C^N)\right\}.$ Secondly,
	\eqref{spec_matrix_mult_op} comes down to
	\begin{equation}\label{spec_matrix_mult_op_my_case}
		\sigma(\mathcal{A}_q) = \overline{\bigcup_{\indextwo\in \Z^d}
		\sigma(q(\indextwo))}
	\end{equation}
	because functions on $\Z^d$ coinciding $\#$-a.e. are equal. Notice that
	$q(\indextwo)$ is a matrix and its spectrum is simply given by the roots of
	$\lambda \mapsto \det(\lambda I_{\C^{N\times N}} - q(\indextwo))$.\newline

	The following lemma formalizes the above mentioned idea and can for
	instance be applied to every reasonable linear differential operator on
	$L^2(\T^d, \, \C^N)$ in order to compute its spectrum.
	
	\begin{proposition}\label{fourier_diagram_spectrum}
		Let $(\mathcal{A}, \, D(\mathcal{A}))$ be a densely defined and closed operator
		on $L^2(\T^d, \, \C^N)$ with non-void resolvent set $\rho(\mathcal{A})$. Assume that there exists a function $q\colon
		\Z^d \to \C^{N\times N}$ such that $\mathcal{F}(D(\mathcal{A})) =
		D(\mathcal{A}_q)$ and
		\begin{equation*}
			\begin{tikzcd}
			\LtwoTCN \arrow[r, hookleftarrow] & D(\mathcal{A}) \arrow[r, "\mathcal{A}"] &
			\LtwoTCN \\
			\ltwoZCN \arrow[r, hookleftarrow] \arrow[u, leftarrow, "\mathcal{F}"] &
			D(\mathcal{A}_q) \arrow[r, "\mathcal{A}_q"] \arrow[u, leftarrow, "\mathcal{F}"]
			& \ltwoZCN \arrow[u, "\mathcal{F}^{-1}"]
			\end{tikzcd}
		\end{equation*}
		commutes for its induced matrix multiplication operator $(\mathcal{A}_q, \,
		D(\mathcal{A}_q)).$ Then it holds
		\begin{equation*}
			\sigma(\mathcal{A}) = \sigma(\mathcal{A}_q) = \overline{\bigcup_{\indextwo\in
			\Z^d} \sigma(q(\indextwo))}.
		\end{equation*}
	\end{proposition}
	\begin{proof}
		Let $\lambda \in \C.$ By assumption,
		\begin{equation*}
			\lambda I_{L^2(\T^d, \, \C^N)} - \mathcal{A} \colon D(\mathcal{A})
			\longrightarrow L^2(\T^d, \, \C^N)
		\end{equation*}
		is invertible if and only if
		\begin{equation*}
			\mathcal{F} \big(\lambda I_{L^2(\T, \, \C^N)} - \mathcal{A}\big)
			\mathcal{F}^{-1} = \lambda I_{\ltwoZCN} - \mathcal{A}_q \colon D(\mathcal{A}_q)
			\longrightarrow \ltwoZCN
		\end{equation*}
		is invertible. The second equality follows from
		\eqref{spec_matrix_mult_op_my_case}.
	\end{proof}
	
	Before applying \Cref{fourier_diagram_spectrum} to the generator $(-A_2+B_2, \,
	D(-A_2))$ of the transport-reaction semigroup $(R(t))_{t\geq0}$, we need to
	characterize the domain $D(-A_2)$ and the application of $-A_2+B_2$ in terms of Fourier coefficients. Let us start with the latter.
	
	\begin{proposition}\label{proposition_generator_application_fourier}
		Let $(-A_2+B_2, \, D(-A_2))$ be the generator of the transport-reaction
		semigroup $(R_2(t))_{t\geq0}$ on $\LtwoTCN.$ Then
		\begin{equation*}
			\big((-A_2+B_2)u\big)(\mathbf{x}) = \sum_{\indextwo \in \Z^d} M(\indextwo)
			\hat{u}(\indextwo) e^{2\pi i \indextwo \cdot \mathbf{x}}
		\end{equation*}
		holds for all $u\in D(-A_2)$, where $M(\indextwo) \in \C^{N\times N}$ are the matrices defined by
		\begin{equation}
			M(\indextwo) \coloneqq -2\pi i 
			\begin{pmatrix}
			\indextwo \cdot \mathbf{v_1} & & \\
			& \ddots & \\
			& & \indextwo \cdot \mathbf{v_N}
			\end{pmatrix}
			+ B \eqqcolon - 2\pi i \mathbf{V}(\indextwo) + B.
		\end{equation}
	\end{proposition}
	\begin{proof}
		The first step is to consider $u\in C^\infty(\T^d, \, \C^N).$ Using \Cref{lemma_A_is_an_extension}, a calculation shows
		\begin{equation}\label{auxiliary_eq34}
			((-A_2+B_2)u)(\mathbf{x}) = -\colvec{\mathbf{v_1} \cdot \nabla u_1(\mathbf{x}) \\ \vdots \\ \mathbf{v_N} \cdot \nabla u_N(\mathbf{x})} + Bu(\mathbf{x}) = \sum_{\indextwo \in \Z^d} M(\indextwo)
			\hat{u}(\indextwo) e^{2\pi i \indextwo \cdot \mathbf{x}}.
		\end{equation}
		Every technicality and convergence of the Fourier series works out nicely because the Fourier series of $u$ converges to $u$ in the $C^\infty$ norm, see \eqref{decay_fourier_smooth} and its subsequent comment. \newline
		Now, take an arbitrary function $u\in D(-A_2).$ By \Cref{corollary_Ccinfty_core_A+B}, there exists a sequence $(u^{(n)})_{n\in\N} \subset C_c^\infty(\T^d, \, \C^N)$ such that $u^{(n)}$ converges to $u$ in the graph norm of $-A_2+B_2.$ Let $v\coloneqq (-A_2+B_2)u \in \LtwoTCN$ with
		\begin{equation*}
			v(\mathbf{x}) = \sum_{\indextwo \in \Z^d} \hat{v}(\indextwo) e^{2\pi i \indextwo \cdot \mathbf{x}}.
		\end{equation*}
		Parseval's identity \eqref{parseval} and \eqref{auxiliary_eq34} imply that the convergence of $u^{(n)}$ to $u$ in the graph norm is equivalent to
		\begin{equation*}
			\sum_{\indextwo \in \Z^d} |\hat{u^{(n)}}(\indextwo) - \hat{u}(\indextwo)|^2 \longrightarrow 0
		\end{equation*}
		and
		\begin{equation*}
			\sum_{\indextwo \in \Z^d} |M(\indextwo)\hat{u^{(n)}}(\indextwo) - \hat{v}(\indextwo)|^2 \longrightarrow 0
		\end{equation*}
		as $n\to\infty.$ In particular, we obtain
		\begin{align*}
			\hat{u^{(n)}}(\indextwo^*) \longrightarrow \hat{u}(\indextwo^*) \qquad \text{ and } \qquad M(\indextwo^*)\hat{u^{(n)}}(\indextwo^*) \longrightarrow \hat{v}(\indextwo^*)
		\end{align*}
		as $n\to\infty$ for every fixed $k^*\in\Z^d,$ which implies $\hat{v}(\indextwo^*) = M(\indextwo^*)\hat{u}(\indextwo^*).$ Since $k^*\in\Z^d$ was chosen arbitrarily, we arrive at
		\begin{equation*}
			((-A_2+B_2)u)(\mathbf{x}) = v(\mathbf{x}) = \sum_{\indextwo \in \Z^d} \hat{v}(\indextwo) e^{2\pi i \indextwo \cdot \mathbf{x}} = \sum_{\indextwo \in \Z^d} M(\indextwo)\hat{u}(\indextwo) e^{2\pi i \indextwo \cdot \mathbf{x}}.
		\end{equation*}
	\end{proof}
	
	With the knowledge from \Cref{proposition_generator_application_fourier}, it is easy to characterize the domain $D(-A_2+B_2)=D(-A_2)$ in terms of Fourier coefficients.
	
	\begin{proposition}\label{proposition_domain_generator_fourier}
		Let $(-A_2+B_2, \, D(-A_2))$ be the generator of the transport-reaction
		semigroup $(R_2(t))_{t\geq0}$ on $\LtwoTCN.$ Then it holds
		\begin{align*}
			\mathcal{F}(D(-A_2)) &= \bigg\{ \hat{u} \in \ltwoZCN \, \colon \,
			\sum_{\indextwo \in \Z^d} |M(\indextwo)\hat{u}(\indextwo)|^2 < \infty \bigg\} \\
			&= \bigg\{ \hat{u} \in \ltwoZCN \, \colon \,
			\sum_{\indextwo \in \Z^d} |\mathbf{V}(\indextwo)\hat{u}(\indextwo)|^2 < \infty \bigg\},
		\end{align*}
		where $M(\indextwo)$ and $ \mathbf{V}(\indextwo) \in \C^{N \times N}$ are the matrices given by
		\begin{equation}
			M(\indextwo) = -2\pi i 
			\begin{pmatrix}
			\indextwo \cdot \mathbf{v_1} & & \\
			& \ddots & \\
			& & \indextwo \cdot \mathbf{v_N}
			\end{pmatrix}
			+ B = - 2\pi i \mathbf{V}(\indextwo) + B.
		\end{equation}
	\end{proposition}
	\begin{proof}
		If $u\in D(-A_2),$ the function $(-A_2+B_2)u \in \LtwoTCN$ is given by
		\begin{equation}
			\big((-A_2+B_2)u\big)(\mathbf{x}) = \sum_{\indextwo \in \Z^d} M(\indextwo)
			\hat{u}(\indextwo) e^{2\pi i \indextwo \cdot \mathbf{x}}
		\end{equation}
		due to \Cref{proposition_generator_application_fourier}. Therefore, Parseval's identity \eqref{parseval} yields
		\begin{equation*}
			\sum_{\indextwo \in \Z^d} |M(\indextwo)\hat{u}(\indextwo)|^2 < \infty.
		\end{equation*}
		Conversely, let $\hat{u}\in \ltwoZCN$ with
		\begin{equation}\label{auxiliary_eq35}
			\sum_{\indextwo \in \Z^d} |M(\indextwo)\hat{u}(\indextwo)|^2 < \infty.
		\end{equation}
		We need to show that the function $u\in\LtwoTCN$ given by
		\begin{equation*}
			u(\mathbf{x}) = \sum_{\indextwo \in \Z^d} \hat{u}(\indextwo) e^{2\pi i \indextwo \cdot \mathbf{x}}
		\end{equation*}
		lies in the domain $D(-A_2)$ of the generator $-A_2+B_2.$ To this end, consider the sequence $(u^{(n)})_{n\in\N}$ defined by
		\begin{equation*}
			u^{(n)}(\mathbf{x}) \coloneqq \sum_{|\indextwo| \leq n} \hat{u}(\indextwo) e^{2\pi i \indextwo \cdot \mathbf{x}}.
		\end{equation*}
		Notice that $u^{(n)} \in C^\infty(\T^d, \, \C^N)$ converges to $u$ in $\LtwoTCN$ by Parseval's identity \eqref{parseval}. Moreover,
		\begin{equation*}
			\big((-A_2+B_2)u^{(n)}\big)(\mathbf{x}) = \sum_{|\indextwo| \leq n} M(\indextwo)
			\hat{u}(\indextwo) e^{2\pi i \indextwo \cdot \mathbf{x}} \longrightarrow \sum_{\indextwo \in \Z^d} M(\indextwo)
			\hat{u}(\indextwo) e^{2\pi i \indextwo \cdot \mathbf{x}}
		\end{equation*}
		in $\LtwoTCN$ as $n\to\infty$  due to \eqref{auxiliary_eq35} and Parseval's identity \eqref{parseval}. This yields $u\in D(-A_2)$ by the closedness of $(-A_2+B_2, \, D(-A_2)),$ which is a consequence of the Hille-Yosida \Cref{hille-yosida}. \newline
		The second equality in the proposition is an immediate consequence of the triangle inequality in $\ltwoZCN.$
	\end{proof}
	\begin{remark}
		Complementary to \Cref{theorem_domain_generator_higher_dimensions}, \Cref{proposition_domain_generator_fourier} shows that the domain $D(-A_2)$ is in general a larger subspace than $H^1(\T^d, \, \C^N)$ for dimensions $d\geq2,$ even if the transport directions $\mathbf{v_1},\cdots,\mathbf{v_N}\in\R^d$ are non-vanishing.
	\end{remark}
	
	Finally, we can rigorously characterize the spectrum of the generator of the transport-reaction semigroup $(R_2(t))_{t\geq0}.$
	
	\begin{theorem}\label{char_spectrum_generator}
		Let $(-A_2+B_2, \, D(-A_2))$ be the generator of the transport-reaction
		semigroup $(R_2(t))_{t\geq0}$ on $\LtwoTCN.$ Then it holds
		\begin{equation*}
			\sigma(-A_2+B_2) = \overline{\bigcup_{\indextwo\in \Z^d}
			\sigma(M(\indextwo))},
		\end{equation*}
		where $M(\indextwo) \in \C^{N\times N}$ is the matrix
		\begin{equation}\label{def_Mk}
			M(\indextwo) = -2\pi i 
			\begin{pmatrix}
			\indextwo \cdot \mathbf{v_1} & & \\
			& \ddots & \\
			& & \indextwo \cdot \mathbf{v_N}
			\end{pmatrix}
			+ B = - 2\pi i \mathbf{V}(\indextwo) + B.
		\end{equation}
	\end{theorem}
	\begin{proof}
		Let $(\mathcal{A}_M, \, D(\mathcal{A}_M))$ be the matrix multiplication
		operator on $\ltwoZCN$ induced by $M = (k\mapsto M(k))$ (in the sense of \Cref{def_matrix_mult_op}). Notice that \Cref{proposition_domain_generator_fourier} shows $\mathcal{F}(D(-A_2)) = D(\mathcal{A}_M)$. Furthermore, the diagram
		\begin{equation*}
			\begin{tikzcd}
			\LtwoTCN \arrow[r, hookleftarrow] & D(-A_2) \arrow[r, "-A_2+B_2"] &
			\LtwoTCN \\
			\ltwoZCN \arrow[r, hookleftarrow] \arrow[u, leftarrow, "\mathcal{F}"] &
			D(\mathcal{A}_M) \arrow[r, "\mathcal{A}_M"] \arrow[u, leftarrow, "\mathcal{F}"]
			& \ltwoZCN \arrow[u, "\mathcal{F}^{-1}"]
			\end{tikzcd}
		\end{equation*}
		commutes by \Cref{proposition_generator_application_fourier}. The resolvent set $\rho(-A_2+B_2)$ is non-empty by the Hille-Yosida \Cref{hille-yosida} and all properties together imply that the assumptions of \Cref{fourier_diagram_spectrum} are fulfilled. An application of \Cref{fourier_diagram_spectrum} finishes the proof.
	\end{proof}
	
	A simple observation in form of the next lemma shows that the spectrum $\sigma(-A_2+B_2)$ is symmetric to the real line.
	
	\begin{lemma}\label{spectrum_symmetric}
		Let $k\in\Z^d$ and let $M(\indextwo) \in \C^{N\times N}$ be the matrix
		\begin{equation*}
			M(\indextwo) = -2\pi i 
			\begin{pmatrix}
			\indextwo \cdot \mathbf{v_1} & & \\
			& \ddots & \\
			& & \indextwo \cdot \mathbf{v_N}
			\end{pmatrix}
			+ B = - 2\pi i \mathbf{V}(\indextwo) + B.
		\end{equation*}
		Then
		\begin{equation*}
			\sigma(M(-k)) = \overline{\sigma(M(k))}
		\end{equation*}
		holds for all $k\in\Z^d,$ where the overline denotes complex conjugation.
	\end{lemma}
	\begin{proof}
		This is a direct consequence of
		\begin{equation*}
			M(-\indextwo) = \overline{M(\indextwo)}.
		\end{equation*}
	\end{proof}
	Carrying over the characterization of $\sigma(-A_2+B_2)$ from $p=2$ to $1\leq p < \infty$ in all dimensions $d\in\N$ seems to be a very difficult task. In the following, we mostly restrict ourselves to the one dimensional case $d=1$ and assume that the transport directions $v_1,\cdots,v_N \in \R$ are non-vanishing, i.e. $v_1,\cdots,v_N\neq0.$ All proofs rely on the fact that $(-A_p+B_p, \, D(-A_p))$ has compact resolvent in these simple cases. Our assumptions are necessary for this approach, since $(-A_p+B_p, \, D(-A_p))$ generally fails to have compact resolvent if one transport direction vanishes or if the model is considered in higher dimensions. \newline
	
	Let us recall that we have $D(-A_p)=W^{1,p}(\T, \, \C^N)$ under the assumptions $d=1$ and non-vanishing transport directions $v_1,\cdots,v_N\neq0$ by \Cref{corollary_domain_generator_one_dimensional}. \newline

	\begin{theorem}\label{theorem_generator_has_compact_resolvent_spectrum=point_spectrum}
		Let $1\leq p < \infty.$ Let $d=1$ and let the transport directions $v_1,
		\cdots, v_N$ be non-vanishing, i.e. $v_1,\cdots,v_N\neq0$. Then the generator $(-A_p+B_p, \, W^{1,p}(\T, \, \C^N))$ of $(R_p(t))_{t\geq0}$ has
		compact resolvent, i.e. $R(\lambda, \, -A_p+B_p)$ is compact for one (and hence all)
		$\lambda \in \rho(-A_p+B_p).$ \newline
		In particular, the following properties hold true:
		\begin{enumerate}
			\item
			$\sigma(-A_p+B_p)=\sigma_p(-A_p+B_p)$ contains at
			most countably many eigenvalues $\lambda_\indexone.$
			\item
			If $\sigma(-A_p+B_p)$ is infinite, then $|\lambda_\indexone| \to \infty $ as
			$\indexone \to \infty.$
			\item
			For all $\lambda \in \sigma(-A_p+B_p), $ the operator $\lambda_\indexone I_{L^p(\T, \, \C^N)}
			- (-A_p+B_p)$ has closed range and $\emph{dim}\ker(\lambda_\indexone I_{L^p(\T, \, \C^N)}
			- (-A_p+B_p)) = \emph{codim}\range(\lambda_\indexone I_{L^p(\T, \, \C^N)}
			- (-A_p+B_p))< \infty.$
		\end{enumerate}
	\end{theorem}
	\begin{proof}
		We start by showing that the generator has compact resolvent. By
		\cite[\Romannum{2}. 4.25, p. 117]{engel2001one}, it is equivalent to show that
		the canonical embedding
		\begin{equation*}
			(W^{1,p}(\T, \, \C^N), \, \|\cdot\|_{-A+B}) \longhookrightarrow (\LpCn, \, \|\cdot\|_{\LpCn})
		\end{equation*}
		is compact. As the graph norms of $(-A+B, \, W^{1,p}(\T, \, \C^N))$ and $(-A, \, W^{1,p}(\T, \, \C^N))$ are
		equivalent by \Cref{theorem_semigroup_generated_by_-A+B}, it is sufficient to
		show compactness of the embedding
		\begin{equation*}
			(W^{1,p}(\T, \, \C^N), \, \|\cdot\|_{-A}) \longhookrightarrow (\LpCn, \, \|\cdot\|_{\LpCn}).
		\end{equation*}
		Notice that the proof of \Cref{corollary_domain_generator_one_dimensional} also implies equivalence of the norms $\|\cdot\|_{-A}$ and \mbox{$\|\cdot\|_{W^{1,p}(\T, \, \C^N)}$} on
		$W^{1,p}(\T, \, \C^N).$ The claim now follows from the classical
		Rellich-Kondrachov embedding 
		\begin{equation*}
			W^{1,p}(\T, \, \C^N) \subset \subset L^p(\T, \, \C^N),
		\end{equation*}
		see \cite[Theorem 9.16, p. 285]{brezis2010functional}. Properties $(i)$ to $(iii)$ are a consequence of \cite[Theorem 2.16, p. 27]{schnaubelt2012lecture} because the spectrum $\sigma(-A_p+B_p)$ is non-empty. The easiest way to see is is to take an eigenvalue $\lambda\in\C$ of $B$ with eigenvector $z\in\C^N\backslash\{0\}.$ Then, $u\equiv z \in W^{1,p}(\T, \, \C^N)$ satisfies $(-A_p+B_p)u \equiv Bz = \lambda z = \lambda u.$
	\end{proof}
	
	\begin{remark}\label{remark_spectrum_generator_p_equal_infty_point_spec}
		Below, it will become important that the theorem also holds true for the generator considered on $L^\infty(\T, \, \C^N)$ with $D(-A_\infty) = W^{1,\infty}(\T, \, \C^N)$ from \Cref{proposition_adjoint_generator_peq1}. This follows from
		$W^{1,\infty}(\T, \C^N) \subset \subset L^\infty(\T, \, \C^N)$ due to
		$W^{1,\infty}(\T, \C^N) \subseteq W^{1,p}(\T, \, \C^N) \subset \subset C(\T, \,
		\C^N) \subseteq L^\infty(\T, \, \C^N)$, where $p>1$ can be chosen arbitrarily.
	\end{remark}
	
	The assumption of non-vanishing transport directions $v_1,\cdots,v_N\neq0$ cannot be dropped. By definition \eqref{generator_by_definition} of the domain of the
	generator, there are no restrictions on $u_\indexone$ if $v_\indexone=0$ and the
	component could be an arbitrary function in $L^p(\T).$ Consequently, the domain of the operator would no longer be compactly embedded in $L^p(\T, \, \C^N).$\newline
	Concerning higher dimensions $d\geq2$, we observed in \Cref{theorem_domain_generator_higher_dimensions} that the domain
	of the generator of the semigroup $(R(t))_{t\geq0}$ is less regular if one component $\indexone$ is transport periodic in the sense of \Cref{definition_transport_periodic}. This suggests that the generator also fails to have compact resolvent. The next lemma shows this negative result.
	
	\begin{proposition}\label{proposition_no_compact_resolvent}
		Let $1 \leq p < \infty.$ Let $d\geq2$ and let the transport directions
		$\mathbf{v_1}, \cdots, \mathbf{v_N}$ be non-vanishing, i.e. $\mathbf{v_1},\cdots,\mathbf{v_N} \neq 0$. Assume that one component
		$\indexone$ is transport periodic. Then the generator $(-A+B, \, D(-A))$ of
		$(R(t))_{t\geq0}$ does not have compact resolvent.
	\end{proposition}
	\begin{proof}
		We show that the embedding
		\begin{equation*}
			(D(-A), \, \|\cdot\|_{-A}) \longhookrightarrow (\LpCn, \, \|\cdot\|_{\LpCn}).
		\end{equation*}
		is not compact. This is equivalent to our statement by \cite[\Romannum{2}.
		4.25, p. 117]{engel2001one} and the equivalence of the graph norms $\|\cdot\|_{-A+B}$ and $\|\cdot\|_{-A}$, which was shown in \Cref{theorem_semigroup_generated_by_-A+B}. \newline
		To this end, we construct a bounded sequence $(u^{(n)})_{n\in\N} \subset (D(-A), \, \|
		\cdot \|_{-A})$, which fails to have a convergent subsequence in $\LpCn.$ The
		idea is very similar to the proof of
		\Cref{theorem_domain_generator_higher_dimensions}. Let $\mathbf{c} = (\half, \cdots, \half) \in \T^d$ be the center of the torus.
		By assumption, the curve
		\begin{equation*}
			\gamma\colon [0,\tau_j] \rightarrow \T^d \quad \text{with} \quad \gamma(t)
			\coloneqq \mathbf{c} + t \mathbf{v_j}
		\end{equation*}
		fulfills $\gamma(0)=\gamma(\tau_j).$ Let $K=\range(\gamma)$ be the range of
		$\gamma$ and consider the sets $U^{(n)}$ given by
		\begin{equation*}
			U^{(n)} = \{\mathbf{x}\in \T^d \, \colon \, \dist(\mathbf{x}, \, K) < \tfrac{1}{n} \}.
		\end{equation*}
		We define 
		\begin{equation*}
			u_\indexone^{(n)}(\mathbf{x}) \coloneqq \begin{cases}
			n^{(d-1)/p} & \mathbf{x} \in U^{(n)}, \\
			0 & \text{otherwise}
			\end{cases}
		\end{equation*}
		and $u_\indexfour^{(n)} \equiv 0$ for all
		$\indexfour=1,\cdots, N$ with $\indexfour\neq \indexone.$ Notice that $u \in \LpCn$ with
		\begin{equation}\label{proposition_no_compact_resolvent_LP_BOUND}
			\LpCnnorm{p}{u^{(n)}}^p = n^{d-1} \tau_\indexone |\mathbf{v_\indexone}| \frac{\pi^{(d-1)/2}}{\Gamma(\tfrac{d+1}{2})} n^{1-d} = \tau_\indexone |\mathbf{v_\indexone}| \frac{\pi^{(d-1)/2}}{\Gamma(\tfrac{d+1}{2})}
		\end{equation}
		for sufficiently large $n\in\N$. The exact formula is a consequence of the volume of the $(d-1)$-dimensional Euclidean ball with radius $1/n$. In addition, $u_\indexone^{(n)}$ is constant along direction $\mathbf{v_\indexone},$ i.e.
		\begin{equation*}
			u_\indexone^{(n)}(\cdot - t\mathbf{v_\indexone}) - u_\indexone^{(n)}(\cdot) \equiv 0.
		\end{equation*}
		This implies $u^{(n)} \in D(-A)$ with $-Au^{(n)} = 0$ for all $n\in\N$ and boundedness of the sequence in	$(D(-A), \, \|\cdot\|_{-A}).$ \newline
		On the other hand, $u^{(n)} \to 0$ pointwise almost everywhere because of $d\geq2$. If a
		subsequence of $(u^{(n)})_{n\in\N}$ converged in $\LpCn$, a subsubsequence would converge to $0$ in $\LpCn$. This contradicts
		\eqref{proposition_no_compact_resolvent_LP_BOUND}.
	\end{proof}
	
	\begin{remark}
		We refer to \Cref{theorem_domain_generator_higher_dimensions} and \Cref{fig:theorem_domain_generator_higher_dimensions} for an illustration of the idea.
	\end{remark}

	\begin{theorem}\label{theorem_spectrum_generator}
		Let $1\leq p < \infty.$ Let $d=1$ and let the transport directions $v_1,
		\cdots, v_N$ be non-vanishing, i.e. $v_1,\cdots,v_N\neq0.$ Let $(-A_p+B_p, \, W^{1,p}(\T, \, \C^N))$ be the generator of the transport-reaction semigroup $(R_p(t))_{t\geq0}$ on $L^p(\T, \, \C^N).$ Then it holds
		\begin{equation*}
			\sigma(-A_p+B_p) = \sigma_p(-A_p+B_p) = \overline{\bigcup_{k\in \Z} \sigma(M(k))},
		\end{equation*}
		where $M(k)\in \C^{N\times N}$ is the matrix
		\begin{equation*}
			M(\indextwo) = -2\pi i \indextwo
			\begin{pmatrix}
			v_1  & & \\
			& \ddots & \\
			& & v_N
			\end{pmatrix}
			+ B \eqqcolon - 2\pi i \indextwo V + B.
		\end{equation*}
		In particular, the spectrum is independent of $p$.
	\end{theorem}
	\begin{proof}
		Let us recall
		\begin{equation*}
			(-A_p+B_p)u = -\colvec{v_1 u_1^\prime \\ \vdots \\ v_N u_N^\prime} + Bu
		\end{equation*}
		from \Cref{corollary_domain_generator_one_dimensional}. We have
		\begin{equation}\label{auxeq4}
		\restr{-A_p+B_p}{W^{1,p}(\T, \, \C^N) \cap W^{1,p'}(\T, \, \C^N)} = \restr{-A_{p'}+B_{p'}}{W^{1,p}(\T, \, \C^N) \cap W^{1,p'}(\T, \, \C^N)}
		\end{equation}
		for all $1\leq p, \, p' \leq \infty$ by \Cref{corollary_domain_generator_one_dimensional}. It should be
		noted that we also defined the generator for $p=\infty$ with $D(-A_\infty+B_\infty) = W^{1,\infty}(\T, \, \C^N)$ in
		\eqref{def_transport_reaction_gen_p_eq_infty} and \eqref{def_domain_transport_reaction_gen_p_eq_infty}, cf. \Cref{proposition_adjoint_generator_peq1}. The proof includes this case because we apply an argument with the adjoints of the
		generators in Step 3. In fact, our proof shows that the theorem is also valid for
		$p=\infty$ but we do not need this result in the sequel. \newline
		
		Step 1: $\sigma(-A_{p'}+B_{p'}) \subseteq \sigma(-A_p+B_p)$ for all $1\leq p
		\leq p' \leq \infty.$ \newline
		Let $1\leq p \leq p' \leq \infty$ and let $\lambda \in \sigma(-A_{p'}+B_{p'}).$
		We know $\sigma(-A_{p'}+B_{p'}) = \sigma_p(-A_{p'}+B_{p'})$ from
		\Cref{theorem_generator_has_compact_resolvent_spectrum=point_spectrum} and
		\Cref{remark_spectrum_generator_p_equal_infty_point_spec}.	Hence, $\lambda$ is
		an eigenvalue and there exists $u\in W^{1,p'}(\T, \, \C^N) \backslash \{0\}$ with
		$(-A_{p'}+B_{p'})u = \lambda u$. The circle $\T$ has finite mass, so we have $W^{1,p'}(\T, \, \C^N) \subseteq W^{1,p}(\T, \, \C^N)$ and consequently
		\begin{equation*}
			(-A_p+B_p)u = (-A_{p'}+B_{p'})u = \lambda u
		\end{equation*}
		by \eqref{auxeq4}. In words, $\lambda$ is also an eigenvalue of $-A_p+B_p.$
		\newline
		
		Step 2: $\sigma(-A_p+B_p) = \sigma(-A_2+B_2)$ for all $2\leq p \leq \infty.$
		\newline
		Let $2\leq p \leq \infty.$ One direction follows from Step 1. It is left to
		show $\sigma(-A_2+B_2) \subseteq \sigma(-A_p+B_p).$ \Cref{char_spectrum_generator} implies
		\begin{equation*}
			\sigma(-A_2+B_2) = \overline{\bigcup_{k\in \Z} \sigma(M(k))}
		\end{equation*}
		where
		\begin{equation*}
			M(k) = -2\pi i \indextwo 
			\begin{pmatrix}
			v_1 & & \\
			& \ddots & \\
			& & v_N
			\end{pmatrix}
			+ B = - 2\pi i \indextwo V + B.
		\end{equation*}
		The spectrum $\sigma(-A_p+B_p)$ is closed and therefore, we only need
		to prove $\sigma(M(\indextwo)) \subseteq \sigma(-A_p+B_p)$ for all $\indextwo
		\in \Z.$ Let $\indextwo\in\Z$ and let $\lambda\in\sigma(M(k)).$ Then there
		exists $z \in \C^N \backslash \{0\}$ with $M(k)z = \lambda z$ and we define
		\begin{equation*}
			u \coloneqq z e^{2\pi i \indextwo x}.
		\end{equation*}
		Clearly, $u \in W^{1,p}(\T, \, \C^N) \cap H^1(\T, \, \C^N), \, u \neq 0$ and \eqref{auxeq4}, together with \Cref{proposition_generator_application_fourier}, implies
		\begin{equation*}
			(-A_p+B_p)u = (-A_2+B_2)u = M(k) z e^{2\pi i \indextwo x} = \lambda z e^{2\pi
			i \indextwo x} = \lambda u.
		\end{equation*}
		We obtain $\lambda \in \sigma(-A_p+B_p)$ and Step 2 follows. \newline
		
		Step 3: $\sigma(-A_p+B_p) = \sigma(-A_2+B_2)$ for all $1\leq p \leq 2.$
		\newline
		Let $1\leq p \leq 2.$ The whole idea is to apply the classical result that
		taking the adjoint does not change the spectrum; see for example \cite[Theorem
		1.24, p. 12]{schnaubelt2012lecture}. Let $q$ be the dual exponent of $p$. We
		have seen $(-A_p+B_p)^* = A_q + B_q^T$ in
		\Cref{proposition_adjoint_semigroup_pneq1} and
		\Cref{proposition_adjoint_generator_peq1} for the cases $1<p\leq2$ and $p=1$ respectively. For $1\leq p \leq 2,$ the dual
		exponent fulfills $2\leq q \leq \infty$ and an application of Step 2 yields
		\begin{align*}
		\sigma(-A_p+B_p) = \sigma((-A_p+B_p)^*) = \sigma(A_q+B_q^T) = \sigma(A_2+B_2^T) = \sigma((A_2+B_2^T)^*) = \sigma(-A_2+B_2).
		\end{align*}
		Notice that the adjoint generators are still of the class covered throughout the whole chapter, since the different sign for $A$ just changes the direction of movement and $B^T$ is still a real $N\times N$ matrix.
	\end{proof}
	
	We continue our study of the spectrum of the generator in \Cref{chapter_pattern_formation}, where we analyze the asymptotics of the eigenvalues of $M(k)$ as $|k|\to\infty.$ The results from \Cref{chapter_pattern_formation} give deep insights into the qualitative behavior of solutions to \eqref{linear_transport_reaction_eq}.
	
	\newpage
	
	\section{Weak Spectral Mapping Theorems}
	
	The idea to use existing theory of matrix multiplication operators is
	sufficient to prove a weak spectral mapping theorem for $p=2$ in all dimensions $d\in\N$. Our strategy is to not only apply \Cref{fourier_diagram_spectrum} for the generator
	$(-A_2+B_2, \, D(-A_2))$ of the semigroup, but also for each semigroup operator
	$R_2(t)$ itself. \newline
	Afterwards, we use the methods from \cite{latrach2008weak} to demonstrate mathematical tools one can use to extend the result to different $p.$ Given that the spectrum $\sigma(-A_p+B_p)$ is independent of $1\leq p <\infty$, the ideas from \cite{latrach2008weak} show that the weak spectral mapping theorem for all $p$ is essentially equivalent to the weak spectral mapping theorem for $p=1.$  \newline
	
	In particular, assuming that the $C_0$-group $(R_p(t))_{t\in\R}$ on $L^p(\T, \, \C^N)$ is positive, abstract theory from \cite{arendt1984spectral} can be used to show a weak spectral mapping theorem for all $1\leq p <\infty$. Unfortunately, the group is positive if and only if $B$ is a	diagonal matrix, see \Cref{corollary_group_positive}. \newline
	
	A weak spectral mapping theorem for $1\leq p <\infty, \, d\in\N$ and all matrices $B\in\R^{N\times N}$ remains an open problem.
	
	\begin{proposition}\label{proposition_semigroup_fourier_representation}
		Let $(R_2(t))_{t\geq0}$ be the transport-reaction semigroup on $\LtwoTCN$ generated by $(-A_2+B_2, \, D(-A_2))$. Then
		\begin{equation}\label{auxeq5}
			(R_2(t)u)(\mathbf{x}) = \sum_{\indextwo \in \Z^d} e^{tM(\indextwo)} \hat{u}(\indextwo)
			e^{2\pi i \indextwo \cdot \mathbf{x}}
		\end{equation}
		holds for all $u\in \LtwoTCN.$ Here, $M(\indextwo)$ is the matrix
		\begin{equation}
			M(\indextwo) = -2\pi i 
			\begin{pmatrix}
			\indextwo \cdot \mathbf{v_1} & & \\
			& \ddots & \\
			& & \indextwo \cdot \mathbf{v_N}
			\end{pmatrix}
			+ B = - 2\pi i \mathbf{V}(\indextwo) + B.
		\end{equation}
	\end{proposition}
	\begin{proof}
		Let us recall that a function $u\in\LtwoTCN$ is given by its Fourier series
		\begin{equation*}
			u(\mathbf{x}) = \sum_{\indextwo \in \Z^d} \hat{u}(\indextwo) e^{2\pi i
			\indextwo \cdot \mathbf{x}},
		\end{equation*}
		where $(\hat{u}(\indextwo))_{\indextwo\in\Z^d}$ are the unique Fourier coefficients of $u$. Moreover, we have already proven
		\begin{equation*}
			\big((-A_2+B_2)u\big)(\mathbf{x}) = \sum_{\indextwo \in \Z^d} M(\indextwo)
			\hat{u}(\indextwo) e^{2\pi i \indextwo \cdot \mathbf{x}}
		\end{equation*}
		for all $u\in D(-A_2)$ in \Cref{proposition_generator_application_fourier} and 
		\begin{align}\label{auxiliary_eq37}
			\mathcal{F}(D(-A_2)) = \bigg\{ \hat{u} \in \ltwoZCN \, \colon \,
			\sum_{\indextwo \in \Z^d} |M(\indextwo)\hat{u}(\indextwo)|^2 < \infty \bigg\}
		\end{align}
		in \Cref{proposition_domain_generator_fourier}. The statement of the proposition is very intuitive but a rigorous justification is a priori unclear and a little technical. \newline
		
		We start by estimating the matrix norm of $e^{tM(k)}.$ According to definition,
		$M(\indextwo) = -2\pi i \mathbf{V}(\indextwo) + B$. Notice that $\mathbf{V}(\indextwo)$ is a diagonal matrix
		with real entries, thus $e^{-t 2\pi i \mathbf{V}(\indextwo)}$ defines an isometry for all
		$t\geq0.$ The Lee-Trotter product formula from
		\Cref{perturbation_theory_Lie_Trotter_formula} yields
		\begin{equation}\label{auxeq6}
			\| e^{tM(\indextwo)} \|_\infty \leq \limsup_{n\to\infty} \| e^{-t/n \, 2 \pi i \mathbf{V}(\indextwo)} e^{t/n \, B} \|_\infty^n \leq \limsup_{n\to\infty} \| e^{t/n \, B}
			\|_\infty^n \leq e^{t\|B\|_\infty}.
		\end{equation}
		We obtain that the matrix norms of $e^{tM(\indextwo)}$ are uniformly bounded in $\indextwo\in\Z^d$.
		Consequently,
		\begin{equation}\label{domain_mmo_semigroup}
			\bigg\{ \hat{u} \in \ltwoZCN \, \colon \, \sum_{\indextwo \in \Z^d}
			|e^{tM(\indextwo)}\hat{u}(\indextwo)|^2 < \infty \bigg\} = \ltwoZCN
		\end{equation}
		and the series in \eqref{auxeq5} is well defined for all $u\in \LtwoTCN.$ We define the linear operators $\widetilde{R}(t)$ on $\LtwoTCN$ via
		\begin{equation*}
			(\widetilde{R}(t) u)(\mathbf{x}) \coloneqq \sum_{\indextwo \in \Z} e^{tM(\indextwo)}
			\hat{u}(\indextwo) e^{2\pi i \indextwo \cdot \mathbf{x}}
		\end{equation*}
		for all $t\geq0.$ Then, \eqref{auxeq6} gives the estimate
		\begin{equation}\label{auxiliary_eq36}
			\|\widetilde{R}(t)u\|_{\LtwoTCN} \leq e^{t\|B\|_\infty} \|u\|_{\LtwoTCN}
		\end{equation}
		and the family $(\widetilde{R}(t))_{t\geq0}$ fulfills the semigroup property. The next step is to show
		\begin{equation}\label{auxeq8}
			\lim_{t \searrow 0} \left(\frac{\widetilde{R}(t)u - u}{t}\right)(\mathbf{x}) = \sum_{\indextwo \in \Z^d}
			M(\indextwo) \hat{u}(\indextwo) e^{2\pi i \indextwo \cdot \mathbf{x}} =
			((-A_2+B_2)u)(\mathbf{x}) 
		\end{equation}
		for all $u\in D(-A_2)$ and with convergence in $\LtwoTCN$.  Let us firstly note that the series is well defined by \eqref{auxiliary_eq37}. Secondly, if \eqref{auxeq8} holds true, \eqref{auxiliary_eq36} and density of $D(-A_2)$ in $\LtwoTCN$ imply that $(\widetilde{R}(t))_{t\geq0}$ is indeed a strongly continuous semigroup.\newline
		We now start with the proof of \eqref{auxeq8}. By Parseval's identity, \eqref{auxeq8} is equivalent to
		\begin{equation*}
			\lim_{t \searrow 0} \sum_{\indextwo \in \Z^d} \bigg|
			\bigg(\frac{e^{tM(\indextwo)} - I_{\C^{N\times N}}}{t} - M(\indextwo) \bigg) \hat{u}(\indextwo) \bigg|^2 = 0.
		\end{equation*}
		The result comes down the question, whether we can interchange the two limits
		because
		\begin{equation*}
			\lim_{t \searrow 0} \sum_{|\indextwo|\leq n} \bigg|
			\bigg(\frac{e^{tM(\indextwo)} - I_{\C^{N\times N}}}{t} - M(\indextwo)\bigg) \hat{u}(\indextwo) \bigg|^2 = 0
		\end{equation*}
		holds for all $n\in\N$. The fundamental theorem of calculus implies
		\begin{equation}\label{auxiliary_eq38}
			e^{tM(\indextwo)} - I_{\C^{N\times N}} = t\int_0^1 M(\indextwo) e^{st M(\indextwo)} \, ds = t  \bigg(\int_0^1 e^{stM(\indextwo)} \, ds\bigg)	M(\indextwo)
		\end{equation}
		for all $\indextwo\in\Z^d.$ The last equality follows from the Riemann definition of the
		integral, boundedness of $M(\indextwo)$ and $M(\indextwo)e^{rM(\indextwo)} =
		e^{rM(\indextwo)} M(\indextwo)$ for all $r\geq0.$ As a result, \eqref{auxiliary_eq38} and \eqref{auxeq6} show that there is a finite constant $C=C(\|B\|_\infty, \, 1)>0$ with
		\begin{align*}
			\bigg| \bigg(\frac{e^{tM(\indextwo)} - I_{\C^{N\times N}}}{t} &- M(\indextwo) \bigg) \hat{u}(\indextwo) \bigg|^2 = \bigg| \bigg( \int_0^1
			e^{stM(\indextwo)} - I_{\C^{N\times N}} \, ds \bigg) M(\indextwo)
			\hat{u}(\indextwo) \bigg|^2 \\
			&\leq \bigg(\int_0^1 \| e^{st M(\indextwo)} - I_{\C^{N\times N}} \|_\infty \,
			ds\bigg)^2 |M(\indextwo) \hat{u}(\indextwo)|^2 \leq C
			|M(\indextwo)\hat{u}(\indextwo)|^2
		\end{align*}
		for all $t\leq1$. Finally, this estimate yields
		\begin{align*}
			\limsup_{t \searrow 0} &\sum_{\indextwo \in \Z^d} \bigg|
			\bigg(\frac{e^{tM(\indextwo)} - I_{\C^{N\times N}}}{t} - M(\indextwo) \bigg) \hat{u}(\indextwo) \bigg|^2 \\
			&\leq \limsup_{t \searrow 0} \sum_{|\indextwo| > n} \bigg|
			\bigg(\frac{e^{tM(\indextwo)} - I_{\C^{N\times N}}}{t} - M(\indextwo) \bigg) \hat{u}(\indextwo) \bigg|^2 \leq C \sum_{|\indextwo| > n} |M(\indextwo)\hat{u}(\indextwo)|^2
		\end{align*}
		for all $n\in\N.$ For $u\in D(-A_2),$ the limit $n\to\infty$ implies the
		desired result by \eqref{auxiliary_eq37}. \newline
		In summary, we have proven that $(\widetilde{R}(t))_{t\geq0}$ defines a strongly continuous semigroup with a generator that extends $(-A_2+B_2, \, D(-A_2))$. In particular, the generator of $(\widetilde{R}(t))_{t\geq0}$ extends $(-A_2+B_2, \, C_c^\infty(\T^d,\,\C^N))$ and \Cref{corollary_Ccinfty_core_A+B} implies $(\widetilde{R}(t))_{t\geq0} = (R_2(t))_{t\geq0}.$
	\end{proof}
	\begin{remark}
		Equation \eqref{auxeq5} is an explicit formula of the solution to the abstract
		Cauchy problem \eqref{equation_abstract_cauchy_problem} (in a $L^2(\T, \, \C^N)$
		setting).
	\end{remark}
	
	\begin{theorem}\label{theorem_WSMTHM_p_eq_2}
		Let $(R_2(t))_{t\geq0}$ be the transport-reaction semigroup on $\LtwoTCN$ generated by $(-A_2+B_2, \, D(-A_2))$. Then the weak spectral mapping theorem
		\begin{equation*}
			\sigma(R_2(t)) = \overline{e^{t\sigma(-A_2+B_2)}} \quad \text{ for all } t\in\R
		\end{equation*}
		holds.
	\end{theorem}
	\begin{proof}
		Let $t\geq0.$ \Cref{proposition_semigroup_fourier_representation} allows us to apply the same methods we used for $(-A_2+B_2, \, D(-A_2))$ to $R_2(t)$. Let us check the assumptions of \Cref{fourier_diagram_spectrum} for $(R_2(t), \, \LtwoTCN)$ and the matrix multiplication operator $(\mathcal{A}_{e^{tM}}, \, D(\mathcal{A}_{e^{tM}}))$ on $\ltwoZCN$ induced by $e^{tM} \coloneqq (\indextwo\mapsto e^{tM(\indextwo)}).$ \newline
		Firstly, we have seen $D(\mathcal{A}_{e^{tM}}) = \ltwoZCN$ in
		\eqref{domain_mmo_semigroup} and $\mathcal{F}(L^2(\T, \,
		\C^N)) = \ltwoZCN$ obviously holds true. Secondly, the diagram
		\begin{equation*}
			\begin{tikzcd}
			D(R_2(t)) = \LtwoTCN \arrow[r, "R_2(t)"] &
			\LtwoTCN \\
			D(\mathcal{A}_{e^{tM}}) = \ltwoZCN \arrow[u, leftarrow, "\mathcal{F}"]
			\arrow[r, "\mathcal{A}_{e^{tM}}"] & \ltwoZCN \arrow[u, "\mathcal{F}^{-1}"]
			\end{tikzcd}
		\end{equation*}
		commutes by \Cref{proposition_semigroup_fourier_representation} and, last but not least, the resolvent set $\rho(R_2(t))$ is non-empty because $R_2(t)$ is bounded. An
		application of \Cref{fourier_diagram_spectrum} implies
		\begin{equation*}
			\sigma(R_2(t)) = \sigma(\mathcal{A}_{e^{tM}}) =
			\overline{\bigcup_{\indextwo\in \Z^d} \sigma(e^{tM(\indextwo)})}.
		\end{equation*}
		The operators $M(\indextwo)$ are just matrices, so we obtain
		\begin{equation*}
			\sigma(e^{tM(\indextwo)}) = e^{t\sigma(M(\indextwo))}
		\end{equation*}
		by \cite[Theorem 5.3, p. 84]{schnaubelt2012lecture}. Now, we apply \Cref{char_spectrum_generator} and conclude
		\begin{align*}
			\sigma(R_2(t)) &= \sigma(\mathcal{A}_{e^{tM}}) = \overline{\bigcup_{\indextwo\in\Z} \sigma(e^{tM(\indextwo)})} = \overline{\bigcup_{\indextwo\in\Z}			e^{t\sigma(M(\indextwo))}} \subseteq \overline{\bigcup_{\indextwo\in\Z} e^{t	\overline{\bigcup_{\indexthree\in\Z}\sigma(M(\indexthree))}}} \\
			&= \overline{e^{t\overline{\bigcup_{k\in \Z} \sigma(M(k))}}} =
			\overline{e^{t\sigma(-A_2+B_2)}}.
		\end{align*} 
		The other direction is the easy one, which follows from the spectral inclusion theorem
		\begin{equation*}
			\sigma(R_2(t)) \supseteq e^{t\sigma(-A_2+B_2)},
		\end{equation*}
		see \Cref{spectral_inclusion_theorem}, and the closedness of
		$\sigma(R_2(t))$. \newline
		For $t\leq0,$ the result can be shown with \Cref{lemma_-A+B_generates_C_0_group} and
		an application of the weak spectral mapping theorem to the semigroup
		$(R_2(-t))_{t\geq0}.$
	\end{proof}
	
	We now follow \cite{latrach2008weak} and present methods to extend the weak spectral mapping
	theorem to other $1\leq p < \infty$. We present the ideas in detail and adapt the
	proofs to our system of PDEs, if necessary. For the rest of the chapter, we will assume
	\begin{assumption}\label{assumption_spectrum}
		\begin{enumerate}
			\item
			The spectrum $\sigma(-A_p+B_p)$ of the generator of the transport-reaction semigroup $(R_p(t))_{t\geq0}$ on $\LpCn$ is independent of $1\leq p < \infty.$
			\item
			The spectrum $\sigma(A_p+B_p^T)$ of $(A_p+B_p^T, \, D(A_p))$ is independent of $1\leq p < \infty.$
		\end{enumerate}
	\end{assumption}
	\begin{remarks}
		\begin{enumerate}
			\item
			Under \Cref{assumption_spectrum} (i), \Cref{proposition_adjoint_semigroup_pneq1} almost implies \Cref{assumption_spectrum} (ii) because of
			\begin{equation*}
				\sigma(A_p+B_p^T) = \sigma(-A_q+B_q)
			\end{equation*}
			for all $1<p<\infty$ with dual exponent $q$. Here, we used that the spectrum is unchanged when taking the adjoint.
			\item
			The operator $(A_p+B_p^T, \, D(A_p))$ is still of the class covered in this thesis. The change of sign for $A_p$ just changes the direction of movement and $B_p^T$ is a real $N\times N$ matrix.
		\end{enumerate}
	\end{remarks}
	This assumption is of course only fulfilled in the case $d=1$ with non-vanishing transport directions $v_1,\cdots,v_N\neq 0,$ see \Cref{theorem_spectrum_generator}. Nevertheless, we decided to formulate all results under \Cref{assumption_spectrum} instead of restricting ourselves to the case $d=1$. The reason for this is to emphasize the independence of the arguments of the spacial dimension $d\in\N$. \newline

	Let us start by formulating some technicalities.
	
	\begin{lemma}\label{lemma_gen_semigroup_resolv_coincide}
		Let $(R_p(t))_{t\in\R}$ be the transport-reaction group on $\LpCn$ generated by $(-A_p+B_p, \, D(-A_p))$. The following properties hold true for all $1\leq p <  p' < \infty$ and all $t\in\R:$
		\begin{enumerate}
			\item
			$D(-A_p')\subseteq D(-A_p),$
			\item
			$\restr{-A_p+B_p}{D(-A_p')} = -A_{p'}+B_{p'}$,
			\item
			$\restr{R_p(t)}{L^{p'}(\T^d, \, \C^N)} = R_{p'}(t).$
		\end{enumerate}
		Under \Cref{assumption_spectrum} (i), it also holds
		\begin{enumerate}
			\setcounter{enumi}{3}
			\item
			$\restr{R(\lambda, \, -A_p+B_p)}{L^{p'}(\T^d, \, \C^N)} = R(\lambda, \, -A_{p'}+B_{p'})$ \quad for all $\lambda \in \rho(-A+B).$
		\end{enumerate}
	\end{lemma}
	\begin{proof}
		We start by showing $(i)$ and $(ii)$ at once. On $C_c^\infty(\T^d,\, \C^N),$ the generators coincide by \Cref{lemma_A_is_an_extension} and the definition of $B$ as a pointwise matrix multiplication. Now, let $u\in D(-A_p')$. By \Cref{corollary_Ccinfty_core_A+B}, there exists a sequence $(u^{(n)})_{n\in\N} \subset C_c^\infty(\T^d, \, \C^N)$ with $u^{(n)} \to u$ in the graph norm of $(-A_{p'}+B_{p'}, \, D(-A_{p'})).$ We also assumed $p< p',$ so $u^{(n)}$ also converges to $u$ in $\LpCn$ and $((-A_p+B_p)u^{(n)})_{n\in\N}$ is a Cauchy sequence in $\LpCn.$ This yields
		\begin{equation*}
			(-A_p+B_p)u^{(n)} \longrightarrow (-A_{p'}+B_{p'})u
		\end{equation*}
		in $\LpCn$ and the closedness of $(-A_p+B_p, \, D(-A_p)),$ see \Cref{hille-yosida}, implies $u\in D(-A_p)$ with $(-A_p+B_p)u = (-A_{p'}+B_{p'})u.$ \newline
		Property $(iii)$ follows from the Dyson-Phillips series in \Cref{theorem_semigroup_generated_by_-A+B} and the fact that the transport groups $(T_p(t))_{t\in\R}$ and $(T_{p'}(t))_{t\in\R}$
		coincide on $L^{p'}(\T^d,\, \C^N)$. \newline
		Concerning property $(iv),$ notice that \Cref{assumption_spectrum}(i) justifies the notation $\rho(-A+B)$ used above. Let $\lambda\in\rho(-A+B)$. For $v\in L^{p'}(\T^d, \, \C^N) \subseteq \LpCn,$ there exist unique $u_p \coloneqq R(\lambda, \, -A_p+B_p)v\in D(-A_p)$ and $u_{p'} \coloneqq R(\lambda, \, -A_{p'}+B_{p'})v\in D(-A_{p'})$ such that
		\begin{equation*}
			\big(\lambda - (-A_p+B_p)\big)u_p = v = \big(\lambda -
			(-A_{p'}+B_{p'})\big)u_{p'}
		\end{equation*}
		holds. Property $(ii)$ yields
		\begin{equation*}
			\big(\lambda - (-A_p+B_p)\big)u_{p'} = \big(\lambda -
			(-A_{p'}+B_{p'})\big)u_{p'} = v,
		\end{equation*}
		which implies $u_{p'} = u_p$ by the injectivity of $\lambda - (-A_p+B_p)$.
	\end{proof}

	Before we show equivalence of the weak spectral mapping theorem for all $1\leq p < \infty$ to two weak spectral mapping theorems for $p=1$ (under \Cref{assumption_spectrum}), we state the two key theorems used in the proof. \newline
	For $1\leq p \leq 2,$ the idea is to use an interpolation argument based on the Riesz-Thorin
	theorem. Adapted to our application, the Riesz-Thorin theorem on a $\sigma$-finite measure space $(X, \Sigma, \mu)$ reads as follows.
	
	\begin{theorem}[{\cite[Theorem 3.16, p. 35]{lax2011complex}}]\label{theorem_Riesz_Thorin}
		Suppose that $S\colon L^1(X, \, \C) + L^2(X, \, \C) \rightarrow L^1(X, \, \C) +
		L^2(X, \, \C)$ is a linear operator such that
		\begin{align*}
			\restr{S}{L^1(X, \C)} &\colon L^1(X, \, \C) \longrightarrow L^1(X, \, \C), \\
			\restr{S}{L^2(X, \C)} &\colon L^2(X, \, \C) \longrightarrow L^2(X, \, \C)
		\end{align*}
		are bounded with norms $C_1$ and $C_2$ respectively. Then for $0<\theta < 1,$
		$S$ defines a bounded linear operator
		\begin{equation*}
			S\colon L^p(X, \, \C) \longrightarrow L^p(X, \C)
		\end{equation*}
		where
		\begin{equation}\label{auxeq9}
			\frac{1}{p} = \frac{1-\theta}{1} + \frac{\theta}{2} = 1 - \frac{\theta}{2}
		\end{equation}
		and $\|S\|_{L^p(X, \C)} \leq C_1^{1-\theta} C_2^{\theta}.$
	\end{theorem}
	\begin{remark}
		For $1\leq p \leq 2,$ it holds $L^p(X, \C^N) \subseteq L^1(X, \C^N) + L^2(X,
		\C^N)$, so it makes sense to consider $S$ as an operator on $L^p(X, \C^N)$.
	\end{remark}
	
	The Riesz-Thorin theorem is going to be applied in the setting of the next result.
	
	\begin{corollary}[{\cite[Corollary 1.2, p. 208]{greiner1991weak}}]\label{corollary_greiner}
		Let $(R(t))_{t\geq0}$ be a semigroup on a Banach space \mbox{$(E, \, \|\cdot\|_E)$} with generator
		$(\mathcal{A}, \, D(\mathcal{A})).$ Let $t>0.$ Then the following assertions are
		equivalent:
		\begin{enumerate}
			\item
			$e^{\lambda t} \in \rho(R(t))$
			\item
			$\lambda + \tfrac{2\pi i}{t}\Z \subseteq \rho(\mathcal{A})$ and the sequence
			of operators
			\begin{equation*}
				S^{(n)} \coloneqq \frac{1}{n} \sum_{m=0}^{n-1}  \sum_{\indextwo=-m}^m
				R\big(\lambda + \tfrac{2\pi i}{t}\indextwo, \, \mathcal{A}\big) 
			\end{equation*}
			is bounded in $\mathcal{L}(E).$
		\end{enumerate}
	\end{corollary}
	
	\begin{proposition}\label{proposition_WSMTHM_1_leq_p_leq_2}
		Let $(R_p(t))_{t\in\R}$ be the transport-reaction group on $\LpCn$ generated by $(-A_p+B_p, \, D(-A_p))$. Under \Cref{assumption_spectrum} (i), the weak spectral mapping theorem
		\begin{equation*}
			\sigma(R_p(t)) = \overline{e^{t\sigma(-A_p+B_p)}} =
			\overline{e^{t\sigma(-A+B)}} \qquad \text{ for all } t\in\R
		\end{equation*}
		for all $1\leq p \leq 2$ is equivalent to the weak spectral mapping theorem
		\begin{equation*}
			\sigma(R_1(t)) = \overline{e^{t\sigma(-A_1+B_1)}} =
			\overline{e^{t\sigma(-A+B)}} \qquad \text{ for all } t\in\R
		\end{equation*}
		for $p=1.$
	\end{proposition}
	\begin{proof}
		Let us assume that the weak spectral mapping theorem holds for $p=1$ and let $1<p<2.$ For $p=2,$ the weak spectral mapping theorem holds by \Cref{theorem_WSMTHM_p_eq_2}. Due to the spectral inclusion \Cref{spectral_inclusion_theorem} and the closedness of spectra, we only have to prove the inclusion
		\begin{equation*}
			\sigma(R_p(t)) \subseteq \overline{e^{t\sigma(-A+B)}}
		\end{equation*}
		for all $t\in\R.$ Let $t>0.$ The goal is to show
		\begin{equation*}
			\Big(\overline{e^{t\sigma(-A+B)}}\Big)^c \subseteq \rho(R_p(t)).
		\end{equation*}
		To this end, let $\mu \in \big(\overline{e^{t\sigma(-A+B)}}\big)^c.$ The semigroup operator $R_p(t)$ is invertible with inverse $R_p(-t)$, so we can assume $\mu \neq 0$ without loss of generality.	In this case, there exists $\lambda\in\C$ with $\mu=e^{\lambda t}$ and if we can show property $(ii)$ from \Cref{corollary_greiner}, then we are done. \newline
		Regarding the first condition of $(ii)$ in \Cref{corollary_greiner}, notice that $\mu = e^{\lambda t} \in
		\rho(R_2(t))$ by \Cref{theorem_WSMTHM_p_eq_2} and the assumed independence of $p$ for the
		spectrum of the generator. \Cref{corollary_greiner} implies
		\begin{equation*}
			\lambda + \tfrac{2\pi i}{t} \Z \subset \rho(-A_2+B_2) = \rho(-A+B) =
			\rho(-A_p+B_p).
		\end{equation*}
		It is left to show that the sequence of operators $(S^{(n)}_p)_{n\in\N}$ defined by
		\begin{equation*}
			S^{(n)}_p \coloneqq \frac{1}{n} \sum_{m=0}^{n-1}  \sum_{\indextwo=-m}^m
			R\big(\lambda + \tfrac{2\pi i}{t}\indextwo, \, -A_p+B_p\big) 
		\end{equation*}
		is bounded in
		$\mathcal{L}(\LpCn).$ From $\mu \in \rho(R_1(t))=\rho(R_2(t))$ and
		\Cref{corollary_greiner}, we already know that $(S^{(n)}_1)_{n\in\N}$ and
		$(S^{(n)}_2)_{n\in\N}$ are bounded sequences in $\mathcal{L}(L^1(\T^d, \, \C^N))$
		and $\mathcal{L}(\LtwoTCN)$ respectively. The idea is now to apply the
		Riesz-Thorin \Cref{theorem_Riesz_Thorin}. Notice that there exists $\theta \in (0,1)$ such that \eqref{auxeq9} holds. Originally, \Cref{theorem_Riesz_Thorin} is formulated for complex valued operators but it is easy to get around this restriction by identifying $\LpCn$ with $L^p(\T^d \times \{1,\cdots,N\}, \, \C)$ through the canonical isomorphism
		\begin{equation*}
			\Phi\colon \LpCn \longrightarrow L^p(\T^d \times \{1,\cdots,N\}, \, \C) 
		\end{equation*}
		with
		\begin{equation*}
			u \longmapsto \big[(\mathbf{x},\indexone) \mapsto u_\indexone(\mathbf{x})\big].
		\end{equation*}
		Let $n\in\N$ and define the operator
		\begin{align*}
			S^{(n)} \colon L^1(\T^d, \C^N) + \LtwoTCN &\longrightarrow L^1(\T^d, \, \C^N) +
			\LtwoTCN \\
			u &\longmapsto S^{(n)}_1u_1 + S^{(n)}_2 u_2
		\end{align*}
		for $u=u_1+u_2$ with $u_1\in L^1(\T^d, \, \C^N)$ and $u_2\in \LtwoTCN.$
		Let us show that this map is well defined. To this end, assume
		\begin{equation*}
			u = u_1+u_2 = v_1+v_2
		\end{equation*}
		for $u_1, \, v_1 \in L^1(\T^d, \, \C^N)$ and $u_2, \, v_2 \in \LtwoTCN.$ We
		obtain
		\begin{equation*}
			u_1-v_1 = v_2-u_2 \in L^1(\T^d, \, \C^N) \cap \LtwoTCN.
		\end{equation*}
		Hence, \Cref{lemma_gen_semigroup_resolv_coincide} $(iv)$ and the definitions
		of $S^{(n)}_1$ and $S^{(n)}_2$ imply
		\begin{align*}
			S^{(n)}_1u_1 + S^{(n)}_2 u_2 &= S^{(n)}_1v_1 + S^{(n)}_2 v_2 +
			S^{(n)}_1(u_1-v_1) + S^{(n)}_2 (u_2-v_2) \\
			&= S^{(n)}_1v_1 + S^{(n)}_2 v_2 + S^{(n)}_2(u_1-v_1) - S^{(n)}_2 (v_2-u_2) \\
			&= S^{(n)}_1v_1 + S^{(n)}_2 v_2.
		\end{align*}
		\Cref{theorem_Riesz_Thorin} implies that $S^{(n)}(p) \coloneqq
		\restr{S^{(n)}}{\LpCn}$ fulfills
		\begin{equation}\label{auxeq10}
			\|S^{(n)}(p)\|_{\mathcal{L}(\LpCn)} \leq \|S^{(n)}_1
			\|_{\mathcal{L}(L^1(\T^d, \C^N))}^{1-\theta} \| S^{(n)}_2\|_{\mathcal{L}(\LtwoTCN)}^{\theta} \eqqcolon C_{1,n}^{1-\theta} C_{2,n}^\theta.
		\end{equation}
		Furthermore, we have
		\begin{equation*}
			\restr{S^{(n)}(p)}{\LtwoTCN} = S^{(n)}_2 = \restr{S^{(n)}_p}{\LtwoTCN}
		\end{equation*}
		by \Cref{lemma_gen_semigroup_resolv_coincide} $(iv)$. Density of $\LtwoTCN$ in $\LpCn$ and continuity of and $S^{(n)}(p)$ and $S^{(n)}_p$
		respectively imply $S^{(n)}(p) = S^{(n)}_p$, which finally results in
		\begin{equation*}
			\sup_{n\in\N} \|S^{(n)}_p\|_{\mathcal{L}(\LpCn)} \leq \sup_{n\in\N}
			C_{1,n}^{1-\theta} \, \sup_{n\in\N} C_{2,n}^\theta < \infty
		\end{equation*}
		by \eqref{auxeq10} and the boundedness of $(S^{(n)}_1)_{n\in\N}$ and
		$(S^{(n)}_2)_{n\in\N}$ in $\mathcal{L}(L^1(\T^d, \, \C^N))$
		and $\mathcal{L}(\LtwoTCN)$ respectively. This shows the weak spectral mapping theorem for $1<p<2$ and all $t\geq0$. \newline
		For $t<0,$ the statement immediately follows from \Cref{lemma_-A+B_generates_C_0_group}.
	\end{proof}
	
	\begin{theorem}\label{theorem_WSMTHM_equiv_peq1}
		Let $(R_p(t))_{t\in\R}$ be the transport-reaction group on $\LpCn$ generated by $(-A_p+B_p, \, D(-A_p))$. Under \Cref{assumption_spectrum}, the weak spectral mapping theorem
		\begin{equation*}
			\sigma(R_p(t)) = \overline{e^{t\sigma(-A_p+B_p)}} =
			\overline{e^{t\sigma(-A+B)}} \qquad \text{ for all } t\in\R
		\end{equation*}
		for all $1\leq p < \infty$ is equivalent to the two weak spectral mapping theorems
		\begin{align*}
			\sigma(R_1(t)) &= \overline{e^{t\sigma(-A_1+B_1)}} =
			\overline{e^{t\sigma(-A+B)}} \qquad \text{ for all } t\in\R \\
			\sigma(\widetilde{R}_1(t)) &= \overline{e^{t\sigma(A_1+B_1^T)}} =
			\overline{e^{t\sigma(A+B^T)}} \qquad \text{ for all } t\in\R
		\end{align*}
		for $p=1,$ where $(\widetilde{R}_1(t))_{t\geq0}$ is the transport-reaction semigroup on $L^1(\T^d, \, \C^N)$ generated by $(A_1+B_1^T, \, D(A_1)).$
	\end{theorem}
	\begin{proof}
		For $1\leq p \leq 2,$ the result simply follows from \Cref{proposition_WSMTHM_1_leq_p_leq_2}. Let $2<p<\infty$ and let $t\geq0.$ The
		adjoint semigroup of $(R_p(t))_{t\geq0}$ is the semigroup
		$(\widetilde{R}_q(t))_{t\geq0}$ on $L^q(\T^d, \, \C^N)$ generated by the adjoint of $(-A_p+B_p, \, D(-A_p))$, namely $(A_q+B_q^T, \,
		D(A_q)),$ see  \Cref{proposition_adjoint_semigroup_pneq1}. Here, $1<q<2$ is the
		dual exponent of $p$ and the adjoint semigroup and generator are still of the class covered by \Cref{proposition_WSMTHM_1_leq_p_leq_2}, since the different sign for $A$ just changes the direction of movement and $B^T$ is still a real $N\times N$ matrix. Hence, we can use \Cref{assumption_spectrum}(ii) to apply \Cref{proposition_WSMTHM_1_leq_p_leq_2} to the adjoint semigroup. We obtain
		\begin{equation*}
			\sigma(R_p(t)) = \sigma(\widetilde{R}_q(t)) = \overline{e^{t\sigma(A_q+B_q^T)}} = \overline{e^{t\sigma(-A_p+B_p)}} = \overline{e^{t\sigma(-A+B)}}.
		\end{equation*}
		As in the proof of \Cref{proposition_WSMTHM_1_leq_p_leq_2}, the statement for $t<0$ follows from \Cref{lemma_-A+B_generates_C_0_group}.
	\end{proof}
	
	Unfortunately, proving \Cref{assumption_spectrum} for arbitrary dimensions $d\in\N$ and showing weak spectral mapping theorems for $p=1$ are difficult tasks. Nevertheless, we can prove some partial results.
	
	\begin{corollary}\label{wsmthm_deq1_equiv}
		Let $d=1$ and let the transport directions $v_1,\cdots,v_N \neq 0 $ be non-vanishing. Let $(R_p(t))_{t\in\R}$ be the transport-reaction group on $L^p(\T, \, \C^N)$ generated by $(-A_p+B_p, \, D(-A_p))$. Then the weak spectral mapping theorem
		\begin{equation*}
			\sigma(R_p(t)) = \overline{e^{t\sigma(-A_p+B_p)}} =
			\overline{e^{t\sigma(-A+B)}} \qquad \text{ for all } t\in\R
		\end{equation*}
		for all $1\leq p < \infty$ is equivalent to to the two weak spectral mapping theorems
		\begin{align*}
			\sigma(R_1(t)) &= \overline{e^{t\sigma(-A_1+B_1)}} =
			\overline{e^{t\sigma(-A+B)}} \qquad \text{ for all } t\in\R \\
			\sigma(\widetilde{R}_1(t)) &= \overline{e^{t\sigma(A_1+B_1^T)}} =
			\overline{e^{t\sigma(A+B^T)}} \qquad \text{ for all } t\in\R
		\end{align*}
		for $p=1,$ where $(\widetilde{R}_1(t))_{t\geq0}$ is the transport-reaction semigroup on $L^1(\T, \, \C^N)$ generated by $(A_1+B_1^T, \, D(A_1)).$
	\end{corollary}
	\begin{proof}
		Given $d=1$ and $v_1,\cdots,v_N \neq 0,$ \Cref{assumption_spectrum} is fulfilled by \Cref{theorem_spectrum_generator} and we can apply \Cref{theorem_WSMTHM_equiv_peq1}.
	\end{proof}

	In particular, the weak spectral mapping theorem holds for all $1\leq p < \infty $ in the case that $(R_p(t))_{t\in\R}$ defines a positive group on $L^p(\T, \, \C^N)$ with non-vanishing transport directions $v_1,\cdots,v_N \neq 0$. This happens if and only if $B$ is a diagonal matrix, see \Cref{corollary_group_positive}, and it is a consequence of the following abstract result by Arendt an Greiner. Again, $(X, \Sigma, \mu)$  is assumed to be a $\sigma$-finite measure space.
	
	\begin{theorem}[{\cite[Corollary 1.4, p. 300]{arendt1984spectral}}]\label{theorem_WSMTHM_Arendt_Greiner}
		Let $(\mathcal{A}, \, D(\mathcal{A}))$ be the generator of a strongly
		continuous group $(R(t))_{t\in\R}$ of positive operators on $L^1(X, \mu).$ Then
		the weak spectral mapping theorem holds, i.e.
		\begin{equation*}
			\sigma(R(t)) = \overline{e^{t\sigma(\mathcal{A})}} \quad \text{ for all } t\in\R.
		\end{equation*}
	\end{theorem}
	
	\begin{corollary}\label{corollary_CN_trick}
		Let $(R_1(t))_{t\in\R}$ be the transport-reaction group on $L^1(\T^d, \, \C^N)$ generated by $(-A_1+B_1, \, D(-A_1))$. Assume that $B\in \R^{N\times N}$ is a diagonal matrix. Then the two weak spectral mapping theorems
		\begin{align*}
			\sigma(R_1(t)) &= \overline{e^{t\sigma(-A_1+B_1)}} =
			\overline{e^{t\sigma(-A+B)}} \qquad \text{ for all } t\in\R \\
			\sigma(\widetilde{R}_1(t)) &= \overline{e^{t\sigma(A_1+B_1^T)}} =
			\overline{e^{t\sigma(A+B^T)}} \qquad \text{ for all } t\in\R
		\end{align*}
		hold, where $(\widetilde{R}_1(t))_{t\geq0}$ is the transport-reaction semigroup on $L^1(\T^d, \, \C^N)$ generated by $(A_1+B_1^T, \, D(A_1)).$
	\end{corollary}
	\begin{proof}
		Given a diagonal matrix $B$, the groups $(R_1(t))_{t\in\R}$ and $(\widetilde{R}_1(t))_{t\geq0}$ are positive due to \Cref{corollary_group_positive}. Now the result is a direct consequence of \Cref{theorem_WSMTHM_Arendt_Greiner}. Technically, we should pause for a moment before simply applying \Cref{theorem_WSMTHM_Arendt_Greiner}. The above theorem only holds for positive groups on $L^1(\T^d, \, \C)$, but we consider positive groups on
		$L^1(\T^d, \, \C^N).$ However, this subtlety causes no problem whatsoever because the canonical mapping
		\begin{equation*}
			\Phi\colon L^1(\T^d, \, \C^N) \longrightarrow L^1(\T^d \times \{1,\cdots,N\}, \C) 
		\end{equation*}
		with
		\begin{equation*}
			u \longmapsto \big[(\mathbf{x},\indexone) \mapsto u_\indexone(\mathbf{x})\big]
		\end{equation*}
		defines an isomorphism.
	\end{proof}

	\begin{corollary}\label{wsmthm_positive_group_all_p}
		Let $d=1$ and let the transport directions $v_1,\cdots,v_N \neq 0 $ be non-vanishing. Let $(R_p(t))_{t\in\R}$ be the transport-reaction group on $L^p(\T, \, \C^N)$ generated by $(-A_p+B_p, \, D(-A_p))$ and assume that $B\in \R^{N\times N}$ is a diagonal matrix. Then the weak spectral mapping theorem
		\begin{equation*}
			\sigma(R_p(t)) = \overline{e^{t\sigma(-A_p+B_p)}} =
			\overline{e^{t\sigma(-A+B)}} \qquad \text{ for all } t\in\R
		\end{equation*}
		holds for all $1\leq p < \infty$.
	\end{corollary}
	\begin{proof}
		Combine \Cref{wsmthm_deq1_equiv} and \Cref{corollary_CN_trick}.
	\end{proof}
	
	At this point, we should critically reflect on \Cref{wsmthm_positive_group_all_p}. If $B$ is a diagonal matrix $D$ with diagonal entries $d_1,\cdots,d_N \in \R$, the transport-reaction model \eqref{equation_abstract_cauchy_problem} reads
	\begin{equation*}
		\partial_t u + \colvec{v_1 u'_1 \\ \vdots \\ v_N u'_N} = \colvec{d_1u_1 \\
		\vdots \\ d_N u_N}.
	\end{equation*}
	There is no interaction and one can directly compute the solution. A short
	computation (use e.g. \Cref{theorem_semigroup_generated_by_-A+B} (v) for the
	rigorous proof) shows
	\begin{equation*}
		R(t) = e^{tD} T(t) \eqqcolon \colvec{e^{d_1t} T_1(t) \\ \vdots \\ e^{d_Nt}
		T_N(t)}
	\end{equation*}
	for all $t\in\R$, where $(T(t))_{t\in\R}$ is the transport group on $L^p(\T,\, \C^N)$ defined in	\eqref{transport_semigroup_definition} and \Cref{remark_-A_generates_C0_group}.
	Together with the computations from	\Cref{example_spectral_analysis_transport_1d} and
	\Cref{example_spectrum_transport_sg_N}, we obtain
	\begin{align*}
		\sigma(R_p(t)) &= \bigcup_{\indexone=1}^N e^{d_\indexone t}
		\sigma(T_\indexone(t)) = \bigcup_{\indexone=1}^N e^{d_\indexone t}
		\overline{e^{t\sigma(-A_j)}} = \bigcup_{\indexone=1}^N
		\overline{e^{t(d_j+\sigma(-A_j))}} = \bigcup_{\indexone=1}^N
		\overline{e^{t\sigma(-A_j+d_j)}}\\
		&= \overline{e^{\cup_{\indexone=1}^N t \sigma(-A_\indexone+d_\indexone)}} =
		\overline{e^{t\sigma(-A+D)}},
	\end{align*}
	where we used the notation $-A_j \coloneqq -v_j \partial_x$ for the $j$-th component of
	$-A.$ This shows that it is actually possible to prove \Cref{wsmthm_positive_group_all_p} with explicit computations. \newline
	
	After considering \cite[Theorem 3.2, Theorem 3.3, p. 1049]{latrach2008weak} and
	\cite{lods2006stability}, a natural idea would be to split the reaction matrix
	$B$ into its diagonal entries and its strictly lower and upper triangular
	matrix, i.e. to write $B=D+(L+U).$ If $(U_p(t))_{t\geq0}$ is the semigroup
	generated by $-A_p+D_p$ and $(R_p(t))_{t\geq0}$ is the semigroup generated by
	$-A_p + B_p,$ compactness of $R_p(t)-U_p(t)$ would imply that the the essential
	spectra of the semigroup operators $U_p(t)$ and $R_p(t)$ coincide for all $t\in\R$. This would allow us to prove a weak spectral mapping theorem in dimension $d=1$ with non-vanishing transport directions $v_1,\cdots,v_N \neq 0$ for all $1\leq p < \infty$, cf. \cite[Theorem 3.3, p. 1049]{latrach2008weak}. Unfortunately, the next example illustrates that this procedure does not work in the case of discrete values for the velocity.
	
	\begin{example}
		Consider $d=1, \, N=2$ and the system
		\begin{equation*}
			\partial_t u + \partial_x u = \begin{pmatrix}
			1 & 1 \\
			0 & 1
			\end{pmatrix} u.
		\end{equation*}
		Then, $D=I_{\C^{2\times 2}}$ and the equations can be solved explicitly. We
		obtain
		\begin{align*}
			U(t)u &= e^t T(t)u = e^t \colvec{u_1(\cdot - t) \\ u_2(\cdot - t)}, \\
			R(t)u &= e^t\colvec{u_1(\cdot - t) + u_2(\cdot - t)t \\ u_2(\cdot -t)}.
		\end{align*}
		Consequently,
		\begin{equation*}
		(R(t) - U(t))u = te^t \colvec{u_2(\cdot-t) \\ 0}
		\end{equation*}
		and the difference of the semigroup operators is not compact for $t\neq 0$ because the
		unit ball in $L^p(\T)$ is not.
	\end{example}
	
	\newpage

	\chapter{Arbitrary Side Lengths and Neumann Boundary Conditions}\label{chapter_side_length_neumann}
	
	So far, we considered the transport-reaction models on the domain $[0,1]^d$ with
	periodic boundary conditions, i.e. the domain $\T^d$ defined in
	\Cref{chapter_Introduction}. An easy scaling argument in the spacial variable
	can be used to see that the assumption of side length $1$ of the torus can be
	made without loss of generality in terms of mathematical treatment - models on
	$\T^d_L \coloneqq L\T^d$ for $L>0$ can be studied completely analogously.
	Nevertheless, the side length can affect stability and pattern formation
	properties of the transport-reaction model and for this reason, it often makes
	sense to incorporate this additional parameter in concrete examples. \newline
	Secondly, we describe the connection of ``symmetric" models with appropriate
	Neumann boundary conditions on $[0,L]$ to models with periodic boundary
	conditions on $[0,2L].$ A precise mathematical explanation of what we mean by
	that is given at the beginning of the second section of this chapter.
	
	\section{Arbitrary Side Lengths}\label{section_arbitrary_side_lengths}
	
	Consider a side length $L>0$, the $d$-dimensional torus $\T^d_L = L \T^d$ and
	recall that the transport-reaction equation for $N\in\N$ particle groups reads
	\begin{equation}\label{recall_transport_reaction_eq}
		\partial_t u + \colvec{\mathbf{v_1} \cdot \nabla u_1 \\ \vdots \\ \mathbf{v_N} \cdot \nabla u_N} = Bu,
	\end{equation}
	where $\mathbf{v_1}, \cdots, \mathbf{v_N} \in \R^d$ are the transport directions
	and $B\in\R^{N\times N}$ describes linear reactions.
	\Cref{chapter_transport_reaction_SG} dealt with the question of well-posedness
	of \eqref{recall_transport_reaction_eq} on $[0,1]^d$ with periodic boundary
	conditions, which corresponds to the choice of $L^p(\T^d, \, \C^N)$ as the
	appropriate Banach space. Recall that the operator defined by
	\begin{equation}\label{auxeq26}
		(-A_p+B_p)u = - \colvec{\mathbf{v_1} \cdot \nabla u_1 \\ \vdots \\ \mathbf{v_N} \cdot \nabla u_N} + Bu
	\end{equation}
	and with domain $D(-A_p+B_p) = D(-A_p)$ studied in \Cref{chapter_transport_semigroup}
	generates the transport-reaction semigroup $(R_p(t))_{t\geq0}$ on $\LpCn$ from
	\Cref{theorem_semigroup_generated_by_-A+B}. \newline
	For this section, it is more important to stress the dependence on $L$ and on the
	directions $\mathbf{v} \coloneqq (\mathbf{v_1},\cdots,\mathbf{v_N}),$ rather
	than to stress the dependence on $p$. Consequently, we use the notation
	$(-A+B)^{L,\mathbf{v}}$ for the operator \eqref{auxeq26} on $L^p(\T^d_L, \,
	\C^N)$ with domain
	\begin{equation*}
		D(-A^{L,\mathbf{v}}) = \left\{ u \in L^p(\T^d_L, \, \C^N) \, \colon \,
		\lim_{t \searrow 0} \frac{u_\indexone(\cdot - t\mathbf{v_\indexone}) -
		u_\indexone(\cdot)}{t} \text{ exists in } \Lp(\T^d_L) \text{ for all }
		\indexone=1,\cdots,N \right\}.
	\end{equation*}
	and $(R^{L,\mathbf{v}}(t))_{t\geq 0}$ for its generated semigroup. It should
	be noted that this operator is indeed a generator. This follows similarly to
	\Cref{chapter_transport_semigroup} and \Cref{chapter_transport_reaction_SG} by
	defining the transport semigroup on $\T^d_L$, see
	\eqref{transport_semigroup_definition}, and using the perturbation argument from
	\Cref{theorem_semigroup_generated_by_-A+B}. If $L=1$, we drop the side length in the notation and
	write $(-A+B)^\mathbf{v}$ or $(R^\mathbf{v}(t))_{t\geq 0}$. To
	achieve rigor and conciseness, we also introduce the scaling operator
	\begin{equation}\label{scaling_operator}
		\begin{aligned}
			\mathcal{S}\colon L^p(\T^d_L, \, \C^{N}) &\longrightarrow L^p(\T^d, \, \C^{N}) \\
			u &\longmapsto (\mathbf{x} \mapsto u(L\mathbf{x}))
		\end{aligned}
	\end{equation}
	with inverse $\mathcal{S}^{-1}.$
	
	\begin{theorem}\label{theorem_side_length}
		Let $1\leq p < \infty, \, d\in\N,\, N\in\N$ and let $L>0$ be the side length of the torus. Moreover, let $\mathbf{v_1},
		\cdots, \mathbf{v_N}\in\R^d$ be arbitrary transport directions. Then
		\begin{equation}\label{auxeq27}
			R^{L,\mathbf{v}}(t) = \mathcal{S}^{-1} R^{\mathbf{v}/L}(t) \mathcal{S}
		\end{equation}
		holds for all $t\geq0$. In particular, we obtain the following properties:
		\begin{enumerate}
			\item
			$\sigma(R^{L,\mathbf{v}}(t)) = \sigma(R^{\mathbf{v}/L}(t)) \quad \text{ for all } t\geq 0,$
			\item
			$D(-A^{L,\mathbf{v}}) = \mathcal{S}^{-1} D(-A^{\mathbf{v}/L}) =
			\mathcal{S}^{-1} D(-A^{\mathbf{v}}),$
			\item
			$(-A+B)^{L, \mathbf{v}} = \mathcal{S}^{-1}
			(-A+B)^{\mathbf{v}/L}\mathcal{S},$
			\item
			$\sigma((-A+B)^{L,\mathbf{v}}) = \sigma((-A+B)^{\mathbf{v}/L}).$
		\end{enumerate}
	\end{theorem}
	\begin{proof}
		Let $u\in C^\infty(\T^d_L)$ and consider $\widetilde{u} \coloneqq
		\mathcal{S}u.$ Then, $\widetilde{u}(t,\cdot) \coloneqq
		(R^{\mathbf{v}/L}(t)\widetilde{u})(\cdot)$ is the solution of
		\begin{equation*}
			\partial_t \colvec{\widetilde{u}_1 \\ \vdots \\ \widetilde{u}_N}(t,\mathbf{x}) + \frac{1}{L}
			\colvec{\mathbf{v_1} \cdot \nabla\widetilde{u}_1 \\ \vdots \\ \mathbf{v_N} \cdot
			\nabla\widetilde{u}_N}(t,\mathbf{x}) = B \colvec{\widetilde{u}_1 \\
			\vdots \\ \widetilde{u}_N}(t,\mathbf{x})
		\end{equation*}
		on $\T^d$ with initial value $\widetilde{u}(0,\cdot) = \widetilde{u}(\cdot).$
		The rescaled function $u(t,\cdot) \coloneqq
		\mathcal{S}^{-1}\widetilde{u}(t,\cdot)$ fulfills
		$u(0,\cdot)=\mathcal{S}^{-1}\widetilde{u}(\cdot)=u(\cdot)$ and solves
		\begin{align*}
			\partial_t \colvec{u_1 \\ \vdots \\ u_N}(t,\mathbf{x}) &+ \colvec{\mathbf{v_1} \cdot \nabla u_1 \\ \vdots \\ \mathbf{v_N} \cdot \nabla u_N}(t,\mathbf{x}) = \partial_t
			\colvec{\widetilde{u}_1 \\ \vdots \\ \widetilde{u}_N}(t,\frac{\mathbf{x}}{L}) +
			\frac{1}{L} \colvec{\mathbf{v_1} \cdot \nabla\widetilde{u}_1 \\ \vdots \\
			\mathbf{v_N}\cdot \nabla\widetilde{u}_N}(t,\frac{\mathbf{x}}{L}) \\
			&= B\widetilde{u}(t,\frac{\mathbf{x}}{L}) = Bu(t,\mathbf{x}).
		\end{align*}
		This shows $R^{L,\mathbf{v}}(t) = \mathcal{S}^{-1}
		R^{\mathbf{v}/L}(t)\mathcal{S}$ on $C^\infty(\T^d_L)$ and density in
		$L^p(\T^d_L)$ implies \eqref{auxeq27}. \newline
		The second equality in $(ii)$ is a direct consequence of the definition of the domain, see \eqref{generator_by_definition}. All other properties result from
		\eqref{auxeq27} and the fact that $\mathcal{S}$ and its inverse are both bounded
		bijections.
	\end{proof}
	\begin{remark}
		Property $(i)$ and $(iv)$ in \Cref{theorem_side_length} are by far the most important one
		for applications. In conjunction with a connection of the spectrum of the
		semigroup to the generator, for instance \Cref{theorem_WSMTHM_p_eq_2}, they allow us to study long-time asymptotics for arbitrary side lengths $L>0$.
	\end{remark}

	\section{Symmetric Models with Neumann Boundary Conditions}\label{section_neumann_boundary}
	
	This section only deals with the one dimensional case $d=1$ and we consider the Banach spaces $L^p((0,L), \, \C^{2N})$ for $L>0.$ Similar to the previous section, the dependence on $1\leq p<\infty$ plays a secondary role and our notation does not include this parameter, but stresses the boundary conditions and the length $L>0$ of the domain $(0,L).$ \newline 
	
	Many times in applications, for example in \cite{lutscher2002emerging,hillen1996turing,hillen2002hyperbolic,hadeler1999reaction}, there are $N$ different particles
	$u_1, \cdots, u_N$, whose density can be split into right and left moving
	subgroups $\alpha_\indexone$ and $\beta_\indexone$ with $u_\indexone =
	\alpha_\indexone + \beta_\indexone$ for $\indexone=1,\cdots,N$. We write
	$v_1,\cdots,v_N>0$ for the speeds of the particles and use the  notation
	\begin{equation*}
		\Gamma \coloneqq \begin{pmatrix}
		v_1  & & \\
		& \ddots & \\
		& & v_N
		\end{pmatrix}, \qquad
		\alpha \coloneqq (\alpha_1, \cdots, \alpha_N)^T \quad \text{ and } \quad \beta
		\coloneqq (\beta_1, \cdots, \beta_N)^T.
	\end{equation*}
	With this notation, many models studied in the above mentioned literature
	read
	\begin{equation}\label{symmetric_transport_reaction}
		\partial_t\colvec{\alpha \\ \beta} + \begin{pmatrix}
		\Gamma & 0 \\
		0 & -\Gamma
		\end{pmatrix} \partial_x \colvec{\alpha \\ \beta} = \begin{pmatrix}
		B_1 & B_2 \\
		B_2 & B_1
		\end{pmatrix} \colvec{\alpha \\ \beta} \eqqcolon B \colvec{\alpha \\ \beta}
	\end{equation}
	for some matrices $B_1, B_2 \in \R^{N\times N}.$ The left-hand side of
	\eqref{symmetric_transport_reaction} simply transcribes the model assumption
	that left and right moving particles have the same speed and the symmetry of $B$
	is always a consequence of symmetry assumptions on the linear turning behavior
	of the particles and on the nonlinearities appearing in the non-linearized
	models. Given this setup, it is very natural to assume that particles hitting
	the boundary of $[0,L]$ are reflected and turn into particles moving in the
	other direction. These considerations yield	 the Neumann boundary conditions
	\begin{equation}\label{neumann_boundary}
		\alpha(t,0) = \beta(t,0) \quad \text{ and } \quad \alpha(t,L) = \beta(t,L)
		\quad \text{ for all } t\geq 0.
	\end{equation}
	We show a connection of \eqref{symmetric_transport_reaction}
	with Neumann boundary conditions \eqref{neumann_boundary} on $(0,L)$ to
	\eqref{symmetric_transport_reaction} with periodic boundary conditions on
	$(0,2L).$ This connection has already been pointed out in \cite[Lemma 2, p. 242]{lutscher2002modeling} and in the proof of \cite[Theorem 4.5, p. 18]{hillen2010existence} for a Goldstein-Kac model, which is probably the most
	famous example of an equation of type \eqref{symmetric_transport_reaction}; see
	\Cref{section_goldstein_kac} for the model description. We show that the idea
	from \cite[Theorem 4.5, p. 18]{hillen2010existence} carries over to the general
	case by using (without proof) that
	\begin{equation}\label{neumann_generator}
		(-A+B)^{\text{Neu}, \, L} \coloneqq - \begin{pmatrix}
		\Gamma & 0 \\
		0 & -\Gamma
		\end{pmatrix} \partial_x + B
	\end{equation}
	with domain
	\begin{equation}\label{neumann_generator_domain}
		D((-A+B)^{\text{Neu}, \, L}) = \bigg\{ (\alpha, \, \beta)^T \in
		W^{1,p}((0,L), \, \C^{2N}) \, \colon \, \alpha(0)=\beta(0) \text{ and }
		\alpha(L) = \beta(L)\bigg\}
	\end{equation}
	generates a strongly continuous semigroup on $L^p((0,L), \, \C^{2N})$, which we
	call $(R^{\text{Neu},L}(t))_{t\geq0}.$ This statement seems very reasonable
	and it should not be too difficult to prove it with, for example, the
	Lumer-Phillips \Cref{lumer_phillips} and
	the same perturbation argument as in \Cref{theorem_semigroup_generated_by_-A+B}.
	\newline
	To make the connection to periodic boundary conditions precise, we introduce the
	extension operator
	\begin{equation}\label{extension_operator}
		\begin{aligned}
		\mathcal{E}\colon L^p((0,L), \, \C^{2N}) &\longrightarrow L^p((0,2L), \, \C^{2N})
		\\
		\colvec{\alpha \\ \beta} &\longmapsto \colvec{\widetilde{\alpha} \\ \widetilde{\beta}}
		\end{aligned}
	\end{equation}
	with
	\begin{equation*}
		\widetilde{\alpha}(x) = \begin{cases}
		\alpha(x) &  x\in (0,L) \\
		\beta(2L-x) & x\in(L, 2L)
		\end{cases}
		\quad \text{ and } \quad
		\widetilde{\beta}(x) = \begin{cases}
		\beta(x) &  x\in (0,L) \\
		\alpha(2L-x) & x\in(L, 2L)
		\end{cases}.
	\end{equation*}
	Furthermore, we introduce the notation $(-A+B)^{\text{per},L}$ for the
	operator defined in \eqref{neumann_generator} with domain
	\begin{equation}\label{periodic_gen_domain}
		D((-A+B)^{\text{per}, \, L}) = \bigg\{ (\alpha, \, \beta)^T \in
		W^{1,p}((0,L), \, \C^{2N}) \, \colon \, \alpha(0)=\alpha(L) \text{ and }
		\beta(0) = \beta(L)\bigg\}.
	\end{equation}
	It should be noted that $(-A+B)^{\text{per},L}$ is the the generator of the transport-reaction semigroup on the circle $\T_L$ from \Cref{theorem_semigroup_generated_by_-A+B}. We introduced this alternative
	definition solely for notation purposes.
	
	\begin{theorem}\label{theorem_neumann_periodic_connection}
		Let $1\leq p < \infty$ and let $L>0$. Moreover, let
		$(R^{\text{Neu},L}(t))_{t\geq0}$ be the semigroup on $L^p((0,L), \,
		\C^{2N})$ generated by $(-A+B)^{\text{Neu},L}$ and let $(R^{\text{per}, 2L}(t))_{t\geq0}$ be the semigroup on $L^p((0,2L), \, \C^{2N})$ generated by
		$(-A+B)^{\text{per},2L}$. Then
		\begin{equation*}
			R^{\text{Neu},L}(t) \colvec{\alpha \\ \beta} = \restr{R^{\text{per}, 2L}(t)
			\mathcal{E}\bigg(\colvec{\alpha \\ \beta}\bigg)}{(0,L)}
		\end{equation*}
		holds for all $(\alpha, \, \beta) \in L^p((0,L), \, \C^{2N})$ and all
		$t\geq0$, where $\mathcal{E}$ is the extension operator defined in
		\eqref{extension_operator}.
	\end{theorem}
	\begin{proof}
		Let $(\alpha, \beta) \in D((-A+B)^{\text{Neu},L})$ and define
		$(\widetilde{\alpha}, \widetilde{\beta}) \coloneqq \mathcal{E}((\alpha,\beta))
		\in D((-A+B)^{\text{per},2L}).$ Notice that the periodic boundary conditions are
		fulfilled due to the Sobolev embedding $W^{1,p}((0,L)) \subseteq
		C([0,L])$. Consider the solution
		\begin{equation}\label{auxeq23}
			\colvec{\widetilde{\alpha} \\ \widetilde{\beta}}(t, \cdot) = R^{\text{per},
			2L}(t) \colvec{\widetilde{\alpha} \\ \widetilde{\beta}}(\cdot)
		\end{equation}
		of the model \eqref{symmetric_transport_reaction} on $(0,2L)$ with periodic boundary conditions. The next step is
		to show that the symmetry induced by the extension $\mathcal{E}$ is preserved
		for all times. To this end, we define the function
		\begin{equation*}
			\colvec{\widehat{\alpha} \\ \widehat{\beta}}(t, \cdot) \coloneqq
			\colvec{\widetilde{\beta} \\ \widetilde{\alpha}}(t, 2L - \cdot),
		\end{equation*}
		which fulfills the initial condition
		\begin{equation*}
			(\widehat{\alpha}, \widehat{\beta})(0,x) = (\widetilde{\beta},
			\widetilde{\alpha})(0,2L-x) = (\widetilde{\alpha}, \widetilde{\beta})(0,x)
		\end{equation*}
		by definition \eqref{extension_operator} of the extension. In addition, $(\widehat{\alpha}, \widehat{\beta})$ fulfills periodic boundary conditions and we
		have
		\begin{align*}
			\partial_t \colvec{\widehat{\alpha} \\ \widehat{\beta}}(t,x) &+
			\begin{pmatrix}
			\Gamma & 0 \\
			0 & -\Gamma
			\end{pmatrix} \partial_x\colvec{\widehat{\alpha}\\ \widehat{\beta}}(t,x) =
			\partial_t \colvec{\widetilde{\beta} \\ \widetilde{\alpha}}(t, 2L-x) +
			\begin{pmatrix}
			-\Gamma & 0 \\
			0 & \Gamma
			\end{pmatrix} \partial_x\colvec{\widetilde{\beta} \\
			\widetilde{\alpha}}(t,2L-x) \\
			&= \begin{pmatrix}
			B_2 & B_1 \\
			B_1 & B_2
			\end{pmatrix} \colvec{\widetilde{\alpha} \\ \widetilde{\beta}}(t,2L-x) =
			\begin{pmatrix}
			B_1 & B_2 \\
			B_2 & B_1
			\end{pmatrix} \colvec{\widehat{\alpha} \\ \widehat{\beta}}(t,x)
		\end{align*}
		almost everywhere. Uniqueness of the solution to \eqref{symmetric_transport_reaction} implies $(\widehat{\alpha},	\widehat{\beta}) = (\widetilde{\alpha}, \widetilde{\beta}),$ i.e.
		\begin{equation}\label{auxeq22}
			\colvec{\widetilde{\alpha} \\ \widetilde{\beta}}(t,\cdot) =
			\colvec{\widetilde{\beta} \\ \widetilde{\alpha}}(t, 2L - \cdot)
		\end{equation}
		and, together with the periodic boundary conditions for
		$(\widetilde{\alpha},\widetilde{\beta})$, equation \eqref{auxeq22} particularly
		shows
		\begin{equation*}
			\widetilde{\alpha}(t,0) = \widetilde{\beta}(t,2L) = \widetilde{\beta}(t,0)
			\quad \text{ and } \quad \widetilde{\alpha}(t,L) = \widetilde{\beta}(t,L).
		\end{equation*}
		That is, the restriction of $(\widetilde{\alpha}, \widetilde{\beta})$ to
		$(0,L)$ solves \eqref{symmetric_transport_reaction} with the Neumann boundary
		conditions \eqref{neumann_boundary}. Recalling \eqref{auxeq23}, this means
		\begin{equation*}
			R^{\text{Neu},L}(t) \colvec{\alpha \\ \beta} = \restr{R^{\text{per}, 2L}(t)
			\mathcal{E}\bigg(\colvec{\alpha \\ \beta}\bigg)}{[0,L]}
		\end{equation*}
		for all $(\alpha, \beta) \in D((-A_p+B_p)^{\text{Neu},L})$ and all $t\geq0.$
		Density of the domain of the generator, see \Cref{hille-yosida}, yields the theorem.
	\end{proof}
	\begin{remarks}
		\begin{enumerate}
			\item
			The result simply states that every solution of the Neumann problem on
			$(0,L)$ can be obtained by solving the periodic problem on $(0,2L)$ with the
			appropriately chosen initial function and subsequently restricting it to the
			original interval.
			\item
			This correspondence heavily relies on the fact that right and left moving
			particles have the same speed. If the speeds differ, concentrations in the
			Neumann model will ``overlap", which is not the case in periodic models (this
			illustration refers to $B=0$).
		\end{enumerate}
	\end{remarks}
	
	The result clearly justifies our strong focus on periodic boundary conditions.
	Nevertheless, it would be desirable to characterize the spectrum of the
	generator $(-A+B)^{\text{Neu},L}$ and its corresponding semigroup
	$(R^{\text{Neu},L}(t))_{t\geq0}.$ Without any additional effort, we can only
	state the following result.
	
	\begin{corollary}\label{corollary_spectrum_neumann}
		Let $1\leq p < \infty$ and $ L>0$. Assume that the transport directions
		$v_1,\cdots,v_N$ are non-vanishing, i.e. $v_1,\cdots,v_N\neq 0$. Moreover, let $(R^{\text{Neu},L}(t))_{t\geq0}$ be the semigroup on $L^p((0,L), \, \C^{2N})$ generated by $(-A+B)^{\text{Neu},L}$ and let $(R^{\text{per}, 2L}(t))_{t\geq0}$ be
		the semigroup on $L^p((0,2L), \, \C^{2N})$ generated by
		$(-A+B)^{\text{per},2L}$. Then the following properties hold true:
		\begin{enumerate}
			\item
			$\sigma((-A+B)^{\text{Neu},L}) = \sigma_p((-A+B)^{\text{Neu},L}),$
			\item
			$\sigma((-A+B)^{\text{Neu},L}) \subseteq
			\sigma((-A+B)^{\text{per},2L}),$
			\item
			$\sigma_p(R^{\text{Neu},L}(t)) \subseteq
			\sigma_p(R^{\text{per},2L}(t)) \quad \text{ for all } t\geq0.$
		\end{enumerate}
	\end{corollary}
	\begin{proof}
		\begin{enumerate}
			\item
			This follows similarly to the proof of
			\Cref{theorem_generator_has_compact_resolvent_spectrum=point_spectrum} and every
			property mentioned there also holds for the generator with Neumann boundary
			conditions. Everything works out nicely due to the compact embedding
			\begin{equation*}
				D((-A+B)^{\text{Neu,L}}) \subset \subset L^p((0,L), \, \C^{2N}),
			\end{equation*}
			which is a consequence of the assumption of non-vanishing transport
			directions. To see that spectrum is non-empty, let $\lambda\in\C$ be an eigenvalue of $B_1+B_2$ with eigenvector $z\in \C^N\backslash \{0\}$ and define the constant function $(\alpha,\beta) \equiv (z,z).$ Then, $(\alpha,\beta)$ fulfills the boundary conditions $\alpha(0) = \beta(0), \,
			\alpha(L)=\beta(L)$ and satisfies
			\begin{equation*}
				(-A+B)^{\text{Neu}, \, L}\colvec{\alpha \\ \beta} = \begin{pmatrix}
				B_1 & B_2 \\
				B_2 & B_1
				\end{pmatrix} \colvec{z \\ z} = \lambda \colvec{\alpha \\ \beta}.
			\end{equation*}
			\item
			Let $\lambda \in \sigma((-A+B)^{\text{Neu},L}).$ We know from $(i)$ that
			$\lambda$ is an eigenvalue, i.e. there exists $(\alpha,\beta)\in
			D((-A+B)^{\text{Neu},L})\backslash \{0\}$ with
			\begin{equation}\label{auxeq24}
				-\begin{pmatrix}
				\Gamma & 0 \\
				0 & -\Gamma
				\end{pmatrix} \partial_x \colvec{\alpha \\ \beta} + \begin{pmatrix}
				B_1 & B_2 \\
				B_2 & B_1
				\end{pmatrix} \colvec{\alpha \\ \beta} = \lambda \colvec{\alpha \\ \beta}.
			\end{equation}
			The extension $(\widetilde{\alpha},\widetilde{\beta}) \coloneqq
			\mathcal{E}((\alpha,\beta))$ fulfills periodic boundary conditions and for
			$x\in(L,2L)$ and we can use \eqref{auxeq24} to obtain
			\begin{align*}
				-\begin{pmatrix}
				\Gamma & 0 \\
				0 & -\Gamma
				\end{pmatrix} &\partial_x\colvec{\widetilde{\alpha} \\ \widetilde{\beta}}(x)
				+ \begin{pmatrix}
				B_1 & B_2 \\
				B_2 & B_1
				\end{pmatrix} \colvec{\widetilde{\alpha} \\ \widetilde{\beta}}(x) \\
				&= \begin{pmatrix}
				\Gamma & 0 \\
				0 & -\Gamma
				\end{pmatrix}\partial_x \colvec{\beta \\ \alpha}(2L-x) + \begin{pmatrix}
				B_1 & B_2 \\
				B_2 & B_1
				\end{pmatrix} \colvec{\beta \\ \alpha} (2L-x) \\
				&= \lambda \colvec{\beta \\ \alpha}(2L-x) = \lambda\colvec{\widetilde{\alpha}
				\\ \widetilde{\beta}}(x)
			\end{align*}
			almost everywhere. Hence, $(\widetilde{\alpha}, \widetilde{\beta}) \in
			D((-A+B)^{\text{per},2L})$ is an eigenfunction of
			$(-A+B)^{\text{per},2L}$ with eigenvalue $\lambda$.
			\item
			Let $t\geq0$ and $\lambda \in \sigma_p(R^{\text{Neu},L}(t))$ with
			eigenfunction $(\alpha, \beta) \in L^p((0,L),\C^{2N}).$ Consider the extension
			$(\widetilde{\alpha}, \widetilde{\beta}) =\mathcal{E}((\alpha,\beta)) \in
			L^p((0,2L), \, \C^{2N}).$ \Cref{theorem_neumann_periodic_connection} yields
			\begin{equation}\label{auxeq25}
				\restr{R^{\text{per},2L}(t)\colvec{\widetilde{\alpha} \\
				\widetilde{\beta}}}{(0,L)} = R^{\text{Neu},L}(t) \colvec{\alpha \\ \beta} =
				\lambda \colvec{\alpha \\ \beta} = \lambda \restr{\colvec{\widetilde{\alpha} \\
				\widetilde{\beta}}}{(0,L)}.
			\end{equation}
			Moreover, density of $D((-A+B)^{\text{Neu},L})$ in $L^p((0,L), \, \C^{2N})$ and \eqref{auxeq22}
			implies
			\begin{equation*}
				R^{\text{per},2L}(t) \colvec{\widetilde{\alpha} \\ \widetilde{\beta}}(x)
				= \begin{pmatrix}
				0 & I_{\C^{N\times N}} \\
				I_{\C^{N\times N}} & 0
				\end{pmatrix} \bigg( R^{\text{per},2L}(t) \colvec{\widetilde{\alpha} \\
				\widetilde{\beta}}(2L-x) \bigg)
			\end{equation*}
			almost everywhere on $(0,2L).$ Thus, almost every $x\in(L,2L)$ fulfills
			\begin{align*}
				R^{\text{per},2L}(t)\colvec{\widetilde{\alpha} \\ \widetilde{\beta}} (x) &=
				\begin{pmatrix}
				0 & I_{\C^{N\times N}} \\
				I_{\C^{N\times N}} & 0
				\end{pmatrix} R^{\text{Neu},L}(t) \colvec{\alpha \\ \beta} (2L-x) \\
				&= \lambda \colvec{\beta \\ \alpha} (2L-x) = \lambda
				\colvec{\widetilde{\alpha} \\ \widetilde{\beta}}(x)
			\end{align*}
			by \eqref{auxeq25} and we obtain $\lambda \in
			\sigma_p(R^{\text{per},2L}(t))$ with eigenfunction $(\widetilde{\alpha},
			\widetilde{\beta}).$
		\end{enumerate}
	\end{proof}

	\newpage
	
	\chapter{Applications: Long-Time Behavior of Hyperbolic Models}\label{chapter_application}
	
	The theory from \Cref{chapter_spectral_analysis} can deal with many existing myxobacteria models on a linearized level. It allows us to study the long-time
	behavior of solutions and the results obtained can, for example, be used to verify already existing qualitative stability theorems from the perspective of semigroup theory. In addition, the weak spectral mapping \Cref{theorem_WSMTHM_p_eq_2} enables us to prove sharp quantitative estimates. \newline
	The methods used do not rely on any symmetry assumption	or on the typically used Kac trick \cite{kac1974stochastic}, which allows to reduce symmetric systems of hyperbolic equations to reaction telegraph equations for sufficiently smooth solutions. It has for example been used in	\cite{hillen1996turing,lutscher2002emerging,eftimie2012hyperbolic} and we refer
	to these references and \cite{hillen2010existence} for any details on the trick.\newline
	
	We show an optimal convergence result for a classical Goldstein-Kac model, apply
	our theory to models ``with killing" and end with an application to a more
	delicate model from \cite{hillen1996turing}.
	
	\section{A Goldstein-Kac Model}\label{section_goldstein_kac}
	
	In \cite{lutscher2002emerging}, the authors proposed a myxobacteria model of
	Goldstein-Kac type. They focused on a system with $N=2$ bacteria groups and
	their equation reads
	\begin{equation}\label{system_goldstein_kac}
		\partial_t \colvec{w_1 \\ w_2} + \begin{pmatrix}
		v & 0 \\
		0 & -v
		\end{pmatrix} \partial_x \colvec{w_1 \\ w_2} = \begin{pmatrix}
		-(\mu+\eta_1) & \mu+\eta_2 \\
		\mu+\eta_1 & -(\mu+\eta_2)
		\end{pmatrix} \colvec{w_1 \\ w_2},
	\end{equation}
	where $v>0$ is the speed, $\mu>0$ is an autonomous turning rate and
	$\eta_\indexone = \eta_\indexone(w_1, \, w_2) \geq0$ for $j=1,2$ are turning rates, which
	depend on the concentrations. Additionally, they assume symmetry with respect to
	a change of left and right, i.e.
	\begin{equation*}
		\eta_1(w_1, w_2) = \eta_2(w_2, w_1) \eqqcolon \eta(w_1, w_2)
	\end{equation*}
	for all $w_1, \, w_2 \in \R.$ System \eqref{system_goldstein_kac} has the
	stationary states $w_1 = w_2 \equiv c$ for arbitrary $c\in\R.$ We define $u_1
	\coloneqq w_1 -c, \, u_2 \coloneqq w_2 - c$ and linearize with Taylor to obtain
	\begin{equation}\label{model_Goldstein_kac}
		\begin{aligned}
		\partial_t u &+ \begin{pmatrix}
		v & 0 \\
		0 & -v
		\end{pmatrix} \partial_x u = \begin{pmatrix}
		-\mu & \mu \\
		\mu & -\mu
		\end{pmatrix} w + \begin{pmatrix}
		-\eta(w_1, w_2) & \eta(w_2, w_1) \\
		\eta(w_1, w_2) & -\eta(w_2, w_1)
		\end{pmatrix} w \\
		&= \begin{pmatrix}
		-\mu & \mu \\
		\mu & -\mu
		\end{pmatrix} u + \colvec{(\eta(c,c)+\partial_1\eta(c,c)c -
			\partial_2\eta(c,c)c )(w_2-w_1) \\ (\eta(c,c)+\partial_1\eta(c,c)c -
			\partial_2\eta(c,c)c )(w_1-w_2)} \\
		&= \begin{pmatrix}
		-\Lambda & \Lambda \\
		\Lambda & -\Lambda
		\end{pmatrix} u
		\end{aligned}
	\end{equation}
	for $\Lambda \coloneqq \mu+\eta(c,c)+\partial_1\eta(c,c)c -
	\partial_2\eta(c,c)c$, cf. \cite[(5), p. 624]{lutscher2002emerging}. \newline
	This linearized problem is well-posed and the solution of the abstract Cauchy
	problem with initial function $u\in L^2(\T, \, \C^2)$ is given by $t \mapsto
	R(t)u$, where $(R(t))_{t\geq0}$ is the transport-reaction semigroup on $L^2(\T, \, \C^2)$ generated by 
	\begin{equation}\label{gen_goldstein_kac}
		-A+B \coloneqq -\begin{pmatrix}
		v & 0 \\
		0 & -v
		\end{pmatrix} \partial_x + \begin{pmatrix}
		-\Lambda & \Lambda \\
		\Lambda & -\Lambda
		\end{pmatrix},
	\end{equation}
	with domain $H^1(\T, \, \C^2)$ see \Cref{theorem_semigroup_generated_by_-A+B} and \Cref{corollary_domain_generator_one_dimensional}.
	Due to \Cref{theorem_WSMTHM_p_eq_2}, a stability analysis reduces to a spectral
	analysis of the generator. In particular, the growth bound of the semigroup
	coincides with the spectral bound of the generator by \Cref{wsmthm_implies_growth_bd_eq_spectral_bound}. Secondly, \Cref{theorem_spectrum_generator}
	implies
	\begin{equation*}
		\sigma(-A+B) = \overline{\bigcup_{k\in \Z} \sigma(M(k))},
	\end{equation*}
	where
	\begin{equation}\label{auxeq16}
		M(k) = -2 \pi i \indextwo \begin{pmatrix}
		v & 0 \\
		0 & -v
		\end{pmatrix} + \begin{pmatrix}
		-\Lambda & \Lambda \\
		\Lambda & -\Lambda
		\end{pmatrix} = \begin{pmatrix}
		-2\pi i \indextwo v -\Lambda & \Lambda \\
		\Lambda & 2\pi i\indextwo v -\Lambda
		\end{pmatrix}.
	\end{equation}
		Hence, we need to compute the roots of
	\begin{align*}
		\det\begin{pmatrix}
		-2\pi i \indextwo v -\Lambda- \lambda & \Lambda \\
		\Lambda & 2\pi i\indextwo v -\Lambda - \lambda
		\end{pmatrix} &= (-2\pi i \indextwo v-\Lambda-\lambda)(2\pi i \indextwo v -
		\Lambda - \lambda) - \Lambda^2 \\
		&= \lambda^2 + 2 \Lambda \lambda + 4 \pi^2 \indextwo^2 v^2.
	\end{align*}
	They are clearly given by
	\begin{equation*}
		\lambda_{1,2}(\indextwo) = - \Lambda\pm \sqrt{\Lambda^2 - 4 \pi^2\indextwo^2 v^2}
	\end{equation*}
	and particularly
	\begin{equation}\label{spectrum_goldstein_kac_generator}
		\sigma(-A+B) = \{ - \Lambda\pm \sqrt{\Lambda^2 - 4 \pi^2\indextwo^2 v^2} \,
		\colon \, \indextwo\in\Z \}
	\end{equation}
	because the union of all eigenvalues is a closed set. It follows immediately that the stationary states are unstable for $\Lambda < 0$ (consider $\indextwo=0$) and it always holds $0\in \sigma(-A+B)$. \newline
	The goal is now to show a precise convergence result if we have $\Lambda>0.$ Then, every eigenvalue except $0$ has negative real part and we essentially need a spectral decomposition of the generator. To this end, let
	\begin{equation}
		\begin{aligned}
			\{1\} &= \{u\in L^2(\T, \, \C^2) \, \colon \, u \equiv const \in \C^2\}, \\
			\{1\}^{\perp} &= \{u\in L^2(\T, \, \C^2) \, \colon \, \int_{\T} u(x) \, dx =
			0\},
		\end{aligned}
	\end{equation}
	such that $L^2(\T, \, \C^2) = \{1\} \oplus \{1\}^\perp.$ Notice that
	$E\coloneqq(\{1\}^\perp, \, \|\cdot\|_{L^2(\T, \C^2)})$ defines a Banach space.
	We consider the part of $(-A+B, \, H^1(\T, \, \C^2))$ in $E,$ i.e.
	\begin{equation*}
		\restr{-A+B}{E}u \coloneqq (-A+B)u
	\end{equation*}
	with
	\begin{equation}\label{restriction_generator_goldstein_kac}
		D(\restr{-A+B}{E}) \coloneqq \left\{u\in H^1(\T, \, \C^2) \, \colon \,
		\int_{\T}u(x) \, dx = 0\right\}.
	\end{equation}
	
	\begin{proposition}\label{proposition_char_spec_goldstein_kac_restr}
		Let $v>0$ and let $ \Lambda \in \R.$ Then $(\restr{-A+B}{E}, \, D(\restr{-A+B}{E}))$
		generates a strongly continuous semigroup on $E.$ This semigroup is given by
		$(\restr{R(t)}{E})_{t\geq0},$ i.e. the restriction of $(R(t))_{t\geq0}$ to $E.$
		Moreover, it holds
		\begin{equation*}
			\sigma(\restr{-A+B}{E}) =\sigma(-A+B)\backslash\{-2\Lambda, 0\} =  \{ -
			\Lambda\pm \sqrt{\Lambda^2 - 4 \pi^2\indextwo^2 v^2} \, \colon \,
			\indextwo\in\Z\backslash\{0\} \}
		\end{equation*}
	\end{proposition}
	\begin{proof}
		Firstly, note that $E$ is $(R(t))_{t\geq0}$ invariant due to \Cref{time_development_averages}. Secondly, $E$ is continuously embedded into $L^2(\T,\C^2).$ We find, using \cite[\Romannum{2} 2.3, p. 60]{engel2001one}, that the restricted
		semigroup $(\restr{R(t)}{E})_{t\geq0}$ has the generator
		\begin{equation*}
			\restr{-A+B}{E}
		\end{equation*}
		with domain
		\begin{align*}
			D(\restr{-A+B}{E}) &= \big\{ u \in H^1(\T,\C^2) \cap E \, \colon \, (-A+B)u \in
			E\big\} = \left\{u\in H^1(\T, \, \C^2) \, \colon \, \int_{\T}u(x) \, dx = 0\right\}.
		\end{align*}
		Here, the second equality follows for example from the $E$-invariance of $(R(t))_{t\geq0}$, \Cref{def_generator} of generators and closedness of $E$ as a subspace of $L^2(\T, \, \C^2)$. \newline
		It is left to show the statements on the spectrum. The second equality is a
		direct consequence of \eqref{spectrum_goldstein_kac_generator} and we start with
		the direction $``\supseteq"$ for the first equality. If $\indextwo\neq0$ and
		$\lambda\in\sigma(M(\indextwo)),$ we can define
		\begin{equation*}
			u(x) \coloneqq z e^{2\pi i \indextwo x} \in E,
		\end{equation*}
		where $z\in\C^2\backslash\{0\}$ is an eigenvector of $M(\indextwo)$ with eigenvalue $\lambda.$ \Cref{proposition_generator_application_fourier} implies $\restr{-A+B}{E}u = \lambda u$ and $\lambda\in\sigma(\restr{-A+B}{E})$ follows. \newline
		Next, we show the direction $``\subseteq".$ Similarly to the proof of	\Cref{theorem_generator_has_compact_resolvent_spectrum=point_spectrum}, one can easily see that $(D(\restr{-A+B}{E}), \, \|\cdot\|_{-A+B})$ is compactly embedded into $E.$ This yields
		\begin{equation*}
			\sigma(\restr{-A+B}{E}) = \sigma_p(\restr{-A+B}{E}),
		\end{equation*}
		since we already know that the spectrum is non-empty from the previously proven
		direction. Notice that every eigenvalue of $\restr{-A+B}{E}$ is also an
		eigenvalue of $-A+B,$ simply by definition. We conclude
		\begin{equation*}
			\sigma(\restr{-A+B}{E})  \subseteq \sigma(-A+B)
		\end{equation*}
		and the final step is to show that neither $-2\Lambda$ nor $0$ can be eigenvalues
		of $\restr{-A+B}{E}.$ We present the proof for $\lambda = -2\Lambda$ and argue
		by contradiction. The proof for $\lambda=0$ is similar. To this end, let $u\in
		D(\restr{-A+B}{E}) \backslash \{0\}$ and assume
		\begin{equation*}
			-\begin{pmatrix}
			v & 0 \\
			0 & -v
			\end{pmatrix} \colvec{u^\prime_1 \\ u'_2} + \begin{pmatrix}
			-\Lambda & \Lambda \\
			\Lambda & -\Lambda
			\end{pmatrix}u = -2\Lambda\colvec{u_1 \\ u_2},
		\end{equation*}
		which is equivalent to
		\begin{equation}\label{auxeq13}
			-\begin{pmatrix}
			v & 0 \\
			0 & -v
			\end{pmatrix} \colvec{u'_1 \\ u'_2} + \begin{pmatrix}
			\Lambda & \Lambda \\
			\Lambda & \Lambda
			\end{pmatrix}u = 0.
		\end{equation}
		We obtain $(u_1+u_2)'=0,$ i.e. $u_1 \equiv -u_2+c$ for some constant $c\in\C.$
		The integral constraint \eqref{restriction_generator_goldstein_kac} implies
		$c=0$ and $u_1 \equiv -u_2.$ Plugging this information back into \eqref{auxeq13}
		yields $u_1 = -u_2 \equiv \widetilde{c}$ for a second constant
		$\widetilde{c}\in\C.$ Using again the integral constraint \eqref{auxeq13} gives
		$u=0,$ which contradicts our assumption of $u$ being an eigenfunction.
	\end{proof}
	
	In what follows, the idea is to use stability of the semigroup
	$(\restr{R(t)}{E})_{t\geq0}$ on $E$. We would expect stability due to shown
	spectral properties of the generator but actually, we still need a weak spectral
	mapping theorem for the restricted semigroup to make this connection rigorous,
	cf. the introduction of \Cref{chapter_spectral_analysis} and Appendix \ref{appendix_Long_time_behavior_SG} for a detailed
	explanation. \newline
	The theory on matrix multiplication operators, see
	\Cref{def_matrix_mult_op} and the subsequent page, can also deal with
	this technical problem. Taking $(X, \Sigma, \, \mu) = (\Z\backslash\{0\},
	\mathcal{P}(\Z\backslash\{0\}), \, \#)$ with counting measure $\#$ yields
	$L^2(X, \, \C^2) = l^2(\Z\backslash\{0\}, \, \C^2)$ and $\mathcal{F}(E) =
	l^2(\Z\backslash\{0\}, \, \C^2)$, where $\mathcal{F}$ is the Fourier transform
	defined in \eqref{def_fourier_transform}. Following the proof of
	\Cref{theorem_WSMTHM_p_eq_2} and using exactly the same arguments, we obtain
	\begin{equation}\label{auxeq14}
		\sigma(\restr{R(t)}{E}) = \overline{e^{t\sigma(\restr{-A+B}{E})}}.
	\end{equation}
	In fact, these abstract arguments with matrix multiplication operators offer an alternative proof for the characterization of $\sigma(\restr{-A+B}{E})$ in
	\Cref{proposition_char_spec_goldstein_kac_restr}. \newline
	
	\Cref{proposition_char_spec_goldstein_kac_restr} decomposes the spectrum of the generator $(-A+B, \, H^1(\T, \, \C^2))$ into its stable and	central spectrum. Together with \eqref{auxeq14}, we are finally able to prove an optimal convergence result for the model \eqref{model_Goldstein_kac} in the case
	$\Lambda>0$. For the sake of readability, we define the negative constant
	\begin{equation}\label{auxeq12}
		\omega \coloneqq -\Lambda + \real(\sqrt{\Lambda^2-4\pi^2v^2}) = \begin{cases}
		-\Lambda & \Lambda^2 \leq 4\pi^2v^2, \\
		-\Lambda + \sqrt{\Lambda^2-4\pi^2v^2} & \text{otherwise}
		\end{cases}
	\end{equation}
	before stating the theorem. Notice that $\omega\in [-\Lambda,0)$ is the largest
	real part of the eigenvalues in the stable spectrum of $-A+B.$
	
	\begin{theorem}\label{theorem_goldstein_kac_conv}
		Let $v>0, \, \Lambda>0$ and let $(R(t))_{t\geq0}$ be the transport-reaction semigroup on $L^2(\T, \, \C^2)$ generated by the operator defined in \eqref{gen_goldstein_kac}. Moreover, let $\omega$ be defined as in \eqref{auxeq12}. Then, for every $\varepsilon>0,$ there exists a constant $C_\varepsilon>0$ such that
		\begin{equation}\label{auxeq11}
			\| R(t)u - \frac{1}{2} \colvec{\int_{\T} u_1 + \int_{\T} u_2 \\ \int_{\T} u_1
			+ \int_{\T} u_2} \|_{L^2(\T, \C^2)} \leq C_\varepsilon e^{(\omega+\varepsilon)t}
			\|u\|_{L^2(\T, \, \C^2)}
		\end{equation}
		holds for all $u\in L^2(\T, \,\C^2).$ The estimate \eqref{auxeq11} does not
		hold for $\omega+\varepsilon$ replaced by any constant $\widetilde{\omega}<\omega.$
	\end{theorem}
	\begin{proof}
		Let $\varepsilon>0.$ \newline
		
		Step 1: Exponential stability of the restricted semigroup
		$(\restr{R(t)}{E})_{t\geq0}.$ \newline
		Recall that \Cref{proposition_char_spec_goldstein_kac_restr} states
		\begin{equation*}
			\sigma(\restr{-A+B}{E}) = \{ - \Lambda\pm \sqrt{\Lambda^2 - 4 \pi^2\indextwo^2
			v^2} \, \colon \, \indextwo\in\Z\backslash\{0\} \}.
		\end{equation*}
		The eigenvalue with the largest real part is $\lambda \coloneqq -\Lambda +
		\sqrt{\Lambda^2 - 4 \pi^2 v^2}$, i.e.
		\begin{equation*}
			s(\restr{-A+B}{E}) = \real(\lambda) = \omega,
		\end{equation*}
		where $s(\restr{-A+B}{E})$ is the spectral bound of $\restr{-A+B}{E},$ see
		\Cref{long_time_behavior_definition_exponential_growth_bound_spectral_bound}.
		The weak spectral mapping theorem \eqref{auxeq14} yields
		\begin{equation*}
			\omega_0(\restr{-A+B}{E}) = \omega
		\end{equation*}
		for the exponential growth bound of the restricted semigroup, see
		\Cref{long_time_behavior_definition_exponential_growth_bound_spectral_bound} and
		\Cref{wsmthm_implies_growth_bd_eq_spectral_bound}. Hence, there exists a constant
		$C_\varepsilon>0$ such that
		\begin{equation}\label{auxeq15}
			\| \restr{R(t)}{E}u \|_{L^2(\T, \, \C^2)} \leq C_\varepsilon
			e^{(\omega+\varepsilon)t} \|u\|_{L^2(\T, \, \C^2)}
		\end{equation}
		holds for all $u\in E.$ \newline
		
		Step 2: A computation of $e^{tB}.$ \newline
		The matrix
		\begin{equation*}
			B = \begin{pmatrix}
			-\Lambda & \Lambda \\
			\Lambda & -\Lambda
			\end{pmatrix} = \Lambda \begin{pmatrix}
			-1 &  1\\
			1 & -1
			\end{pmatrix}
		\end{equation*}
		can be diagonalized with
		\begin{equation*}
			B = \begin{pmatrix}
			1 & 1 \\
			1 & -1
			\end{pmatrix} \begin{pmatrix}
			0 & 0 \\
			0 & -2\Lambda
			\end{pmatrix} \begin{pmatrix}
			1 & 1 \\
			1 & -1
			\end{pmatrix}^{-1} = \half \begin{pmatrix}
			1 & 1 \\
			1 & -1
			\end{pmatrix} \begin{pmatrix}
			0 & 0 \\
			0 & -2\Lambda
			\end{pmatrix} \begin{pmatrix}
			1 & 1 \\
			1 & -1
			\end{pmatrix}
		\end{equation*}
		and we obtain
		\begin{equation*}
			e^{tB} = \begin{pmatrix}
			1 & 1 \\
			1 & -1
			\end{pmatrix} \begin{pmatrix}
			1 & 0 \\
			0 & e^{-2\Lambda t}
			\end{pmatrix} \begin{pmatrix}
			1 & 1 \\
			1 & -1
			\end{pmatrix}^{-1}
			= \half \begin{pmatrix}
			1+e^{-2\Lambda t} & 1-e^{-2\Lambda t} \\
			1-e^{-2\Lambda t} & 1+e^{-2\Lambda t}
			\end{pmatrix}.
		\end{equation*}
		In particular, $\Lambda>0$ implies
		\begin{equation}\label{auxeq17}
			e^{tB}y - \half\colvec{y_1 + y_2 \\ y_1+y_2} = \half e^{-2\Lambda t}
			\colvec{y_1-y_2 \\ y_2-y_1} \longrightarrow 0
		\end{equation}
		as $t\to\infty$ and for all $y\in\C^2.$ \newline
		
		Step 3: The convergence result for the full model. \newline
		Notice that, if $u\equiv const,$ an application of
		\Cref{proposition_semigroup_fourier_representation} shows
		\begin{equation*}
			R(t)u = e^{tM(0)}u = e^{tB} u.
		\end{equation*}
		Finally, using \eqref{auxeq15} and \eqref{auxeq17}, we obtain
		\begin{align*}
			\| R(t)u &- \frac{1}{2} \colvec{\int_{\T} u_1 + \int_{\T} u_2 \\ \int_{\T} u_1
			+ \int_{\T} u_2} \|_{L^2(\T, \C^2)} \\
			&= \| R(t)(u-\int_{\T} u) + R(t)\int_{\T}u - \frac{1}{2} \colvec{\int_{\T} u_1
			+ \int_{\T} u_2 \\ \int_{\T} u_1 + \int_{\T} u_2} \|_{L^2(\T, \C^2)} \\
			&\leq \| \restr{R(t)}{E} (u-\int_{\T} u) \|_{L^2(\T, \, \C^2)} + \| e^{tB}
			\int_{\T}u - \frac{1}{2} \colvec{\int_{\T} u_1 + \int_{\T} u_2 \\ \int_{\T} u_1
			+ \int_{\T} u_2} \|_{L^2(\T, \C^2)} \\
			&\leq C_\varepsilon e^{(\omega+\varepsilon)t} \| u - \int_{\T} u \|_{L^2(\T,
			\, \C^N)} + \half e^{-2\Lambda t} \bigg| \colvec{\int_{\T} u_1 - \int_{\T} u_2
			\\ \int_{\T} u_2 - \int_{\T } u_1} \bigg| \\
			&\leq C_\varepsilon e^{(\omega+\varepsilon)t} \bigg(\|u-\int_{\T} u\|_{L^2(\T,
			\, \C^2)} + \bigg|\int_{\T} u_1 - \int_{\T} u_2 \bigg| \bigg) \\
			&\leq C_\varepsilon e^{(\omega+\varepsilon)t} \|u\|_{L^2(\T, \, \C^2)}.
		\end{align*}
		In the last two steps, we used new constants $C_\varepsilon$ without
		relabeling. \newline
		
		Step 4: The estimate does not hold for $\omega+\varepsilon$ replaced by
		$\widetilde{\omega} < \omega$. \newline
		To see this claim, notice that $\lambda = -\Lambda + \sqrt{\Lambda^2 - 4 \pi^2
			v^2}$ is an eigenvalue of $M(1)$, where $M(1)$ was defined in \eqref{auxeq16}. We
		take an associated eigenvector $z\in\C^2$ with $|z|=1$ and define
		\begin{equation*}
			u \coloneqq z e^{2\pi i x} \in E.
		\end{equation*}
		\Cref{proposition_semigroup_fourier_representation} implies
		\begin{equation*}
			(R(t) u)(x) = e^{tM(1)}z e^{2\pi i x} = e^{t\lambda} z e^{2\pi i x}.
		\end{equation*}
		Therefore, we have
		\begin{equation*}
			\| R(t)u \|_{L^2(\T, \, \C^2)} = |e^{\lambda t}| = e^{\omega t} \|u\|_{L^2(\T,
			\, \C^2)}
		\end{equation*}
		and the exponential rate of convergence cannot be smaller than $\omega.$
	\end{proof}
	
	\begin{remarks}\label{remark_Goldstein_kac}
		\begin{enumerate}
			\item
			Notice that $\sigma(B) = \{-2\Lambda,0\}$. The qualitative behavior, namely
			convergence to the mean of the initial value or blow-up, does not change due to the additional spacial variable, cf. \eqref{auxeq17}.
			This is contrary to the pattern formation phenomena discussed in
			\Cref{chapter_pattern_formation}
			\item
			One might think that \Cref{theorem_goldstein_kac_conv} can be easily obtained
			from the exact solution formula given in
			\Cref{proposition_semigroup_fourier_representation}, the eigenvalues of $M(\indextwo)$ and Parseval's identity. However, this is a
			fallacy and incorrect. The matrices $M(\indextwo)$ are not normal for
			$\indextwo \neq 0$ and therefore, it is non-trivial to find estimates of the
			operator norms $\|M(\indextwo)\|$ and $\|e^{tM(\indextwo)}\|$ in terms of the
			eigenvalues of $M(\indextwo).$
			\item
			\Cref{theorem_side_length} allows us to state the corresponding result on
			$\T_L = L \T$ for arbitrary length $L>0$ of the circle. Let $\varepsilon>0.$
			Using the properties $(i)$ and $(iv)$ from \Cref{theorem_side_length} yields the
			optimal (up to $\varepsilon$) estimate
			\begin{equation*}
				\| R^{L,v}(t)u - \frac{1}{2} \colvec{\dashint_{\T_L} u_1 + \dashint_{\T_L}
				u_2 \\ \dashint_{\T_L} u_1 + \dashint_{\T_L} u_2} \|_{L^2(\T_L, \C^2)} \leq
				C_\varepsilon e^{(\omega_L+\varepsilon)t} \|u\|_{L^2(\T_L, \, \C^2)}
			\end{equation*}
			for all $u\in L^2(\T_L, \, \C^2),$ where $\omega_L$ is the constant
			\begin{equation*}
				\omega_L \coloneqq -\Lambda + \real(\sqrt{\Lambda^2-\tfrac{4\pi^2}{L^2}v^2})
				= \begin{cases}
				-\Lambda & \Lambda^2 \leq \tfrac{4\pi^2}{L^2}v^2, \\
				-\Lambda + \sqrt{\Lambda^2-\tfrac{4\pi^2}{L^2}v^2} & \text{otherwise}
				\end{cases}.
			\end{equation*}
			\item
			The Goldstein-Kac model \eqref{model_Goldstein_kac} is an example of a
			symmetric model discussed in \Cref{section_neumann_boundary}. Therefore,
			\Cref{theorem_neumann_periodic_connection} can be applied to study the model on
			$(0,L)$ with the Neumann boundary conditions $u_1(t,0)=u_2(t,0)$ and
			$u_1(t,L)=u_2(t,L)$ for all $t\geq0.$ For all $\varepsilon>0$, we obtain
			\begin{equation*}
				\| R^{\text{Neu},L}(t)u - \frac{1}{2} \colvec{\dashint_{[0,L]} u_1 +
				\dashint_{[0,L]} u_2 \\ \dashint_{[0,L]} u_1 + \dashint_{[0,L]} u_2}
				\|_{L^2([0,L], \C^2)} \leq C_\varepsilon e^{(\omega_{2L}+\varepsilon)t}
				\|u\|_{L^2([0,L], \, \C^2)}.
			\end{equation*}
			With Neumann boundary conditions, however, it is unclear whether the
			convergence rate is optimal, cf. \Cref{corollary_spectrum_neumann}.
		\end{enumerate}
	\end{remarks}
	
	The visualization of \Cref{theorem_goldstein_kac_conv} is intuitive, but also surprising. For small transport speeds $v>0,$ the movement of the particles is the factor slowing down the convergence and an increasing speed also increases the rate of convergence towards the equilibrium. Then, for transport faster than the critical velocity $v^* = \tfrac{\Lambda}{2\pi}$, a higher speed / faster ``stirring" motion does not correspond to a faster rate of convergence and the reactions described by $B$ become the factor slowing down the convergence. Somehow, however, the fastest rate is capped by half of the rate of
	the single cell model, cf. \eqref{auxeq17}.
	
	\section{Models with Killing}
	
	A second way to apply the theory from \Cref{chapter_spectral_analysis} is to
	carry over convergence results to different norms. We illustrate the main idea
	with the example of positive and mass conserving models \textit{exposed to killing}
	in the case $d=1$. As a first step, we briefly discuss what we mean by
	\textit{killing} and present the most important properties of this artificial
	modification of the model. The latter can be done in arbitrary dimension $d$.
	
	\begin{definition}
		Let $d\in\N, \, N\in\N$ and let $\delta>0.$ The (original) model
		\begin{equation}\label{original_model}
			\partial_t u + \colvec{\mathbf{v_1} \cdot \nabla u_1 \\ \vdots \\ \mathbf{v_N} \cdot \nabla u_N} = Bu
		\end{equation}
		on $\T^d$ with $B \in \R^{N\times N}$ and $\mathbf{v_1},\cdots,\mathbf{v_N}
		\in\R^d$ is exposed to $\delta$-killing if we modify the model to
		\begin{equation}\label{killing_model}
			\partial_t u + \colvec{\nabla u_1 \cdot \mathbf{v_1} \\ \vdots \\ \nabla u_N
			\cdot \mathbf{v_N}} = Bu - \delta u.
		\end{equation}
	\end{definition}
	\begin{remarks}
		\begin{enumerate}
			\item
			The additional term $-\delta u$ simply changes the birth/death rate of the
			$N$ subgroups in a uniform manner.
			\item
			The mathematically rigorous and consistent notation will often be $-\delta
			I_{E}$ for an appropriate Banach space $E$. But we will neglect this subtlety
			and always use the notation $-\delta$ from above.
		\end{enumerate}
		
	\end{remarks}
	
	Let $1\leq p < \infty.$ Recall that the operator
	\begin{equation*}
		(-A+B)u = -\colvec{\mathbf{v_1} \cdot \nabla u_1 \\ \vdots \\ \mathbf{v_N} \cdot \nabla u_N} + Bu
	\end{equation*}
	with appropriately chosen domain $D(-A)$ generates the transport-reaction semigroup $(R(t))_{t\geq0}$ on $\LpCn$, see \Cref{theorem_semigroup_generated_by_-A+B}. We refer to \Cref{chapter_transport_semigroup} and \Cref{chapter_transport_reaction_SG} for
	the precise setup and the definition of $D(-A)$. Of course, $B-\delta \in
	\R^{N\times N}$ is still a matrix and the whole theory from
	\Cref{chapter_transport_reaction_SG} and \Cref{chapter_spectral_analysis} can
	still be applied. The appropriate point of view, however, is to see
	$\delta$-killing as a modification of the original model. The next lemma
	clarifies this perturbative effect.
	
	\begin{lemma}\label{lemma_killing_sg}
		Let $1\leq p < \infty$ and let $(R(t))_{t\geq0}$ be the transport-reaction semigroup on $\LpCn$ generated by $(-A+B, \, D(-A)).$ Moreover, let $\delta>0.$ The semigroup $(R_\delta(t))_{t\geq0}$ on $\LpCn$ generated by $(-A+B-\delta, \, D(-A))$ is given by
		\begin{equation*}
			R_\delta(t) = e^{-\delta t} R(t)
		\end{equation*}
		for all $t\geq0$.
	\end{lemma}
	\begin{proof}
		Let $(R_\delta(t))_{t\geq0}$ be the semigroup generated by
		$(-A+B-\delta, \, D(-A)),$ which exists due to
		\Cref{theorem_semigroup_generated_by_-A+B}, and let $(\widetilde{R}(t))_{t\geq0} \coloneqq (e^{-\delta t} R(t))_{t\geq0}$. It is clear that $(\widetilde{R}(t))_{t\geq0}$ defines a strongly continuous semigroup. In addition, it holds
		\begin{equation*}
			\lim_{t \searrow 0} \frac{\widetilde{R}(t)u - u}{t} = e^{-\delta \cdot 0} (-A+B)u
			-\delta e^{-\delta \cdot 0} u = (-A+B-\delta)u
		\end{equation*}
		for all $u\in D(-A)$ by the product rule, i.e. the generator of
		$(\widetilde{R}(t))_{t\geq0}$ extends $(-A+B-\delta, \, D(-A)).$ However, $(R_\delta(t))_{t\geq0}$ is the only $C_0$-semigroup with a generator that extends $(-A+B-\delta, \, D(-A))$ by \Cref{corollary_Ccinfty_core_A+B}, so we obtain $(R_\delta(t))_{t\geq0} = (\widetilde{R}(t))_{t\geq0}$.
	\end{proof}
	
	\begin{theorem}\label{theorem_killing_spectral_prop_growth_bound}
		Let $1\leq p < \infty$ and let $\delta>0.$ Then
		\begin{align}
			\omega_0(-A+B-\delta) &= \omega_0(-A+B) - \delta,
			\label{killing_growth_bound}\\
			\sigma(-A+B-\delta) &= \sigma(-A+B) -\delta, \label{killing_spec_gen}\\
			\sigma(R_\delta(t)) &= e^{-\delta t} \sigma(R(t)) \label{killing_spec_sg},
		\end{align}
		where $\omega_0(\cdot)$ is the growth bound of the respective semigroup, see
		\Cref{long_time_behavior_definition_exponential_growth_bound_spectral_bound}. In
		particular, an exposure to $\delta$-killing turns the model exponentially stable
		if and only if $\delta>\omega_0(-A+B)$. Secondly, the weak spectral mapping
		theorem for the model exposed to $\delta$-killing holds if and only if it holds for the original model.
	\end{theorem}
	\begin{remark}
		The existence of a weak spectral mapping theorem implies $\omega_0(-A+B) =
		s(-A+B)$, see \Cref{wsmthm_implies_growth_bd_eq_spectral_bound}.
	\end{remark}
	\begin{proof}
		Let $\omega \in \R.$ There exists $M_\omega \geq 1$ with
		\begin{equation*}
			\|R_\delta(t)\|_{\mathcal{L}(\LpCn)} \leq M_\omega e^{\omega t}
		\end{equation*}
		for all $t\geq0$ if and only if there exists $M_\omega\geq1$ such that
		\begin{equation*}
			\|R(t)\|_{\mathcal{L}(\LpCn)} \leq M_\omega e^{(\omega+\delta)t}
		\end{equation*}
		holds for all $t\geq0$ by \Cref{lemma_killing_sg}. Statement \eqref{killing_growth_bound} can now be shown with the definition of the growth bound, see
		\Cref{long_time_behavior_definition_exponential_growth_bound_spectral_bound},
		and the respective addendum on exponential stability results from
		\Cref{long_time_behavior_exponential_stability_equivalences} $(iv).$ \newline
		Equations \eqref{killing_spec_gen} and \eqref{killing_spec_sg} are immediate
		consequences of the definition of the spectrum and \Cref{lemma_killing_sg}. Both
		equations and a routine computation for sets yield that the weak
		spectral mapping theorem
		\begin{equation*}
			\sigma(R_\delta(t)) = \overline{e^{t\sigma(-A+B-\delta)}} \quad \text{ for all } t\geq 0
		\end{equation*}
		holds iff we have
		\begin{equation*}
			e^{-\delta t}  \sigma(R(t)) = e^{-\delta t} \overline{e^{t(\sigma(-A+B))}} \quad \text{ for all } t\geq 0.
		\end{equation*}
		The latter is true if and only if the weak spectral mapping theorem holds for
		the original model.
	\end{proof}
	
	\subsection{Positive and Mass Conserving Models with Killing}
	Assume that the transport-reaction model is positive and mass conserving in the
	sense of \Cref{definition_positive} and
	\Cref{definition_conservation}. Both properties have been characterized in
	\Cref{theorem_semigroup_positive} and \Cref{theorem_conservation} respectively.
	Together, given $B=(b_{\indexfour \indexone})_{\indexfour,\indexone=1,\cdots,N},$ they are equivalent to
	\begin{equation}\label{condition_positive}
		b_{\indexfour \indexone} \geq 0 \qquad \text{ for } \indexfour, \indexone = 1,
		\cdots, N \quad  \text{ with } \quad \indexfour \neq \indexone
	\end{equation}
	and
	\begin{equation}\label{condition_mass}
		b_{\indexone\indexone} = - \sum^N_{\substack{\indexfour=1 \\ \indexfour \neq \indexone}}
		b_{\indexfour\indexone}  \qquad \text{ for } \indexone=1,\cdots,N.
	\end{equation}
	Positivity of the semigroup and mass conservation allow us to bound the
	$L^1(\T^d, \C^N)$ norm of the solution with the $L^1(\T^d, \, \C^N)$ norm of the
	initial function. The next lemma is typical for positive semigroups and is
	similar to  a computation from \cite[p. 106]{schnaubelt2011lecture}.
	
	\begin{lemma}\label{positive_sg_estimate}
		Let $1\leq p < \infty$ and assume that $B\in \R^{N\times N}$ fulfills \eqref{condition_positive}. Let $(R(t))_{t\geq0}$ be the positive transport-reaction semigroup on $\LpCn$ generated by $(-A+B, \, D(-A))$. Then,
		\begin{equation*}
			|R(t)u| \leq R(t)|u|
		\end{equation*}
		holds for all $u\in\LpCn$ and all $t\geq0$, i.e.
		\begin{equation*}
			|(R(t)u)_\indexone(\mathbf{x})| \leq (R(t)|u|)_\indexone(\mathbf{x})
		\end{equation*}
		holds for all $\indexone=1,\cdots,N$ and almost all $\mathbf{x}\in\T^d.$
		Here, $|u| \in \LpCn$ is the componentwise absolute value of $u\in\LpCn$.
	\end{lemma}
	\begin{proof}
		Let $u\in C^\infty(\T^d, \, \C^N)$ and $\mathbf{x}\in\T^d.$ For $\indexone=1,\cdots,N$, choose
		$\alpha_\indexone = \alpha_\indexone(\mathbf{x}) \in \C$ with
		$|\alpha_\indexone|=1$ and
		\begin{equation*}
			\R\ni \alpha_j \big((R(t)u)_\indexone(\mathbf{x})\big) =
			(R(t)(\alpha_\indexone u))_\indexone(\mathbf{x}) = (R(t)\real(\alpha_\indexone
			u))_\indexone (\mathbf{x}).
		\end{equation*}
		The second equality follows by decomposing $\alpha_\indexone u$ into its real
		and imaginary part and using that $R(t)$ maps real valued functions to real
		valued ones. Positivity implies
		\begin{align*}
			\pm R(t)\real(\alpha_\indexone u) &= \pm R(t) \real(\alpha_\indexone u )_+ \mp
			R(t)\real(\alpha_\indexone u )_- \leq R(t) \real(\alpha_\indexone u )_+ +
			R(t)\real(\alpha_\indexone u )_-\\
			&= R(t) |\real(\alpha_\indexone u)| \leq R(t) |\alpha_\indexone u| =  R(t) |u| 
		\end{align*}
		and therefore $|R(t)\real(\alpha_\indexone u)| \leq R(t) |u|,$
		i.e. it holds
		\begin{equation}\label{auxiliary_eq39}
			|(R(t)\real(\alpha_\indexone u))_\indexfour(\mathbf{y})| \leq
			(R(t)|u|)_\indexfour(\mathbf{y})
		\end{equation}
		for all $\indexfour=1,\cdots,N$ and all $\mathbf{y}\in\T^d$. Using this and $|\alpha_\indexone|=1$, we obtain
		\begin{align*}
			|(R(t)u)_\indexone(\mathbf{x})| = |\alpha_j \big((R(t)u)_\indexone(\mathbf{x})\big)| = |(R(t)\real(\alpha_\indexone u))_\indexone (\mathbf{x})| \leq (R(t)|u|)_\indexone(\mathbf{x})
		\end{align*}
		for all $\indexone=1,\cdots,N$. Notice that this estimate holds for all $\mathbf{x}\in\T^d.$ Given an arbitrary function $u\in\LpCn$, the claim follows by taking an approximating sequence $(u^{(n)})_{n\in\N} \subset C^\infty(\T^d, \, \C^N)$ of $u$.
	\end{proof}
	
	\begin{corollary}\label{pos_mass_con_L1_bound}
		Let $p=1$ and assume that $B\in \R^{N\times N}$ fulfills \eqref{condition_positive} and \eqref{condition_mass}. Let $(R_1(t))_{t\geq0}$ be the positive and mass conserving transport-reaction semigroup on $L^1(\T^d, \, \C^N)$ generated by $(-A_1+B_1, \, D(-A_1))$. Then,
		\begin{equation*}
			\| R_1(t)u\|_{L^1(\T^d, \, \C^N)} \leq \sqrt{N} \|u\|_{L^1(\T^d, \, \C^N)}
		\end{equation*}
		holds for all $u\in L^1(\T^d, \, \C^N)$ and all $t\geq0$. Furthermore, we have $\omega_0(-A_1+B_1) = 0.$
	\end{corollary}
	\begin{proof}
		Let $u\in L^1(\T^d, \, \C^N).$ Applying \Cref{positive_sg_estimate}, positivity and mass conservation yields
		\begin{align*}
			\|R(t)u\|_{L^1(\T^d, \, \C^N)} &= \|\int_{\T^d} |R(t)u|(\mathbf{x}) \,
			d\mathbf{x}\|_2 \leq \|\int_{\T^d} (R(t)|u|)(\mathbf{x}) \, d\mathbf{x}\|_2 \\
			&\leq (1,\cdots,1)^T \cdot \int_{T^d} (R(t)|u|)(\mathbf{x}) \, d\mathbf{x} =
			(1,\cdots,1)^T \cdot \int_{T^d} |u|(\mathbf{x}) \, d\mathbf{x} \\
			&\leq \sqrt{N} \, \|\int_{\T^d} |u|(\mathbf{x}) \, d\mathbf{x} \|_2 = \sqrt{N}
			\|u\|_{L^1(\T^d, \, \C^N)}.
		\end{align*}
		This estimate particularly shows $\omega_0(A_1+B_1) \leq 0.$ Equality follows
		from \Cref{long_time_behavior_exponential_stability_equivalences} $(iv):$ if we
		had $\omega_0(-A_1+B_1) < 0,$ the semigroup $(R_1(t))_{t\geq0}$ would be
		exponentially stable and we would obtain
		\begin{equation*}
			\|R(t)u\|_{L^1(\T^d, \, \C^N)} \longrightarrow 0
		\end{equation*}
		as $t\to\infty$ for all $u\in L^1(\T^d, \, \C^N).$ This would yield a contradiction to mass
		conservation for all $u> 0$ due to
		\begin{equation*}
			0 < (1,\cdots,1)^T \cdot \int_{\T^d} u(\mathbf{x}) \, d\mathbf{x} =
			(1,\cdots,1)^T \cdot \int_{\T^d} (R(t)u)(\mathbf{x}) \, d\mathbf{x} \leq \sqrt{N}
			\|R(t)u\|_{L^1(\T^d, \, \C^N)} \to 0.
		\end{equation*}
	\end{proof}
	\begin{remark}
		Neither positivity nor mass conservation alone are sufficient to get a bound as
		in \Cref{pos_mass_con_L1_bound}. This is easy to see if we solely assume
		positivity. For a diagonal matrix $B=D$ with positive diagonal entries, the transport-reaction semigroup is positive and every component grows exponentially. For mass conservation, consider the following example.
	\end{remark}
	\begin{example}
		Consider the mass conserving equation
		\begin{equation*}
			\partial_t u(t,\mathbf{x}) = \begin{pmatrix}
			1 & -1 \\
			-1 & 1
			\end{pmatrix} u(t,\mathbf{x}) = Bu(t,\mathbf{x})
		\end{equation*}
		on $\T^d$ with some initial function $u \in L^1(\T^d, \, \C^2)$ with $u_1 \neq
		u_2.$ The computation of the matrix exponential in step (ii) of
		\Cref{theorem_goldstein_kac_conv} shows that the solution is given by
		\begin{equation*}
			(R(t)u)(\mathbf{x}) = \half \colvec{u_1(\mathbf{x}) + u_2(\mathbf{x}) \\
			u_1(\mathbf{x}) + u_2(\mathbf{x})} + \half e^{2t} \colvec{u_1(\mathbf{x}) -
			u_2(\mathbf{x}) \\ u_2(\mathbf{x}) - u_1(\mathbf{x})}
		\end{equation*}
		and consequently
		\begin{equation*}
			\|R(t)u\|_{L^1(\T^d, \, \C^2)} \geq \frac{1}{\sqrt{2}} \bigg( e^{2t} \intT
			|u_1(\mathbf{x}) - u_2(\mathbf{x}) | \, d\mathbf{x} - \intT |u_1(\mathbf{x}) +
			u_2(\mathbf{x}) | \, d\mathbf{x}  \bigg) \longrightarrow \infty
		\end{equation*}
		as $t\to\infty.$ A similar blow-up behavior with motion of the particles
		included happens in the mass conserving model \eqref{model_Goldstein_kac} with
		$\Lambda < 0.$
	\end{example}
	
	The explicit computation $\omega_0(-A_1+B_1) = 0$ in the case of positive and
	mass conserving model done in \Cref{pos_mass_con_L1_bound} and
	\Cref{theorem_killing_spectral_prop_growth_bound} imply that $\delta$-killing
	for any $\delta>0$ results in exponential stability with respect to the $L^1(\T^d, \,
	\C^N)$ norm. More precisely, $\omega_0(-A_1+B_1-\delta) = -\delta$ and
	\Cref{lemma_killing_sg} and \Cref{pos_mass_con_L1_bound} show the intuitive
	estimate
	\begin{equation}\label{convergence_killing_mass_cons_and_positive}
		\|R_{1,\delta}(t)u\|_{L^1(\T^d, \, C^N)} \leq \sqrt{N} e^{-\delta t} \|u\|_{L^1(\T^d,
		\, \C^N)}.
	\end{equation}
	for all $u\in L^1(\T^d, \, \C^N).$ The result is sharp in the sense that the
	convergence rate is optimal. \newline
	
	So far, every result holds true for arbitrary dimension $d$ but in the
	following, we will assume $d=1$. The weak spectral mapping \Cref{theorem_WSMTHM_p_eq_2} and the independence of the spectrum of the generator $(-A_p+B_p, \, D(-A_p))$ of $1\leq p< \infty$ shown in \Cref{theorem_spectrum_generator} allow us to expand the convergence result
	\eqref{convergence_killing_mass_cons_and_positive} to the $L^2(\T, \, \C^N)$
	norm. In the process, we ``lose" an arbitrary small fraction of the convergence
	rate.
	
	\begin{theorem}
		Let $d=1$ and assume that $B\in \R^{N\times N}$ fulfills \eqref{condition_positive} and \eqref{condition_mass}. Moreover, assume that the transport directions $v_1,\cdots,v_N$ are non-vanishing, i.e. $v_1,\cdots,v_N\neq0.$ Let $(R_2(t))_{t\geq0}$ be the positive and
		mass conversing semigroup on $L^2(\T, \, \C^N)$ generated by $(-A_2+B_2, \,
		D(-A_2)).$ For every $\varepsilon>0,$ there exists a constant $C_\varepsilon>0$
		such that
		\begin{equation}\label{auxeq21}
			\| R_{2,\delta}(t)u\|_{L^2(\T, \, \C^N)} \leq C_\varepsilon
			e^{(-\delta+\varepsilon)t} \|u\|_{L^2(\T, \, \C^N)}
		\end{equation}
		holds for all $u\in L^2(\T, \, \C^N).$ The estimate \eqref{auxeq21} does not
		hold for $-\delta+\varepsilon$ replaced by any constant $\widetilde{\delta} <
		-\delta.$
	\end{theorem}
	\begin{proof}
		We have already seen
		\begin{equation*}
			\omega_0(-A_1+B_1-\delta) = -\delta
		\end{equation*}
		and it holds
		\begin{equation*}
			s(-A_1+B_1-\delta) \leq \omega_0(-A_1+B_1-\delta) = -\delta
		\end{equation*}
		because the spectral bound of a semigroup is always dominated by the growth
		bound, see \Cref{exp_stab_proposition}. In fact,
		we have equality: \Cref{theorem_spectrum_generator} characterizes the spectrum
		of the generator, which is independent of $1\leq p<\infty$, and implies
		\begin{equation*}
			\sigma(-A+B-\delta) \supseteq \sigma(M(0)) = \sigma(B-\delta).
		\end{equation*}
		Furthermore, the model was assumed to be mass conserving, so we obtain
		$0\in\sigma(B^T) = \sigma(B)$ by \Cref{theorem_conservation}. Hence, there
		exists $z\in\C^N\backslash \{0\}$ with $Bz = 0$ and $(B-\delta)z = -\delta z$,
		i.e. $-\delta \in \sigma(-A+B-\delta)$ and
		\begin{equation*}
			s(-A+B-\delta) = -\delta.
		\end{equation*}
		Now, the weak spectral mapping \Cref{theorem_WSMTHM_p_eq_2} for $p=2$ and
		\Cref{wsmthm_implies_growth_bd_eq_spectral_bound} result in
		\begin{equation*}
			\omega_0(-A_2+B_2-\delta) = s(-A+B-\delta) = -\delta
		\end{equation*}
		and the theorem follows from the definition of the growth bound.
	\end{proof}
	\begin{remark}
		If we had a weak spectral mapping theorem for $1\leq p < \infty$, the same
		result would be true on $L^p(\T, \, \C^N)$.
	\end{remark}

	\section{A Reaction Random Walk System}\label{reaction_random_walk}
	
	In this section, we always assume $d=1$ and consider the domain $\T$. The goal
	is to apply our theory from \Cref{chapter_spectral_analysis} to the main model
	from \cite{hillen1996turing} with periodic boundary conditions. This is
	different to the original paper, where Neumann boundary conditions in the sense
	of \Cref{section_neumann_boundary} were studied. Essentially, we give an easier and shorter proof of \cite[Lemma 4.1, p. 61 and Theorem 4.3, p. 63]{hillen1996turing}. \newline
	
	The author of \cite{hillen1996turing} assumed that the particle densities $u_1$
	and $u_2$ of two different species can be split into particle densities of right
	and left moving particles, i.e. $u_1=\alpha_1 + \beta_1$ and $u_2 = \alpha_2 +
	\beta_2.$ Each species $\indexone=1,2$ has a speed $v_\indexone > 0$ and a
	turning rate $\mu_\indexone > 0$. Moreover, the reactions of the particles are
	assumed to be described by a continuously differentiable function $F\colon \R^2
	\to \R^2,$ which depends on the concentrations of the \textit{species}.
	Introducing the variables
	\begin{align*}
		\Gamma \coloneqq \begin{pmatrix}
		v_1 & 0 \\
		0 & v_2
		\end{pmatrix},
		\qquad M \coloneqq \begin{pmatrix}
		\mu_1 & 0 \\
		0 & \mu_2
		\end{pmatrix}
	\end{align*}
	and
	\begin{equation*}
		\alpha \coloneqq (\alpha_1, \, \alpha_2)^T, \qquad \beta \coloneqq (\beta_1,
		\,\beta_2)^T, \qquad u \coloneqq \alpha + \beta,
	\end{equation*}
	the model, cf. \cite[(15), p. 54 and (38), p. 58]{hillen1996turing}, reads
	\begin{equation}\label{model_hillen}
		\begin{aligned}
		\partial_t \alpha + \Gamma \partial_x\alpha &= M(\beta-\alpha) + \tfrac{1}{2}
		F(\alpha+\beta), \\
		\partial_t\beta - \Gamma \partial_x\beta &= M(\alpha-\beta) + \tfrac{1}{2}
		F(\alpha+\beta).
		\end{aligned}
	\end{equation}
	Assuming that the reactions $F$ have a stationary state $u^* \in \R^2,$ i.e.
	$F(u^*) = 0,$ the equilibrium of \eqref{model_hillen} is given by
	\begin{equation*}
		\alpha^* = \beta^* \equiv \frac{u^*}{2}.
	\end{equation*}
	Let $N\coloneqq DF(u^*) \in \R^{2\times 2}$ be the Jacobian of $F$ at the
	stationary state. A simple computation shows that the linearization
	around the equilibrium of \eqref{model_hillen} reads
	\begin{equation}\label{model_hillen_linear}
		\partial_t \colvec{\alpha \\ \beta} + \begin{pmatrix}
		\Gamma & 0 \\
		0 & -\Gamma
		\end{pmatrix} \partial_x \colvec{\alpha \\ \beta} = \begin{pmatrix}
		-M & M \\
		M & -M
		\end{pmatrix} \colvec{\alpha \\ \beta} + \half \begin{pmatrix}
		N & N \\
		N & N
		\end{pmatrix} \colvec{\alpha \\ \beta}.
	\end{equation}
	It should be noted that \eqref{model_hillen_linear} describes the time evolution
	of the functions $\alpha-u^*/2$ and $\beta-u^*/2$ but we relabeled them back to
	$\alpha$ and $\beta$ respectively. System \eqref{model_hillen_linear} is
	equivalent to \cite[(20), p. 55]{hillen1996turing} with the introduction of
	the variables $\alpha+\beta$ and $\alpha-\beta.$ In \cite{hillen1996turing},
	these new variables were used to reduce \eqref{model_hillen} to a reaction
	telegraph system, see \cite[p. 54]{hillen1996turing} for details on this
	connection. \newline
	The linearized model \eqref{model_hillen_linear} is a transport-reaction model
	studied in \Cref{chapter_transport_reaction_SG}. In order to be consistent with
	our previous notation, we introduce the block matrices
	\begin{equation*}
		V \coloneqq \begin{pmatrix}
		\Gamma & 0 \\
		0 & -\Gamma
		\end{pmatrix}, \qquad B \coloneqq \begin{pmatrix}
		-M & M \\
		M & -M
		\end{pmatrix} + \half \begin{pmatrix}
		N & N \\
		N & N
		\end{pmatrix}
	\end{equation*}
	and the operator
	\begin{equation*}
		-A+B = -V \partial_x + B
	\end{equation*}
	on $L^2(\T, \, \C^4)$ with domain $D(-A+B)=D(-A)=H^1(\T, \, \C^4).$ Then, $-A+B$
	generates the strongly continuous transport-reaction semigroup $(R(t))_{t\geq0}$ on $L^2(\T, \, \C^4)$ by \Cref{theorem_semigroup_generated_by_-A+B} and \Cref{corollary_domain_generator_one_dimensional}. The solution of
	\eqref{model_hillen_linear} with initial function $(\alpha,\beta) \in L^2(\T, \,
	C^4)$ is given by $t\mapsto R(t)(\alpha,\beta)$, see Appendix \ref{appendix_semigroup_theory} for details on the solution concepts used in a semigroup context. \newline
	The stability analysis can be performed with \Cref{theorem_spectrum_generator}
	and \Cref{theorem_WSMTHM_p_eq_2}. If all eigenvalues of the generator $(-A+B, \,
	H^1(\T, \, \C^4))$ have a negative real part, the zero solution of
	\eqref{model_hillen_linear} is exponentially stable because of \Cref{wsmthm_implies_growth_bd_eq_spectral_bound}. Conversely, existence of an	eigenvalue $\lambda \in \sigma(-A+B)$ with $\real \lambda>0$ means that the zero solution is unstable. Therefore, by \Cref{theorem_spectrum_generator}, the question of stability comes down to the computation of the eigenvalues of
	\begin{equation*}
		M(\indextwo) = -2\pi i \indextwo V + B
	\end{equation*}
	for $k\in\Z.$ It is left to determine the roots of
	\begin{equation*}
		\lambda \mapsto P(\lambda) \coloneqq \det(-2\pi i \indextwo V + B -\lambda
		I_{\C^{4\times 4}}) = \det\begin{pmatrix}
		-2\pi i \indextwo \Gamma - M + \tfrac{1}{2}N - \lambda & M + \tfrac{1}{2}N \\
		M + \tfrac{1}{2}N & 2\pi i\indextwo \Gamma - M + \tfrac{1}{2} N - \lambda
		\end{pmatrix}
	\end{equation*}
	with the abbreviation $\lambda$ for $\lambda I_{\C^{2\times 2}}.$ At first glance,
	this looks quite laborious but some clever manipulations reduce the length of
	the computation tremendously. Notice that we have
	\begin{align*}
		\begin{pmatrix}
		I & -I \\
		0 & I
		\end{pmatrix} &\begin{pmatrix}
		-2\pi i \indextwo \Gamma - M + \tfrac{1}{2}N - \lambda & M + \tfrac{1}{2}N \\
		M + \tfrac{1}{2}N & 2\pi i\indextwo \Gamma - M + \tfrac{1}{2} N - \lambda
		\end{pmatrix} \begin{pmatrix}
		I & I \\
		0 & I
		\end{pmatrix} \\
		&= \begin{pmatrix}
		I & -I \\
		0 & I
		\end{pmatrix} \begin{pmatrix}
		-2\pi i \indextwo\Gamma - M + \tfrac{1}{2}N - \lambda & -2\pi i \indextwo
		\Gamma + N - \lambda \\
		M + \tfrac{1}{2}N & 2\pi i\indextwo\Gamma + N - \lambda
		\end{pmatrix} \\
		&= \begin{pmatrix}
		-2\pi i \indextwo \Gamma - 2M - \lambda & -4\pi i \indextwo \Gamma \\
		M + \tfrac{1}{2} N & 2 \pi i \indextwo \Gamma + N - \lambda
		\end{pmatrix}.
	\end{align*}
	This transformation does not change the determinant and now, the upper left and
	upper right matrices are diagonal matrices because $\Gamma$ and $M$ are diagonal
	by definition. They particularly commute and we can use \cite[Theorem 3, p. 
	4]{silvester2000determinants} to deduce
	\begin{align*}
		P(\lambda) &= \det\begin{pmatrix}
		-2\pi i \indextwo \Gamma - 2M - \lambda & -4\pi i \indextwo \Gamma \\
		M + \tfrac{1}{2} N & 2 \pi i \indextwo \Gamma + N - \lambda
		\end{pmatrix} \\
		&= \det \begin{pmatrix}
		(2\pi i \indextwo \Gamma + N - \lambda)(-2\pi i \indextwo \Gamma - 2M -
		\lambda) - (M+\tfrac{1}{2}N)(-4\pi i \indextwo \Gamma)
		\end{pmatrix} \\
		&= \det(\lambda^2 + \lambda(2M-N) + (4\pi^2 \indextwo^2 \Gamma^2 - 2 NM)).
	\end{align*}
	With the notation
	\begin{equation*}
		N = \begin{pmatrix}
		\nu_1 & \nu_2 \\
		\nu_3 & \nu_4
		\end{pmatrix},
	\end{equation*}
	the obtained $2\times 2$-matrix reads
	\begin{equation*}
		\begin{pmatrix}
		\lambda^2 + (2\mu_1-\nu_1)\lambda + (4\pi^2\indextwo^2v_1^2-2\nu_1\mu_1) &
		-\nu_2 \lambda - 2\nu_2\mu_2 \\
		-\nu_3\lambda - 2\nu_3\mu_1 & \lambda^2 + (2\mu_2-\nu_4)\lambda +
		(4\pi^2\indextwo^2v_2^2-2\nu_4\mu_2)
		\end{pmatrix}
	\end{equation*}
	and its determinant is given by
	\begin{alignat}{2}
		P(\lambda) &= \, && \lambda^4 + \big[(2\mu_2-\nu_4) +
		(2\mu_1-\nu_1)\big]\lambda^3 \nonumber \\
		& &&+  \big[ (4\pi^2\indextwo^2v_2^2-2\nu_4\mu_2) +
		(2\mu_1-\nu_1)(2\mu_2-\nu_4) +(4\pi^2\indextwo^2v_1^2 - 2\nu_1\mu_1) -
		\nu_2\nu_3\big] \lambda^2\nonumber \\
		& &&+  \big[(2\mu_1 - \nu_1)(4\pi^2\indextwo^2v_2^2-2\nu_4\mu_2) +
		(4\pi^2\indextwo^2v_1^2-2\nu_1\mu_1)(2\mu_2-\nu_4) - 2\nu_2\nu_3\mu_1 -
		2\nu_2\nu_3\mu_2\big]\lambda \nonumber \\
		& &&+ \big[(4\pi^2\indextwo^2v_1^2 - 2
		\nu_1\mu_1)(4\pi^2\indextwo^2v_2^2-2\nu_4\mu_2) - 4\nu_2\nu_3\mu_1\mu_2\big]
		\nonumber \\
		&= \, && \lambda^4 +  \big[2(\mu_1+\mu_2) - (\nu_1+\nu_4)\big]\lambda^3
		\nonumber \\
		& &&+ \big[4\pi^2\indextwo^2(v_1^2+v_2^2) - 2(\mu_1+\mu_2)(\nu_1+\nu_4) +
		4\mu_1\mu_2 + \det N \big]\lambda^2 \nonumber \\
		& &&+  \big[4\pi^2\indextwo^2\big(v_1^2(2\mu_2-\nu_4) + v_2^2(2\mu_1 -
		\nu_1)\big) - 4(\nu_1 + \nu_4)\mu_1 \mu_2 + 2 (\mu_1+\mu_2)\det N\big]\lambda
		\nonumber \\
		& &&+ \big[16\pi^4\indextwo^4 v_1^2 v_2^2 -
		8\pi^2\indextwo^2\big(v_1^2\nu_4\mu_2 + v_2^2\nu_1\mu_1\big) + 4 \mu_1\mu_2 \det
		N\big].
	\end{alignat}
	
	Including the side length $L>0$ to the model by applying \Cref{theorem_side_length} and using the notations $(R^{L,v}(t))_{t\geq0}$ and $(-A+B)^{L,v}$ for the semigroup and generator on $L^2(\T_L, \, \C^4)$ yields the following theorem.
	
	\begin{theorem}\label{theorem_random_walk}
		A complex number $\lambda\in\C$ is an eigenvalue of the operator
		$(-A+B)^{L,v}$ if and only if there is a mode $k\in\Z$ such that $P_L(\lambda)=0,$ where $P_L$ is the
		polynomial
		\begin{equation*}
			P_L(\lambda) = \lambda^4 + a_3 \lambda^3 + a_2 \lambda^2 + a_1 \lambda + a_0
		\end{equation*}
		with coefficients
		\begin{align*}
			a_3 &= 2(\mu_1+\mu_2) - (\nu_1+\nu_4), \\
			a_2 &= \tfrac{4\pi^2\indextwo^2}{L^2}(v_1^2+v_2^2) -
			2(\mu_1+\mu_2)(\nu_1+\nu_4) + 4\mu_1\mu_2 + \det N,\\
			a_1 &= \tfrac{4\pi^2\indextwo^2}{L^2}\big(v_1^2(2\mu_2-\nu_4) + v_2^2(2\mu_1 -
			\nu_1)\big) - 4(\nu_1 + \nu_4)\mu_1 \mu_2 + 2 (\mu_1+\mu_2)\det N,\\
			a_0 &= \tfrac{16\pi^4\indextwo^4}{L^4} v_1^2 v_2^2 - 2
			\tfrac{4\pi^2\indextwo^2}{L^2}\big(v_1^2\nu_4\mu_2 + v_2^2\nu_1\mu_1\big) + 4
			\mu_1\mu_2 \det N.
		\end{align*}
	\end{theorem}
	In \cite{hillen1996turing}, the author computed the eigenvalues of the generator
	$(-A+B)^{\text{Neu},L}$ on $(0,L)$ with the Neumann boundary conditions
	introduced in \Cref{section_neumann_boundary}. The astonishing thing is the
	correspondence between both results.
	
	\begin{theorem}[{\cite[Lemma 4.1, p. 61; Theorem 4.3, p. 63]{hillen1996turing}}]
		A complex number $\lambda\in\C$ with $\real\lambda>\max\{-2\mu_1,-2\mu_2\}$
		is an eigenvalue of the operator $(-A+B)^{\text{Neu},L}$ if and only if there
		is a mode $k\in\Z$ such that $P_{2L}(\lambda) = 0,$ where $P_{2L}$ is exactly
		the same polynomial as in \Cref{theorem_random_walk}.
	\end{theorem}
	\begin{remark}
		The restriction $\real\lambda>\max\{-2\mu_1,-2\mu_2\}$ is a technical
		assumption needed for the computations in \cite{hillen1996turing} and, due to
		$\mu_1,\mu_2>0,$ it is irrelevant if one is only interested in stability or
		blow-up results.
	\end{remark}
	The theorems strongly support the presumption that there is a deeper connection
	of the spectral properties between the generator of symmetric models on $(0,L)$ with Neumann boundary conditions to the same models on $(0,2L)$ with periodic boundary conditions than the one pointed out in \Cref{section_neumann_boundary}. We conjecture
	that all inclusions in \Cref{corollary_spectrum_neumann} are actually
	equalities. \newline
	
	The next logical step is to study the polynomial from \Cref{theorem_random_walk}
	and to characterize parameter constellations which turn the model
	\eqref{model_hillen_linear} stable or unstable. T. Hillen studied this question
	in the second half of his paper \cite{hillen1996turing} under the two
	assumptions that $u^* \in \R^2$ is a stable equilibrium of the ODE model
	$y'=F(y)$ and that $N=DF(u^*)$ corresponds to an activator-inhibitor system, see
	\cite[(H1) and (H2), p. 56]{hillen1996turing} for the exact parameter choices. We
	briefly elucidate his main idea, which we believe is currently the standard approach to
	deal with stability questions of sophisticated transport-reaction models. The
	author used, cf. \cite[p. 63]{hillen1996turing}, the Routh-Hurwitz criterion
	from \cite[p. 194ff.]{gantmakher2000theory}: \newline
	
	All roots of a polynomial
	\begin{equation*}
		P(\lambda) = \lambda^4 + a_3 \lambda^3 + a_2 \lambda^2 + a_1 \lambda + a_0
	\end{equation*}
	have negative real part if and only if
	\begin{align*}
		0 &< a_3, \\
		0 &< a_2a_3 - a_1, \\
		0 &< (a_2a_3-a_1)a_1-a_3^2a_0, \\
		0 &< a_0 \big((a_2a_3-a_1)a_1-a_3^2a_0\big).
	\end{align*}
	These conditions are equivalent to
	\begin{equation*}
		a_3, \, a_2, \, a_1, \, a_0 > 0 \quad \text{ and } \quad
		(a_2a_3-a_1)a_1-a_3^2a_0 > 0.
	\end{equation*}
	We refer to \cite{hillen1996turing} for the actual computations and findings for
	the reaction random walk system \eqref{model_hillen_linear}.
	
	\chapter{Pattern Formation}\label{chapter_pattern_formation}
	
	The goal of this chapter is to study transport-driven instabilities. This means that we are interested in the question whether an additional spacial variable and transport can cause stable reactions to become unstable. If so, what are the patterns or phenomena occurring and how does the solution of the transport-reaction equation \eqref{linear_transport_reaction_eq} behave qualitatively? \newline
	
	Originally, this questions was studied by Turing in the context of reaction-diffusion equations \cite{turing1990chemical}. Although chemical reactions and diffusion are both homogenizing and stabilizing mechanisms, together they can drive very regular instabilities, so called Turing patterns. His discovery led to a broad study of these patterns in reaction-diffusion equations, especially in the context of biology. \newline
	The first section gives an introduction to the mathematical modeling of chemical reactions and the concept of Turing patterns, where the second part of it follows the recent review \cite{woolley2017turing}. Beyond motivating reaction-diffusion equations and the concept ``Turing patterns", we also emphasize that the typical parameter choices causing these instabilities are of activator-inhibitor type. \newline
	
	For the transport-reaction system
	\begin{equation}\label{transport_reaction_system}
		\partial_t u + \colvec{\mathbf{v_1} \cdot \nabla u_1\\ \vdots \\ \mathbf{v_N}
		\cdot \nabla u_N} = Bu
	\end{equation}
	with $N\in\N$ components, transport directions $\mathbf{v_1},\cdots,\mathbf{v_N} \in \R^d$ and linear reactions described by a matrix $B\in \R^{N\times N},$ pattern formation mechanisms and instabilities caused by movement are less studied than for reaction-diffusion equations. Even though linear transport is not homogenizing, the detailed study of the simple Goldstein-Kac model in \Cref{section_goldstein_kac} and the \Cref{remark_Goldstein_kac} could suggest that it corresponds to a ``stirring motion", which has no influence on the stability of an equilibrium. On top of that, \Cref{time_development_averages} shows that the averages of the concentrations will converge to zero, if $B$ is stable, i.e. if $B$ has only eigenvalues with a negative real part. \newline
	However, it is known that \eqref{transport_reaction_system} can also generate Turing patterns. This has for example been shown in \cite{hillen1996turing} for a concrete system of $N=4$ equations. In \cite[Theorem 11, p. 20]{stevens2008partial}, specific conditions for a system with an arbitrary number $N$ of components were found, which lead to Turing patterns. These conditions from \cite{stevens2008partial} are unfortunately a little elusive, less concrete as for reaction-diffusion equations and require a priori assumptions on $B$. One reason for this is that Turing patterns in transport-reactions models require a minimal degree of complexity in the sense that the number of components $N$ needs to be greater than two.\newline
	To the best of our knowledge, there is currently no general analysis of transport-driven instabilities and we aim to close this gap.  Restricting ourselves to the case $d=1$, the transport reaction system \eqref{transport_reaction_system} reads
	\begin{equation}\label{transport_reaction_one_dimensional}
		\partial_t u + \colvec{v_1 u_1'\\ \vdots \\ v_N u_N'} = Bu
	\end{equation}
	with $N\in\N$ components, transport directions $v_1,\dots,v_N \in \R$ and linear reactions described by a matrix $B\in \R^{N\times N}$. We consider the equation on the interval $(0,1)$ with periodic boundary conditions, i.e. on the circle $\T.$ At this point, it should be mentioned that this restriction to the one-dimensional case is, in principle, unnecessary. \Cref{char_spectrum_generator}, \Cref{theorem_WSMTHM_p_eq_2} and the methods we present in this chapter also allow a rigorous pattern formation analysis on the $d$-dimensional torus $\T^d.$ Nevertheless, we do not address the higher dimensional case here and propose it as a topic for future research. \newline
	
	In contrast to reaction-diffusion systems, Turing patterns are not the only instabilities caused by the additional spacial variable. We introduce the notion of hyperbolic instabilities which allows us to categorize all possible instabilities for \eqref{transport_reaction_one_dimensional}: an instability is either a Turing pattern or a hyperbolic instability. Intuitively, hyperbolic instabilities correspond to an ever growing oscillation due to the dominance of increasingly high frequencies. As a side product of our analysis, we find a fairly general condition which ensures the existence of Turing patterns.
	
	\section{Reaction-Diffusion Equations and Turing Patterns}
	
	Before introducing reaction-diffusion equations and diffusion-driven instabilities, we give a brief introduction to the mathematical modeling of chemical reaction networks.
	
	\subsection{Reaction Networks}\label{section_reaction_networks}
	This subsection will only deal with	the dynamics in a single cell. As a mathematical consequence, there is no spacial
	variable and the reaction network can be fully described by a system of ODEs. The short
	overview given here is inspired by the lecture "Computational Systems Biology"
	held by Professor Jan Hasenauer in the winter term 2020 at the University of
	Bonn and \cite{kuttler2011reaction}.\newline
	To begin with, let us put chemical reactions in concrete and formal terms.
	Chemical species $Y_1, \cdots, Y_N $ are chemically identical molecular entities
	and their state $Y =(Y_1, \cdots, Y_N)$ is given by a $\N_0^N$ valued vector
	describing the current number of molecules. These chemical species interconvert
	by reacting. In general, a single chemical reaction
	\begin{align}
		\underbrace{\sum_{\indexone=1}^{N} s_{\indexone}^-
		Y_\indexone}_{\text{reactants}} \quad \xrightarrow{a(Y)} \quad
		\underbrace{\sum_{\indexone=1}^{N} s_{\indexone}^+
		Y_\indexone}_{\text{products}}
	\end{align}
	is characterized by the \textit{stoichiometric coefficients} $s_{\indexone}^-,
	\, s_{\indexone}^+ \in \N_0$ for $\indexone=1, \cdots N$ and by the
	state-dependent \textit{reaction propensity} $a(Y) \in \R_+$. The latter is
	defined such that $a(Y)h + o(h)$ gives the probability of the reaction taking
	place in $(t, \, t+h)$ for $h>0.$ In principle, this propensity can have the
	time $t$ as an additional argument but under the assumption of thermodynamic
	equilibrium it is (almost) proportional to the product of molecule counts
	$\prod_{\indexone=1}^{N} Y_\indexone^{s_\indexone^-}$
	\cite{gillespie1992rigorous}. Prototypes of these reactions are zeroth order
	reactions $\emptyset \xrightarrow{k} products$ with propensity $k$ for some
	constant $k \in \R_+$ or first order reactions $Y_\indexone \xrightarrow{k}
	products$ and second order reactions $Y_\indexfour + Y_\indexone
	\xrightarrow{k} products$ with propensities $kY_\indexone$ and $kY_\indexfour
	Y_\indexone$ respectively (in the case $\indexfour \neq \indexone$). \newline
	Instead of tracking the number of molecules, we are interested in the
	concentrations $y_\indexone = [Y_\indexone]$ of the chemical species and accordingly, the vector $y=(y_1,\cdots,y_N)$ of all concentrations is $\R_{\geq0}^N$ valued. This
	macroscopic perspective is mainly motivated by the fact that molecule numbers
	are often high and continuous models are much easier to study qualitatively and
	quantitatively with analytical tools. The propensities $a(y)$ in this
	macroscopic setting are proportional to $\prod_{\indexone=1}^{N}
	y_\indexone^{s_\indexone^-}$. In particular, the propensities of the zeroth,
	first and second order reactions are $k, \, ky_\indexone$ and $ky_\indexfour
	y_\indexone$ respectively. One can use the above mentioned considerations to
	obtain a system of ODEs given by
	\begin{align*}
		\frac{d}{dt} \colvec{y_1 \\ \vdots \\ y_N} = \colvec{(s_1^+ - s_1^-)a(y) \\
		\vdots \\ (s_N^+ - s_N^-)a(y)}.
	\end{align*}
	Let us now consider an arbitrary number of reactions $R_1, \cdots, R_M$ with
	reaction $m=1,\cdots,M$ being
	\begin{align*}
		\sum_{\indexone=1}^{N} s_{m\indexone}^- Y_\indexone \quad \xrightarrow{a_m(y)}
		\quad \sum_{\indexone=1}^{N} s_{m\indexone}^+ Y_\indexone
	\end{align*}
	and define the stoichiometric matrix $S \in \R^{N\times M}$ by $S_{\indexone m}
	\coloneqq s_{m \indexone}^+ - s_{m \indexone}^-.$ Summing up all reaction fluxes
	yields the ODE system
	\begin{equation}\label{reaction_rate_equation_ODE}
		\frac{d}{dt} y = S a(y) \eqqcolon F(y)
	\end{equation}
	for the concentrations of the chemical species and $a(y) \coloneqq (a_1(y),
	\cdots, a_M(y)).$ 
	
	\begin{example}[Reversible reaction]
		Consider three chemical species and the reversible reaction
		\begin{equation*}
			Y_1 + Y_2 \xrightleftharpoons[k_-]{k_+} Y_3,
		\end{equation*}
		which can be split up into the two reactions
		\begin{equation*}
			\begin{array}{lrcl}
			R_1\colon & Y_1 + Y_2 &\overset{k_+}{\longrightarrow}&  Y_3,\\
			R_2\colon & Y_3 &\overset{k_-}{\longrightarrow}&  Y_1 + Y_2. \\
			\end{array}
		\end{equation*}
		The corresponding system of ODEs reads
		\begin{align*}
			\frac{d}{dt} \colvec{y_1 \\ y_2 \\ y_3} = 
			\begin{pmatrix}
			-1 & 1 \\
			-1 & 1 \\
			1 & -1
			\end{pmatrix}
			\colvec{k_+ \, y_1 y_2 \\ k_- \, y_3} = \colvec{-k_+ \, y_1 y_2 + k_- \, y_3
				\\ -k_+ \, y_1 y_2 + k_- \, y_3 \\ k_+ \, y_1 y_2 - k_- \, y_3}.
		\end{align*}
	\end{example}
	
	\subsection{The Full Model and Diffusion-Driven Instabilities}
	
	This subsection follows \cite{woolley2017turing}. For reasons of consistency, we consider all equations on the one-dimensional domain $(0,L)$ for $L>0$ and with periodic boundary conditions, i.e. on the circle $\T_L.$ \newline
	
	Reaction-diffusion equations in the context of chemicals appear, if one adds a spacial variable to the models from \Cref{section_reaction_networks}. In contrast to bacteria and without any external force, the movement of chemicals is determined by Fick's Law of Diffusion \cite{fick1855v}. Therefore, the full model reads
	\begin{equation}\label{reaction_diff_eq}
		\partial_t u(t,x) - D \Delta u(t,x) = F(u(t,x))
	\end{equation}
	where $u=(u_1,\cdots,u_N)$ is the vector of all concentrations and $D$ is a diagonal matrix with entries $D_{\indexone}>0$ for $\indexone=1,\cdots,N.$ These diagonal entries are often called diffusivity constants and they control how quickly chemicals spread along the circle \cite[p. 221]{woolley2017turing}. The function $F\colon \R^N \to \R^N$ describes the chemical reactions and is typically given by \eqref{reaction_rate_equation_ODE}. \newline
	
	For clarity, we now consider the reaction-diffusion equation describing two concentrations $u_1$ and $u_2.$ In this simple case, \eqref{reaction_diff_eq} reads
	\begin{equation}\label{reaction_diffusion_2D}
		\partial_t \colvec{u_1 \\ u_2} - \begin{pmatrix}
		D_1 & 0 \\
		0 & D_2
		\end{pmatrix} \partial_x^2 \colvec{u_1 \\ u_2} = \colvec{F_1(u_1,u_2) \\ F_2(u_1,u_2)}.
	\end{equation}
	Talking about a diffusion-driven instability only makes sense if there exists a homogeneous steady state $c\in\C^2$ such that $F(c) = 0$ and $y\equiv c$ is a stable equilibrium of the ODE model $y'=F(y),$ i.e. all eigenvalues of the Jacobian of $F$ evaluated at the steady state
	\begin{equation*}
		B \coloneqq DF(c)
	\end{equation*}
	have a negative real part. Let us write
	\begin{equation*}
		B = \begin{pmatrix}
		a & b \\
		c & d
		\end{pmatrix}.
	\end{equation*}
	Then, both eigenvalues of $B$ having a negative real part is equivalent to
	\begin{align}
		\trace(B) &= a + d < 0, \label{neg_trace} \\
		\det(B) &= ad - bc > 0. \label{pos_det}
	\end{align}
	Notice that this works out so nicely because $B$ is only a $2\times 2$ matrix. Also, these conditions cover both the cases of complex conjugate eigenvalues and two real eigenvalues. \newline
	In order to characterize conditions which turn the stable equilibrium $y\equiv c$ to an unstable steady state $u\equiv c$ of \eqref{reaction_diffusion_2D}, we study the linearization
	\begin{equation}\label{reaction_diffusion_2D_linearized}
		\partial_t u - \begin{pmatrix}
		D_1 & 0 \\
		0 & D_2
		\end{pmatrix} \partial_x^2 u = Bu
	\end{equation}
	of \eqref{reaction_diffusion_2D} with the same Fourier approach as in \Cref{chapter_spectral_analysis}. From the perspective of semigroup theory, the operator
	\begin{equation}\label{generator_reaction_diff}
		D + B \coloneqq \begin{pmatrix}
		D_1 & 0 \\
		0 & D_2
		\end{pmatrix}\partial_x^2 + B
	\end{equation}
	with domain $H^2(\T_L, \, \C^2)$ generates an analytic semigroup on $L^2(\T_L, \, \C^2).$ This semigroup fulfills the spectral mapping theorem, see \Cref{smtm}, and the question of stability comes down to studying the spectrum of $D+B.$ If there exists $\lambda \in \sigma(D+B)$ with $\real\lambda>0,$ the equilibrium $y\equiv c$ of the ODE model will turn unstable in the reaction-diffusion model. Now, any function $u\in L^2(\T_L, \, \C^2)$ can be written as its unique Fourier series
	\begin{equation}\label{fourier_series_reaction_diff}
		u(x) = \sum_{\indextwo\in\Z} \hat{u}(\indextwo) e^{2\pi i \tfrac{k}{L} x}
	\end{equation}
	with Fourier coefficients 
	\begin{equation*}
		\hat{u}(\indextwo) = \int_{\T_L} u(x)e^{-2\pi i \tfrac{k}{L}x} \, dx \in \C^2.
	\end{equation*}
	An informal application of the generator \eqref{generator_reaction_diff} to the Fourier series \eqref{fourier_series_reaction_diff} of $u$ yields
	\begin{equation*}
		\big((D+B)u\big)(x) = \sum_{\indextwo\in\Z} M(\indextwo) \hat{u}(\indextwo)  e^{2\pi i \tfrac{k}{L} x},
	\end{equation*}
	where $M(\indextwo) \in \R^{2\times 2}$ are the matrices
	\begin{equation}\label{matrices_Mk_react_diff}
		M(\indextwo) = \begin{pmatrix}
		-\frac{4\pi^2 \indextwo^2}{L^2} D_1 + a & b \\
		c & -\frac{4\pi^2 \indextwo^2}{L^2} D_2 + d
		\end{pmatrix} \eqqcolon \begin{pmatrix}
		-\indextwo_L^2 D_1 + a & b \\
		c & - \indextwo_L^2 D_2 + d
		\end{pmatrix}.
	\end{equation}
	These considerations can be made rigorous on the domain $H^2(\T_L, \, \C^2)$ of $D+B$ and 
	\begin{equation}\label{spec_reaction_diff}
		\sigma(D+B) = \bigcup_{\indextwo\in\Z} \sigma(M(k))
	\end{equation}
	follows similarly to \Cref{proposition_generator_application_fourier} and \Cref{char_spectrum_generator}. Because we assumed $D_{1,2}>0$, the eigenvalues of $M(\indextwo)$ never converge to a complex number as $|k|\to \infty$ and we do not need to take the closure on the right-hand side of \eqref{spec_reaction_diff}. It is left to characterize all parameter constellations which imply existence of an eigenvalue $\lambda$ of some $M(\indextwo)$ with a positive real part. \newline
	
	Hence, we need to compute the roots of
	\begin{align*}
		\det (( M(\indextwo) - \lambda I_{\C^{2\times 2}}) &= \det \begin{pmatrix}
		-\indextwo_L^2 D_1 + a - \lambda & b \\
		c & - \indextwo_L^2 D_2 + d - \lambda
		\end{pmatrix} = (\indextwo_L^2 D_1 + \lambda - a)(\indextwo_L^2 D_2 + \lambda - d) - bc \\
		&= \lambda^2 + \lambda \big( \indextwo_L^2(D_1+D_2) - \trace(B) \big) + \indextwo_L^4 D_1 D_2 - \indextwo_L^2(D_1 d + D_2 a) + \det(B) \\
		&\eqqcolon \lambda^2 + \lambda \big( \indextwo_L^2(D_1+D_2) - \trace(B) \big) + h(\indextwo_L^2).
	\end{align*}
	They are given by
	\begin{equation*}
		\lambda_{\pm}(\indextwo) = \frac{\trace(B) - \indextwo_L^2(D_1+D_2)}{2} \pm \sqrt{\left(\frac{\trace(B)-\indextwo_L^2(D_1+D_2)}{2}\right)^2 - h(\indextwo_L^2)}
	\end{equation*}
	and \eqref{neg_trace} implies $\real(\lambda_-(\indextwo))<0$ for all $\indextwo\in\Z.$ Also, $\real(\lambda_+(\indextwo))$ is positive if and only if $h(\indextwo_L^2)<0,$ which is equivalent to
	\begin{equation*}
		\indextwo_L^4 - \frac{D_1 d + D_2 a}{D_1 D_2} \indextwo_L^2 + \frac{\det(B)}{D_1 D_2} < 0
	\end{equation*}
	by the definition of $h(\indextwo_L^2)$. This implies the necessary condition
	\begin{equation}\label{condition_Turing_instab_hL}
		\indextwo_-^2 < \indextwo_L^2 < \indextwo_+^2,
	\end{equation}
	where
	\begin{equation}\label{def_k_pm_reaction_diffusion}
		\indextwo_\pm^2 = \frac{D_1 d + D_2 a}{2 D_1 D_2} \pm \sqrt{\left(\frac{D_1 d + D_2 a}{2D_1 D_2}\right)^2 - \frac{\det(B)}{D_1 D_2}}.
	\end{equation}
	Condition \eqref{condition_Turing_instab_hL} can be realized iff $\indextwo_+^2>0,$ i.e. if and only if
	\begin{equation}\label{auxiliary_eq40}
		\indextwo_+^2 = \frac{D_1 d + D_2 a}{2 D_1 D_2} + \sqrt{\left(\frac{D_1 d + D_2 a}{2D_1 D_2}\right)^2 - \frac{\det(B)}{D_1 D_2}} > 0.
	\end{equation}
	Given $\det(B)>0$, see \eqref{pos_det}, \eqref{auxiliary_eq40} is equivalent to the two conditions
	\begin{equation*}
		\frac{D_1 d + D_2 a}{2 D_1 D_2} > 0 \quad \text{ and } \quad \left(\frac{D_1 d + D_2 a}{2D_1 D_2}\right)^2 - \frac{\det(B)}{D_1 D_2} > 0.
	\end{equation*}
	These two inequalities yield the single condition
	\begin{equation}\label{auxiliary_eq41}
		D_1 d + D_2 a > 2 \sqrt{D_1 D_2} \sqrt{\det(B)} > 0.
	\end{equation}
	Finally, putting \eqref{neg_trace}, \eqref{pos_det}, \eqref{auxiliary_eq40} and \eqref{auxiliary_eq41} together yields a characterization of the existence of Turing patterns for the reaction-diffusion equation \eqref{reaction_diffusion_2D}.
	
	\begin{theorem}[{\cite[p. 225]{woolley2017turing}}]\label{turing_instab_reaction_diffusion_eq}
		Let $L>0$ be the length of the circle $\T_L$. Moreover, let $c\in \R^2$ with $F(c)=0$ and let $B=DF(c).$ The reaction-diffusion equation \eqref{reaction_diffusion_2D} generates Turing patterns if and only if the diffusion and reaction parameters fulfill
		\begin{align*}
			\trace(B) &= a + d < 0, \\
			\det(B) &= ad - bc > 0, \\
			D_1 d + D_2 a > 2 &\sqrt{D_1 D_2} \sqrt{\det(B)} > 0
		\end{align*}
		and there exists $k\in\Z$ with
		\begin{equation*}
			\indextwo_-^2 < \frac{4\pi^2 \indextwo^2}{L^2} < \indextwo_+^2,
		\end{equation*}
		where $\indextwo_\pm^2$ are defined in \eqref{def_k_pm_reaction_diffusion}.
	\end{theorem}
	\begin{remarks}
		\begin{enumerate}
			\item
			The reaction-diffusion equation \eqref{reaction_diffusion_2D} is symmetrical in $u_1$ and $u_2.$ Moreover, either $a$ or $d$ has to be negative for Turing patterns to exist. Let $d<0$ without loss of generality. Then, for Turing patterns to exist, $a$ has to be positive and $b$ and $c$ need to have different signs. \newline
			The situation in which $b, \, d<0 $ is often interpreted as an activator-inhibitor system. The activator $u_1$ increases by itself and its growth can be offset by the inhibitor $u_2,$ which increases only in the presence of the activator \cite[p. 175]{evans10}.
			\item
			The wave number $|\indextwo|$ with the largest possible real part for an eigenvalue of $M(\indextwo)$ is dominant for large times and indicates the number of peaks in the emerging pattern \cite[cf. p. 223]{woolley2017turing}.
			\item
			The last condition in the above theorem shows that a decreasing domain size shrinks the window of viable wave numbers $\indextwo$ \cite[p. 226]{woolley2017turing}.
		\end{enumerate}
	\end{remarks}

	We end this introduction to Turing pattern formation with the general definition of Turing patterns, which can be applied in a broader context, e.g. in the context of linear systems of $N\in\N$ equations. To this end, we introduce the notation
	\begin{equation*}
		\Sigma(\indextwo) = \max_{\lambda \in \sigma(M(\indextwo))} \real\lambda.
	\end{equation*}
	
	\begin{definition}[Turing pattern formation]
		Model \eqref{reaction_diffusion_2D_linearized} generates Turing patterns if $\sigma(B) \subseteq \C_- = \{ \lambda \in \C \, \colon \, \real\lambda < 0\}$ and if there exist finitely many $k_1, \cdots, k_n \in \Z$ such that the following holds:
		\begin{enumerate}
			\item
			\begin{equation*}
				\Sigma(\indextwo_\indexfour) = \Sigma(\indextwo_\indexone) > 0 \qquad \text{for all } \indexfour,\indexone=1,\cdots,n,
			\end{equation*}
			\item
			\begin{equation*}
				\sup_{\substack{\indextwo\in\Z \\ \indextwo\neq \indextwo_1,\cdots, \indextwo_n}} \Sigma(\indextwo) < \Sigma(\indextwo_\indexfour)\qquad \text{for all } \indexfour=1,\cdots,n.
			\end{equation*}
		\end{enumerate}
	\end{definition}
	
	\section{Transport-Driven Instabilities}\label{section_transport_driven_instab}
	
	As mentioned in the introduction of this chapter, the goal is now to study the phenomenon of instabilities driven by an additional spacial variable in the context of the transport-reaction equation
	\begin{equation}\label{transport_driven_instab_eq}
		\partial_t u + \colvec{v_1 u_1'\\ \vdots \\ v_N u_N'} = Bu
	\end{equation}
	on the circle $\T$ and with transport directions $v_1, \cdots, v_N \in \R.$ Here, $B\in \R^{N\times N}$ is a real matrix. The length of the circle can be easily incorporated with the knowledge from \Cref{section_arbitrary_side_lengths} but we decided to neglect this additional index for the sake of readability. \newline
	
	\setcounter{assumption}{0}
	
	The main difference to reaction-diffusion equations is that transport does not damp high frequencies. Mathematically, this corresponds to the transport-reaction semigroup $(R(t))_{t\geq0}$ on $L^2(\T, \, \C^N)$ generated by
	\begin{equation*}
		-A+B = -\begin{pmatrix}
		v_1  & & \\
		& \ddots & \\
		& & v_N
		\end{pmatrix} \partial_x + B = - V \partial_x + B,
	\end{equation*}
	see \Cref{theorem_semigroup_generated_by_-A+B}, not being analytic. Nevertheless, the question of stability still reduces to the study of the spectrum of the generator $\sigma(-A+B)$ because of the weak spectral mapping \Cref{theorem_WSMTHM_p_eq_2}. By \Cref{char_spectrum_generator}, this spectrum is given by
	\begin{equation*}
		\sigma(-A+B) = \overline{\bigcup_{\indextwo\in \Z} \sigma(M(\indextwo))},
	\end{equation*}
	where $M(\indextwo)\in \C^{N\times N}$ are the matrices
	\begin{equation}\label{matrices_Mk_trans_react}
		M(\indextwo) = -2\pi i \indextwo
		\begin{pmatrix}
		v_1  & & \\
		& \ddots & \\
		& & v_N
		\end{pmatrix}
		+ B = - 2\pi i \indextwo V + B
	\end{equation}
	for $\indextwo\in\Z.$ Compared to \eqref{matrices_Mk_react_diff}, the matrices \eqref{matrices_Mk_trans_react} are complex valued. Furthermore, there exist constants $c_1, \, c_2 \in \R$ such that $\sigma(-A+B)$ is contained in a strip $\C_{c_1,c_2} \coloneqq \{\lambda \in \C \, \colon \, c_1 < \real\lambda < c_2 \}$ by \Cref{lemma_-A+B_generates_C_0_group}. In particular, we will see that infinitely many eigenvalues of $-A+B$ can have a positive real part, even if $B$ is stable and has only eigenvalues with a negative real part. However, there is still an underlying structure in the spectrum $\sigma(-A+B),$ which becomes apparent after a detailed study of the eigenvalues of $M(\indextwo).$ We emphasize that the following considerations are valid for arbitrary matrices $B$. \newline
	
	The main idea is to write $M(\indextwo)$ as
	\begin{equation*}
		M(\indextwo) = -2\pi i \indextwo \left(V + \frac{i}{2\pi \indextwo} B\right),
	\end{equation*}
	so it suffices to compute the eigenvalues of
	\begin{equation}\label{Mk_factored_out}
		V + \frac{i}{2\pi \indextwo} B.
	\end{equation}
	Now for large $|k|$, the matrices \eqref{Mk_factored_out} are a small perturbation of $V$ and perturbation theory from \cite[Chapter 2]{kato2013perturbation} can be applied. The theory in \cite{kato2013perturbation} is presented in the most general way possible, which makes the application not too straight forward, even though our situation is essentially the easiest case regarding \cite[Chapter 2, (2.1), p. 74]{kato2013perturbation}. Therefore, we follow \cite{kato2013perturbation} carefully, add some intermediate steps and try to attach importance to bibliographical references. \newline
	
	In the following, we always assume:
	
	\begin{assumption}\label{assumption_transport_dir}
		The transport-directions $v_1,\cdots, v_N \in \R$ are pairwise different, i.e. $v_\indexfour \neq v_\indexone$ holds for all $\indexfour,\indexone =1 , \cdots, N.$
	\end{assumption}
	
	\Cref{assumption_transport_dir} is necessary in order to get exact formulas for the eigenvalues of $M(\indextwo)$ for large $|\indextwo|$ from an application of \cite[Chapter 2]{kato2013perturbation}. It has also been assumed in \cite{stevens2008partial} and is very unlikely to be violated in applications. \newline
	
	Now, we follow \cite{kato2013perturbation} to study the eigenvalues of
	\begin{equation}\label{perturbed_V}
		V(z) \coloneqq V + zB
	\end{equation}
	for $z\in\C$ with small $|z|.$ This eigenvalue problem falls into the category of \cite[(1.2), p. 63]{kato2013perturbation} with $B^{(1)} = B$ and $B^{(n)}=0$ for all $n\geq2.$ Notice that we adapted the notation from \cite{kato2013perturbation} to our setup by using the letters $V$ and $B^{(1)}$ instead of $T$ and $T^{(1)}$ respectively. \newline
	Under \Cref{assumption_transport_dir}, the unperturbed matrix $V(0) = V$ has the distinct eigenvalues $v_\indexone$ with eigenprojections $P_\indexone = e_\indexone \otimes e_\indexone$, eigennilpotents $D_\indexone=0$ and algebraic and geometric multiplicities $m_\indexone=1$ for $\indexone=1,\cdots,N.$ \newline
	Let $\indexone=1,\cdots,N$ and let us introduce the matrices
	\begin{align*}
		S_\indexone(\zeta) = - \sum^N_{\substack{\indexfour=1 \\ \indexfour \neq \indexone}} \left[ (\zeta-v_\indexfour)^{-1} (e_\indexfour \otimes e_\indexfour) 	+ \sum_{n=1}^{m_\indexfour - 1} (\zeta - v_\indexfour)^{-n-1}D_\indexfour^n \right] = - \sum^N_{\substack{\indexfour=1 \\ \indexfour \neq \indexone}} (\zeta-v_\indexfour)^{-1} (e_\indexfour \otimes e_\indexfour)
	\end{align*}
	from \cite[(5.32), p. 40]{kato2013perturbation} and
	\begin{equation}\label{matrices_Sj}
		S_\indexone \coloneqq S_\indexone(v_\indexone) = - \sum^N_{\substack{\indexfour=1 \\ \indexfour \neq \indexone}} (v_\indexone-v_\indexfour)^{-1} (e_\indexfour \otimes e_\indexfour) = \sum^N_{\substack{\indexfour=1 \\ \indexfour \neq \indexone}} (v_\indexfour-v_\indexone)^{-1} (e_\indexfour \otimes e_\indexfour)
	\end{equation}
	from \cite[(5.28), p. 40]{kato2013perturbation}. Given this notation, the Laurent series of the resolvent $R(\zeta) = (V-\zeta)^{-1}$ (caution: the definition from \cite{kato2013perturbation} differs to the definition of the resolvent we used in the previous chapters) at $\zeta = v_\indexone$ is given by
	\begin{align}\label{Laurent_series_resolvent}
		R(\zeta) &= -(\zeta - v_\indexone)^{-1} (e_\indexone \otimes e_\indexone) + \sum_{n=0}^\infty (\zeta - v_\indexone)^n S_\indexone^{n+1} \eqqcolon \sum_{n=-1}^\infty (\zeta - v_\indexone)^n S_\indexone^{(n+1)},
	\end{align}
	see \cite[(5.18), p. 39]{kato2013perturbation} and \cite[(2.9), p. 76]{kato2013perturbation}. Here, we used the notation
	\begin{equation}\label{definition_Sj^n}
		S_\indexone^{(0)} = -P_\indexone = - e_\indexone \otimes e_\indexone \quad \text{ and } \quad S_\indexone^{(n)} = S_\indexone^n \text{ for } n\geq 1.
	\end{equation}
	Formula \eqref{Laurent_series_resolvent} does not come as a surprise and can also be computed with the geometric series because $R(\zeta)$ is of course just given by
	\begin{equation}\label{resolvent}
		R(\zeta) = \begin{pmatrix}
		(v_1-\zeta)^{-1}  & & \\
		& \ddots & \\
		& & (v_N-\zeta)^{-1}
		\end{pmatrix}
	\end{equation}
	on $\C \backslash \{v_1,\cdots,v_N\}.$ \newline
	For the perturbed matrix $V(z)$ from \eqref{perturbed_V}, the resolvent $R(\zeta, z) = (V(z) - \zeta)^{-1}$ can be written as the power series
	\begin{equation}\label{perturbed_resolvent_power_series}
		R(\zeta, z) = R(\zeta) + \sum_{n=1}^\infty z^n R^{(n)}(\zeta),
	\end{equation}
	where
	\begin{align*}
		R^{(n)}(\zeta) &= \sum_{p=1}^n \sum_{\substack{\nu_1 + \cdots + \nu_p = n \\ \nu_\indexthree \geq 1}} (-1)^p R(\zeta) B^{(\nu_1)} R(\zeta) B^{(\nu_2)} \cdots B^{(\nu_p)} R(\zeta) = (-1)^n R(\zeta) \big(B R(\zeta)\big)^n,
	\end{align*}
	see \cite[(1.13) and (1.14), p. 67]{kato2013perturbation}. Notice that we used $B^{(n)}=0$ for all $n\geq 2$ because $V(z)$ is just a linear perturbation of $V$. \newline
	
	Let $\indexone=1,\cdots,N$ and let $\Gamma_\indexone$ be a closed positively-oriented circle in the resolvent set $\rho(V) = \C \backslash \{v_1,\cdots,v_N\}$ enclosing $v_\indexone$ but no other $v_\indexfour$ for $\indexfour \neq \indexone.$ Let $|z|$ be sufficiently small. Then, the operator
	\begin{equation}\label{projection}
		P_\indexone(z) = - \frac{1}{2\pi i} \int_{\Gamma_\indexone} R(\zeta,z) \, d\zeta
	\end{equation}
	is equal to the sum of the eigenprojections for all the eigenvalues of $V(z)$ lying inside $\Gamma_\indexone$  \cite[p. 67]{kato2013perturbation}. The eigenvalues $\lambda_\indexone(z)$ are continuous functions in $z$ by \cite[Summary, p. 73]{kato2013perturbation} and therefore, assuming small $|z|,$ \Cref{assumption_transport_dir} ensures that there is exactly one eigenvalue $\lambda_\indexone(z)$ lying inside $\Gamma_\indexone$ and $P_\indexone(z)$ is itself the eigenprojection of this eigenvalue $\lambda_\indexone(z)$ \cite[p. 68]{kato2013perturbation}. At $z=0,$ no eigenvalues depart from the unperturbed eigenvalue $\lambda_\indexone(0) = v_\indexone$, i.e. there is no splitting, cf. \cite[Chapter 2, §1.2, p. 65f., p. 68]{kato2013perturbation}. \newline
	The trick from \cite{kato2013perturbation} is now the following: the trace of a matrix is the sum of all its eigenvalues (including multiplicities). Consequently, we obtain the formula
	\begin{equation}\label{first_formula_eigenvalues}
		\lambda_\indexone(z) = \lambda_\indexone(z) + (N-1) \cdot 0  = \trace(V(z)P_\indexone(z)) = v_\indexone + \trace\big((V(z) - v_\indexone) P_\indexone(z)\big),
	\end{equation}
	see \cite[(2.5), p. 75]{kato2013perturbation}. \newline
	
	It is left to find more concrete formulas for the trace appearing in \eqref{first_formula_eigenvalues}. Combining \eqref{perturbed_resolvent_power_series} and \eqref{projection} yields
	\begin{align}\label{series_Pj}
		P_\indexone(z) = e_\indexone \otimes e_\indexone + \sum_{n=1}^{\infty} z^n P_\indexone^ {(n)}
	\end{align}
	with
	\begin{equation*}
		P_\indexone^{(n)} = - \frac{1}{2\pi i} \int_{\Gamma_\indexone} R^{(n)}(\zeta) \, d\zeta = (-1)^{n+1} \frac{1}{2\pi i} \int_{\Gamma_\indexone}  R(\zeta) \big(B R(\zeta)\big)^n \, d\zeta
	\end{equation*}
	see \cite[(1.17) and (1.18), p. 68]{kato2013perturbation}. Now, we plug the Laurent series \eqref{Laurent_series_resolvent} of the resolvent $R(\zeta)$ at $\zeta = v_\indexone$ into the integrand. This gives a Laurent series in $\zeta - v_\indexone,$ of which only the term with power $(\zeta-v_\indexone)^{-1}$ contributes to the integral \cite[p. 76]{kato2013perturbation}. More precisely, the substitution yields
	\begin{alignat*}{2}
		P_\indexone^{(n)} &=  (-1)^{n+1} \frac{1}{2\pi i} \int_{\Gamma_\indexone} && R(\zeta) \big(B R(\zeta)\big)^n \, d\zeta  \\
		&= (-1)^{n+1} \frac{1}{2\pi i} \int_{\Gamma_\indexone} &&\Bigg[\bigg( \sum_{\indextwo_1 = -1}^\infty (\zeta - v_\indexone)^{\indextwo_1} S_\indexone^{(\indextwo_1+1)} \bigg) \bigg( B \sum_{\indextwo_2 = -1}^\infty (\zeta - v_\indexone)^{\indextwo_2} S_\indexone^{(\indextwo_2+1)} \bigg) \\
		&  && \cdots \bigg( B \sum_{\indextwo_{n+1} = -1}^\infty (\zeta - v_\indexone)^{\indextwo_{n+1}} S_\indexone^{(\indextwo_{n+1}+1)} \bigg) \Bigg] \, d\zeta \\
		&= (-1)^{n+1} \sum_{\indextwo_{1} = -1}^\infty \cdots && \sum_{\indextwo_{n+1} = -1}^\infty \bigg(\frac{1}{2\pi i} \int_{\Gamma_\indexone} (\zeta-v_\indexone)^{\indextwo_1+\cdots+\indextwo_{n+1}} \, d\zeta \bigg) S_\indexone^{(\indextwo_{1}+1)} B S_\indexone^{(\indextwo_{2}+1)} \cdots S_\indexone^{(\indextwo_{n}+1)} B S_\indexone^{(\indextwo_{n+1}+1)},
	\end{alignat*}
	which results in the finite sum
	\begin{equation}\label{formula_Pjn}
		\begin{aligned}
		P_\indexone^{(n)} &= (-1)^{n+1} \sum_{\substack{\indextwo_1+\cdots + \indextwo_{n+1} = -1 \\ \indextwo_\indexthree \geq -1}} S_\indexone^{(\indextwo_{1}+1)} B S_\indexone^{(\indextwo_{2}+1)} \cdots S_\indexone^{(\indextwo_{n}+1)} B S_\indexone^{(\indextwo_{n+1}+1)} \\
		&= (-1)^{n+1} \sum_{\substack{\indextwo_1+\cdots + \indextwo_{n+1} = n \\ \indextwo_\indexthree \geq 0}} S_\indexone^{(\indextwo_{1})} B S_\indexone^{(\indextwo_{2})} \cdots S_\indexone^{(\indextwo_{n})} B S_\indexone^{(\indextwo_{n+1})},
		\end{aligned}
	\end{equation}
	cf. \cite[(2.12), p. 76]{kato2013perturbation}. \newline
	
	Concerning the computation of \eqref{first_formula_eigenvalues}, the same idea can be used. Notice that \eqref{projection} and
	\begin{equation*}
		(V(z) - v_\indexone) R(\zeta, z) = I_{\C^{N\times N}} + (\zeta - v_\indexone) R(\zeta,z)
	\end{equation*}
	for all $\zeta\in \range(\Gamma_\indexone)$ gives
	\begin{equation}
		(V(z) - v_\indexone) P_\indexone(z) = - \frac{1}{2\pi i} \int_{\Gamma_\indexone} (\zeta - v_\indexone) R(\zeta,z) \, d\zeta,
	\end{equation}
	see \cite[(2.15), p. 77]{kato2013perturbation}. Similar to above, using the power series \eqref{perturbed_resolvent_power_series} for $R(\zeta,z)$ and holomorphy of 
	\begin{equation*}
		\zeta \longmapsto (\zeta - v_\indexone) R(\zeta)
	\end{equation*}
	in the interior of the circle $\Gamma_\indexone$, see \eqref{resolvent}, yields
	\begin{align}\label{auxiliary_eq42}
		(V(z) - v_\indexone) P_\indexone(z) &= - \frac{1}{2 \pi i} \int_{\Gamma_\indexone} (\zeta - v_\indexone) R(\zeta) \, d\zeta + \sum_{n=1}^\infty z^n \widetilde{B}_\indexone^{(n)} = \sum_{n=1}^\infty z^n \widetilde{B}_\indexone^{(n)}.
	\end{align}
	Here, $\widetilde{B}_\indexone^{(n)}$ are given by
	\begin{equation}\label{formula_tilde_Bj}
		\begin{aligned}
			\widetilde{B}_\indexone^{(n)} &= - \frac{1}{2\pi i} \int_{\Gamma_\indexone} (\zeta - v_\indexone) R^{(n)}(\zeta) \, d\zeta = (-1)^{n+1} \frac{1}{2\pi i} \int_{\Gamma_\indexone} (\zeta-v_\indexone) R(\zeta) \big( BR(\zeta) \big)^n \, d\zeta \\
			&= (-1)^{n+1} \sum_{\substack{\indextwo_1+\cdots + \indextwo_{n+1} = n-1 \\ \indextwo_\indexthree \geq 0}} S_\indexone^{(\indextwo_{1})} B S_\indexone^{(\indextwo_{2})} \cdots S_\indexone^{(\indextwo_{n})} B S_\indexone^{(\indextwo_{n+1})},
		\end{aligned}
	\end{equation}
	cf. \cite[(2.18), p. 77]{kato2013perturbation}. Notice that the last step follows analogously to the computation of $P_\indexone^{(n)}$ with the only difference being the additional factor $(\zeta-v_\indexone)$. \newline
	
	Finally, \eqref{first_formula_eigenvalues} and \eqref{auxiliary_eq42} gives the exact formula
	\begin{equation}\label{formula_eval}
		\lambda_\indexone(z) = v_\indexone + \sum_{n=1}^\infty z^n \hat{\lambda}_\indexone^{(n)} \coloneqq v_\indexone + \sum_{n=1}^\infty z^n \trace \widetilde{B}_\indexone^{(n)}
	\end{equation}
	for the eigenvalues of $V(z) = V + zB,$ assuming that $|z|$ is sufficiently small, cf. \cite[(2.22), p. 78]{kato2013perturbation}. Although being horribly complicated, the formula for the coefficients $\hat{\lambda}_\indexone^{(n)}$ only depend on the model parameters $v_1,\cdots,v_N$ and $B$, see \eqref{matrices_Sj} and \eqref{formula_tilde_Bj}. Moreover, these coefficients $\hat{\lambda}_\indexone^{(n)}$ are always real. \newline
	There is a second more convenient expression for $\hat{\lambda}_\indexone^{(n)},$ which is
	\begin{equation}\label{formula_lamdas}
		\begin{aligned}
		\hat{\lambda}_\indexone^{(n)} &= \frac{1}{m_\indexone} \sum_{p=1}^n \frac{(-1)^p}{p} \sum_{\substack{\nu_1+\cdots+\nu_p = n \\ \indextwo_1+\cdots+\indextwo_p = p-1 \\ \nu_\indexthree \geq 1, \, \indextwo_p \geq 0}} \trace \big( B^{(\nu_1)} S_\indexone^{(\indextwo_1)} \cdots B^{(\nu_p)} S_\indexone^{(\indextwo_p)} \big) \\
		&= \frac{(-1)^n}{n} \sum_{\substack{\indextwo_1+\cdots+\indextwo_n = n-1 \\ \indextwo_\indexthree \geq 0}} \trace \big( B S_\indexone^{(\indextwo_1)} \cdots B S_\indexone^{(\indextwo_n)}\big),
		\end{aligned}
	\end{equation}
	see \cite[(2.31), p. 79]{kato2013perturbation}. We refer to \cite[p. 78f.]{kato2013perturbation} for any details on the proof. Notice that \eqref{formula_lamdas} and \eqref{formula_Pjn} also prove \cite[Problem 2.1, p. 80]{kato2013perturbation}, namely the expression
	\begin{equation*}
		\hat{\lambda}_\indexone^{(n)} = \frac{1}{n} \trace \big( B P_\indexone^{(n-1)} \big).
	\end{equation*}
	For convenience, we summarize everything in a theorem.
	
	\begin{theorem}\label{theorem_perturbed_eigenvalues}
		Let $V \in \R^{N\times N}$ be a diagonal matrix with pairwise different diagonal entries and let $B\in \R^{N\times N}.$ Then for $z\in\C$ with sufficiently small $|z|$, the matrix
		\begin{equation*}
			V(z) = V + zB
		\end{equation*}
		has $N$ distinct eigenvalues $\lambda_1(z), \cdots, \lambda_N(z)$ with
		\begin{equation}\label{series_evalues}
			\lambda_\indexone(z) = v_\indexone + \sum_{n=1}^\infty z^n \hat{\lambda}_\indexone^{(n)}
		\end{equation}
		for all $\indexone=1,\cdots,N.$ The coefficients $\hat{\lambda}_\indexone^{(n)} \in \R$ are given by
		\begin{align}
			\hat{\lambda}_\indexone^{(n)} &= (-1)^{n+1} \sum_{\substack{\indextwo_1+\cdots + \indextwo_{n+1} = n-1 \\ \indextwo_\indexthree \geq 0}} \trace \big(S_\indexone^{(\indextwo_{1})} B S_\indexone^{(\indextwo_{2})} \cdots S_\indexone^{(\indextwo_{n})} B S_\indexone^{(\indextwo_{n+1})}\big) \nonumber \\
			&= \frac{(-1)^n}{n} \sum_{\substack{\indextwo_1+\cdots+\indextwo_n = n-1 \\ \indextwo_\indexthree \geq 0}} \trace \big( B S_\indexone^{(\indextwo_1)} \cdots B S_\indexone^{(\indextwo_n)}\big) \label{second_formula},
		\end{align}
		where $S_\indexone^{(n)}$ are the matrices
		\begin{equation*}
			S_\indexone^{(0)} = -P_\indexone = - e_\indexone \otimes e_\indexone \quad \text{ and } \quad S_\indexone^{(n)} = S_\indexone^n \text{ for } n\geq 1
		\end{equation*}
		with
		\begin{equation*}
			S_\indexone = \sum^N_{\substack{\indexfour=1 \\ \indexfour \neq \indexone}} (v_\indexfour-v_\indexone)^{-1} (e_\indexfour \otimes e_\indexfour).
		\end{equation*}
	\end{theorem}

	Before applying this result in the context of transport-driven instabilities, we actually compute the first three coefficients $\hat{\lambda}_\indexone^{(n)}$ for $\indexone=1,\cdots,N$ and estimate $|\hat{\lambda}_\indexone^{(n)}|.$ The latter verifies convergence of the series in \Cref{theorem_perturbed_eigenvalues} and is also useful for estimates of the remainder of the series in \eqref{series_evalues}.
	
	\begin{lemma}\label{formulas_small_lambda_hat_jn}
		Consider the setup from \Cref{theorem_perturbed_eigenvalues}. For $\indexone=1,\cdots,N$ and $n=1,2,3$, the coefficients $\hat{\lambda}_\indexone^{(n)} \in \R$ are given by
		\begin{align*}
			\hat{\lambda}_\indexone^{(1)} &= b_{\indexone\indexone}, \\
			\hat{\lambda}_\indexone^{(2)} &= - \sum^N_{\substack{\indexfour=1 \\ \indexfour \neq \indexone}} (v_\indexfour - v_\indexone)^{-1} b_{\indexone \indexfour} b_{\indexfour \indexone}, \\
			\hat{\lambda}_\indexone^{(3)} &= \sum_{\substack{i,l=1 \\ i,l\neq j}}^{N} (v_l-v_j)^{-1}(v_i-v_j)^{-1}b_{jl}b_{li}b_{ij} - \sum_{\substack{i=1 \\ i\neq j}}^{N} (v_i-v_j)^{-2}b_{ij}b_{ji}b_{jj}. \\
		\end{align*}
	\end{lemma}
	\begin{proof}
		Let $\indexone=1,\cdots,N.$ We always use formula \eqref{second_formula}. For $n=1$, we have
		\begin{equation*}
			\hat{\lambda}_\indexone^{(1)} = - \trace \big(B S_\indexone^{(0)}\big) = \trace \big( B (e_\indexone \otimes e_\indexone)\big) = b_{\indexone \indexone}.
		\end{equation*}
		The cyclic property of the trace yields
		\begin{equation*}
			\hat{\lambda}_\indexone^{(2)} = \half \sum_{\substack{\indextwo_1 + \indextwo_2 = 1 \\ k_l \geq 0}} \trace \big( B S_j^{(k_1)} B S_j^{(k_2)} \big) = \half \left( \trace \big(B S_j B (-P_j)\big) + \trace \big( B (-P_j) B S_j \big) \right) = - \trace \big(B S_j B P_j \big)
		\end{equation*}
		and
		\begin{alignat*}{2}
			\hat{\lambda}_\indexone^{(3)} &= - \frac{1}{3} \Big( &&\trace\big( B S_j B S_j B (-P_j)\big) + \trace\big( B S_j B (-P_j) B S_j \big)+ \trace\big( B (-P_j) B S_j B S_j \big) \\
			& &&+ \trace\big( B S^2_j B (-P_j) B (-P_j)\big) + \trace\big( B (-P_j) B S_j^2 B (-P_j)\big) + \trace\big( B (-P_j) B (-P_j) B S_j \big) \, \Big) \\
			&= \trace \big(B &&S_j B S_j BP_j\big) - \trace\big( B S_j^2 B P_j BP_j\big).
		\end{alignat*}
		We refer to Appendix \ref{proof_formulas_small_lambdahatjn} for the computation of the traces.
	\end{proof}
	
	\begin{lemma}\label{estimate_lamda_hat_jn}
		Consider the setup from \Cref{theorem_perturbed_eigenvalues}. Then
		\begin{equation*}
			|\hat{\lambda}_\indexone^{(n)}| \leq \max_{i\neq l} |v_i - v_l| \left(\frac{2 \|B\|_\infty}{\min_{i\neq l} |v_i - v_l|}\right)^n
		\end{equation*}
		holds for all $j=1,\cdots,N$ and all $n\in\N.$
	\end{lemma}
	\begin{proof}
		The proof follows \cite[p. 88]{kato2013perturbation}. Notice that the power series of $R(\zeta, z)$ given in \eqref{perturbed_resolvent_power_series} converges for
		\begin{equation*}
			|z| \, \|BR(\zeta)\| < 1.
		\end{equation*}
		Moreover, the resolvent
		\begin{equation}
			R(\zeta) = \begin{pmatrix}
			(v_1-\zeta)^{-1}  & & \\
			& \ddots & \\
			& & (v_N-\zeta)^{-1}
			\end{pmatrix}
		\end{equation}
		has the matrix norm
		\begin{equation*}
			\|R(\zeta)\|_\infty = \frac{1}{\min_{i=1,\cdots,N} |v_i - \zeta|}
		\end{equation*}
		for $\zeta \in \C \backslash \{v_1,\cdots,v_N\}.$ Let $j=1,\cdots,N$ and let $\Gamma_j$ be the closed positively-oriented circles around $v_j$ with radii
		\begin{equation*}
			\half \dist \Big( v_j, \, \bigcup_{i\neq j} \{v_i\} \Big).
		\end{equation*}
		The series \eqref{perturbed_resolvent_power_series} is uniformly convergent for $\zeta\in\Gamma_j$ if
		\begin{equation*}
			|z| < \min_{\zeta\in\Gamma_j} \frac{1}{\|BR(\zeta)\|},
		\end{equation*}
		cf. \cite[(3.3), p. 88]{kato2013perturbation}. In particular, it is uniformly convergent for
		\begin{equation}\label{z_small}
			|z| < \frac{\min_{i\neq l} |v_i-v_l|}{2 \|B\|_\infty}
		\end{equation}
		because of
		\begin{align*}
			\frac{\min_{i\neq l} |v_i-v_l|}{2 \|B\|_\infty} &\leq \frac{\dist \big( v_j, \, \cup_{i\neq j} \{v_i\}\big)}{2 \|B\|_\infty} = \frac{1}{\|B\|_\infty} \min_{\zeta\in\Gamma_j} |v_j-\zeta| = \frac{1}{\|B\|_\infty} \min_{\zeta\in\Gamma_j} \min_{i=1,\cdots,N} |v_i-\zeta| 
			\\
			&= \frac{1}{\|B\|_\infty} \min_{\zeta\in\Gamma_j} \frac{1}{\|R(\zeta)\|_\infty} \leq \min_{\zeta\in\Gamma_j} \frac{1}{\|BR(\zeta)\|}.
		\end{align*}
		Under the assumption of $z$ satisfying \eqref{z_small}, als	o the power series \eqref{series_Pj} for $P_j(z)$ and \eqref{series_evalues} for $\lambda_j(z) - v_j$ is convergent \cite[p. 88]{kato2013perturbation}, which ensures that every function we are dealing with is indeed holomorphic for small $|z|.$ As a consequence, Cauchy's inequality can be used to estimate the coefficients $|\hat{\lambda}_\indexone^{(n)}|.$\newline
		For $z\in\C$ satisfying \eqref{z_small}, the eigenvalue $\lambda_j(z)$ lies inside $\Gamma_j$ by \cite[p. 88]{kato2013perturbation} and it follows
		\begin{equation*}
			|\lambda_j(z) - v_j| \leq \max_{\zeta\in\Gamma_j} |\zeta-v_j| \leq \max_{i\neq l} |v_i - v_l|,
		\end{equation*}
		cf. \cite[(3.4), p. 88]{kato2013perturbation}. Finally, Cauchy's inequality for the Taylor coefficients of holomorphic functions gives
		\begin{equation*}
			|\hat{\lambda}_\indexone^{(n)}| \leq \max_{i\neq l} |v_i - v_l| \left(\frac{2 \|B\|_\infty}{\min_{i\neq l} |v_i - v_l|}\right)^n,
		\end{equation*}
		see \cite[(3.5), p. 88]{kato2013perturbation}.
	\end{proof}
	
	Given all the knowledge of the behavior of the eigenvalues of $V(z)$ for small $|z|,$ we can return our attention to the problem of transport-driven instabilities.
	
	\begin{corollary}\label{eigenvalues_Mk}
		Assume that the transport directions satisfy \Cref{assumption_transport_dir}. Then for $k\in\Z$ with sufficiently large absolute value, the matrix
		\begin{equation*}
			M(\indextwo) = - 2\pi i \indextwo V + B
		\end{equation*}
		has $N$ distinct eigenvalues $\lambda_1(k),\cdots,\lambda_N(k)$ with
		\begin{equation*}
			\lambda_j(k) = b_{jj} - 2\pi i k v_j + \sum_{n=1}^\infty \left(\frac{i}{2\pi k}\right)^n \hat{\lambda}_\indexone^{(n+1)}
		\end{equation*}
		for all $j=1,\cdots,N.$ The coefficients $\hat{\lambda}_\indexone^{(n)} \in \R$ are defined in \Cref{theorem_perturbed_eigenvalues}. In particular, the real parts of the eigenvalues are given by
		\begin{equation}\label{formula_real_parts_ev}
			\real \lambda_j(k) = b_{jj} + \sum_{n=1}^\infty \frac{(-1)^n}{(2\pi k)^{2n}} \hat{\lambda}_\indexone^{(2n+1)}.
		\end{equation}
	\end{corollary}
	\begin{proof}
		Notice that $M(k)$ can be written as
		\begin{equation*}
			M(\indextwo) = -2\pi i \indextwo \left(V + \frac{i}{2\pi \indextwo} B\right).
		\end{equation*}
		Then for $k\in\Z$ with large $|k|,$ \Cref{theorem_perturbed_eigenvalues} implies that $M(k)$ has the $N$ distinct eigenvalues given by
		\begin{align*}
			\lambda_j(k) &= -2\pi i k \left( v_j + \sum_{n=1}^\infty \left(\frac{i}{2\pi k}\right)^n \hat{\lambda}_\indexone^{(n)} \right) = b_{jj} - 2\pi i k v_j + \sum_{n=1}^\infty \left(\frac{i}{2\pi k}\right)^n \hat{\lambda}_\indexone^{(n+1)},
		\end{align*}
		where we used \Cref{formulas_small_lambda_hat_jn} in the last step. The addendum on the real part of the eigenvalues follows from the fact that all coefficients $\hat{\lambda}_\indexone^{(n)}$ are real.
	\end{proof}
	\begin{remark}\label{remark_conv_real_parts}
		Formula \eqref{formula_real_parts_ev} for the real parts of the eigenvalues of $\lambda_j(k)$ and \Cref{estimate_lamda_hat_jn} directly imply
		\begin{equation}\label{limit_real_parts_ev}
			\lim_{|k| \to \infty} \real \lambda_j(k) = b_{jj}.
		\end{equation}
	\end{remark}
	
	The statement from \eqref{limit_real_parts_ev} can be strengthened in the sense that the convergence in \eqref{limit_real_parts_ev} is ``almost always" strictly monotone. This turns out to be the core result for the study of transport-driven instabilities. Before proving a result on the monotonicity of the convergence in \eqref{limit_real_parts_ev}, we need to estimate the remainder of \eqref{formula_real_parts_ev}.
	
	\begin{lemma}\label{estimate_remainder}
		Assume that the transport directions satisfy \Cref{assumption_transport_dir}. Let $j=1,\cdots,N$ and let $n^*\in\N.$ Then it holds
		\begin{equation*}
			\sum_{n=n^*+1}^\infty \frac{(-1)^n}{(2\pi k)^{2n}} \hat{\lambda}_\indexone^{(2n+1)} = \mathcal{O}\left( \frac{1}{k^{2(n^*+1)}}\right) \quad \text{ as } |k| \to \infty.
		\end{equation*}
	\end{lemma}
	\begin{proof}
		This is a simple consequence of \Cref{estimate_lamda_hat_jn}. Firstly, notice that we have
		\begin{align*}
			\sum_{n=n^*+1}^\infty  \frac{(-1)^n}{(2\pi k)^{2n}} \hat{\lambda}_\indexone^{(2n+1)} &= \frac{1}{(2\pi k)^{2(n^*+1)}} \sum_{n=n^*+1}^\infty \frac{(-1)^n}{(2\pi k)^{2(n-n^*-1)}} \hat{\lambda}_\indexone^{(2n+1)} \\
			&= \frac{1}{(2\pi k)^{2(n^*+1)}} \sum_{n=0}^\infty \frac{(-1)^{n+n^*+1}}{(2\pi k)^{2n}} \hat{\lambda}_\indexone^{(2(n+n^*+1)+1)}.
		\end{align*}
		Secondly, \Cref{estimate_lamda_hat_jn} implies existence of constants $C_1, \, C_2 > 0,$ independent of $j$, with
		\begin{align*}
			\left| \sum_{n=0}^\infty \frac{(-1)^{n+n^*+1}}{(2\pi k)^{2n}} \hat{\lambda}_\indexone^{(2(n+n^*+1)+1)} \right| \lesssim \sum_{n=0}^\infty k^{-2n} C_1 C_2^{2(n+n^*+1)+1} \lesssim \sum_{n=0}^\infty \left(\frac{C}{k}\right)^n \lesssim 1
		\end{align*}
		as $|k| \to \infty.$
	\end{proof}
	
	\begin{theorem}\label{theorem_eventually_monotone}
		Assume that the transport directions satisfy \Cref{assumption_transport_dir} and let $\lambda_1(k), \cdots, \lambda_N(k)$ be the eigenvalues of
		\begin{equation*}
			M(\indextwo) = -2\pi i k V + B.
		\end{equation*}
		Then one of the following is true:
		\begin{enumerate}
			\item
			$\real\lambda_j(k)$ is eventually constant,
			\item
			$\real\lambda_j(k)$ is eventually strictly monotone.
		\end{enumerate}
	\end{theorem}
	\begin{remarks}
		\begin{enumerate}
			\item
			We say that $(a_k)_{k\in\Z} \subset \R$ is eventually constant / (strictly) monotone, if there exists some $K\in \N$ such that the sequences $(a_k)_{k\geq K}$ and $(a_{-k})_{k\geq K}$ are both constant / (strictly) monotone.
			\item
			The symmetry
			\begin{equation}
				\real\lambda_j(k) = \real\lambda_j(-k)
			\end{equation}
			holds for all $j=1,\cdots,N$ and all $k\in\Z$ by \Cref{spectrum_symmetric}.
		\end{enumerate}
	\end{remarks}
	\begin{proof}
		Let $j=1,\cdots,N$ and define
		\begin{equation*}
			n^* = \inf \big\{ n\in\N \, \colon \, \hat{\lambda}_\indexone^{(2n+1)} \neq 0 \big\}.
		\end{equation*}
		If $n^* = \infty,$ then the real part of $\lambda_j(k)$ is eventually constant and equal to $b_{jj}$ by \eqref{formula_real_parts_ev} and \eqref{limit_real_parts_ev}. Otherwise, they are given by
		\begin{equation*}
			\real \lambda_j(k) = b_{jj} + \sum_{n=n^*}^\infty \frac{(-1)^n}{(2\pi k)^{2n}} \hat{\lambda}_\indexone^{(2n+1)} = b_{jj} + \frac{(-1)^{n^*}}{(2\pi k)^{2n^*}} \hat{\lambda}_\indexone^{(2n^*+1)} + \sum_{n=n^*+1}^\infty \frac{(-1)^n}{(2\pi k)^{2n}} \hat{\lambda}_\indexone^{(2n+1)}.
		\end{equation*}
		 Now, \Cref{estimate_remainder} implies
		\begin{align*}
			\real \lambda_j(k+1) - \real \lambda_j(k) &= \frac{(-1)^{n^*}}{(2\pi)^{2n^*}} \hat{\lambda}_\indexone^{(2n^*+1)} \left( \frac{1}{(k+1)^{2n^*}} - \frac{1}{k^{2n^*}} \right) + \mathcal{O}\left( \frac{1}{k^{2(n^*+1)}}\right)\\
			&=  \frac{(-1)^{n^*}}{(2\pi)^{2n^*}} \hat{\lambda}_\indexone^{(2n^*+1)} \frac{(-2n^*)}{|k|^{2n^*+1}} + \mathcal{O}\left( \frac{1}{k^{2(n^*+1)}}\right).
		\end{align*}
		Consequently, in the case $(-1)^{n^*+1} \hat{\lambda}_\indexone^{(2n^*+1)} > 0$ there exists a constant $C = C(n^*)>0$ with
		\begin{equation*}
			\real \lambda_j(k+1) - \real \lambda_j(k) \geq \frac{2n^*(-1)^{n^*+1}}{(2\pi)^{2n^*}} \hat{\lambda}_\indexone^{(2n^*+1)} \frac{1}{|k|^{2n^*+1}}  - \frac{C}{k^{2(n^*+1)}} > 0
		\end{equation*}
		for large $|k|,$ i.e. the real part of $\lambda_j(k)$ is eventually strictly monotonically increasing. \newline
		Similarly, $(-1)^{n^*+1} \hat{\lambda}_\indexone^{(2n^*+1)} < 0$ implies that the real part of $\lambda_j(k)$ is eventually strictly monotonically decreasing.
	\end{proof}
	\begin{remarks}\label{remarks_eventually_monotone}
		\begin{enumerate}
			\item		
			The set of parameter constellations $(V,B)$, which imply existence of an eventually constant real part of an eigenvalue $\lambda_j(k)$ of $M(k)$ is closed in the set of all parameters $(V,B)$ fulfilling \Cref{assumption_transport_dir}. This is a consequence of the continuity of the map
			\begin{equation*}
			(V,B) \longmapsto \hat{\lambda}_\indexone^{(2n+1)}
			\end{equation*}
			with domain
			\begin{equation*}
			\{(V,B) \, \colon \, V\in\R^{N\times N} \text{ diagonal with } v_i\neq v_j \text{ for all } i,j=1,\cdots,N \text{ and } B\in\R^{N\times N} \},
			\end{equation*}
			see \eqref{second_formula}.
			\item
			In some sense, it is ``very unlikely" that one eigenvalue $\lambda_j(k)$ has eventually constant real part: the coefficients $\hat{\lambda}_\indexone^{(2n+1)}$ have to vanish for all $n\in\N$.
			\item
			Typically, $\hat{\lambda}_\indexone^{(3)} \neq 0$. In this case, $\real \lambda_j(k)$ is eventually strictly monotonically increasing if and only if
			\begin{equation*}
				\sum_{\substack{i,l=1 \\ i,l\neq j}}^{N} (v_l-v_j)^{-1}(v_i-v_j)^{-1}b_{jl}b_{li}b_{ij} - \sum_{\substack{i=1 \\ i\neq j}}^{N} (v_i-v_j)^{-2}b_{ij}b_{ji}b_{jj} = \hat{\lambda}_\indexone^{(3)} > 0
			\end{equation*}
			and strictly monotonically decreasing if and only if
			\begin{equation*}
				\sum_{\substack{i,l=1 \\ i,l\neq j}}^{N} (v_l-v_j)^{-1}(v_i-v_j)^{-1}b_{jl}b_{li}b_{ij} - \sum_{\substack{i=1 \\ i\neq j}}^{N} (v_i-v_j)^{-2}b_{ij}b_{ji}b_{jj} = \hat{\lambda}_\indexone^{(3)} < 0.
			\end{equation*}
		\end{enumerate}
	\end{remarks}
	
	Although being very unlikely, the real part of an eigenvalue $\lambda_j(k)$ can be eventually constant. For $N=2,$ we can give a characterization of all parameter choices leading to this instance.
	
	\begin{lemma}\label{eventually_constant_N2}
		Let $N=2.$ Assume that the transport directions satisfy \Cref{assumption_transport_dir}, i.e. $v_1 \neq v_2.$ Let $\lambda_1(k)$ and $ \lambda_2(k)$ be the eigenvalues of
		\begin{equation*}
			M(\indextwo) = -2\pi i k V + B.
		\end{equation*} 
		Then, $\real\lambda_j(k)$ is eventually constant if and only if $B$ fulfills (at least) one of the following three conditions:
		\begin{enumerate}
			\item
			$b_{12} = 0$,
			\item
			$b_{21} = 0$,
			\item
			$b_{11} = b_{22}$.
		\end{enumerate}
	\end{lemma}
	\begin{remark}
		For $N=2$, the real parts of both eigenvalues $\lambda_j(k)$ of $M(k)$ are eventually constant if and only if the real part of one sequence $\lambda_1(k)$ or $\lambda_2(k)$ is eventually constant.
	\end{remark}
	\begin{proof}
		The proof is given in Appendix \ref{proof_eventually_constant_N2}
	\end{proof}
	
	As of now, we assume that $(V,B)$ is not such an exceptional parameter constellation yielding existence of some eigenvalue $\lambda_j(k)$ of $M(k)$ with eventually constant real part. \newline
	
	\Cref{theorem_eventually_monotone} and \eqref{limit_real_parts_ev} have massive consequences for the long-term behavior of solutions of the transport-reaction system \eqref{transport_driven_instab_eq}. \newline
	
	In the context of transport-driven instabilities, we assume
	
	\begin{assumption}\label{assumption_B}
		The spectrum $\sigma(B)$ is contained in the left half plane of $\C,$ that is
		\begin{equation*}
			\sigma(B) \subseteq \C_- = \{ \lambda \in \C \, \colon \, \real\lambda < 0 \}.
		\end{equation*}
	\end{assumption}
	
	Let us recall the definition of Turing patterns. To this end, we again use the notation
	\begin{equation}\label{defSigma}
		\Sigma(\indextwo) = \max_{\lambda \in \sigma(M(\indextwo))} \real\lambda.
	\end{equation}
	
	\begin{definition}[Turing pattern formation]
		The transport-reaction model \eqref{transport_driven_instab_eq} generates Turing patterns if $B$ fulfills \Cref{assumption_B} and if there exist finitely many $k_1, \cdots, k_n \in \Z$ such that the following holds:
		\begin{enumerate}
			\item
			\begin{equation*}
				\Sigma(\indextwo_\indexfour) = \Sigma(\indextwo_\indexone) > 0 \qquad \text{for all } \indexfour,\indexone=1,\cdots,n,
			\end{equation*}
			\item
			\begin{equation*}
				\sup_{\substack{\indextwo\in\Z \\ \indextwo\neq \indextwo_1,\cdots, \indextwo_n}} \Sigma(\indextwo) < \Sigma(\indextwo_\indexfour)\qquad \text{for all } \indexfour=1,\cdots,n.
			\end{equation*}
		\end{enumerate}
	\end{definition}
	
	In addition to Turing patterns, there is a second possible qualitative behavior for the solutions of \eqref{transport_driven_instab_eq}.
	
	\begin{definition}[{Hyperbolic instabilities}]\label{def_hyperbolic_instability}
		The transport-reaction model \eqref{transport_driven_instab_eq} generates hyperbolic instabilities if $B$ fulfills \Cref{assumption_B} and if there exists $b > 0$ and $K\in\N$ such that the following holds:
		\begin{enumerate}
			\item
			The sequence $(\Sigma(k))_{k\geq K}$ is strictly monotonically increasing with limit $b$,
			\item
			\begin{equation*}
				\sup_{\substack{k\in\Z \\ |k|\leq K}} \Sigma(k) < b.
			\end{equation*}
		\end{enumerate}
	\end{definition}
 	\begin{remark}
 		The definition relies on the symmetry $\sigma(M(-k)) = \overline{\sigma(M(k))}$ for all $k\in\N,$ see \Cref{spectrum_symmetric}.
 	\end{remark}
 	
 	Intuitively, hyperbolic instabilities describe chaotic and increasingly oscillating behavior: for large $t$, high wave numbers $|k|$ are dominant and for even larger $t$, even higher waves numbers are dominant due to the eventual strict monotonicity of $(\Sigma(k))_{k\in\Z}$. The remarkable thing is that \textit{every single} high frequency is important. Given that the spectrum of the generator has to be contained in a strip $\C_{c_1,c_2} = \{\lambda \in \C \, \colon \, c_1 < \real\lambda < c_2 \}$ for some constants $c_1,\,c_2\in\R$ by \Cref{lemma_-A+B_generates_C_0_group}, hyperbolic instabilities in some sense correspond to the most chaotic qualitative behavior theoretically possible. \newline
 	In general, if the initial function has an infinite Fourier series, increasingly high frequencies are dominant. If the initial function has a finite Fourier series, either the highest wave number or finitely many wave numbers independent of the initial function ($\text{argmax}_{|k|\leq K} \Sigma(k)$) are dominant. The number of peaks of the solution typically increases over time.\newline
 	
 	Regarding transport-driven instabilities, there is the following dichotomy.
 	
 	\begin{theorem}\label{theorem_dichotomy}
 		Let the transport-directions $v_1,\cdots,v_N$ be pairwise different. Moreover, assume that $B$ is a stable matrix, i.e. \Cref{assumption_transport_dir} and \Cref{assumption_B} are fulfilled. If the transport-reaction model \eqref{transport_driven_instab_eq} generates an instability, the instability is either a Turing pattern or a hyperbolic instability.
 	\end{theorem}
 	\begin{proof}
 		Notice that \Cref{theorem_eventually_monotone} implies
 		\begin{align*}
 			\Sigma(k) &= \max_{\lambda \in \sigma(M(\indextwo))} \real\lambda = \max_{j=1,\cdots,N} \real \lambda_j(k) = \max \big\{ \Sigma_{\text{incr}}(k), \, \Sigma_{\text{decr}}(k) \big\},
 		\end{align*}
 		with
 		\begin{equation*}
 			\Sigma_{\text{incr}}(k) \coloneqq \max_{\substack{j=1,\cdots,N \\ \real \lambda_j(k) \text{ ev. strictly } \\ \text{ increasing }}} \real \lambda_j(k)
 		\end{equation*}
 		and
 		\begin{equation*}
	 		\Sigma_{\text{decr}}(k) \coloneqq \max_{\substack{j=1,\cdots,N \\ \real \lambda_j(k) \text{ ev. strictly } \\ \text{ decreasing }}} \real \lambda_j(k).
 		\end{equation*}
 		Hence, $\Sigma(k)$ is eventually the maximum of a strictly monotonically increasing and a strictly monotonically decreasing function. Moreover, \eqref{limit_real_parts_ev} implies
 		\begin{align}\label{limit_sigma}
 			\lim_{|k|\to\infty} \Sigma_{\text{incr}}(k) &= \max\{ b_{jj} \, \colon \, \real \lambda_j(k) \text{ eventually strictly monotonically increasing} \} \eqqcolon b_1, \nonumber \\
 			\lim_{|k|\to\infty} \Sigma_{\text{decr}}(k) &= \max\{ b_{jj} \, \colon \, \real \lambda_j(k) \text{ eventually strictly monotonically decreasing} \} \eqqcolon b_2, \nonumber\\
 			\lim_{|k|\to\infty} \Sigma(k) &=  \max \{b_1, \, b_2\} = \max_{j=1,\cdots,N} b_{jj} \eqqcolon b.
 		\end{align}
 		A case distinction yields the dichotomy: if we have $b_1 \leq b_2$, the sequence $\Sigma(k)$ converges eventually strictly monotonically decreasing to $b$ and Turing patterns emerge. If we have $b_1 > b_2,$ there exists some $K\in\N$ such that the sequence $(\Sigma(k))_{k\geq K}$ converges strictly monotonically increasing to $b$. Now, if there exists some $k$ with absolute value smaller than $K$ and
 		\begin{equation*}
 			\Sigma(k)  > b \qquad,
 		\end{equation*}
 		the transport-reaction model generates Turing patterns. If not,
 		\begin{equation*}
 			\Sigma(k) < b
 		\end{equation*}
 		holds for all $k\in\Z$ and the instability is hyperbolic.
 	\end{proof}
 
 	We also obtain a useful criteria, which guarantees the existence of transport-driven instabilities.
 	
 	\begin{corollary}
 		Let the transport-directions $v_1,\cdots,v_N$ be pairwise different. Moreover, assume that $B$ is a stable matrix, i.e. \Cref{assumption_transport_dir} and \Cref{assumption_B} are fulfilled. Assume that $B$ has a positive diagonal entry $b_{jj}>0$. Then the transport-reaction model \eqref{transport_driven_instab_eq} is unstable and it either generates a Turing pattern or a hyperbolic instability.
 	\end{corollary}
 	\begin{proof}
 		By assumption and \eqref{limit_sigma}, we have
 		\begin{equation*}
 			\lim_{|k|\to\infty} \Sigma(k) = \max_{j=1,\cdots,N} b_{jj} > 0
 		\end{equation*}
 		and the transport-reaction model \eqref{transport_driven_instab_eq} generates an instability. It is either a Turing pattern or a hyperbolic instability by the dichotomy \Cref{theorem_dichotomy}.
 	\end{proof}
 
 	\Cref{remarks_eventually_monotone} and \eqref{limit_real_parts_ev} also reveal a useful criterion for the existence of Turing patterns.
 	
 	\begin{corollary}\label{corollary_existence_turing_patterns}
 		Let the transport-directions $v_1,\cdots,v_N$ be pairwise different. Moreover, assume that $B$ is a stable matrix, i.e. \Cref{assumption_transport_dir} and \Cref{assumption_B} are fulfilled. Let $j^*$ be the index with $b_{j^*j^*} = \max_{j=1,\cdots,N} b_{jj}$ (wlog unique). Assume that the model is unstable and assume
 		\begin{equation*}
	 		\lambda_{j^*}^{(3)} = \sum_{\substack{i,l=1 \\ i,l\neq j^*}}^{N} (v_l-v_{j^*})^{-1}(v_i-v_{j^*})^{-1}b_{{j^*}l}b_{li}b_{i{j^*}} - \sum_{\substack{i=1 \\ i\neq {j^*}}}^{N} (v_i-v_{j^*})^{-2}b_{i{j^*}}b_{{j^*}i}b_{{j^*}{j^*}} < 0.
 		\end{equation*}
 		Then the model generates Turing patterns.
 	\end{corollary}
 	\begin{proof}
 		We eventually have
 		\begin{equation*}
 			\Sigma(k) = \real \lambda_{j^*}(k),
 		\end{equation*}
 		which is eventually strictly monotonically decreasing due to $\lambda_{j^*}^{(3)} < 0,$ see \Cref{remarks_eventually_monotone}.
 	\end{proof}
 
 	\subsection{Transport-Driven Instabilities for $N=2$}
 	
 	Until now, we almost always considered the general case of arbitrary many components $N\in \N$. This subsection gives a very brief excursion to the case $N=2.$ In this simple case, the transport-reaction model \eqref{transport_driven_instab_eq} reads
 	\begin{equation}\label{transport_reaction_N2}
 		\partial_t \colvec{u_1 \\ u_2} + \colvec{v_1 u_1'\\ v_2 u_2'} = \begin{pmatrix}
 		a & b \\
 		c & d
 		\end{pmatrix} \colvec{u_1 \\ u_2},
 	\end{equation}
 	where we used the notation
 	\begin{equation*}
 		B = \begin{pmatrix}
 		a & b \\
 		c & d
 		\end{pmatrix}
 	\end{equation*}
 	for the matrix $B\in \R^{2\times 2}.$ The key takeaway is that basically the same parameter choices for $B$ which lead to Turing pattern formation in reaction-diffusion equations, see \Cref{turing_instab_reaction_diffusion_eq}, lead to hyperbolic instabilities in transport-reaction equations. The difference is that the transport speeds $v_1$ and $v_2$ as well as the length $L$ of the circle $\T_L$ have no influence on the instability to occur. \newline
 	In contrast to the general case with $N$ components, $N=2$ ensures that the sequences of the real parts of the eigenvalues of
 	\begin{equation*}
 		M(k) = - 2\pi i k \begin{pmatrix}
 		v_1 & 0 \\
 		0 & v_2
 		\end{pmatrix} + \begin{pmatrix}
 		a & b \\
 		c & d
 		\end{pmatrix}
 	\end{equation*}
	are not only eventually monotone, but even monotone from $k=0$ on. That is, $\big(\real \lambda_1(k)\big)_{k\geq0}$ and $\big(\real \lambda_2(k)\big)_{k\geq0}$ are monotone sequences. Also the assumption $v_1 \neq v_2$ from \Cref{eigenvalues_Mk} is no mathematical restriction at all. However, the case $v_1 = v_2$ is uninteresting, because thereby the eigenvalues of $M(k)$ lie on straight lines. \newline
	
	In any case, for a transport-driven instability to occur, $B$ has to be stable, i.e.
	\begin{align*}
		\trace(B) &= a + d < 0, \\
		\det(B) &= ad - bc > 0.
	\end{align*}
	Additionally, the case of eventually constant real parts of the eigenvalues of $M(k)$ in \Cref{theorem_eventually_monotone} can be ruled out by \Cref{eventually_constant_N2}. If both eigenvalues are eventually constant, the eigenvalues lie either on straight lines or eventually have a real part equal to $a=d.$ Under the assumption that $B$ is stable and given the immediate monotonicity of the real parts of the eigenvalues, the whole spectrum of $-A+B$ lies in $\C_-$ and the transport-reaction model does not generate instabilities. \newline
	This consideration, \eqref{limit_real_parts_ev} and the immediate monotonicity of the real parts of the eigenvalues can be summarized in the following theorem.
	
	\begin{theorem}
		Let $N=2$. The transport-reaction model \eqref{transport_reaction_N2} generates instabilities if and only if $v_1 \neq v_2$ and the matrix $B$ fulfills
		\begin{align*}
			\trace(B) &= a + d < 0, \\
			\det(B) &= ad - bc > 0, \\
			a > 0 &\text{ or } d>0.
		\end{align*}
		In this case, the instability is a hyperbolic instability in the sense of \Cref{def_hyperbolic_instability}.
	\end{theorem}
	\begin{remark}
		Wlog, let $d<0.$ For hyperbolic instabilities to exist, $a$ has to be positive and $b$ and $c$ need to have different signs. \newline
		The situation in which $b,\,d<0$ and $ a, \, c > 0$ fulfill all conditions is remarkably similar to the activator-inhibitor scenario for reaction-diffusion equations. The ``hyperbolic interpretation" of these parameters is to see $u_1$ as a prey and $u_2$ as a predator. The prey $u_1$ increases by itself and its growth can be offset by the predator $u_2,$ which increases only in the presence of the prey.
	\end{remark}
	
	\section{Supplementary Figures and Simulations}
	
	Let $v_1,\cdots,v_N\in\R$ and let $B\in\R^{N\times N}.$ Consider the transport-reaction equation
	\begin{equation}\label{transport_reaction_eq}
		\partial_t u + \colvec{v_1 u_1'\\ \vdots \\ v_N u_N'} = Bu
	\end{equation}
	Complementary to our abstract findings from \Cref{eigenvalues_Mk}, we add figures to demonstrate how the spectrum
	\begin{equation*}
		\sigma(-A+B) = \overline{\bigcup_{\indextwo\in \Z} \sigma(M(\indextwo))}
	\end{equation*}
	of generator of the transport-reaction semigroup $(R(t))_{t\geq0}$ on $L^2(\T, \, \C^N)$ typically looks like. As always, the matrices $M(k)$ are given by
	\begin{equation*}
		M(\indextwo) = -2\pi i \indextwo
		\begin{pmatrix}
		v_1  & & \\
		& \ddots & \\
		& & v_N
		\end{pmatrix}
		+ B = - 2\pi i \indextwo V + B.
	\end{equation*}
	
	In addition, we simulated the solution of \eqref{transport_reaction_eq} in one exemplary case of Turing patterns and hyperbolic instabilities respectively. The simulations are based on formula \eqref{auxeq5}. We have no expertise in numerical simulations of hyperbolic equations and used the exact formula in order to avoid difficulties like numerical diffusion.

	\subsection{Figures of The Spectrum for $N=2$}
	
	In the case $N=2$, the transport-reaction model \eqref{transport_reaction_eq} reads
	\begin{equation*}
		\partial_t \colvec{u_1 \\ u_2} + \colvec{v_1 u_1'\\ v_2 u_2'} = B \colvec{u_1 \\ u_2},
	\end{equation*}
	for a matrix $B\in \R^{2\times 2}.$ What follows are plots of
	\begin{equation*}
		\sigma(-A+B) = \overline{\bigcup_{\indextwo\in \Z} \sigma(M(\indextwo))},
	\end{equation*}
	where the matrices $M(k)$ are given by
	\begin{equation*}
		M(\indextwo) = -2\pi i \indextwo
		\begin{pmatrix}
		v_1  & 0 \\
		0 & v_2
		\end{pmatrix}
		+ B.
	\end{equation*}
	Obviously, the plots only contain a part of the spectrum and there is a cutoff in $|k|.$ Each dot in the pictures corresponds to one eigenvalue of $M(k), \, k\in\Z$. The black crosses are the eigenvalues of $B$ and the parameter choices are written down on the right of each plot.
	
	\begin{minipage}{0.75\textwidth}
		\begin{figure}[H]
		\centering
		\includegraphics[width = 4in]{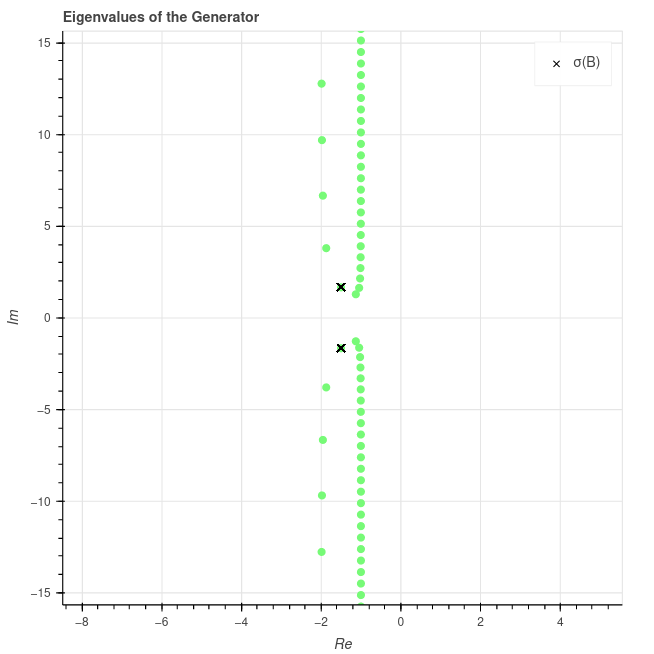}
		\caption[Typical Spectrum for $N=2.$ Example 1]{A typical spectrum of $-A+B$ for $N=2.$}
		\label{fig:typical_spec_N2}
		\end{figure}
	\end{minipage} \hfill
	\begin{minipage}{0.2\textwidth}
		Parameter choices: \newline
		\begin{equation*}
			B = \begin{pmatrix}
			-2 & 3 \\
			-1 & -1
			\end{pmatrix}
		\end{equation*}
		\begin{align*}
			v_1 &= 0.5 \\
			v_2 &= -0.1
		\end{align*}
	\end{minipage}

	\begin{minipage}{0.75\textwidth}
		\begin{figure}[H]
			\centering
			\includegraphics[width = 4in]{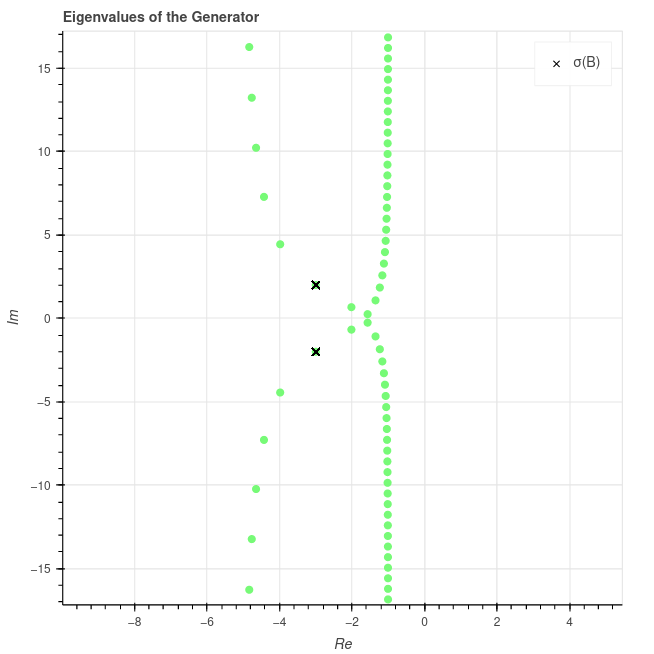}
			\caption[Typical Spectrum for $N=2.$ Example 2]{A typical spectrum of $-A+B$ for $N=2.$}
			\label{fig:typical_spec2_N2}
		\end{figure}
	\end{minipage} \hfill
	\begin{minipage}{0.2\textwidth}
		Parameter choices: \newline
		\begin{equation*}
		B = \begin{pmatrix}
		-5 & 2 \\
		-4 & -1
		\end{pmatrix}
		\end{equation*}
		\begin{align*}
		v_1 &= -0.5 \\
		v_2 &= -0.1
		\end{align*}
	\end{minipage}

	\begin{minipage}{0.75\textwidth}
		\begin{figure}[H]
			\centering
			\includegraphics[width = 4in]{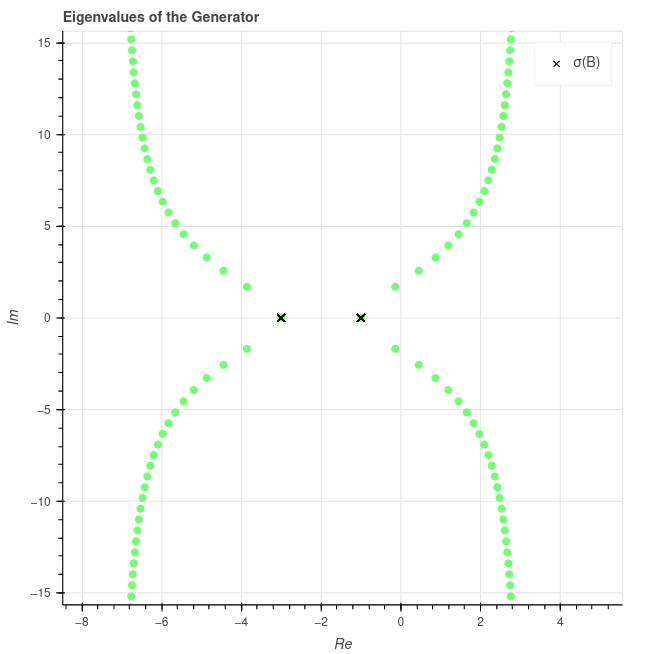}
			\caption[Hyperbolic Instabilities for $N=2.$ Example 1]{An example of the spectrum of $-A+B$ in the case of hyperbolic instabilities for $N=2.$}
			\label{fig:hyperbolic_instab_N2}
		\end{figure}
	\end{minipage} \hfill
	\begin{minipage}{0.2\textwidth}
		Parameter choices: \newline
		\begin{equation*}
			B = \begin{pmatrix}
			3 & 8 \\
			-3 & -7
			\end{pmatrix}
		\end{equation*}
		\begin{align*}
			v_1 &= 0.1 \\
			v_2 &= -0.1
		\end{align*}
	\end{minipage}

	\begin{minipage}{0.75\textwidth}
		\begin{figure}[H]
			\centering
			\includegraphics[width = 4in]{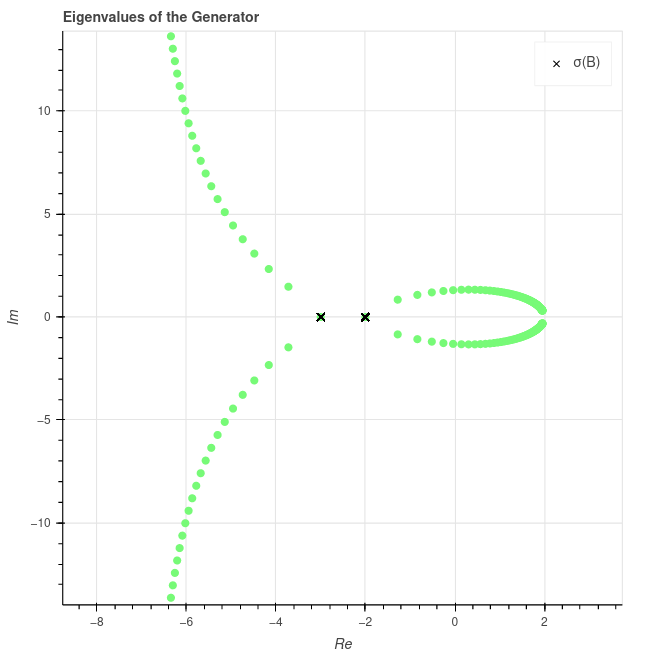}
			\caption[Hyperbolic Instabilities for $N=2.$ Example 2]{An example of the spectrum of $-A+B$ in the case of hyperbolic instabilities for $N=2$ with only one moving component.}
			\label{fig:hyperbolic_instab2_N2}
		\end{figure}
	\end{minipage} \hfill
	\begin{minipage}{0.2\textwidth}
		Parameter choices: \newline
		\begin{equation*}
			B = \begin{pmatrix}
			-7 & 4 \\
			-5 & 2
			\end{pmatrix}
		\end{equation*}
		\begin{align*}
			v_1 &= -0.1 \\
			v_2 &= 0
		\end{align*}
	\end{minipage}

	\begin{minipage}{0.75\textwidth}
		\begin{figure}[H]
			\centering
			\includegraphics[width = 4in]{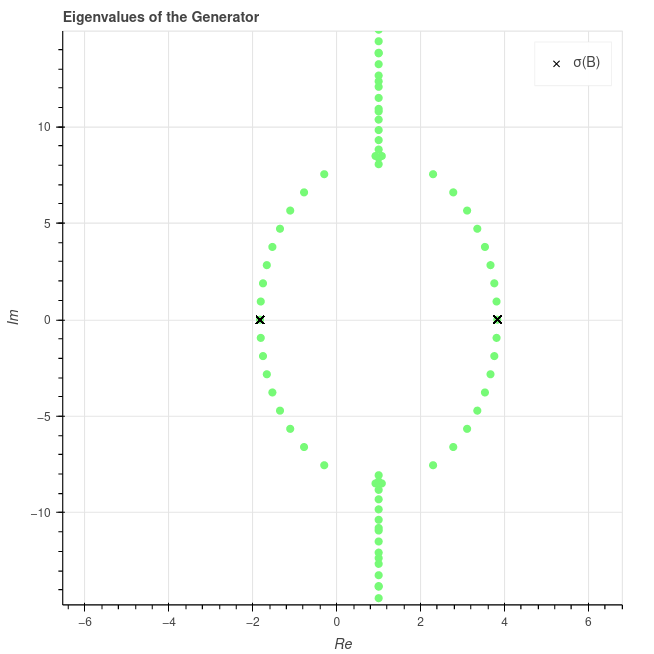}
			\caption[Eventually Constant Real Parts for $N=2.$ An Example]{An example of the spectrum of $-A+B$ in the case of eventually constant real parts for $N=2,$ see \Cref{eventually_constant_N2}.}
			\label{fig:ev_const_real_partN2}
		\end{figure}
	\end{minipage} \hfill
	\begin{minipage}{0.2\textwidth}
		Parameter choices: \newline
		\begin{equation*}
			B = \begin{pmatrix}
			1 & 2 \\
			4 & 1
			\end{pmatrix}
		\end{equation*}
		\begin{align*}
			v_1 &= -0.1 \\
			v_2 &= -0.2
		\end{align*}
	\end{minipage}

	\subsection{Figures of The Spectrum for $N=3$}
	
	In the case $N=3$, the transport-reaction model \eqref{transport_reaction_eq} reads
	\begin{equation*}
		\partial_t \colvec{u_1 \\ u_2 \\ u_3} + \begin{pmatrix}
		v_1  & 0 & 0\\
		0 & v_2 & 0 \\
		0 & 0 & v_3
		\end{pmatrix} \colvec{ u_1'\\ u_2' \\u_3'} = B \colvec{u_1 \\ u_2 \\ u_3}
	\end{equation*}
	with $B\in\R^{3\times 3}.$  What follows are plots of
	\begin{equation*}
		\sigma(-A+B) = \overline{\bigcup_{\indextwo\in \Z} \sigma(M(\indextwo))},
	\end{equation*}
	where the matrices $M(k)$ are given by
	\begin{equation}\label{Mk_case_N3}
		M(\indextwo) = -2\pi i \indextwo
		\begin{pmatrix}
		v_1  & 0 & 0\\
		0 & v_2 & 0 \\
		0 & 0 & v_3
		\end{pmatrix}
		+ B.
	\end{equation}
	Again, the plots only contain a part of the spectrum and there is a cutoff in $|k|.$ Each dot in the pictures corresponds to one eigenvalue of $M(k), \, k\in\Z$. The black crosses are the eigenvalues of $B$ and the parameter choices are written down on the right of each plot. \newline
	
	In contrast to $N=2,$ there are a variety of possibilities for the location of the eigenvalues for small $|k|.$ In particular, the real parts of the eigenvalues are generally no longer immediately monotone.
	
	\begin{minipage}{0.75\textwidth}
		\begin{figure}[H]
			\centering
			\includegraphics[width = 4in]{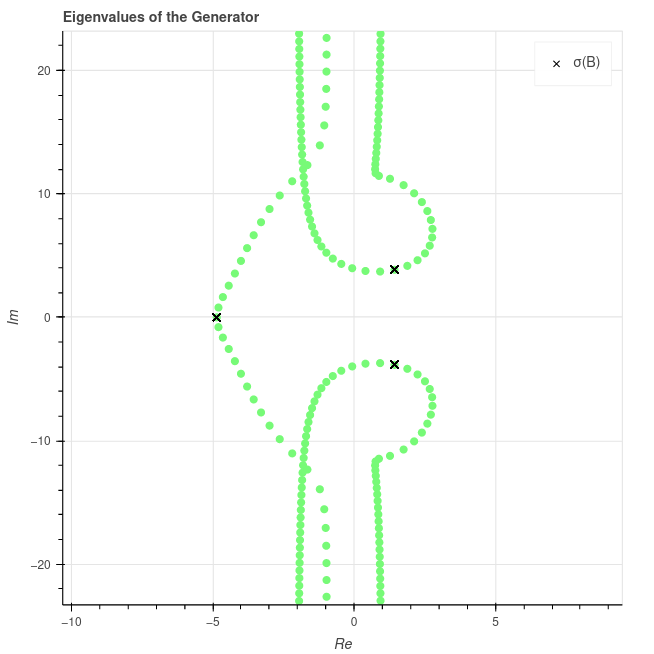}
			\caption[Spectrum for $N=3$. An Example]{An example of the spectrum of $-A+B$ for $N=3.$}
			\label{fig:spec_N3}
		\end{figure}
	\end{minipage} \hfill
	\begin{minipage}{0.2\textwidth}
		Parameter choices: \newline
		\begin{equation*}
		B = \begin{pmatrix}
		1 & 2 & -4\\
		-5 & -2 & 2 \\
		-4 & -5 & -1
		\end{pmatrix}
		\end{equation*}
		\begin{align*}
		v_1 &= 0.1 \\
		v_2 &= -0.1 \\
		v_3 &= 0.2
		\end{align*}
	\end{minipage}

	\begin{minipage}{0.75\textwidth}
		\begin{figure}[H]
			\centering
			\includegraphics[width = 4in]{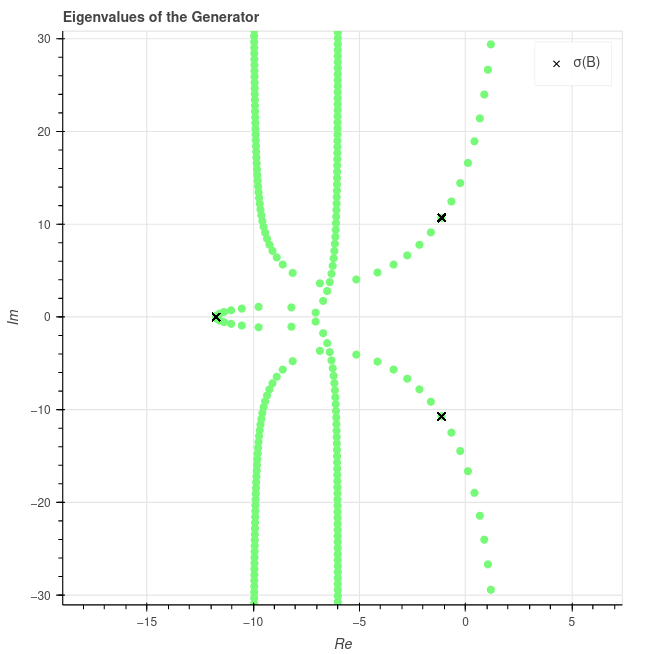}
			\caption[Hyperbolic Instabilities for $N=3$. An Example]{An example of the spectrum of $-A+B$ in the case of hyperbolic instabilities for $N=3.$}
			\label{fig:hyperbolic_instab_N3}
		\end{figure}
	\end{minipage} \hfill
	\begin{minipage}{0.2\textwidth}
		Parameter choices: \newline
		\begin{equation*}
			B = \begin{pmatrix}
			-6 & 2 & -9\\
			4 & -10 & -5 \\
			8 & 10 & 2
			\end{pmatrix}
		\end{equation*}
		\begin{align*}
			v_1 &= 0.1 \\
			v_2 &= -0.1 \\
			v_3 &= 0.5
		\end{align*}
	\end{minipage}

	\begin{minipage}{0.75\textwidth}
		\begin{figure}[H]
			\centering
			\includegraphics[width = 4in]{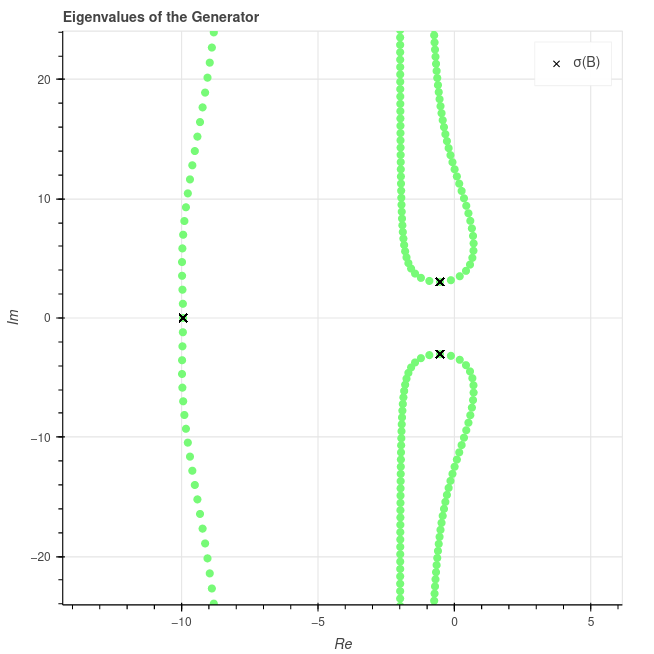}
			\caption[Turing Patterns for $N=3$. Example 1]{An example of the spectrum of $-A+B$ in the case of Turing patterns for $N=3.$}
			\label{fig:turing_N3}
		\end{figure}
	\end{minipage} \hfill
	\begin{minipage}{0.2\textwidth}
		Parameter choices: \newline
		\begin{equation*}
			B = \begin{pmatrix}
			-1 & 2 & -4\\
			-2 & -2 & 2 \\
			-6 & -7 & -8
			\end{pmatrix}
		\end{equation*}
		\begin{align*}
			v_1 &= 0.1 \\
			v_2 &= -0.1 \\
			v_3 &= 0.2
		\end{align*}
	\end{minipage}

	\begin{minipage}{0.75\textwidth}
		\begin{figure}[H]
			\centering
			\includegraphics[width = 4in]{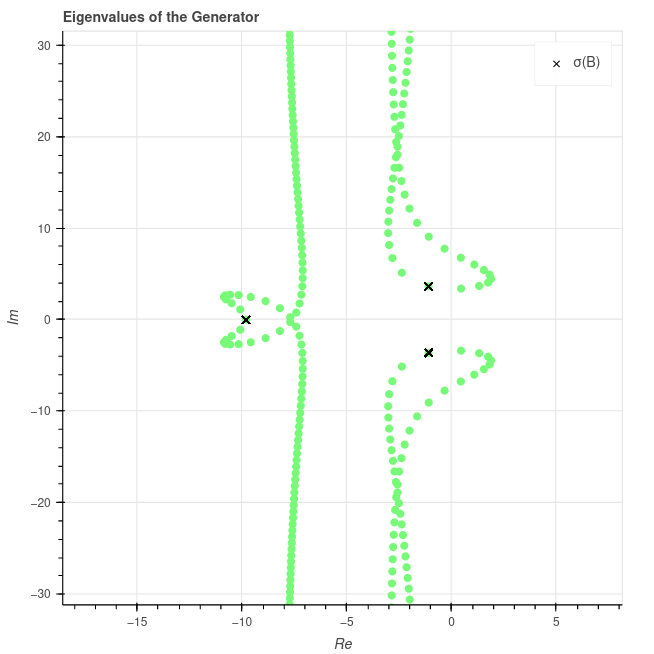}
			\caption[Turing Patterns for $N=3$. Example 2]{An example of the spectrum of $-A+B$ in the case of Turing patterns for $N=3.$}
			\label{fig:turing2_N3}
		\end{figure}
	\end{minipage} \hfill
	\begin{minipage}{0.2\textwidth}
		Parameter choices: \newline
		\begin{equation*}
			B = \begin{pmatrix}
			-8 & 2 & -9\\
			-5 & -3 & -10 \\
			9 & -9 & -1
			\end{pmatrix}
		\end{equation*}
		\begin{align*}
			v_1 &= 0.1 \\
			v_2 &= -0.2 \\
			v_3 &= 0.2
		\end{align*}
	\end{minipage}

	\begin{minipage}{0.75\textwidth}
		\begin{figure}[H]
			\centering
			\includegraphics[width = 4in]{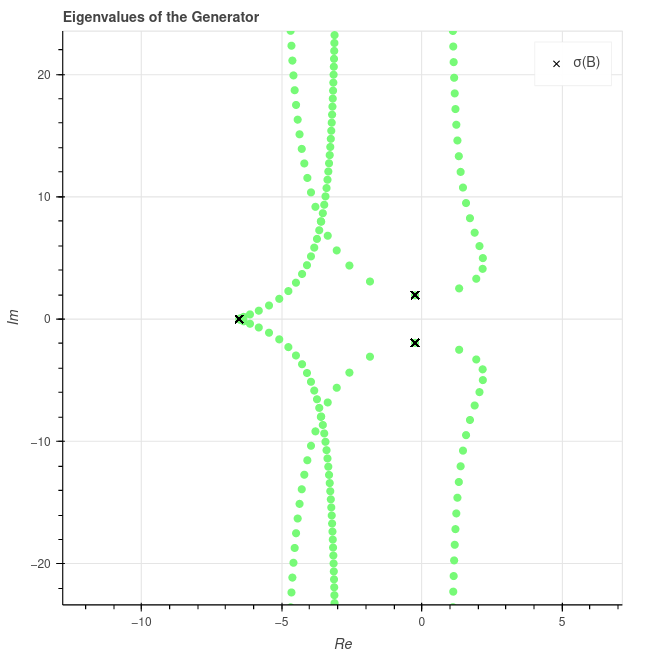}
			\caption[Turing Patterns for $N=3$. Example 3]{An example of the spectrum of $-A+B$ in the case of Turing patterns due to \Cref{corollary_existence_turing_patterns}.}
			\label{fig:turing_corollary_N3}
		\end{figure}
	\end{minipage} \hfill
	\begin{minipage}{0.2\textwidth}
		Parameter choices: \newline
		\begin{equation*}
		B = \begin{pmatrix}
		-3 & 2 & -4\\
		-5 & -5 & 2 \\
		-5 & -5 & 1
		\end{pmatrix}
		\end{equation*}
		\begin{align*}
		v_1 &= -0.1 \\
		v_2 &= -0.2 \\
		v_3 &= 0.2
		\end{align*}
	\end{minipage}

	\begin{minipage}{0.75\textwidth}
		\begin{figure}[H]
			\centering
			\includegraphics[width = 4in]{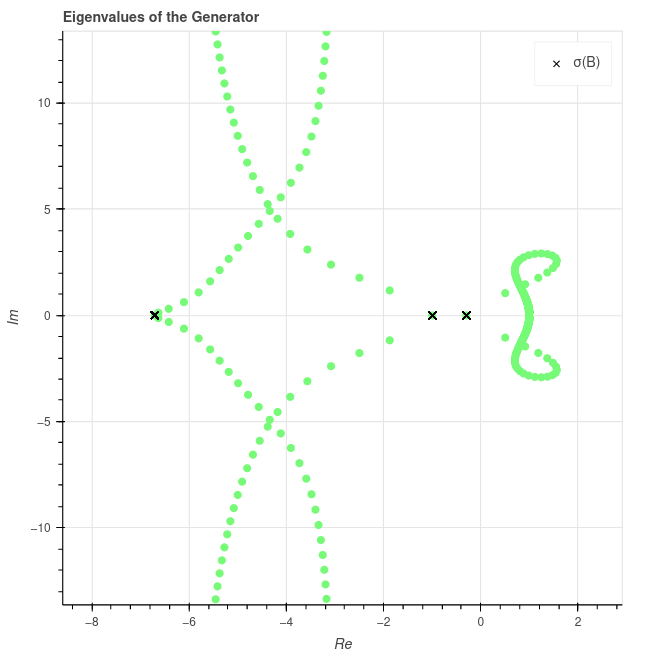}
			\caption[Turing Patterns for $N=3$. Example 4]{An example of the spectrum of $-A+B$ in the case of Turing patterns for $N=3$ with only two moving components.}
			\label{fig:turing_N3_not_moving}
		\end{figure}
	\end{minipage} \hfill
	\begin{minipage}{0.2\textwidth}
		Parameter choices: \newline
		\begin{equation*}
			B = \begin{pmatrix}
			-3 & 2 & -4\\
			-5 & -6 & 2 \\
			-5 & -5 & 1
			\end{pmatrix}
		\end{equation*}
		\begin{align*}
			v_1 &= -0.1 \\
			v_2 &= 0.1 \\
			v_3 &= 0
		\end{align*}
	\end{minipage}

	\subsection{A Simulation for $N=3$}
	
	As already mentioned, the solution formula \eqref{auxeq5} was used for the simulations. We cut off the frequencies at $K = 100,$ i.e. we considered an initial function $u_0$ with
	\begin{equation*}
		\hat{u}_0(k) = 0
	\end{equation*}
	for all $|k| > 100.$ Then, the formula
	\begin{equation*}
		u(t,x) = \sum_{k = - 100}^{100} e^{tM(\indextwo)} \hat{u}_0(\indextwo)
		e^{2\pi i \indextwo x}
	\end{equation*}
	for the solution to \eqref{transport_reaction_eq} with initial function $u_0$ is exact. Regarding the initial function, we perturbed each component $j=1,2,3$ of the constant equilibrium by sampling the Fourier coefficients as random normal variables with
	\begin{equation*}
		\hat{u}_{0,j}(k) \sim \mathcal{N}\left(0, 10^{-4}(k+1)^{-2} \right)
	\end{equation*}
	for $k=0,\cdots,100$ and setting $\hat{u}_{0,j}(-k) = \overline{\hat{u}_{0,j}(k)}$. The computation of the matrix exponential in the solution formula was implemented with \texttt{scipy.linalg.expm} from the SciPy library. \newline
	Concerning the $3$D-plots below, an exponential scaling of the solution is necessary in order to get a reasonable visualization of the time evolution. In both our examples, we plotted the rescaled solution
	\begin{equation}\label{expo_scaling_sol}
		\widetilde{u}(t,x) \coloneqq e^{-\frac{3}{2}t} u(t,x).
	\end{equation}
	Given a specific model, the appropriate rate of the rescaling depends on the maximum time $t_{\text{max}}$ of the simulation and on the growth rate of the solution - the latter can be read off from a plot of the spectrum of the generator. \newline

	The simulation of Turing patterns was performed in the exemplary case of
	\begin{equation}\label{turing_example}
		\partial_t \colvec{u_1 \\ u_2 \\ u_3} + \begin{pmatrix}
		0.1 & 0 & 0 \\
		0 & -0.2 & 0 \\
		0 & 0 & 0.2
		\end{pmatrix} \partial_x \colvec{u_1 \\ u_2 \\ u_3} = \begin{pmatrix}
		-8 & 2 & -9\\
		-5 & -3 & -10 \\
		9 & -9 & -1
		\end{pmatrix} \colvec{u_1 \\ u_2 \\ u_3}.
	\end{equation}
	This parameter choice coincides with the one from \Cref{fig:turing2_N3}. In particular, the spectrum of the generator of the transport-reaction semigroup, see \Cref{fig:turing2_N3}, indicates that Turing patterns emerge. It is probably not easily apparent from \Cref{fig:turing2_N3} but the wave numbers $k=\pm 4$ maximize $\Sigma(k),$ with $\Sigma(k)$ being defined in \eqref{defSigma}. The matrices $M(k)$ are accordingly given by
	\begin{equation*}
		M(k) = - 2\pi i k \begin{pmatrix}
		0.1 & 0 & 0 \\
		0 & -0.2 & 0 \\
		0 & 0 & 0.2
		\end{pmatrix} + \begin{pmatrix}
		-8 & 2 & -9\\
		-5 & -3 & -10 \\
		9 & -9 & -1
		\end{pmatrix},
	\end{equation*}
	see \eqref{Mk_case_N3}. Therefore, the number of peaks of the solution of \eqref{turing_example} for large times should be equal to $4$. Our simulation verifies these theoretical considerations.
	
	\begin{figure}[H]
		\centering
		\includegraphics[width=\textwidth]{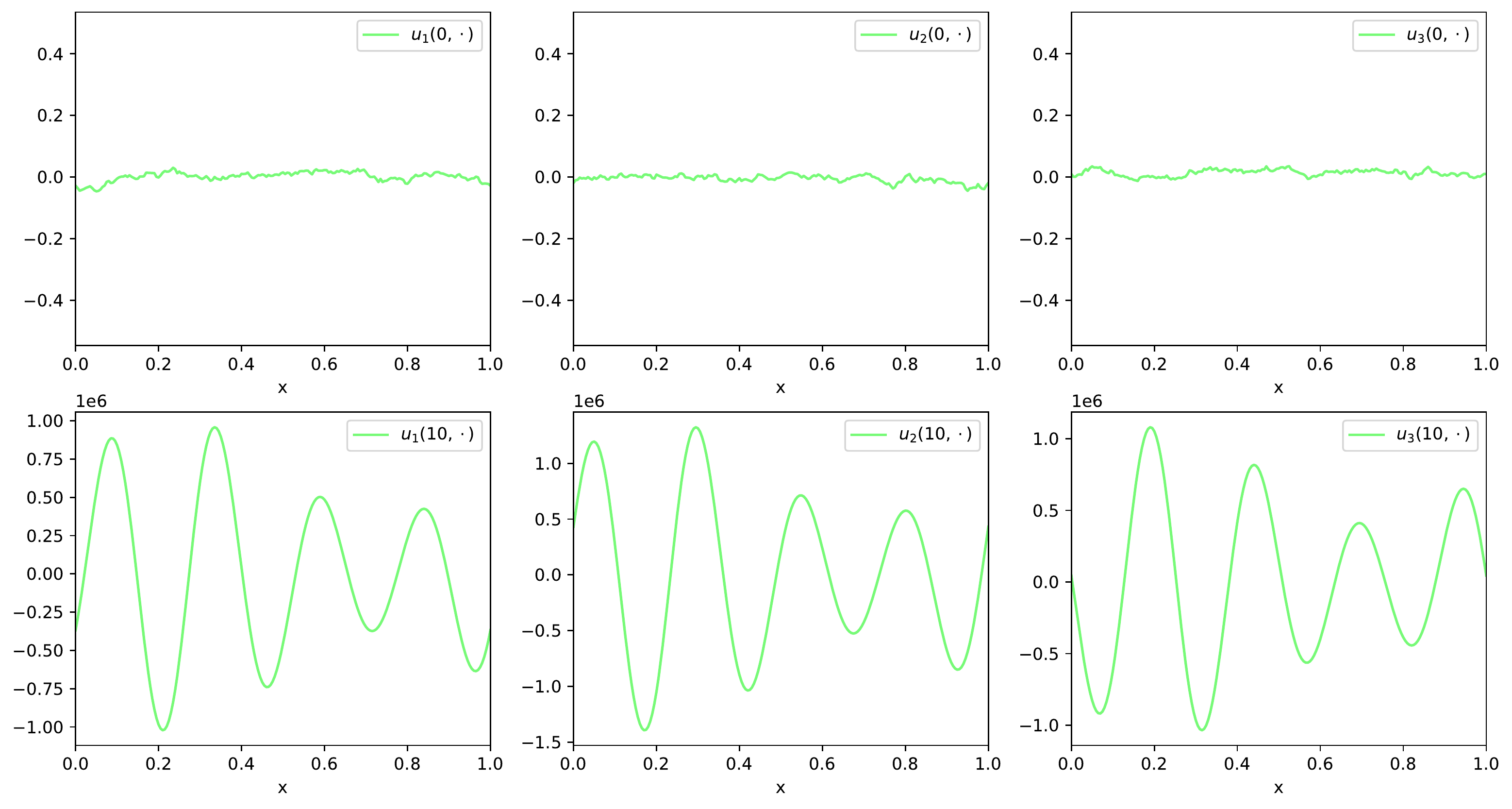}
		\caption[A Simulation of Turing Patterns for $N=3$.]{A Simulation of Turing pattern formation in the exemplary case \eqref{turing_example}. The graphs in the upper row are plots of the three components of the initial function $u_0$. The lower row shows plots of each component of the solution at the time $t=10.$}
		\label{fig:simulation_turing_graphs}
	\end{figure}
	
	\begin{figure}[H]
		\subfloat[The time evolution of $u_1$]{\includegraphics[width = .5\textwidth]{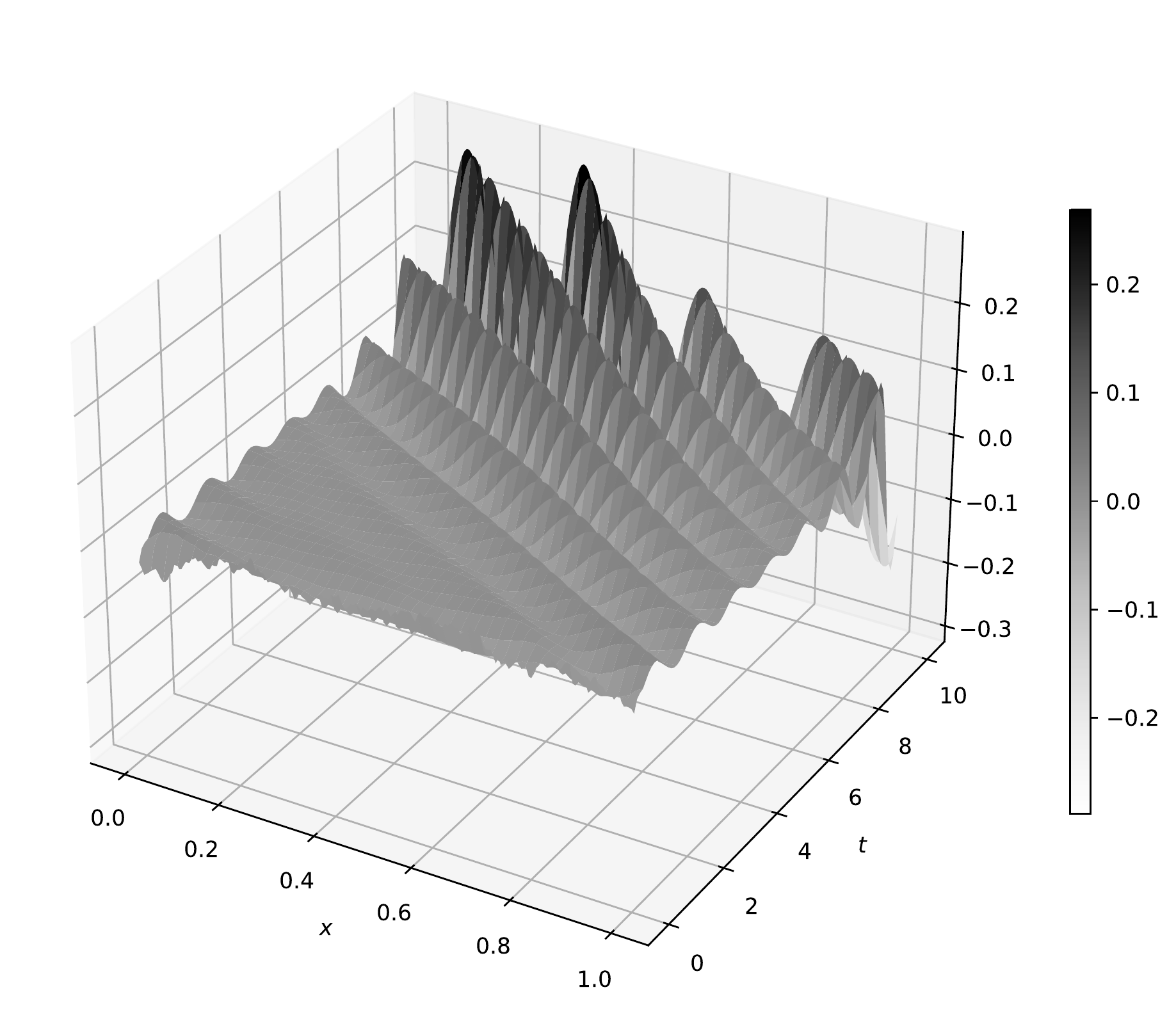}} 
		\subfloat[The time evolution of $u_2$]{\includegraphics[width = .5\textwidth]{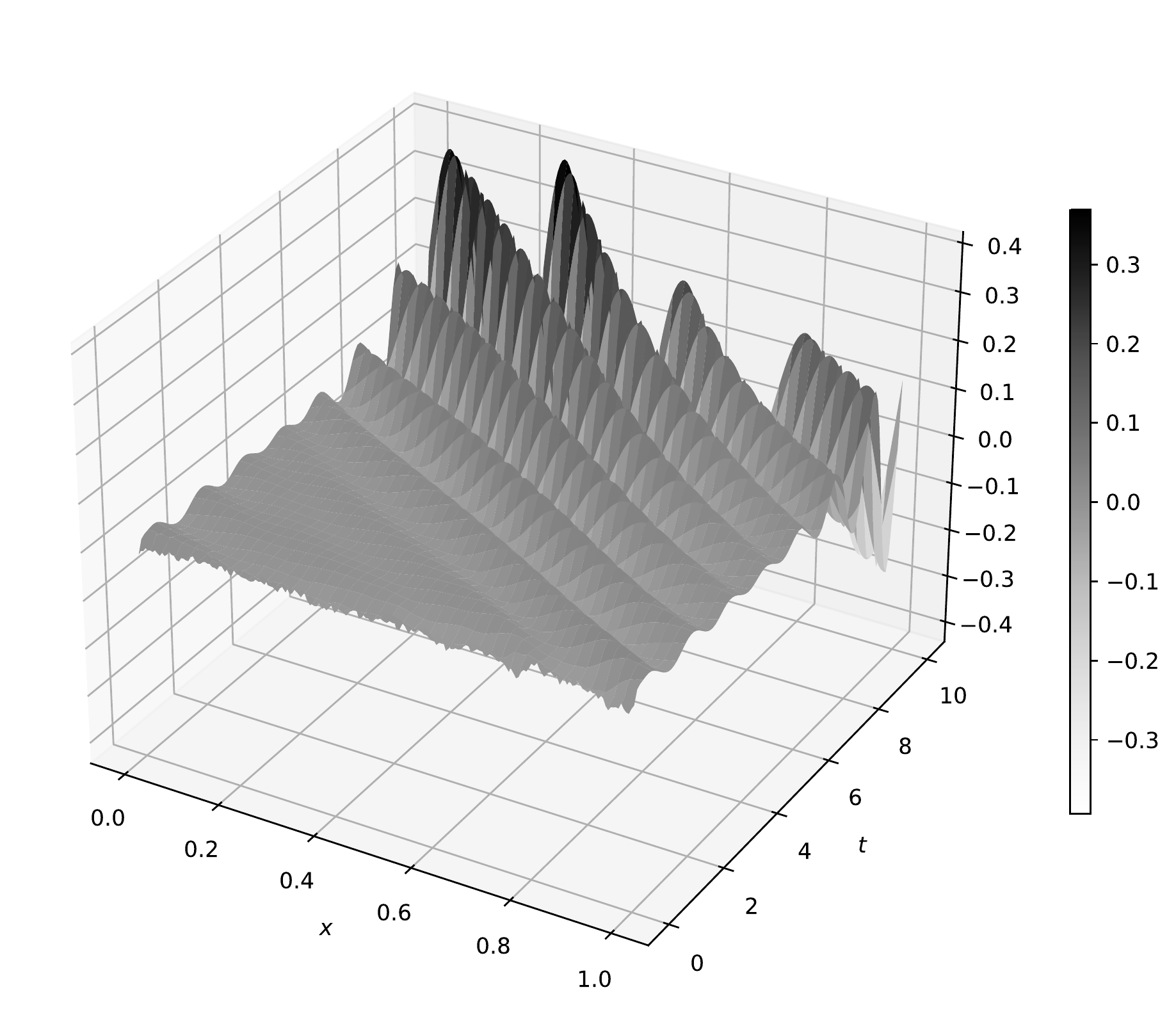}} \\
		\subfloat[The time evolution of $u_3$]{\includegraphics[width = .5\textwidth]{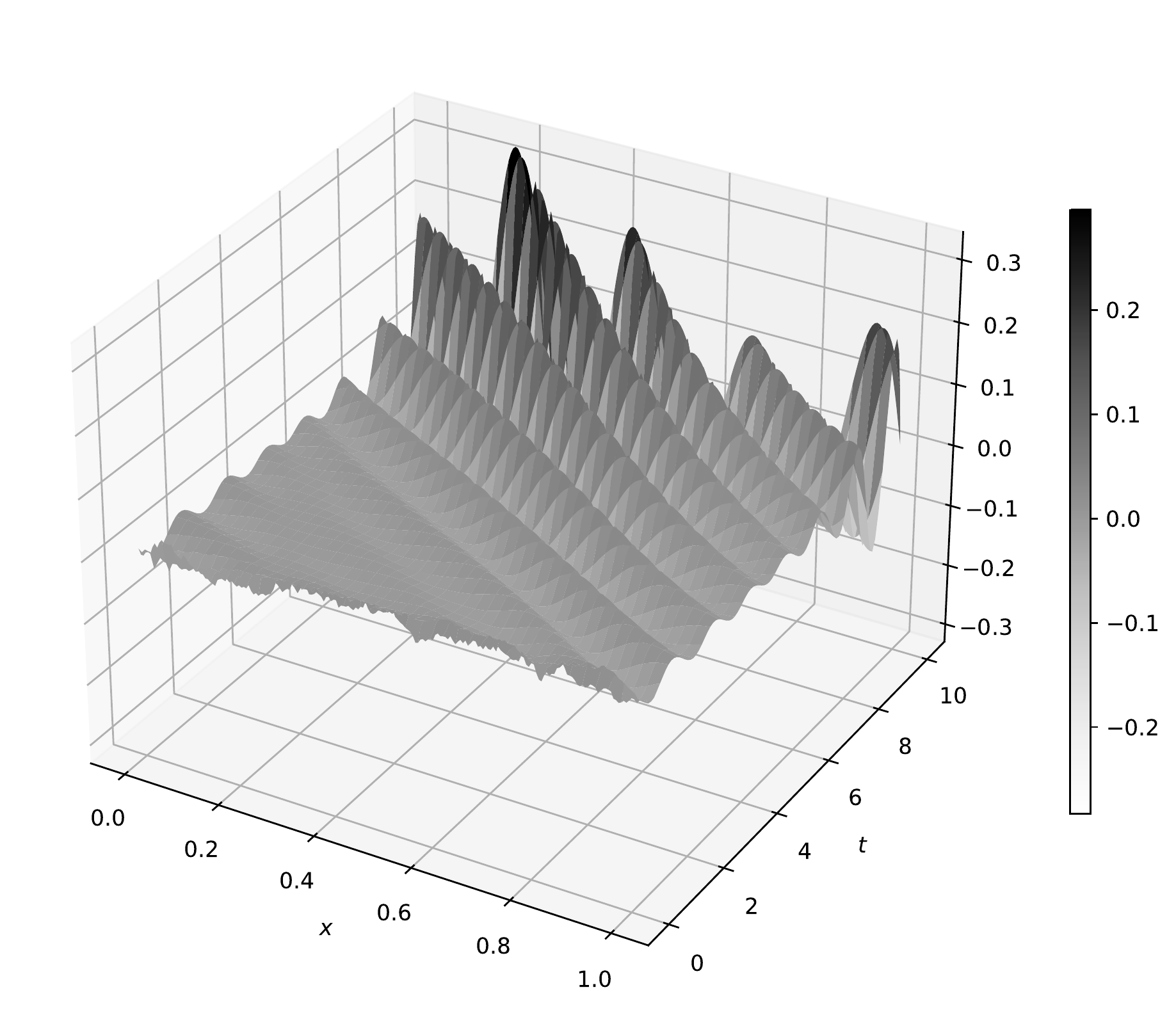}}
		\caption[A Simulation of the Time Evolution of Turing Patterns]{The emerging of Turing patterns in the exemplary case \eqref{turing_example}. Notice that the values of $u$ are scaled, see \eqref{expo_scaling_sol}.}
		\label{time_evo_turing_pattern}
	\end{figure}

	The simulation of hyperbolic instabilities was performed in the exemplary case of
	\begin{equation}\label{hyperbolic_instab_example}
		\partial_t \colvec{u_1 \\ u_2 \\ u_3} + \begin{pmatrix}
		0.1 & 0 & 0 \\
		0 & -0.2 & 0 \\
		0 & 0 & 0.5
		\end{pmatrix} \partial_x \colvec{u_1 \\ u_2 \\ u_3} = \begin{pmatrix}
		-6 & 2 & -9\\
		4 & -10 & -5 \\
		8 & 10 & 2
		\end{pmatrix} \colvec{u_1 \\ u_2 \\ u_3}.
	\end{equation}
	This parameter choice coincides with the one from \Cref{fig:hyperbolic_instab_N3}. In particular, the spectrum of the generator of the transport-reaction semigroup, see \Cref{fig:hyperbolic_instab_N3}, indicates that hyperbolic instabilities emerge. Therefore, the number of peaks of the solution to \eqref{hyperbolic_instab_example} should increase and chaotic, increasingly oscillating behavior should emerge. Indeed, the simulation verifies our theoretical considerations.
	
	\begin{figure}[H]
		\centering
		\includegraphics[width=\textwidth]{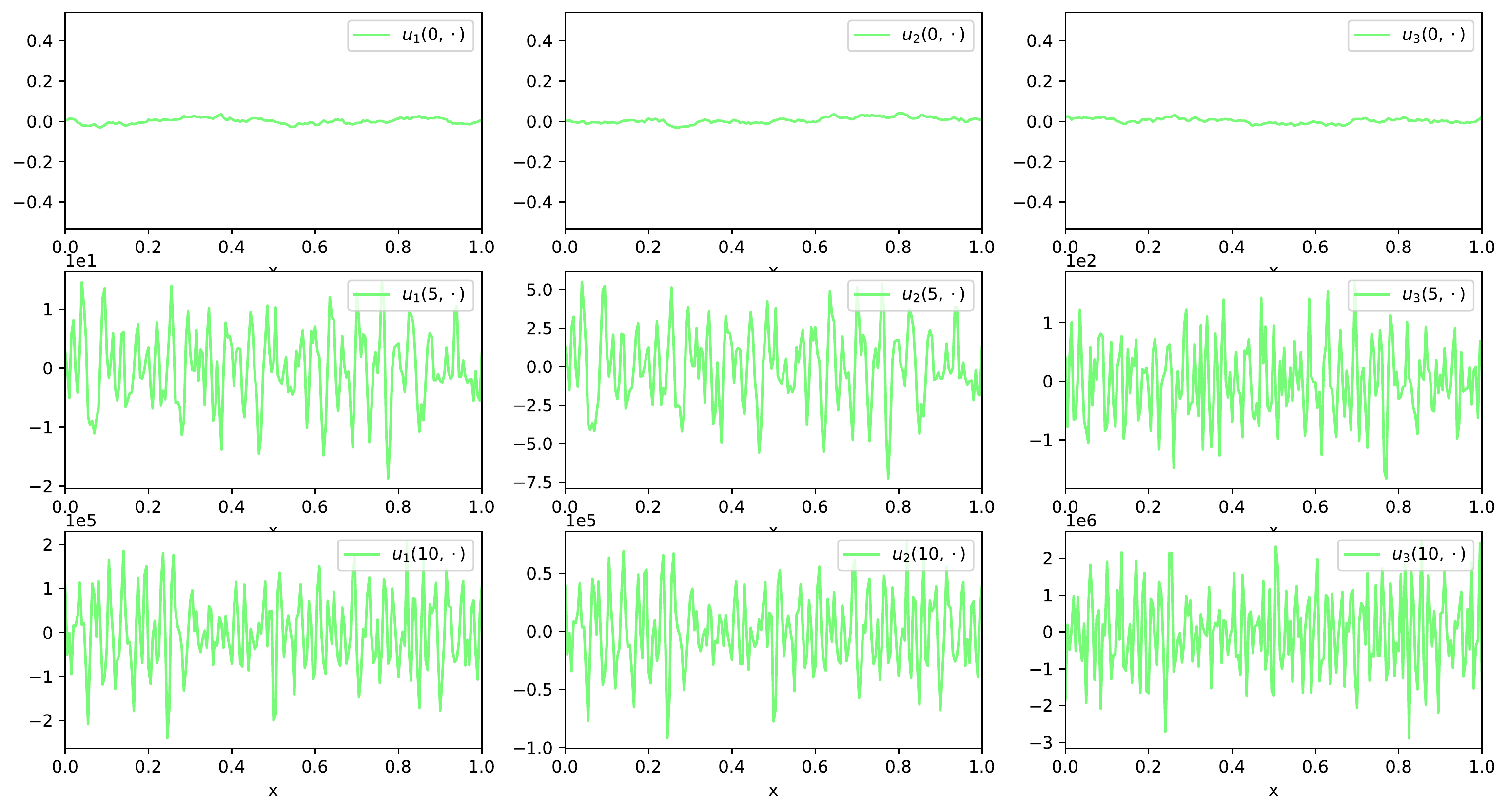}
		\caption[A Simulation of Hyperbolic Instabilities for $N=3$.]{A Simulation of hyperbolic instabilities in the exemplary case \eqref{hyperbolic_instab_example}. The graphs in the upper row are plots of the three components of the initial function $u_0$. The lower rows show plots of each component of the solution at the times $t=5,10.$}
		\label{fig:simulation_hyperbolic_instability}
	\end{figure}
	
	\begin{figure}[H]
		\subfloat[The time evolution of $u_1$]{\includegraphics[width = .5\textwidth]{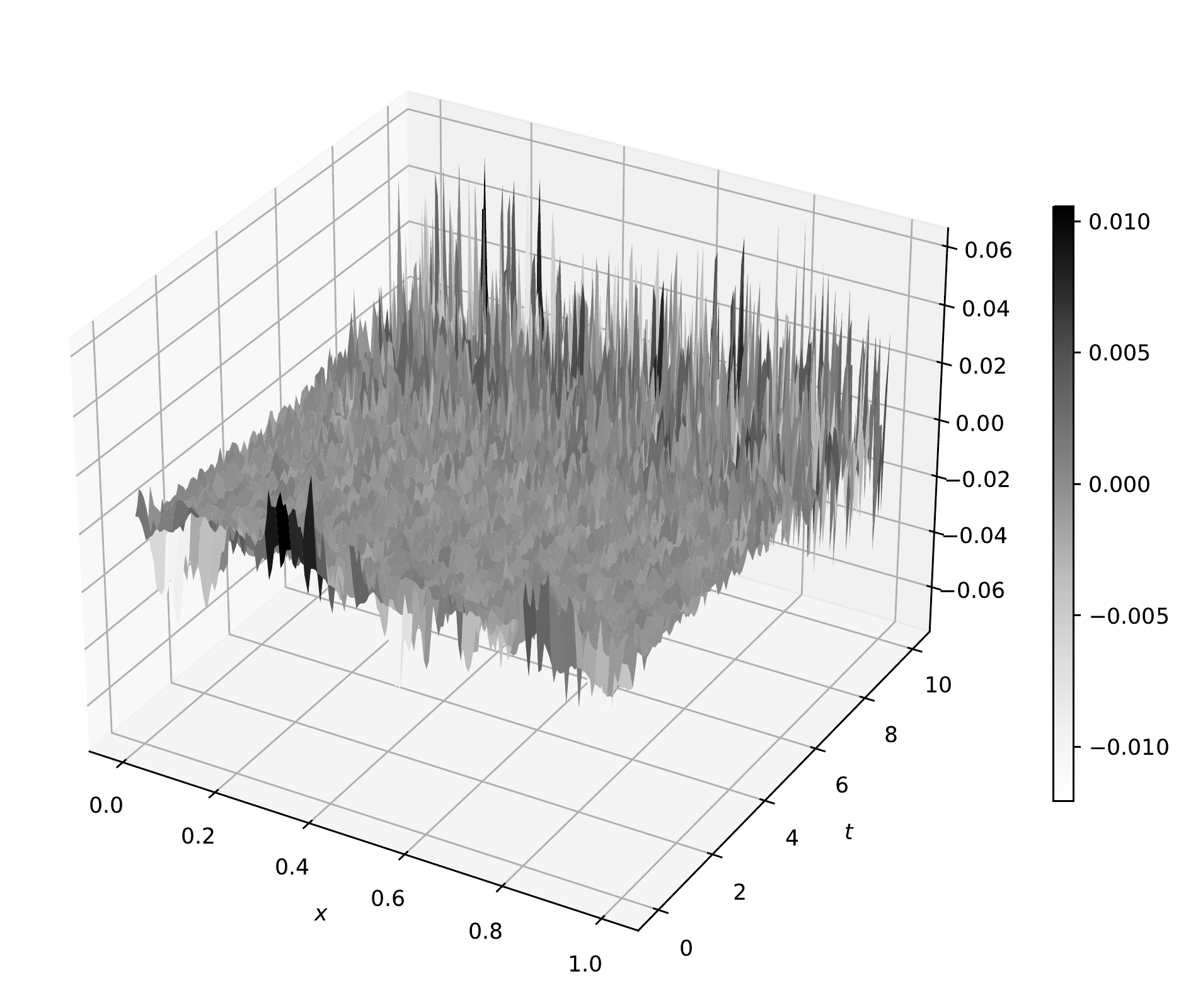}} 
		\subfloat[The time evolution of $u_2$]{\includegraphics[width = .5\textwidth]{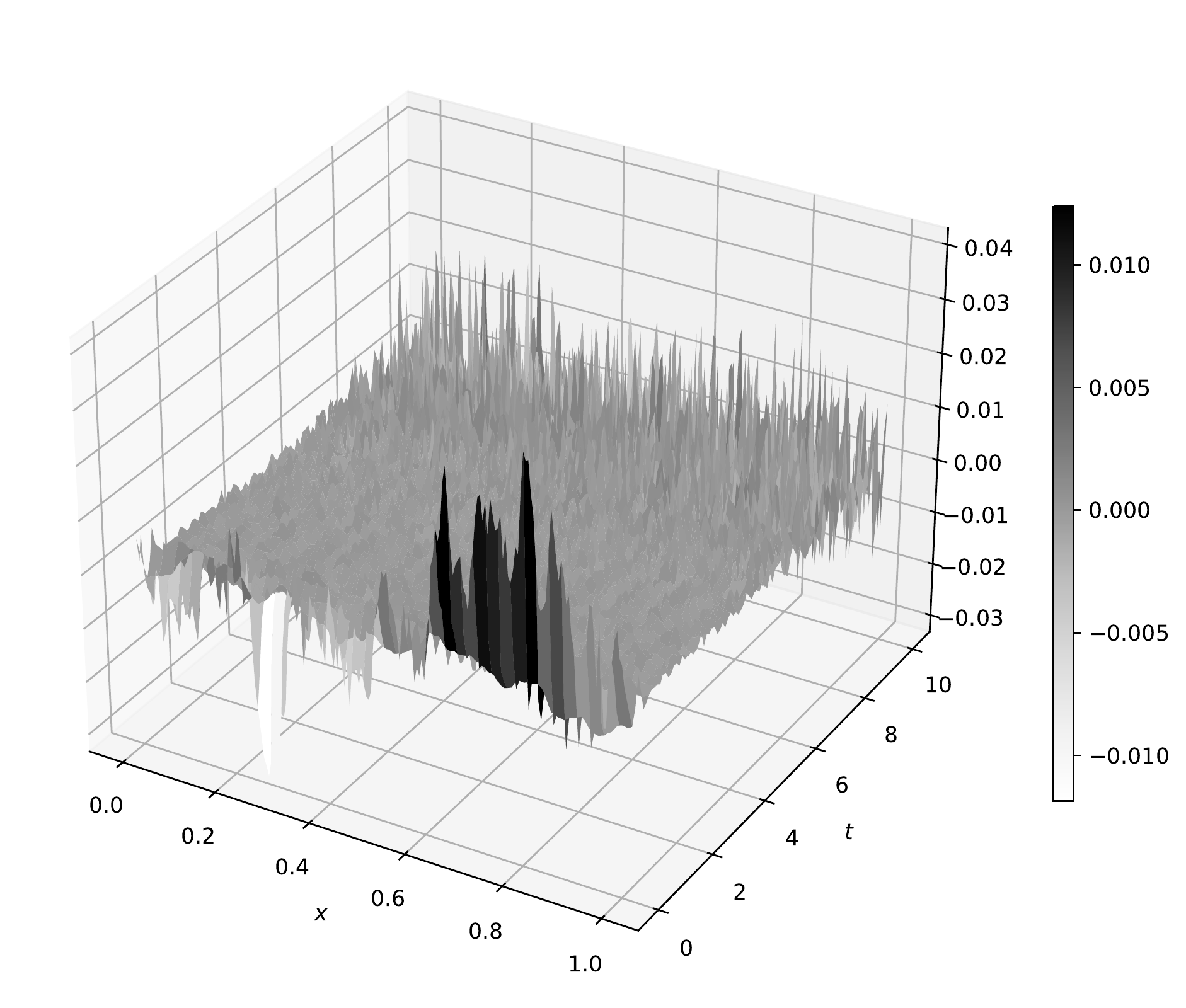}} \\
		\subfloat[The time evolution of $u_3$]{\includegraphics[width = .5\textwidth]{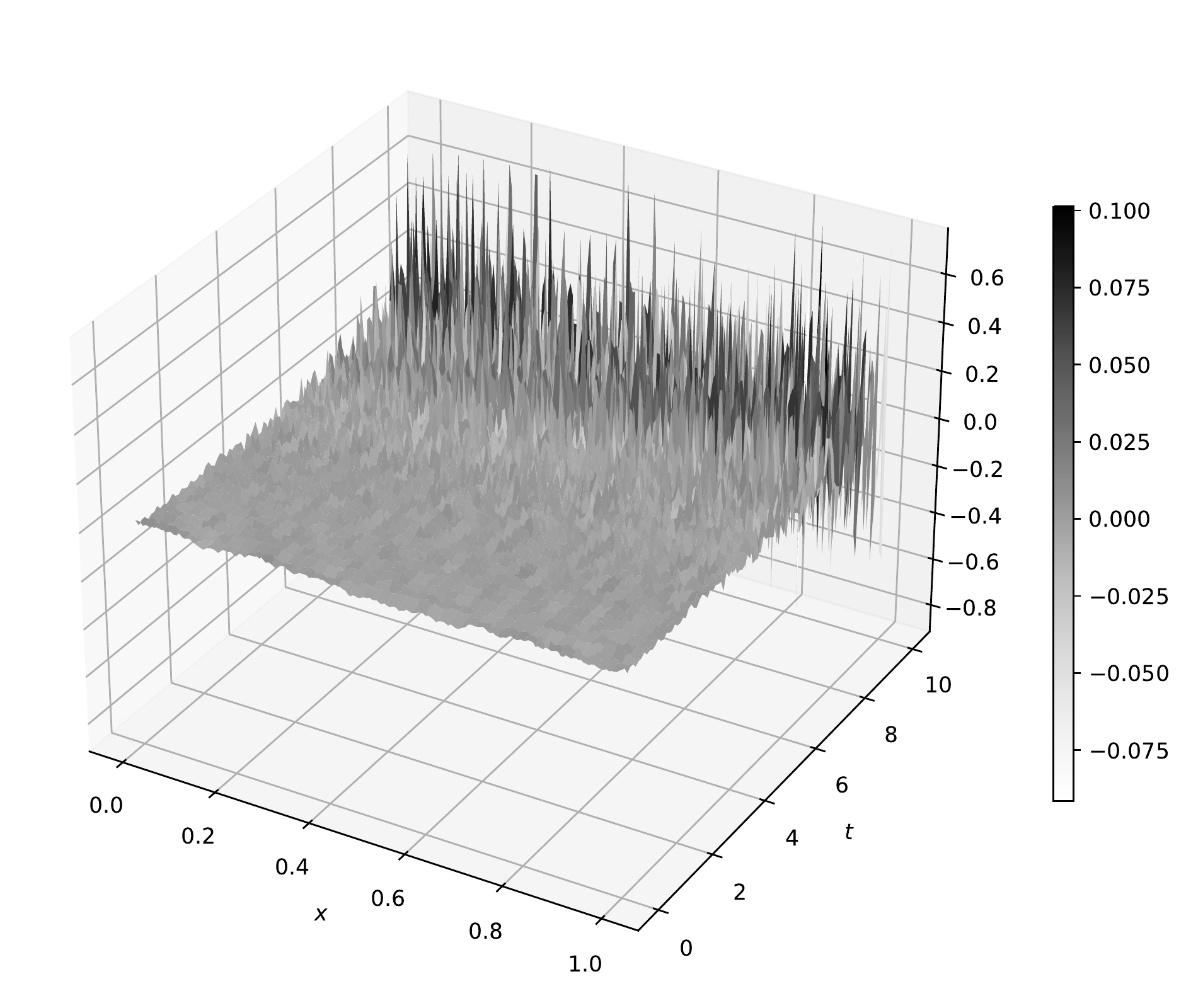}}
		\caption[A Simulation of the Time Evolution of Hyperbolic Instabilities]{The emerging of hyperbolic instabilities in the exemplary case \eqref{hyperbolic_instab_example}. Notice that the values of $u$ are scaled, see \eqref{expo_scaling_sol}.}
		\label{time_evo_hyperbolic_instab}
	\end{figure}

	\begin{appendices}
		
		\vspace*{2cm}
		The purpose of these detailed appendices is to strongly facilitate reading the
		thesis and to write the paper as self-contained as possible.\newline
		We decided to present the most important results from semigroup and spectral
		theory we applied, including standard references and comments for readers
		unfamiliar with semigroup theory. \newline
		
		Our summaries of semigroup and spectral theory are based on the excellent books
		from A. Pazy \cite{pazy2012semigroups} and K.-J. Engel and R. Nagel
		\cite{engel2001one}.
		
		\chapter{Postponed Proofs}\label{appendix_postponed_proofs}
		\renewcommand{\thechapter}{\arabic{chapter}}
		\setcounter{chapter}{2}
		\setcounter{theorem}{1}
		\begin{lemma}
			Let $u \in \Lp(\T^d)$ and let $k=1,\cdots,d.$
			\begin{enumerate}
				\item
				Let $D_\indextwo^h u \rightharpoonup v$ in $\Lp(\T^d)$. Then the
				$\indextwo$-th weak derivative exists in $\Lp(\T^d)$ and $\partial_\indextwo u =
				v.$
				\item
				Conversely, if $u \in \Wkp{1}{p}$, it follows that $D_\indextwo^h u \to
				\partial_\indextwo u$ in $\Lp(\T^d)$ for $p>1$ and $D_\indextwo^h u
				\rightharpoonup \partial_\indextwo u$ for $p=1$.
			\end{enumerate}
			All convergences are with respect to $h \to 0.$
		\end{lemma}
		\renewcommand{\thechapter}{\Alph{chapter}}
		\setcounter{chapter}{1}
		\begin{proof}[Proof of
			\Cref{lemma_convergence_unit_difference_quotients}]\label{proof_lemma_diff_quot}
			\begin{enumerate}
				\item
				Let $\varphi \in C_c^\infty(\T^d)$ be a test function. A translation yields
				\begin{equation*}
				\intT D_\indextwo^h u (\mathbf{x}) \varphi(\mathbf{x}) \, d\mathbf{x} =
				-\intT u(\mathbf{x}) D_\indextwo^{-h}\varphi(\mathbf{x}) \, d\mathbf{x}
				\end{equation*}
				and by assumption, the left-hand side converges to
				\begin{equation*}
				\intT v(\mathbf{x}) \varphi(\mathbf{x}) \, d\mathbf{x}.
				\end{equation*}
				On the other hand $\varphi$ is smooth, so the difference quotient converges
				and dominated convergence implies
				\begin{equation*}
				\intT v(\mathbf{x}) \varphi(\mathbf{x}) \, d\mathbf{x} =  - \intT
				u(\mathbf{x}) \partial_\indextwo \varphi(\mathbf{x}) \, d\mathbf{x},
				\end{equation*}
				i.e. the $\indextwo$-th weak derivative of $u$ exists and is given by $v$.
				
				\item
				The converse direction is a little bit more involved and we first prove a
				uniform $\Lp$ bound of the difference quotients for $1 \leq p < \infty$. To this
				end, let us firstly consider $u \in C^1(\T^d).$ By the fundamental theorem of
				calculus
				\begin{equation*}
				\left| \frac{u(\mathbf{x} + h\mathbf{e_\indextwo}) - u(\mathbf{x})
				}{h}\right| \leq \int_{0}^{1} \left|\partial_\indextwo u(\mathbf{x} +
				sh\mathbf{e_\indextwo}) \right| \, ds
				\end{equation*}
				holds for arbitrary $\mathbf{x} \in \T^d.$ We now integrate over the torus,
				apply Jensen's inequality and use Fubini to obtain
				\begin{align*}
				\intT | D_\indextwo^h u (\mathbf{x})|^p \, d\mathbf{x} &\leq \intT \left(
				\int_0^1 |\partial_\indextwo u(\mathbf{x} + sh \mathbf{e_\indextwo})|\, ds
				\right)^p d \mathbf{x} \leq \intT \int_0^1 |\partial_\indextwo u(\mathbf{x} + sh
				\mathbf{e_\indextwo})|^p \, ds d\mathbf{x} \\
				&\leq \int_0^1 \intT |\partial_\indextwo u(\mathbf{x} + sh
				\mathbf{e_\indextwo})|^p \, d\mathbf{x} ds = \pnorm{p}{\partial_\indextwo u}^p. 
				\end{align*}
				For $u \in \Wkp{1}{p},$ the bound follows by the density of $C^1(\T^d)$ with
				respect to $\|\cdot\|_{W^{1,p}(\T^d)}.$ Notice that the difference quotients of
				an approximation sequence $(u^{(n)})_{n\in\N}$ converge for fixed $h$ as
				\begin{equation*}
				\pnorm{p}{D_\indextwo^h u - D_\indextwo^h u^{(n)}} \leq 2 |h|^{-1}
				\pnorm{p}{u - u^{(n)}} \to 0 \qquad \text{ as } n\to\infty.
				\end{equation*}
				This shows
				\begin{equation}\label{auxeq31}
				\|D_\indextwo^h u\|_{L^p(\T^d)} \leq \|\partial_\indextwo u \|_{L^p(\T^d)}
				\end{equation}
				for all $W^{1,p}(\T^d).$ \newline
				In the following, let $1<p<\infty.$ Since $\Lp(\T^d)$ is reflexive in this
				case, we can choose a subsequence (without relabeling) such that $D_\indextwo^h
				u \rightharpoonup v$ in $\Lp(\T^d).$ It follows by (i) that $v =
				\partial_\indextwo u.$ Notice that starting with a subsequence of
				$(D_\indextwo^h u)_h$ implies that every subsequence has a subsequence which
				converges weakly to the same limit in $\Lp(\T^d),$ namely $\partial_\indextwo
				u$. This implies weak convergence for $p>1$. Since \eqref{auxeq31} implies
				\begin{equation*}
				\limsup_{h \to 0} \pnorm{p}{D_\indextwo^h u} \leq
				\pnorm{p}{\partial_\indextwo u},
				\end{equation*}
				strong convergence follows from \cite[Proposition 3.32,
				p. 78]{brezis2010functional} as $\Lp(\T^d)$ is uniformly convex for $p>1.$
				\newline
				Unfortunately, $L^1(\T^d)$ is not reflexive. However, the Dunford-Pettis
				theorem, see \cite[Theorem 4.30, p. 115]{brezis2010functional}, characterizes
				weakly compact sets and it suffices to show that $(D_\indextwo^hu)_h$ is
				equiintegrable. The same subsubsequence argument as before then implies weak
				convergence of the whole sequence. Let $A \in \mathcal{B}(\T^d)$ and let $u \in
				C^1(\T^d).$ Note that we also have
				\begin{equation*}
				\int_A |D_\indextwo^h u(\mathbf{x})| \, d\mathbf{x} \leq \int_0^1 \int_A
				|\partial_\indextwo u(\mathbf{x}+sh\mathbf{e_\indextwo})| \, d\mathbf{x} ds =
				\int_0^1 \int_{A_{s,h}} |\partial_\indextwo u (\mathbf{x})| \, d\mathbf{x} ds
				\end{equation*}
				for arbitrary measurable $A$ and with $A_{s,h}$ being the translated set.
				Now, consider an arbitrary function $u \in \Wkp{1}{1}$ and take a sequence
				$(u^{(n)})_{n\in\N} \subset C^1(\T^d)$ such that $u^{(n)} \to u$ in
				$\Wkp{1}{1}$. In particular, $(\partial_\indextwo u^{(n)})_{n\in\N}$ is
				uniformly integrable and $D_\indextwo^h u^{(n)} \to D_\indextwo^h u$ in
				$L^1(\T^d)$ for every fixed $h$. Let $\varepsilon>0.$ By the uniform
				integrability of the approximation sequence, there exists $\delta > 0$ such that
				\begin{equation*}
				\mathcal{L}^d(B) \leq \delta \quad \implies \quad \sup_{n \in \N} \int_B
				|\partial_\indextwo u^{(n)}| \, d\mathbf{x} \leq \varepsilon
				\end{equation*}
				for all measurable sets $B\in\mathcal{B}(\T^d).$ Now take $A \in
				\mathcal{B}(\T^d)$ with $\mathcal{L}^d(A) \leq \delta$ and fix $h>0.$ It holds
				\begin{equation*}
				\int_A |D_\indextwo^h u(\mathbf{x})| \, d\mathbf{x} = \lim_{n\to\infty}
				\int_A |D_\indextwo^h u^{(n)}(\mathbf{x})| \, d\mathbf{x} \leq
				\limsup_{n\to\infty} \int_0^1 \int_{A_{s,h}} |\partial_\indextwo
				u^{(n)}(\mathbf{x})| \, d\mathbf{x} ds \leq \varepsilon
				\end{equation*}
				due to $\mathcal{L}^d(A_{s,h}) = \mathcal{L}^d(A).$ This estimate is uniform
				in $h$ and hence, it holds
				\begin{equation*}
				\sup_{h\neq 0} \int_A |D_\indextwo^h u| \, d\mathbf{x} \leq \varepsilon,
				\end{equation*}
				i.e. $(D_\indextwo^h u)_h$ is equiintegrable.
			\end{enumerate}
		\end{proof}

		\renewcommand{\thechapter}{\arabic{chapter}}
		\setcounter{chapter}{2}
		\setcounter{theorem}{14}
		\begin{lemma}
			Let $\mathbf{v_1}, \cdots, \mathbf{v_N} \in \R^d$ be non-vanishing, i.e.
			$\mathbf{v_1,} \cdots, \mathbf{v_N} \neq 0.$ Moreover, let $\indexone = 1,
			\cdots, N.$ The following assertions are equivalent:
			\begin{enumerate}
				\item
				Component $j$ is transport periodic.
				\item
				$\mathbf{v_{\indexone,\indextwo}} \in \mathbf{v_{\indexone,\indextwo^*}} \Q$
				for all $\indextwo=1,\cdots,d$ and one (and hence all) $\indextwo^*$ with
				$\mathbf{v_{\indexone,\indextwo^*}} \neq 0.$
				\item
				$ \displaystyle\bigcap^d_{\substack{\indextwo=1 \\
						\mathbf{v_{\indexone,\indextwo}} \neq 0}} \mathbf{v_{\indexone,\indextwo}} \Z
				\neq \{0\}.$
			\end{enumerate}
		\end{lemma}
		\begin{proof}[Proof of
			\Cref{lemma_component_transport_periodic_characterization}]\label{proof_lemma_periodicity}
			$(i) \implies (ii)$ If component $\indexone$ is transport periodic, there
			exists $t^*>0$ with $t^*\mathbf{v_\indexone} = (\indexthree_1, \cdots,
			\indexthree_d)^T \in \Z^d.$ Then
			\begin{equation*}
			\mathbf{v_{\indexone,\indextwo}} = \mathbf{v_{\indexone,\indextwo^*}}
			\frac{\indexthree_\indextwo}{\indexthree_{\indextwo^*}} \in
			\mathbf{v_{\indexone,\indextwo^*}} \Q
			\end{equation*}
			for all $\indextwo=1,\cdots,d$ and all $\indextwo^*$ with
			$\mathbf{v_{\indexone,\indextwo^*}} \neq 0.$ \newline
			
			$(ii) \implies (iii)$ By assumption, there exists $\indextwo^*$ with
			$\mathbf{v_{\indexone,\indextwo^*}} \neq 0$ and
			\begin{equation*}
			\mathbf{v_{\indexone,\indextwo}} = \mathbf{v_{\indexone,\indextwo^*}}
			\frac{p_\indextwo}{q_\indextwo}
			\end{equation*}
			for all $\indextwo=1,\cdots,d$ and some $p_\indextwo\in\Z, \, q_\indextwo \in
			\N.$ Take an arbitrary $\indextwo$ with $\mathbf{v_{\indexone,\indextwo}} \neq
			0.$ It follows that
			\begin{equation*}
			0 \neq \mathbf{v_{\indexone,\indextwo^*}} \prod_{\substack{\indexthree=1 \\
					p_\indexthree \neq 0}}^{d} p_\indexthree = \mathbf{v_{\indexone,\indextwo}}
			\bigg(q_\indextwo \prod_{\substack{\indexthree\neq \indextwo \\ p_\indexthree
					\neq 0}} p_\indexthree\bigg) \in \mathbf{v_{\indexone,\indextwo}} \Z.
			\end{equation*}
			
			$(iii) \implies (ii) \implies (i)$ Condition $(iii)$ implies that there exists
			$\lambda \neq 0$ and $p_\indextwo \in \Z$ such that
			\begin{equation*}
			\lambda = \mathbf{v_{\indexone,\indextwo}} p_\indextwo
			\end{equation*}
			for all $\indextwo$ with $\mathbf{v_{\indexone,\indextwo}} \neq 0.$ In this
			case
			\begin{equation*}
			\mathbf{v_{\indexone,\indextwo}} = \frac{\lambda}{p_\indextwo} =
			\mathbf{v_{\indexone,\indextwo^*}} \frac{p_{\indextwo^*}}{p_\indextwo} \in
			\mathbf{v_{\indexone,\indextwo^*}}\Q
			\end{equation*}
			for one (and hence all) $\indextwo^*$ with $\mathbf{v_{\indexone,\indextwo^*}}
			\neq 0.$ For
			\begin{equation*}
			t\coloneqq \frac{1}{\mathbf{v_{\indexone,\indextwo^*}}}
			\prod_{\substack{\indexthree=1 \\ \mathbf{v_\indexone^\indexthree \neq 0}}}^d
			p_\indexthree > 0
			\end{equation*}
			we obtain
			\begin{equation*}
			t \mathbf{v_{\indexone,\indextwo}} = t \mathbf{v_{\indexone,\indextwo^*}}
			\frac{p_{\indextwo^*}}{p_\indextwo} = p_{\indextwo^*}
			\prod_{\substack{\indexthree\neq \indextwo \\ \mathbf{v_\indexone^\indexthree}
					\neq 0}} p_\indexthree \in \Z
			\end{equation*}
			for all $\indextwo$ with $\mathbf{v_{\indexone,\indextwo}} \neq0.$ For the
			other $\indextwo,$ we trivially have $t\mathbf{v_{\indexone,\indextwo}} = 0 \in
			\Z.$
		\end{proof}
		\renewcommand{\thechapter}{\Alph{chapter}}
		\setcounter{chapter}{1}
		
		\renewcommand{\thechapter}{\arabic{chapter}}
		\setcounter{chapter}{7}
		\setcounter{theorem}{5}
		\begin{lemma}
			Consider the setup from \Cref{theorem_perturbed_eigenvalues}. For $j=1,\cdots,N$ and $n=1,2,3$, the coefficients $\hat{\lambda}_\indexone^{(n)} \in \R$ are given by
			\begin{align*}
				\hat{\lambda}_\indexone^{(1)} &= b_{\indexone\indexone}, \\
				\hat{\lambda}_\indexone^{(2)} &= - \sum^N_{\substack{\indexfour=1 \\ \indexfour \neq \indexone}} (v_\indexfour - v_\indexone)^{-1} b_{\indexone \indexfour} b_{\indexfour \indexone}, \\
				\hat{\lambda}_\indexone^{(3)} &= \sum_{\substack{i,l=1 \\ i,l\neq j}}^{N} (v_l-v_j)^{-1}(v_i-v_j)^{-1}b_{jl}b_{li}b_{ij} - \sum_{\substack{i=1 \\ i\neq j}}^{N} (v_i-v_j)^{-2}b_{ij}b_{ji}b_{jj}. \\
			\end{align*}
		\end{lemma}
		\begin{proof}[Proof of
			\Cref{formulas_small_lambda_hat_jn}]\label{proof_formulas_small_lambdahatjn}
			Recall
			\begin{align*}
				\hat{\lambda}_\indexone^{(2)} &= - \trace \big(B S_j B P_j \big), \\
				\hat{\lambda}_\indexone^{(3)} &= \trace \big(B S_j B S_j BP_j\big) - \trace\big( B S_j^2 B P_j BP_j\big),
			\end{align*}
			where $P_j$ and $S_j$ are the matrices
			\begin{equation*}
				P_j = e_j \otimes e_j \quad \text{ and } \quad S_j = \sum^N_{\substack{\indexfour=1 \\ \indexfour \neq \indexone}} (v_\indexfour-v_\indexone)^{-1} (e_\indexfour \otimes e_\indexfour).
			\end{equation*}
			For the following matrices, the column differing from all other columns is always the $j$-th one. We have
			\begin{gather*}
				BP_j = \begin{pmatrix}
					0 & \cdots & 0 & b_{1j} & 0 & \cdots & 0 \\
					\vdots & \vdots & \vdots & \vdots & \vdots & \vdots & \vdots \\
					0 & \cdots & 0 & b_{Nj} & 0 & \cdots & 0
				\end{pmatrix}, \\
				BS_j = \begin{pmatrix}
					(v_1-v_j)^{-1} b_{11} & \cdots & (v_{j-1}-v_j)^{-1} b_{1,j-1} & 0 & (v_{j+1}-v_j)^{-1} b_{1,j+1} & \cdots & (v_N-v_j)^{-1} b_{1N} \\
					\vdots & \vdots & \vdots & \vdots &\vdots &\vdots & \\
					(v_1-v_j)^{-1} b_{N1} & \cdots & (v_{j-1}-v_j)^{-1} b_{N,j-1} & 0 & (v_{j+1}-v_j)^{-1} b_{N,j+1} & \cdots & (v_N-v_j)^{-1} b_{NN}
				\end{pmatrix}, \\
				BS_j BP_j = \begin{pmatrix}
				0 & \cdots & 0 & \sum_{i\neq j} (v_i-v_j)^{-1}b_{1i}b_{ij} & 0 & \cdots & 0 \\
				\vdots & \vdots & \vdots & \vdots & \vdots & \vdots & \vdots \\
				0 & \cdots & 0 & \sum_{i\neq j} (v_i-v_j)^{-1}b_{Ni}b_{ij} & 0 & \cdots & 0
				\end{pmatrix}.
			\end{gather*}
			This implies
			\begin{equation*}
				\hat{\lambda}_\indexone^{(2)} = - \trace \big(B S_j B P_j \big) = - \big(BS_j BP_j\big)_{jj} = - \sum_{i\neq j} (v_i-v_j)^{-1}b_{ji}b_{ij}
			\end{equation*}
			and
			\begin{align}\label{auxiliary_eq43}
				\trace \big(B S_j (B S_j BP_j)\big) &= \big( BS_j BS_j BP_j \big)_{jj} = \sum_{l\neq j} (v_l-v_j)^{-1} b_{jl} \sum_{i\neq j} (v_i-v_j)^{-1} b_{li} b_{ij} \nonumber \\
				&= \sum_{i,l\neq j} (v_l-v_j)^{-1} (v_i-v_j)^{-1} b_{jl} b_{li} b_{ij}.
			\end{align}
			Concerning the computation of $\hat{\lambda}_\indexone^{(3)},$ we also need to compute
			\begin{gather*}
				(BP_j)^2 = \begin{pmatrix}
					0 & \cdots & 0 & b_{1j}b_{jj} & 0 & \cdots & 0 \\
					\vdots & \vdots & \vdots & \vdots & \vdots & \vdots & \vdots \\
					0 & \cdots & 0 & b_{Nj}b_{jj} & 0 & \cdots & 0
				\end{pmatrix}, \\
				BS_j^2 = \begin{pmatrix}
					(v_1-v_j)^{-2} b_{11} & \cdots & (v_{j-1}-v_j)^{-2} b_{1,j-1} & 0 & (v_{j+1}-v_j)^{-2} b_{1,j+1} & \cdots & (v_N-v_j)^{-2} b_{1N} \\
					\vdots & \vdots & \vdots & \vdots &\vdots &\vdots & \\
					(v_1-v_j)^{-2} b_{N1} & \cdots & (v_{j-1}-v_j)^{-2} b_{N,j-1} & 0 & (v_{j+1}-v_j)^{-2} b_{N,j+1} & \cdots & (v_N-v_j)^{-2} b_{NN}
				\end{pmatrix}.
			\end{gather*}
			We obtain
			\begin{align}\label{auxiliary_eq44}
				\trace\big( B S_j^2 B P_j BP_j\big) = \big( B S_j^2 B P_j BP_j \big)_{jj} = \sum_{i\neq j} (v_i - v_j)^{-2} b_{ji} b_{ij} b_{jj}
			\end{align}
			and finally, combining \eqref{auxiliary_eq43} and \eqref{auxiliary_eq44} gives
			\begin{equation*}
				\hat{\lambda}_\indexone^{(3)} = \sum_{i,l\neq j} (v_l-v_j)^{-1} (v_i-v_j)^{-1} b_{jl} b_{li} b_{ij} - \sum_{i\neq j} (v_i - v_j)^{-2} b_{ji} b_{ij} b_{jj}.
			\end{equation*}
		\end{proof}
		\renewcommand{\thechapter}{\Alph{chapter}}
		\setcounter{chapter}{1}
		
		\renewcommand{\thechapter}{\arabic{chapter}}
		\setcounter{chapter}{7}
		\setcounter{theorem}{13}
		\begin{lemma}
			Let $N=2.$ Assume that the transport directions satisfy \Cref{assumption_transport_dir}, i.e. $v_1 \neq v_2.$ Let $\lambda_1(k), \lambda_2(k)$ be the eigenvalues of
			\begin{equation*}
				M(\indextwo) = -2\pi i k V + B.
			\end{equation*} 
			Then $\real\lambda_j(k)$ is eventually constant if and only if $B$ fulfills (at least) one of the following three conditions:
			\begin{enumerate}
				\item
				$b_{12} = 0$,
				\item
				$b_{21} = 0$,
				\item
				$b_{11} = b_{22}$.
			\end{enumerate}
		\end{lemma}
	\begin{proof}[Proof of \Cref{eventually_constant_N2}]\label{proof_eventually_constant_N2}
		Wlog, consider $j=1.$ Let the real parts of $\lambda_1(k)$ be eventually constant, i.e. constant for large $|k|.$ In particular, $\hat{\lambda}_1^{(3)} = 0$ has to hold due to the proof of \Cref{theorem_eventually_monotone}. Given the formula for $\hat{\lambda}_1^{(3)}$ from \Cref{formulas_small_lambda_hat_jn} and $N=2$, we obtain
		\begin{align*}
			0 &= \hat{\lambda}_1^{(3)} = \sum_{\substack{i,l=1 \\ i,l\neq 1}}^{2} (v_l-v_1)^{-1}(v_i-v_1)^{-1}b_{1l}b_{li}b_{i1} - \sum_{\substack{i=1 \\ i\neq 1}}^{2} (v_i-v_1)^{-2}b_{i1}b_{1i}b_{11} \\
			&= (v_2-v_1)^{-2} b_{12}b_{22}b_{21} - (v_2-v_1)^{-2}b_{21}b_{12}b_{11} \\
			&= (v_2-v_1)^{-2} b_{12} b_{21} (b_{22}-b_{11}).
		\end{align*}
		Hence, one of the three conditions of $(i), \, (ii)$ and $(iii)$ has to be fulfilled. \newline
		
		\noindent Conversely, if $(i)$ holds true, i.e. $b_{12}=0$, the matrices $M(k)$ read
		\begin{equation*}
			M(k) = -2\pi i k \begin{pmatrix}
			v_1 & 0 \\
			0 & v_2
			\end{pmatrix} + \begin{pmatrix}
			b_{11} & 0 \\
			b_{21} & b_{22}
			\end{pmatrix}
		\end{equation*}
		and the eigenvalues of $M(k)$ are given by
		\begin{equation*}
			\lambda_j(k) = b_{jj} - 2\pi i k v_{j}
		\end{equation*}
		for $j=1,2.$ In particular, their real parts are constantly equal to $b_{jj}.$ In the case $(ii),$ $M(k)$ is an upper triangular matrix and the computation is the same as in $(i).$ \newline
		Given $b_{11} = b_{22} \eqqcolon b,$ the matrices $M(k)$ read
		\begin{equation*}
			M(k) = -2\pi i k \begin{pmatrix}
			v_1 & 0 \\
			0 & v_2
			\end{pmatrix} + \begin{pmatrix}
			b & b_{12} \\
			b_{21} & b
			\end{pmatrix}
		\end{equation*}
		with
		\begin{align*}
			\det(M(k)-\lambda I_{\C^{2\times 2}}) &= (b-2\pi i k v_1 - \lambda )(b-2\pi i k v_2 - \lambda ) - b_{12}b_{21} \\
			&= \lambda^2 - \big( 2b - 2\pi i k(v_1+v_2)\big)\lambda + (b-2\pi i k v_1)(b-2\pi i k v_2) - b_{12}b_{21}
		\end{align*}
		Consequently, the eigenvalues of $M(k)$ are given by
		\begin{align*}
			\lambda_k(k) &= b-\pi i k (v_1+v_2) \pm \sqrt{\big(b-\pi i k (v_1+v_2)\big)^2 - (b-2\pi i k v_1)(b-2\pi i k v_2) + b_{12}b_{21}} \\
			&= b-\pi i k (v_1+v_2) \pm \sqrt{\pi^2 k^2 \big(4v_1 v_2 - (v_1+v_2)^2\big) + b_{12}b_{21} } \\
			&= b-\pi i k (v_1+v_2) \pm \sqrt{-\pi^2 k^2 (v_1-v_2)^2 + b_{12}b_{21}}.
		\end{align*}
		For
		\begin{equation*}
			k^2 > \frac{b_{12}b_{21}}{\pi^2 (v_1-v_2)^2},
		\end{equation*}
		the term under the root is negative and $\real\lambda_j(k) = b.$
	\end{proof}
		\renewcommand{\thechapter}{\Alph{chapter}}
		\setcounter{chapter}{1}

		\chapter{$C_0$-Semigroups, Generators and Abstract Cauchy
			Problems}\label{appendix_semigroup_theory}
		Throughout this section, $(E, \| \cdot \|_E)$ is a complex Banach space and
		$\Omega \subseteq \R^d$ is an open domain. \newline
		
		In this chapter, we give a heuristic introduction of semigroups and explain
		central definitions, generator theorems and solution concepts. \newline
		
		Consider the partial differential equation
		\begin{equation}\label{auxeq29}
			\begin{cases}
			\begin{array}{rrll}
			\partial_t u(t,\mathbf{x}) + Au(t,\mathbf{x}) & = & F(t,\mathbf{x},u(t,\mathbf{x})) & \qquad (t,\mathbf{x})\in
			(0,\infty)\times \Omega, \\
			u(0,\mathbf{x}) & = & u_0(\mathbf{x}) & \qquad \mathbf{x}\in\Omega
			\end{array}
			\end{cases}
		\end{equation}
		for a spacial differential operator $-A$, e.g. the Laplacian $\Delta,$ reasonable boundary conditions on $\partial \Omega$ and a function $F\colon (0,\infty) \times \Omega \times \C \to \C$.\newline
		The basic idea in semigroup theory is to consider the mapping $t\mapsto
		u(t,\cdot) \in E$ for fixed initial value $u_0$ and some Banach space $(E,
		\|\cdot\|_E).$ Then, $u(t) \coloneqq u(t,\cdot) \in E$ is the solution at time
		$t$ and \eqref{auxeq29} is interpreted as an ordinary differential equation on
		$E,$ i.e. as
		\begin{equation}\label{abstract_cp}
			\begin{cases}
			\begin{array}{rrll}
			\dot{u}(t) + Au(t) & = & F(t,u(t)) & \qquad t>0, \\
			u(0) & = & u_0
			\end{array}
			\end{cases}
		\end{equation}
		for a densely defined operator $-A$ on $E$ with domain $D(-A)$  and a function $F\colon (0,\infty) \times E \to E.$ Equation \eqref{abstract_cp} is called
		abstract Cauchy problem and the family of operators $(T(t))_{t\geq0}$ on $E$,
		which maps an initial function $u_0$ to the solution for all
		$t\geq0$, is called the semigroup generated by $(-A, \, D(-A)).$ This family
		should clearly satisfy $T(0) = I_E,$ because the solution at time $t=0$ is the
		initial function itself. Secondly, taking $t,s>0$ and restarting
		\eqref{abstract_cp} after time $s$ with initial function $T(s)u_0$ should give
		the same solution as if we did not restart the equation. Thirdly, the solution
		map $t\mapsto T(t)u_0$ should fulfill a continuity property for every fixed
		initial function. These considerations motivate the definition of a strongly
		continuous semigroup.
		
		\begin{definition}[{\cite[Chapter 1, Definition 1.1, p. 1 and Definition 2.1,
				p. 4]{pazy2012semigroups}}]
			A one parameter family $(T(t))_{t\geq0}$ of bounded linear operators from $E$
			into $E$ is a strongly continuous semigroup if
			\begin{enumerate}
				\item
				$T(0) = I_E,$
				\item
				$T(t+s) = T(t)T(s) \quad \text{ for all } t,s>0,$
				\item
				$\displaystyle\lim_{t \searrow 0} T(t)u = u \quad \text{ for every } u\in E.$
			\end{enumerate}
		\end{definition}
		\begin{remark}
			Strongly continuous semigroups are often called $C_0$-semigroups and $(ii)$ is
			the so called semigroup property.
		\end{remark}
		
		\begin{definition}[{\cite[Chapter 1, Definition 1.1, p. 1]{pazy2012semigroups}}]\label{def_generator}
			Let $(T(t))_{t\geq0}$ be a strongly continuous semigroup. The linear operator
			$(-A, \, D(-A))$ defined by
			\begin{equation*}
			D(-A) = \bigg\{ u\in E \, \colon \, \lim_{t \searrow 0} \frac{T(t)u-u}{t}
			\text{ exists}\bigg\}
			\end{equation*}
			and
			\begin{equation*}
			-Au = \lim_{t \searrow 0} \frac{T(t)u-u}{t} \quad \text{ for } u\in D(-A)
			\end{equation*}
			is the infinitesimal generator of the semigroup, $D(-A)$ is the domain of
			$-A.$
		\end{definition}
		
		In applications, one almost always starts with an abstract Cauchy problem
		\eqref{abstract_cp} and well-posedness comes down to the question, whether $-A$
		generates a $C_0$-semigroup. The domain of the operator has to be chosen
		carefully and should take the boundary values of the PDE into account. With this
		in mind, criteria which characterize infinitesimal generators of
		$C_0$-semigroups are extremely useful and are given through the theorems of
		Hille-Yosida and Lumer-Phillips.
		
		\begin{theorem}[{\cite[Chapter 1, Theorem 2.2, p. 4]{pazy2012semigroups}}]
			Let $(T(t))_{t\geq0}$ be a $C_0$-semigroup. There exists constants
			$\omega\geq0$ and $M\geq1$ such that
			\begin{equation}\label{auxeq30}
			\|T(t)\|_{\mathcal{L}(E)} \leq M e^{\omega t} \quad \text{ for all } t\geq 0.
			\end{equation}
		\end{theorem}
		
		\begin{definition}[{\cite[Chapter 1, p. 8]{pazy2012semigroups}}]
			A strongly continuous semigroup $(T(t))_{t\geq0}$ which fulfills the estimate
			\eqref{auxeq30} with $\omega=0$ is called uniformly bounded and if moreover
			$M=1$, it is called a semigroup of contractions.
		\end{definition}
		
		\begin{theorem}[Hille-Yosida, {\cite[Chapter 1, Theorem 5.3, p. 20]{pazy2012semigroups}}]\label{hille-yosida}
			A linear operator $(-A, \, D(-A))$ is the infinitesimal generator of a
			$C_0$-semigroup $(T(t))_{t\geq0}$, satisfying $\|T(t)\|_{\mathcal{L}(E)} \leq M
			e^{\omega t},$ if and only if
			\begin{enumerate}
				\item
				$-A$ is closed and $D(-A)$ is dense in $E$,
				\item
				The resolvent set $\rho(-A)$ of $-A$ contains the ray $(\omega,\infty)$ and
				\begin{equation}\label{resolvent_est}
				\|R(\lambda, \, -A)^n \|_{\mathcal{L}(E)} \leq \frac{M}{(\lambda-\omega)^n}
				\quad \text{ for all } \lambda > \omega \text{ and } n\in\N,
				\end{equation}
				where $R(\lambda, \, -A)=(\lambda+A)^{-1}$ is the resolvent.
			\end{enumerate}
		\end{theorem}
		\begin{remark}
			This general version of the Hille-Yosida theorem is mainly of theoretical
			importance. In practice, the estimate \eqref{resolvent_est} is hard to check for
			all $n\in\N.$ For contraction semigroups, there is a version of the Hille-Yosida
			theorem which avoids this problem, see \cite[Chapter 1, Theorem 3.1, p. 8]{pazy2012semigroups}.
		\end{remark}
		
		\begin{definition}[{\cite[Chapter 1, Definition 4.1, p. 13 and Theorem 4.2,
				p. 14]{pazy2012semigroups}}]
			A linear operator $(-A, \, D(-A))$ is dissipative if
			\begin{equation*}
			\|(\lambda I_E + A)u \|_E \geq \lambda \|u\|_E \quad \text{ for all } u\in
			D(-A) \text{ and all } \lambda > 0.
			\end{equation*}
		\end{definition}
		
		\begin{theorem}[Lumer-Phillips{\cite[Chapter 1, Theorem 4.3, p. 14]{pazy2012semigroups}}]\label{lumer_phillips}
			Let $(-A, \, D(-A))$ be a linear operator with dense domain $D(-A)$ in $E.$
			\begin{enumerate}
				\item
				If $-A$ is dissipative and there is a $\lambda_0>0$ such that the range
				$\range(\lambda_0I_E+A)$ of $\lambda_0I_E+A$ is $E$, then $-A$ is the
				infinitesimal generator of a $C_0$-semigroup of contractions.
				\item
				Conversely, if $-A$ is the infinitesimal generator of a $C_0$-semigroup of
				contractions on $E,$ then $\range(\lambda I_E + A) = E$ for all $\lambda>0$ and $-A$ is dissipative.
			\end{enumerate}
		\end{theorem}
		
		After ensuring that $(-A, \, D(-A))$ generates a strongly continuous semigroup,
		the abstract Cauchy Problem \eqref{abstract_cp} can be ``solved". There are
		specific solution concepts for semigroup theory, since \eqref{auxeq29} was
		reformulated into a Banach space setting and we state the basic concepts and
		results. \newline
		\indent	Let $T>0,$ possibly $T=\infty,$ and let $I=[0,T) \subseteq \R_{\geq0}	$
		be an interval. For simplicity, we start with the cases $F\equiv 0$ and
		$F(t,u)=F(t)$ in order to obtain a feeling for abstract Cauchy problems. In the
		homogeneous case, the abstract Cauchy problem \eqref{abstract_cp} on $I=\R_{\geq0}$ reads
		\begin{equation}\label{abstract_cp_homo}
			\begin{cases}
			\begin{array}{rrll}
			\dot{u}(t) + Au(t) & = & 0 & \qquad t>0, \\
			u(0) & = & u_0
			\end{array}
			\end{cases}
		\end{equation}
		and the next theorem validates the concept of semigroups.
		
		\begin{theorem}[{\cite[Chapter 4, Theorem 1.3, p. 102]{pazy2012semigroups}}]\label{existence_homogeneous_eq}
			Let $(-A, \, D(-A))$ be a densely defined linear operator with non-empty
			resolvent set $\rho(-A).$ The initial value problem \eqref{abstract_cp_homo} has
			a unique solution $u \in C^1([0,\infty), \, E)$ for every $u_0 \in D(-A)$ if and
			only if $-A$ is the infinitesimal generator of a $C_0$-semigroup
			$(T(t))_{t\geq0}.$ In this case, the solution is given by $u(t)=T(t)u_0.$
		\end{theorem}
		
		Hereafter, we will assume that $(-A, \, D(-A))$ generates the $C_0$-semigroup
		$(T(t))_{t\geq0}$ so that the homogeneous equation \eqref{abstract_cp_homo} has
		a unique solution for $u_0\in D(-A)$. \newline
		The inhomogeneous abstract Cauchy problem on $I=[0,T)$ reads
		\begin{equation}\label{abstract_cp_linear}
		\begin{cases}
		\begin{array}{rrll}
		\dot{u}(t) + Au(t) & = & F(t) & \qquad 0 < t < T, \\
		u(0) & = & u_0
		\end{array}
		\end{cases}.
		\end{equation}	
		
		\begin{definition}[{\cite[Chapter 4, Definition 2.1, p. 105]{pazy2012semigroups}}]\label{classical_solution}
			A function $u\colon I \to E$ is a classical solution on $I$ to the abstract Cauchy problem \eqref{abstract_cp} if
			\begin{enumerate}
				\item
				$u \in C(I, \, E) \cap C^1((0,T), \, E),$
				\item
				$u(t) \in D(-A) $ for all $ t \in (0,T)$ and \eqref{abstract_cp} is satisfied
				on $I$.
			\end{enumerate}
		\end{definition}
		
		\begin{definition}[{\cite[Chapter 4, Definition 2.3, p. 106]{pazy2012semigroups}}]
			Let $u\in E$ and $F\in L^1((0,T), \, E).$ The function $u\in C([0,T], \, E)$
			given by
			\begin{equation}\label{mild_solution}
			u(t) = T(t)u + \int_0^t T(t-s)F(s) \, ds, \qquad t \in [0,T],
			\end{equation}
			is the mild solution to the initial value problem \eqref{abstract_cp_linear}
			on $[0,T].$
		\end{definition}
		\begin{remark}[{\cite[Chapter 4, Corollary 2.2, p. 106]{pazy2012semigroups}}]
			Under the assumption $F\in L^1((0,T), \, E)$, a classical solution coincides
			with the mild solution and \eqref{abstract_cp_linear} has at most one solution.
			If \eqref{abstract_cp_linear} has a solution, it is given by
			\eqref{mild_solution}.
		\end{remark}
		
		\begin{theorem}[{\cite[Chapter 4, Corollary 2.5, p. 107 and Corollary 2.6,
				p. 108]{pazy2012semigroups}}]\label{classical_sol_linear}
			The initial value problem \eqref{abstract_cp_linear} has a unique classical
			solution on $I$ for every $u_0\in D(-A),$ if $F$ satisfies one of the following
			two properties:
			\begin{enumerate}
				\item
				$F\in C^1([0,T], \, E),$
				\item
				$F \in C((0,T), \, E) \cap L^1((0,T), \, E) \quad \text{ and } \quad F(s) \in
				D(-A) \text{ for all } s\in (0,T) \text{ with } AF(\cdot) \in L^1((0,T), \, E).$
			\end{enumerate}
		\end{theorem}
		
		We now summarize the main results for the semilinear initial value problem
		\begin{equation}\label{abstract_cp_semilinear}
			\begin{cases}
			\begin{array}{rrll}
			\dot{u}(t) + Au(t) & = & F(t,u(t)) & \qquad 0 < t < T, \\
			u(0) & = & u_0
			\end{array}
			\end{cases}.
		\end{equation}
		
		The definition of classical solutions for \eqref{abstract_cp_semilinear} is
		the same as before, see \Cref{classical_solution}. Similar to the linear
		inhomogeneous case, classical solutions are mild solutions in the sense of the
		next definition.
		
		\begin{definition}[{\cite[Chapter 6, Definition 1.1, p. 184]{pazy2012semigroups}}]
			A function $u\in C(I, \, E)$ which solves the integral equation
			\begin{equation*}
			u(t) = T(t)u_0 + \int_0^t T(t-s)F(s, u(s)) \,ds, \qquad t \in I,
			\end{equation*}
			is called a mild solution to the initial value problem
			\eqref{abstract_cp_semilinear} on $I$ .
		\end{definition}
		
		A typical Lipschitz assumption on $F$ ensures local existence of mild
		solutions.
		
		\begin{definition}[{\cite[Chapter 6, p. 185]{pazy2012semigroups}}]\label{lipschitz_assumption}
			A function $F\colon [0,\infty)\times E \to E$ is called locally Lipschitz
			continuous in $u$, uniformly in $t$ on bounded intervals, if for every $t'\geq0$
			and $n\in\N$, there is a constant $C(t',n)$ such that
			\begin{equation*}
			\| F(t,u) - F(t,\widetilde{u}) \|_E \leq C(t',n) \| u -\widetilde{u} \|_E
			\end{equation*}
			holds for all $u, \widetilde{u} \in E$ with $\|u\|_E, \, \|\widetilde{u} \|_E
			\leq n$ and $t\in [0,t'].$
		\end{definition}
		
		\begin{theorem}[{\cite[Chapter 6, Theorem 1.4, p. 185]{pazy2012semigroups}}]\label{theorem_mildsol}
			Let $F\colon [0,\infty)\times E \to E$ be continuous in $t$ and locally
			Lipschitz continuous in $u$, uniformly in $t$ on bounded intervals. Then for
			every $u_0\in E,$ there is a time $T_{\text{max}} = T_{\text{max}}(\|u_0\|_E) >
			0,$ possibly $T_{\text{max}}=T,$ such that the initial value problem
			\eqref{abstract_cp_semilinear} has a unique mild solution $u$ on
			$[0,T_{\text{max}}).$ Moreover, if $T_{\text{max}} < T,$ then
			\begin{equation*}
			\lim_{t \nearrow T_{\text{max}}} \| u(t) \|_E = \infty.
			\end{equation*}
		\end{theorem}
		\begin{remarks}\label{remark_global_mild_sol}
			\begin{enumerate}
				\item
				The mild solution on $[0,T_{\text{max}})$ is often called maximal mild
				solution.
				\item
				If for each $t'>0,$ there exists a constant $C(t') >0$ such that
				\begin{equation*}
				\|F(t,u)\|_E \leq C (1+\|u\|_E)
				\end{equation*}
				holds for all $t\in [0,t']$ and all $u\in E$, then
				\eqref{abstract_cp_semilinear} has a mild solution on $[0,\infty),$ i.e. one can
				choose $T=\infty$ and it holds $T_{\text{max}}=T.$
			\end{enumerate}
		\end{remarks}
		
		There are multiple criteria which ensure that a mild solution is a classical
		solution (see \cite[Chapter 6, p. 187ff.]{pazy2012semigroups}) and they
		typically involve stronger regularity of the right-hand side. We endow the
		domain $D(-A)$ with the graph norm $\|u\|_{-A} = \|u\|_E + \|Au\|_E,$ which
		turns \mbox{$(D(-A),\|\cdot\|_{-A})$} to a Banach space due to the closedness of
		$-A,$ see \Cref{hille-yosida}. Then, for example, \Cref{classical_sol_linear}
		and \Cref{theorem_mildsol} imply the next statement.
		
		\begin{theorem}\label{semilinear_strong_sol}
			Let $F\in C([0,\infty)\times E, \, D(-A))$ be uniformly (in t) locally
			Lipschitz continuous in $u$. Then for every $u_0\in D(-A),$ the initial value
			problem \eqref{abstract_cp_semilinear} possesses a unique maximal strong solution
			\begin{equation*}
			u \in C([0,T_{\text{max}}), \, D(-A)) \cap C^1((0,T_{\text{max}}), \, E).
			\end{equation*}
		\end{theorem}
		
		We finish our overview with a result concerning the continuous dependence of
		the solution on the initial data.
		
		\begin{corollary}[{\cite[Chapter 6, Theorem 1.2, p. 184]{pazy2012semigroups}}]\label{semilinear_lipschitz_dependence}
			Let $F\in C([0,\infty)\times E, \, E)$ be globally Lipschitz continuous in
			$u.$  Then the mild solution depends Lipschitz continuously on $u_0\in E,$ i.e.
			for all $T>0,$ the mapping $u_0\mapsto u$ is Lipschitz continuous from $E$ into
			$C([0,T], \, E)$.
		\end{corollary}
		
		\chapter{Perturbation and Approximation of
			$C_0$-Semigroups}\label{appendix_perturbation_theory}
		
		We assume throughout this chapter, similar to the previous one, that $(E, \,
		\|\cdot\|_E)$ is a complex Banach space. \newline
		
		\indent	Many times in applications, the verification of the generator theorems
		from Appendix \ref{appendix_semigroup_theory} is a difficult task and cannot be
		performed in a direct way. Therefore, one tries to build up the given operator
		and its semigroup from simpler ones. Perturbation and approximation are the
		standard methods for this approach \cite[Chapter 3, Introduction, p. 157]{engel2001one}. \newline
		We present the main results for the perturbation of an infinitesimal generator
		$(-A, \, D(-A))$ of a semigroup $(T(t))_{t\geq0}$ with a bounded linear operator
		$B \in \mathcal{L}(E).$ \newline
		
		The following theorem summarizes multiple results from
		\cite{engel2001one,pazy2012semigroups}, namely \cite[\Romannum{3}. Theorem 1.3,
		p. 158]{engel2001one} and \cite[\Romannum{3}. Corollary 1.7, p. 161]{engel2001one}. The uniqueness claim concerning properties $(ii)$ and
		$(iii)$ is proven in \cite[Chapter 3, Proposition 1.2, p. 77]{pazy2012semigroups} and a proof of the (in our case obvious) equivalence of
		the graph norms can be found in \cite[\Romannum{3}. Lemma 2.4, p. 171]{engel2001one} for a more general framework.
		
		\begin{theorem}\label{perturbation_theory_-A+B_gen_integral_identities}
			Let $(-A, \, D(-A))$ generate the $C_0$-semigroup $(T(t))_{t\geq0}$ satisfying $\| T(t) \|_{\mathcal{L}(E)} \leq M e^{\omega t}$ for all $t\geq0$ and constants $M\geq
			1$ and $\omega \in \R.$ Let $B \in \mathcal{L}(E).$ Then the sum $-A+B$ with
			$D(-A+B)=D(-A)$ generates a $C_0$-semigroup $(R(t))_{t\geq0}$ which fulfills
			\begin{enumerate}
				\item
				$\|R(t)\|_{\mathcal{L}(E)} \leq M e^{(\omega+M\|B\|_{\mathcal{L}(E)})t},$
				\item
				$R(t)u = T(t)u + \int_{0}^{t} T(t-s)BR(s)u \, ds,$
				\item
				$R(t)u = T(t)u + \int_{0}^{t} R(t-s)BT(s)u \, ds$
			\end{enumerate}
			for all $t\geq0$ and $u\in E.$ The operator family $(R(t))_{t\geq0}$ is the
			only strongly continuous family of operators solving $(ii)$ and/or $(iii).$ The
			graph norms of $-A$ and $-A+B$ on $D(-A)$ are equivalent.
		\end{theorem}
		
		Aside from the implicit representation in
		\Cref{perturbation_theory_-A+B_gen_integral_identities}, there are other, more
		useful formulas for the semigroup $(R(t))_{t\geq0}.$
		
		\begin{theorem}[Dyson-Phillips series, {\cite[\Romannum{3}. Theorem 1.10, p. 163]{engel2001one}}]\label{perturbation_theory_-A+B_dyson_phillips_series}
			Let $(-A, \, D(-A))$ generate the $C_0$-semigroup $(T(t))_{t\geq0}$ and let
			$B\in\mathcal{L}(E).$ The semigroup $(R(t))_{t\geq0}$ generated by $-A+B$ with
			$D(-A+B)=D(-A)$ satisfies
			\begin{equation}\label{perturbation_theory_Dyson_Phillips_series_formula}
			R(t) = \sum_{k=0}^{\infty} R_k(t)
			\end{equation}
			with
			\begin{equation*}
			R_0(t) = T(t), \quad R_{k+1}(t)u = \int_{0}^{t} T(t-s)BR_k(t)u \, ds
			\end{equation*}
			for all $t\geq0$ and $u\in E.$ The Dyson-Phillips series
			\eqref{perturbation_theory_Dyson_Phillips_series_formula} converges in
			$\mathcal{L}(E)$ and uniformly on compact subsets of $\R_{\geq0}.$
		\end{theorem}

		\begin{theorem}[Lie-Trotter product forumla, {\cite[\Romannum{3}. Corollary
				5.8, p. 227]{engel2001one}}]\label{perturbation_theory_Lie_Trotter_formula}
			Assume that $(-A, \,D(-A))$ and $(B, \, D(B))$ generate the $C_0$-semigroups
			$(T(t))_{t\geq 0}$ and $(S(t))_{t\geq0}$ respectively, subject to the stability
			bound
			\begin{equation*}
			\| (T(\tfrac{t}{n}) S(\tfrac{t}{n}))^n \|_{\mathcal{L}(E)} \leq M e^{\omega
				t}
			\end{equation*}
			for all $n\in\N, \, t\geq0$ and some $M\geq1$ and $\omega\in\R.$ Let
			$D\coloneqq D(-A) \cap D(B)$ and $(\lambda I_E- (-A+B))D$ be dense in $E$ for
			some $\lambda > \omega.$ Then the sum $-A+B$ with domain $D(-A+B)\coloneqq D$
			has a closure $\overline{-A+B}$ which generates the $C_0$-semigroup
			$(R(t))_{t\geq0}$ given by
			\begin{equation*}
			R(t)u = \lim_{n \to \infty} \left(T(\tfrac{t}{n}) S(\tfrac{t}{n})\right)^nu
			\end{equation*}
			for all $t\geq0$ and $u\in E.$ For every fixed $u\in E,$ the convergence is
			uniform on compact subsets of $\R_{\geq0}.$
		\end{theorem}
		
		\chapter{Spectral Theory and Long-Term Behavior of $C_0$-Semigroups}\label{appendix_Long_time_behavior_SG}
		
		Again, $(E, \, \|\cdot\|_E)$ is assumed to be a complex Banach space. \newline
		
		The purpose of this chapter is to give an overview of the main theory we used
		when studying the spectral properties and the long-term behavior of the
		transport-reaction semigroups from \Cref{chapter_transport_reaction_SG}.
		\newline
		It is necessary to recall definitions from spectral theory before presenting
		statements on the asymptotics of strongly continuous semigroups as many results
		characterize asymptotical properties in terms of spectral properties of the
		semigroup.\newline
		As a guideline, the goal is often to deduce stability or more complex long-term
		behavior of the semigroup $(T(t))_{t\geq0}$ from spectral properties of its
		generator $(-A, \, D(-A))$. To this end, it is necessary to study the connection
		of the spectrum of the generator and the spectrum of the semigroup. This
		connection should intuitively be given by
		\begin{equation}\label{intuition_smthm}
			``\quad\sigma(T(t)) = e^{t\sigma(-A)}\quad",
		\end{equation}
		but generally, this formula is incorrect. Especially for hyperbolic equations,
		proving a connection in the fashion of \eqref{intuition_smthm} is non-trivial.
		\newline
		
		In infinite dimensional spaces, a linear operator is no longer injective if and
		only if it is surjective. Therefore, one needs a more refined definition for the
		spectrum than in finite dimensions, where the spectrum of a matrix is simply
		given by the set of its eigenvalues. The next definition and its remark unite
		\cite[\Romannum{4}. Definition 1.1, p. 239]{engel2001one}, \cite[\Romannum{4}.
		Definition 1.6, p. 241]{engel2001one}, \cite[\Romannum{4}. Definition 1.8 and
		Lemma 1.9, p. 242]{engel2001one} and \cite[\Romannum{4}. Definition 1.11, p. 243]{engel2001one}.
		
		\begin{definition}\label{long_time_behavior_definition_spectrum}
			Let $(-A, \, D(-A))$ be a closed operator. We define
			\begin{enumerate}
				\item
				the spectrum $\sigma(-A) = \left\{ \lambda \in \C \, \colon \, \lambda I_E +
				A \text{ is not bijective} \right\},$
				\item
				the resolvent set $\rho(-A) = \sigma(-A)^c,$
				\item
				the point spectrum $\sigma_p(-A) = \left\{ \lambda \in \C \, \colon \,
				\lambda I_E + A \text{ is not injective} \right\},$
				\item
				the approximate point spectrum
				\begin{align*}
				\sigma_{ap}(-A) &= \{ \lambda \in \C \, \colon \, \lambda I_E+A \text{ is
					not injective} \, \text{ or } \, R(\lambda+A) \text{ is not closed in } E \} \\
				&= \{ \lambda \in \C \, \colon \, \exists (u^{(n)})_{n\in\N} \subset D(-A)
				\colon \| u^{(n)} \|_E = 1, \, \lambda u^{(n)} + Au^{(n)} \to 0 \text{ as }
				n\to\infty \},
				\end{align*}
				\item
				the residual spectrum $\sigma_r(-A) = \{ \lambda \in \C \, \colon \, (\lambda
				I_E + A)D(-A) \text{ is not dense in E} \}.$
			\end{enumerate}
		\end{definition}
		\begin{remarks}
			\begin{enumerate}
				\item
				For $\lambda \in \rho(-A), $ the inverse
				\begin{equation*}
				R(\lambda, -A) \coloneqq (\lambda+A)^{-1}
				\end{equation*}
				is, by the closed graph theorem, a bounded linear operator on $E$ and is
				called the resolvent (of $-A$ at the point $\lambda$).
				\item
				Each $\lambda \in \sigma_p(-A)$ is called an eigenvalue and each $0\neq u \in
				D(-A)$ satisfying $(\lambda+A)u = 0$ is an eigenvector of $-A$ (corresponding to
				$\lambda$).
			\end{enumerate}
		\end{remarks}
		
		\begin{definition}[{\cite[\Romannum{4}. Corollary 1.4, p. 241]{engel2001one}}]\label{long_time_behavior_definition_spectral_radius}
			Let $T \in \mathcal{L}(E).$ The quantity
			\begin{equation*}
			r(T) = \max \left\{ |\lambda| \, \colon \, \lambda \in \sigma(T) \right\}
			\end{equation*}
			is called the spectral radius of $T.$
		\end{definition}
		
		\begin{definition}[{\cite[\Romannum{1}. Definition 5.6, p. 40 and
				\Romannum{4}. Definition 2.1, p. 250]{engel2001one}}]\label{long_time_behavior_definition_exponential_growth_bound_spectral_bound}
			Let $\mathcal{T} \coloneqq (T(t))_{t\geq0}$ be a $C_0$-semigroup on $E$ with
			generator $(-A, \, D(-A)).$ The quantity 
			\begin{equation*}
				\omega_0 = \omega_0(\mathcal{T}) = \inf \left\{ \omega \in \R \, \colon \,
				\exists M_\omega \geq 1 \, \forall t \geq 0 \colon \|T(t)\|_{\mathcal{L}(E)}
				\leq M_\omega e^{\omega t} \right\}
			\end{equation*}
			is called its growth bound and
			\begin{equation*}
				s(-A) = \sup \left\{ \real \lambda \, \colon \, \lambda \in \sigma(-A)
				\right\}
			\end{equation*}
			is called the spectral bound of its generator.
		\end{definition}
		\begin{remark}
			Another common notation for the growth bound of a semigroup $(T(t))_{t\geq0}$ with generator $(-A, \, D(-A))$ is $\omega_0(-A).$
		\end{remark}
		
		The spectral radius of the semigroup operators, the growth bound of the
		semigroup and the spectral bound of the generator are closely connected via the
		next proposition.
		
		\begin{proposition}[{\cite[\Romannum{4}. Proposition 2.2, p. 251]{engel2001one}}]\label{exp_stab_proposition}
			Let $\mathcal{T} = (T(t))_{t\geq 0}$ be a semigroup with generator $(-A, \,
			D(-A)).$ One has
			\begin{equation*}
			-\infty \leq s(-A) \leq \omega_0(\mathcal{T}) = \inf_{t>0} \frac{1}{t} \log
			\| T(t) \|_{\mathcal{L}(E)}= \lim_{t \to \infty} \frac{1}{t} \log
			\|T(t)\|_{\mathcal{L}(E)}= \frac{1}{t_0} \log r(T(t_0)) < \infty
			\end{equation*}
			for all $t_0>0.$ In particular, the spectral radius of the operator $T(t)$ is
			given by
			\begin{equation*}
			r(T(t)) = e^{\omega_0 t} \qquad \text{ for all } t\geq 0.
			\end{equation*}
		\end{proposition}
		
		\begin{definition}[{equivalent to \cite[\Romannum{5}. Definition 1.1 (a), p. 296]{engel2001one}}]\label{long_time_behavior_definition_exponentially_stable}
			A $C_0$-semigroup $(T(t))_{t\geq0}$ is called uniformly exponentially stable
			if there exist constants $M, \, \eps >0$ such that
			\begin{equation*}
			\| T(t) \|_{\mathcal{L}(E)} \leq M e^{-\eps t}
			\end{equation*}
			holds for all $t\geq0.$
		\end{definition}
		
		An application of \Cref{exp_stab_proposition} yields the following
		characterizations of exponential stability.
		
		\begin{proposition}[{\cite[\Romannum{5}. Proposition 1.7, p. 299]{engel2001one}}]\label{long_time_behavior_exponential_stability_equivalences}
			Let $\mathcal{T} = (T(t))_{t\geq0}$ be a $C_0$-semigroup with generator $(-A,
			\, D(-A)).$ Then the following assertions are equivalent:
			\begin{enumerate}
				\item
				$(T(t))_{t\geq0}$ is uniformly exponentially stable.
				\item
				$\exists t_0>0$ such that $\|T(t_0) \|_{\mathcal{L}(E)} < 1.$
				\item
				$\exists t_1>0$ such that $r(T(t_1)) < 1.$
				\item
				$\omega_0(\mathcal{T}) < 0.$
			\end{enumerate}
			If this is the case, $(ii)$ is valid for all sufficiently large $t_0>0,$
			assertion $(iii)$ is true for all $t_1 > 0$ and we have $s(-A) < 0.$
		\end{proposition}
		
		The above given criteria for uniform exponential stability have a huge
		disadvantage: they rely on properties of the semigroup itself. Typically in
		applications, only the generator is known. It is therefore desirable to
		characterize uniform exponential stability directly in terms of the generator
		\cite[\Romannum{5}. Introduction, p. 301]{engel2001one}. \newline
		One would hope that the inequality $s(-A) \leq \omega_0$ from
		\Cref{exp_stab_proposition} is actually an equality. This would imply that
		\begin{equation*}
		s(-A) < 0
		\end{equation*}
		is equivalent to uniform exponential stability. Unfortunately, this is generally incorrect
		and the situation is non-trivial. There are (even positive) semigroups
		$(T(t))_{t\geq0}$ with $s(-A)<0$ and $\omega_0 \geq 0;$ see for example
		\cite[\Romannum{4}. Counterexample 2.7, p. 254]{engel2001one}. The equality
		can also fail on less artificial spaces and there are counterexamples	on
		reflexive Banach spaces or even Hilbert spaces, cf. \cite[\Romannum{4}.
		Counterexamples 3.3 and 3.4, p. 271-275]{engel2001one}. \newline
		
		The so called weak spectral mapping theorem
		\begin{equation}\label{wsmthm}
			\sigma(T(t)) \cup \{0\} = \overline{e^{t\sigma(-A)}} \cup \{0\} \quad \text{
			for all } t\geq 0 \tag{WSMT}
		\end{equation}
		is one possibility to get around the problem.
		
		\begin{lemma}[{\cite[\Romannum{5}. Lemma 1.9, p. 301]{engel2001one}}]\label{wsmthm_implies_growth_bd_eq_spectral_bound}
			If for a strongly continuous semigroup $(T(t))_{t\geq0}$ and its generator
			$(-A, \, D(-A))$ the weak spectral mapping theorem \eqref{wsmthm} holds, then
			the growth bound $\omega_0$ and the spectral bound $s(-A)$ coincide, i.e.
			\begin{equation*}
			s(-A) = \omega_0.
			\end{equation*}
			In this case, the semigroup is uniformly exponentially stable if and only if
			$s(-A)<0.$
		\end{lemma}
		
		The question remaining is which assumptions on the semigroup $(T(t))_{t\geq0}$ ensure the existence of a weak spectral mapping theorem. The next result shows that neither the point, nor the residual spectrum,
		cause any problems.
		
		\begin{theorem}[{\cite[\Romannum{4}. Theorem 3.7, p. 277]{engel2001one}}]\label{spectral_theorem_point_spec}
			For a generator $(-A, \, D(-A))$ of a strongly continuous semigroup
			$(T(t))_{t\geq0}$, we have the identities
			\begin{align*}
			\sigma_p(T(t)) \backslash \{0\} &= e^{t\sigma_p(-A)}, \\
			\sigma_r(T(t)) \backslash \{0\} &= e^{t\sigma_r(-A)}
			\end{align*}
			for all $t\geq0.$
		\end{theorem}
		
		Moreover, there is the spectral inclusion theorem.
		
		\begin{theorem}[Spectral Inclusion Theorem, {\cite[\Romannum{4}. Theorem 3.6, p. 276]{engel2001one}}]\label{spectral_inclusion_theorem}
			For the generator $(-A, \, D(-A))$ of a strongly continuous semigroup
			$(T(t))_{t\geq0}$, we have the inclusions
			\begin{align*}
				\sigma(T(t)) &\supseteq e^{t\sigma(-A)} \\
				\sigma_{ap}(T(t)) &\supseteq e^{t\sigma_{ap}(-A)}
			\end{align*}
			for all $t\geq0.$
		\end{theorem}
		
		Special classes of semigroups fulfill a stronger version of \eqref{wsmthm}, the
		spectral mapping theorem. A prominent example is the class of analytic
		semigroups, which appears in the study of parabolic equations; see
		\cite[\Romannum{2}. Definition 4.5, p. 101]{engel2001one} for their
		definition.
		
		\begin{corollary}[{\cite[\Romannum{4}. Corollary 3.12, p. 281]{engel2001one}}]\label{smtm}
			The spectral mapping theorem
			\begin{equation}
				\sigma(T(t)) \backslash \{0\} = e^{t\sigma(-A)} \quad \text{
					for all } t\geq 0  \tag{SMT}
			\end{equation}
			holds for the following class of strongly continuous semigroups:
			\begin{enumerate}
				\item
				eventually compact semigroups,
				\item
				eventually differentiable semigroups,
				\item
				analytic semigroups,
				\item
				uniformly continuous semigroups.
			\end{enumerate}
		\end{corollary}
		\begin{remark}
			For the definitions of eventually compact, eventually differentiable and
			uniformly continuous semigroups, we refer to \cite[\Romannum{2}. Definition
			4.23, p. 117]{engel2001one}, \cite[\Romannum{2}. Definition 4.13, p. 109]{engel2001one} and \cite[\Romannum{1}. Definition 3.6, p. 16]{engel2001one} respectively.
		\end{remark}
		
		Unfortunately, the transport-reaction semigroups appearing when studying the
		linearized version of the hyperbolic equation
		\begin{equation}\label{transport-reaction-pde}
		\partial_t \colvec{u_1 \\ \vdots \\ u_N} + \colvec{\mathbf{v_1} \cdot \nabla u_1 \\ \vdots \\ \mathbf{v_N} \cdot \nabla u_N} = F(u)
		\end{equation}
		on the $d$-dimensional torus $\T^d$ with $N\in\N$ population subgroups, transport directions
		$\mathbf{v_1}, \, \cdots, \mathbf{v_N} \in \R^d$ and a sufficiently smooth
		reaction function $F$ are less regular than all classes of semigroups from
		\Cref{smtm}. Having said this, they can be expanded to a strongly continuous
		group $(T(t))_{t\in\R}$ and for these groups, there is the following result.
		
		\begin{theorem}[{\cite[\Romannum{4}. Theorem 3.16 and Exercise 3.22, p. 283ff.]{engel2001one}}]\label{theorem_wsmt_bounded_grop}
			Let $(T(t))_{t\in\R}$ be a polynomially bounded stongly continuous group with
			generator $(-A, \, D(-A))$. Then the weak spectral mapping theorem
			\begin{equation*}
				\sigma(T(t)) = \overline{e^{t\sigma(-A)}}
			\end{equation*}
			for $t\in\R$ holds.
		\end{theorem}
		\begin{remark}
			Polynomially bounded means that there is a polynomial $P\colon \R \to \R$ with
			\begin{equation*}
				\|T(t)\|_{\mathcal{L}(E)} \leq P(t) \quad \text{ for all } t\in\R.
			\end{equation*}
		\end{remark}
		
		Without any reactions, i.e. the homogeneous case of
		\eqref{transport-reaction-pde}, the generated transport group is bounded and
		therefore, \Cref{theorem_wsmt_bounded_grop} applies. This is, to some extent,
		the motivation for a weak spectral mapping theorem for the linearized version of
		\eqref{transport-reaction-pde}. It should be noted that the transport-reaction
		groups generated in this case are almost never polynomially bounded. Depending
		on the reactions, some components of the solution grow exponentially, either for
		$t\to \infty$ or for $t\to -\infty.$

		\nomenclature{$\T^d$}{$d$-dimensional torus with side length $1$}
		\nomenclature{$\Ind{U}$}{characteristic function of the set $U$}
		\nomenclature{$\lVert\,{\cdot}\,\rVert_E$}{norm on the Banach space $E$}
		\nomenclature{$\lVert\,{\cdot}\,\rVert_{-A}$}{graph norm of $-A$}
		\nomenclature{$\langle\cdot, \, \cdot \rangle_H$}{inner product on the Hilbert
			space $H$}
		\nomenclature{$\langle\cdot, \, \cdot \rangle$}{canonical bilinear form on
			$E\times E^*$ for the Banach space $E$ and its dual space $E^*$}
		\nomenclature{$-A^*$}{adjoint of the linear operator $-A$}
		\nomenclature{$\overline{-A}$}{closure of $-A$}
		\nomenclature{$\restr{-A}{Y}$}{part of $-A$ in $Y$}
		\nomenclature{$C^\infty(\Omega)$}{space of infinitely many times differentiable
			functions}
		\nomenclature{$C^k(\Omega)$}{space of $k$-times differentiable functions,
			$k\in\N$}
		\nomenclature{$C(\Omega)$}{space of continuous functions}
		\nomenclature{$C_c(\Omega)$}{space of continuous functions having compact
			support}
		\nomenclature{$D(-A)$}{domain of $-A$}
		\nomenclature{$\mathcal{F}$}{Fourier transform}
		\nomenclature{$\hat{u}(\indextwo)$}{$\indextwo$-th Fourier coefficient of the
			periodic function $u$ with $\indextwo\in\Z^d$}
		\nomenclature{$\range(-A)$}{range of $-A$}
		\nomenclature{$H^\indextwo(\Omega)$}{Sobolev space of order $(k,2)$}
		\nomenclature{$L^p(\Omega)$}{space of $p$-integrable functions}
		\nomenclature{$\mathcal{L}(E)$}{space of bounded linear operators on $E$}
		\nomenclature{$\mathcal{A}_q$}{matrix multiplication operator associated to
			$q$}
		\nomenclature{$\omega_0(\mathcal{T})$}{growth bound of the semigroup
			$\mathcal{T}=(T(t))_{t\geq0}$}
		\nomenclature{$\sigma(-A)$}{spectrum of $-A$}
		\nomenclature{$\sigma_p(-A)$}{point spectrum of $-A$}
		\nomenclature{$\sigma_{ap}(-A)$}{approximate point spectrum of $-A$}
		\nomenclature{$\sigma_r(-A)$}{residual spectrum of $-A$}
		\nomenclature{$\rho(-A)$}{resolvent set of $-A$}
		\nomenclature{$r(T)$}{spectral radius of $T$}
		\nomenclature{$s(-A)$}{spectral bound of $-A$}
		\nomenclature{$R(\lambda, -A)$}{resolvent of $-A$ at the point $\lambda$}
		\nomenclature{$(T(t))_{t\geq0}$}{semigroup of bounded linear operators}
		\nomenclature{$(T(t))_{t\in\R}$}{group of bounded linear operators}
		\nomenclature{$W^{\indextwo,p}(\Omega)$}{Sobolev space of order $(k,p)$}
		\nomenclature{$(E, \, \lVert\,{\cdot}\,\rVert_E)$}{Banach space}
		\nomenclature{$(X, \, \Sigma, \, \mu)$}{$\sigma$-finite measure space}
		\nomenclature{$\lesssim$}{less or equal up to a constant}
		
		\printnomenclature[1in]
		
	\end{appendices}

	\addcontentsline{toc}{chapter}{Bibliography}

	\bibliography{mather_thesis.bbl}
	
\end{document}